\documentclass[12pt]{report}   
\usepackage{graphicx}  
\usepackage[letterpaper, left=1.5in, right=1in, top=1in, bottom=1in]{geometry}
\usepackage{setspace}  
\usepackage{times}  
\usepackage[explicit]{titlesec}  
\usepackage[titles]{tocloft}  
\usepackage[bookmarks=true, hidelinks]{hyperref}
\usepackage[page]{appendix}  
\usepackage{rotating}  
\usepackage[normalem]{ulem}  
\usepackage{amsfonts}
\usepackage{setspace}
\usepackage{pgf,tikz}
\usepackage{fullpage}
\usepackage{anysize}
\usepackage{float}
\usepackage{url}
\usepackage{bm}

\usepackage{amssymb}
\usepackage{latexsym}
\usepackage{amsmath}
\usepackage{mathrsfs}
\usepackage{amsthm, amssymb}
\theoremstyle{definition}

\usepackage{abstract}
\usepackage{amsmath}
\usepackage{lmodern}

\usepackage[english]{babel}
\usepackage{graphicx}

\def\XXint#1#2#3{{\setbox0=\hbox{$#1{#2#3}{\int}$ }
\vcenter{\hbox{$#2#3$ }}\kern-.6\wd0}}

\newtheorem{theorem}{Theorem}[section]
\newtheorem{definition}[theorem]{Definition}
\newtheorem{lemma}[theorem]{Lemma}
\newtheorem{proposition}[theorem]{Proposition}
\newtheorem{corollary}[theorem]{Corollary}
\newtheorem{example}[theorem]{Example}
\newtheorem{remark}[theorem]{Remark}
\makeatletter \renewenvironment{proof}[1][\proofname] {\par\pushQED{\qed}\normalfont\topsep6\p@\@plus6\p@\relax\trivlist\item[\hskip\labelsep\bfseries#1\@addpunct{.}]\ignorespaces}{\popQED\endtrivlist\@endpefalse} \makeatother

\newtheorem{rem}{Remark}

\newcommand{\EM}{{\mathbb E}}

\newcommand{\diam}{\mbox{\rm diam}}              
\newcommand{\dist}{\mbox{\rm dist}}              
                
\newcommand{\loc}{\mbox{\rm loc}}
\newcommand{\epi}{\mbox{\rm epi}} 
\newcommand{\convn}{\;\overrightarrow{_{_{n\rightarrow \infty}}}\;}

\newcommand{\convk}{\;\overrightarrow{_{_{k\rightarrow \infty}}}\;}

\mathchardef\mhyphen="2D

\newcommand{\Glim}{\Gamma \mhyphen \lim_{n\rightarrow \infty}}
\newcommand{\liminfn}{\liminf_{n\rightarrow \infty}}

\newcommand{\supp}{\mbox{\rm supp}}
\newcommand{\av}{\mbox{\rm av}}

\begin{document}
\doublespacing  



\newcommand{\thesisTitle}{Local Space and Time Scaling Exponents for Diffusion on Compact Metric Spaces}
\newcommand{\yourName}{John William Dever}
\newcommand{\yourSchool}{Mathematics}
\newcommand{\yourMonth}{August}
\newcommand{\yourYear}{2018}


\begin{titlepage}
\begin{center}

\begin{singlespacing}

\textbf{\MakeUppercase{\thesisTitle}}\\
\vspace{10\baselineskip}
A Dissertation\\
Presented to\\
The Academic Faculty\\
\vspace{3\baselineskip}
By\\
\vspace{3\baselineskip}
\yourName\\
\vspace{3\baselineskip}
In Partial Fulfillment\\
of the Requirements for the Degree\\
Doctor of Philosophy in the\\
School of \yourSchool\\
\vspace{3\baselineskip}
Georgia Institute of Technology\\
\vspace{\baselineskip}
\yourMonth{} \yourYear{}
\vfill
Copyright \copyright{} \yourName{} \yourYear{}

\end{singlespacing}

\end{center}
\end{titlepage}

\currentpdfbookmark{Title Page}{titlePage}  



\newcommand{\committeeMemberOne}{Dr. Jean Bellissard, Advisor}
\newcommand{\committeeMemberOneDepartment}{School of Mathematics}
\newcommand{\committeeMemberOneAffiliation}{Georgia Institute of Technology}

\newcommand{\committeeMemberTwo}{Dr. Evans Harrell, Advisor}
\newcommand{\committeeMemberTwoDepartment}{School of Mathematics}
\newcommand{\committeeMemberTwoAffiliation}{Georgia Institute of Technology}

\newcommand{\committeeMemberThree}{Dr. Michael Loss}
\newcommand{\committeeMemberThreeDepartment}{School of Mathematics}
\newcommand{\committeeMemberThreeAffiliation}{Georgia Institute of Technology}

\newcommand{\committeeMemberFour}{Dr. Molei Tao}
\newcommand{\committeeMemberFourDepartment}{School of Mathematics}
\newcommand{\committeeMemberFourAffiliation}{Georgia Institute of Technology}

\newcommand{\committeeMemberFive}{Dr. Yuri Bakhtin}
\newcommand{\committeeMemberFiveDepartment}{Courant Institute of Mathematical Sciences}
\newcommand{\committeeMemberFiveAffiliation}{New York University}

\newcommand{\committeeMemberSix}{Dr. Predrag Cvitanovi\'{c}}
\newcommand{\committeeMemberSixDepartment}{School of Physics}
\newcommand{\committeeMemberSixAffiliation}{Georgia Institute of Technology}

\newcommand{\committeeMemberSeven}{Dr. Alexander Teplyaev}
\newcommand{\committeeMemberSevenDepartment}{School of Mathematics}
\newcommand{\committeeMemberSevenAffiliation}{The University of Connecticut}

\newcommand{\approvalDay}{30}
\newcommand{\approvalMonth}{April}
\newcommand{\approvalYear}{2018}


\begin{titlepage}
\begin{singlespacing}
\begin{center}

\textbf{\MakeUppercase{\thesisTitle}}\\
\vspace{10\baselineskip}

\end{center}
\vfill

\ifdefined\committeeMemberFour

Approved by:
\vspace{2\baselineskip}		

\begin{minipage}[b]{0.4\textwidth}
	
	\committeeMemberOne\\
	\committeeMemberOneDepartment\\
	\textit{\committeeMemberOneAffiliation}\\
	
	\committeeMemberTwo\\
	\committeeMemberTwoDepartment\\
	\textit{\committeeMemberTwoAffiliation}\\
	
	\committeeMemberThree\\
	\committeeMemberThreeDepartment\\
	\textit{\committeeMemberThreeAffiliation}\\
	
	\committeeMemberFour\\
	\committeeMemberFourDepartment\\
	\textit{\committeeMemberFourAffiliation}\\
	
	\vspace{2\baselineskip}		
	
\end{minipage}
\hspace{0.1\textwidth}
\begin{minipage}[b]{0.4\textwidth}

	\committeeMemberFive\\
	\committeeMemberFiveDepartment\\
	\textit{\committeeMemberFiveAffiliation}\\
	
	\committeeMemberSix\\
	\committeeMemberSixDepartment\\
	\textit{\committeeMemberSixAffiliation}\\
	
	\committeeMemberSeven\\
	\committeeMemberSevenDepartment\\
	\textit{\committeeMemberSevenAffiliation}\\

	Date Approved: \approvalMonth{} \approvalDay, \approvalYear
	\vspace{1\baselineskip}		

\end{minipage}

\end{singlespacing}
\end{titlepage}


\pagenumbering{roman}
\addcontentsline{toc}{chapter}{Acknowledgments}
\setcounter{page}{3} 
\clearpage
\begin{centering}
\textbf{ACKNOWLEDGEMENTS}\\
\vspace{\baselineskip}
\end{centering}

I would like to thank my advisors Dr. Jean Bellissard and Dr. Evans Harrell for their patient help, guidance, and encouragement. 

I would also like to thank Dr. Gerard Buskes and Dr. Luca Bombelli for inspiring me to become interested in mathematics in the first place. 

I would like to thank my family for their love and support, especially my mother Sharon Dever, my father William Dever, and my stepmother Marci Dever. 

Finally, I would like to thank the people at St. John the Wonderworker Orthodox Church, especially Fr. Chris Williamson, Fr. Tom Alessandroni, Pamela Showalter, and Marcia Shafer, for their support and inviting community. 

\clearpage




\renewcommand{\cftchapdotsep}{\cftdotsep}  
\renewcommand{\cftchapfont}{\bfseries}  
\renewcommand{\cftchappagefont}{}  
\renewcommand{\cftchappresnum}{Chapter }
\renewcommand{\cftchapaftersnum}{:}
\renewcommand{\cftchapnumwidth}{6em}
\renewcommand{\cftchapafterpnum}{\vskip\baselineskip} 
\renewcommand{\cftsecafterpnum}{\vskip\baselineskip}  
\renewcommand{\cftsubsecafterpnum}{\vskip\baselineskip} 
\renewcommand{\cftsubsubsecafterpnum}{\vskip\baselineskip} 

\titleformat{\chapter}[display]
{\normalfont\bfseries\filcenter}{\chaptertitlename\ \thechapter}{0pt}{\MakeUppercase{#1}}

\renewcommand\contentsname{Table of Contents}

\begin{singlespace}
\tableofcontents
\end{singlespace}

\currentpdfbookmark{Table of Contents}{TOC}

\clearpage


\addcontentsline{toc}{chapter}{List of Figures}
\begin{singlespace}
\setlength\cftbeforefigskip{\baselineskip}  
\listoffigures
\end{singlespace}

\clearpage


\clearpage
\begin{centering}
\textbf{SUMMARY}\\
\vspace{\baselineskip}
We provide a new definition of a local walk dimension $\beta$ that depends only on the metric and not on the existence of a particular regular Dirichlet form or heat kernel asymptotics.  
Moreover, we study the local Hausdorff dimension $\alpha$ and prove that any variable Ahlfors regular measure of variable dimension $Q$ is strongly equivalent to the local Haudorff measure with $Q=\alpha,$ generalizing the constant dimensional case. Additionally, we provide constructions of several variable dimensional spaces, including a new example of a variable dimensional Sierpinski carpet. We show also that there exist natural examples where $\alpha$ and $\beta$ both vary continuously. We prove $\beta\geq 2$ provided the space is doubling.  
We use the local exponent $\beta$ in time-scale renormalization of discrete time random walks, that are approximate at a given scale in the sense that the expected jump size is the order of the space scale. In analogy with the variable Ahlfors regularity space scaling condition involving $\alpha$, we consider the condition that the expected time to leave a ball scales like the radius of the ball to the power $\beta$ of the center. 
Under this local time scaling condition along with the local space scaling condition of Ahlfors regularity,  we then study the $\Gamma$ and Mosco convergence of the resulting continuous time approximate walks as the space scale goes to zero. We prove that a non-trivial Dirichlet form with Dirichlet boundary conditions on a ball exists as a Mosco limit of approximate forms. One of the novel ideas in this construction is the use of exit time functions, analogous to the torsion functions of Riemannian geometry, as test functions to ensure the resulting domain contains enough functions. We also prove tightness of the associated continuous time processes.
\end{centering}

%
%


\clearpage
\pagenumbering{arabic}
\setcounter{page}{1} 

\titleformat{\chapter}[display]
{\normalfont\bfseries\filcenter}{\MakeUppercase\chaptertitlename\ \thechapter}{0pt}{\MakeUppercase{#1}}  
\titlespacing*{\chapter}
  {0pt}{0pt}{30pt}	
  
\titleformat{\section}{\normalfont\bfseries}{\thesection}{1em}{#1}

\titleformat{\subsection}{\normalfont}{\uline{\thesubsection}}{0em}{\uline{\hspace{1em}#1}}

\titleformat{\subsubsection}{\normalfont\itshape}{\thesubsection}{1em}{#1}

\chapter{Introduction}
Classical Brownian motion in Euclidean space largely characterized by its space and time scaling, that is the square of the distance traversed in a time period $\delta t$ by a particle undergoing Brownian motion is proportional to $\delta t.$  Since time interval $\delta t$ approaches zero quadratically faster than the root mean square of the distance the particle travels in $\delta t$, Brownian paths have a fractal character. This scaling guides any method of approximation of Brownian motion by random walks, in that if an approximating discrete random walk takes jumps with distance of order $\delta x$ then the time between jumps $\delta t$ should be of order $(\delta x)^2.$ 
\vspace*{1 in}
\begin{figure}[H]
\centering
\includegraphics[scale=.2]{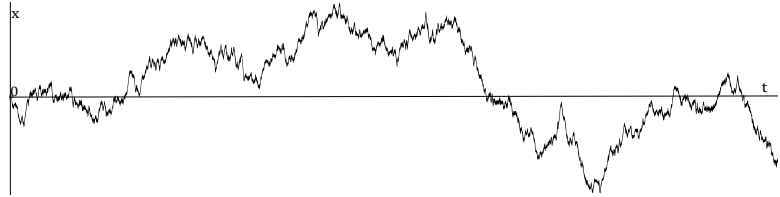}
\caption{A graph of a sample path segment of an approximation of 1 dimensional Brownian motion. The horizontal coordinate is time and the vertical coordinate space.}
\end{figure}
\vspace*{1 in}
In many cases of interest such as classical Brownian motion and various cases of diffusion on fractals, the process $(\bm{x}_t)_{t\geq 0}$ obeys a power law scaling $\mathbb{E}^x(|\bm{x}_t-\bm{x}_0|^2)\asymp t^{2/\beta}$ for some exponent $\beta.$ Cases in which $\beta\neq 2$ are called anomalous; superdiffusive if $\beta<2$ and subdiffusive if $\beta>2.$ 

We study the relationship between the space and time scaling of approximate random walks on a compact metric space. Our approach is novel in that we allow both the space scaling exponent $\alpha$ and the time scaling exponent $\beta$ to be variable in space. Moreover, we provide a new intrinsic definition of a variable walk dimension $\beta$ that depends only on the given metric measure space structure. We then consider functional and probabilistic convergence of approximating walks.

\section{Background}
It has become clear that the domain of the generator of a diffusion process on many non-homogeneous metric spaces such as fractals is often a type of Besov-Lipschitz function space characterized by an exponent $\beta.$ Unlike the case of Euclidean space, it often happens that Lipschitz functions are not in the domain of the generator, or Laplacian. Informally, to define the quadratic form of the Laplacian, instead of integrating the square of a gradient, one must integrate the square of a ``fractal gradient" of the form $``\frac{df}{dx^{\beta/2}}"$ for some exponent $\beta.$ From heat kernel asymptotics in many notable examples, it has been found that the heat kernel bounds are characterized by two scaling exponents, $\alpha$ and $\beta.$ The exponent $\alpha$ is the Hausdorff dimension, often appearing as a space scaling exponent in a suitable geometric measure. From the heat kernel bounds, one often finds that the expected square of the metric distance traveled by the process in a time $t$ scales like $t^{\frac{2}{\beta}}.$ Hence one gets the interpretation of $\beta$ as a kind of walk dimension. In many cases of interest this $\beta,$ in its guise as a walk dimension, is precisely the same $\beta$ that one should use in the ``fractal gradient" ``$\frac{df}{dx^{\beta/2}}$" in order to define the Laplacian. The problem, however, is how might one define this exponent $\beta$, preferably in a primarily geometric manner, without first knowing about functions in the domain of a possibly existing diffusion process. 

\subsection{Overview and Results}
In this paper we propose for a compact metric space a primarily geometric definition of a walk dimension exponent $\beta,$ defined purely in terms of the metric. Moreover, we show that this exponent may be localized, and indeed, there are natural examples where $\beta$ takes on a continuum of values. Our definition of the exponent $\beta$ appears to be new. As we shall see, it may be informally interpreted as a localized ``walk packing dimension" or as a local time scaling exponent.

We also consider another local exponent $\alpha,$ the local Hausdorff dimension. The exponent $\alpha$ is the local Hausdorff dimension, which has been considered previously in \cite{Loc} and \cite{dever} as well as implicitly in \cite{Sob}. It may be informally interpreted as a space scaling exponent, especially when considered in relation to a variable Ahlfors regular measure. A measure is variable Ahlfors regular when it satisfies the geometric property that the measure of a ball of a given radius scales like the radius to some power $Q$ depending on the center of the ball. We show that we must have $Q=\alpha.$ Moreover, we prove a kind of uniqueness result for variable Ahlfors regular measures, showing that any such measure is essentially equivalent, in a precise sense, to a local Hausdorff measure. Additionally, we provide several new examples of spaces in which $\alpha$ varies continuously, including a variable dimensional Sierpi\'nski carpet. See the figure below for an example.
\vspace*{1 in}

\begin{figure}[H]
\centering
\includegraphics[scale=.25]{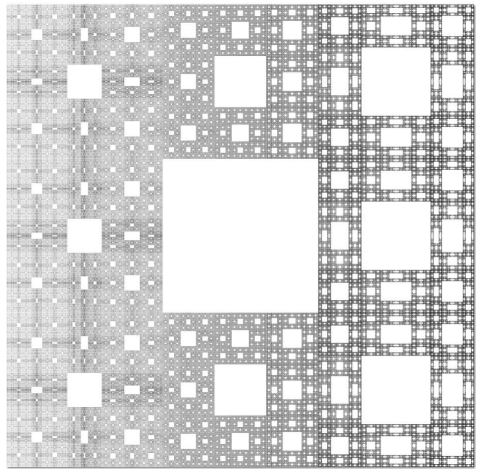}
\caption{A variable dimensional Sierpi\'nski carpet.}
\end{figure}

\newpage
An overview of the definition of $\beta$ is as follows. Given a compact metric space $X$ and a positive scale $\epsilon$, one may discretely approximate $X$ by a maximally separated set at that scale. Given such a discrete approximation, one may define in a geometric manner a graph whose vertices are the elements of the approximating set. This then induces a discrete time random walk on the graph defined by jumping uniformly, or with an appropriate ``partial symmetrization" of uniform probabilities, to an adjacent vertex. Given a ball $B$ of radius $R$ in $X$ and a vertex $y$, one may consider the expected time it takes for a 
walker on the approximating graph at scale $\epsilon$ starting at $y$ to 
leave the ball. One then defines $\beta(B)$ as a critical exponent where the behavior of the maximum exit time from the ball at scale $\epsilon$ multiplied by $\epsilon^\gamma$ changes, when $\gamma$ varies after letting $\epsilon\rightarrow 0$. One may show that if $B\subset B'$ where $B'$ is a ball, then $\beta(B)\leq \beta(B').$ We then define 
$\beta(x)$ as the infimum of the $\beta(B)$ where $B$ is a ball about $x.$ A precise definition is given in Section 4. We also argue there for our interpretation of $\beta$ as a kind of local walk packing dimension.

However, in this thesis, we mainly consider a $\beta$ defined with respect to partial symmetrizations of $\epsilon \mhyphen$jump random walks on $X$ with respect to a given Borel measure $\mu$ of full support. That is, at stage $\epsilon$, given $x\in X,$ we assume that the walker jumps from $x$ to $y,$ for $y$ a distance at most $\epsilon$ from $x,$ with $\mu$-density proportional to $1+\frac{v_\epsilon(x)}{v_\epsilon(y)},$ where $v_\epsilon$ of a point is the $\mu$ volume of the ball of radius $\epsilon$ about that point.  Similarly we may consider the expected number of steps, $E_{\epsilon, B}(x)$, needed for a walker at stage $\epsilon$ to leave a ball $B$ starting at $x.$ One may then again define $\beta(B)$ as a critical exponent where $\sup_{y\in B}E_{\epsilon,B}(y)\epsilon^\gamma$ changes behavior in $\gamma$ as $\epsilon\rightarrow 0$ and $\beta(x)$ as a limit of $\beta(B_r(x))$ as $r\rightarrow 0.$

\begin{figure}[H]
\centering
\includegraphics[scale=.2]{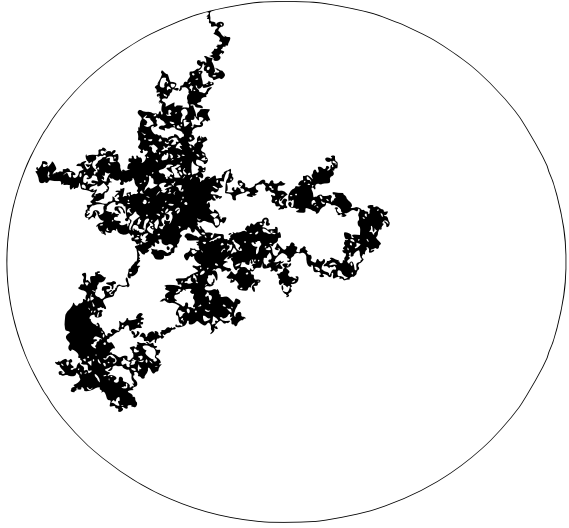}
\caption{A discrete random walk exiting a 2 dimensional ball. }
\end{figure}
\vspace*{1 in}

Once we have constructed $\beta,$ we use it to re-normalize the time scale of the discrete time walks by requiring that a walker at stage $\epsilon$ at site $x$ wait on average $\epsilon^{\beta(x)}$ before jumping, with probability proportional to $(1+\frac{d_\epsilon(x)}{d_\epsilon(y)}),$ where $d_\epsilon(\cdot)=\mu(B_\epsilon(\cdot))\epsilon^{\beta(\cdot)},$ to a neighboring site. This induces a continuous time walk $(\bm{x}(\epsilon)_t)_{t\geq 0}$. Given a ball $B$ we examine the expected exit time of the continuous time walk from $B.$ We will especially be interested in studying the case where the maximum expected exit time from a ball $B_r(x)$ scales like $r^{\beta(x)}$. Let us call such a condition the variable time regularity condition, or $E_\beta$ (see also \cite{grigortelcs} for constant $\beta$). Under $E_\beta$, the exponent $\beta$ has the interpretation of a local time scaling exponent.

Our primary motivation for the definition of $\beta$ is to attempt to define a suitable notion of a Laplace-Beltrami operator on $X$. As such, we consider (minus) the generator $\mathscr{L}_r$ of the continuous time walk $(\bm{x}(r)_t)_{t\geq 0}$ at stage $r.$ We then have that for $f\in L^2(X,\mu),$ 
\[\mathscr{L}_{r}f(x)=\frac{q_r(x)^{-1}}{r^{\beta(x)}\mu(B_r(x))}\int_{B_r(x)}\left(1+\frac{d_r(x)}{d_r(y)}\right) (f(y)-f(x))d\mu(y),\] where $q_r(x)$ is the normalization factor $q_r(x)=\frac{1}{\mu(B_r(x))}\int_{B_r(x)}\left(1+\frac{d_r(x)}{d_r(y)}\right)d\mu(y).$

It is known that on $\mathbb{R}^n$ (and indeed on any Riemannian $n-$manifold) that if $\mu$ is the standard Lebesgue measure (or volume measure on a Riemannian manifold) and if $f$ is any smooth function, then $\frac{1}{r^2\mu(B_r(x))}\int_{B_r(x)}(f(y)-f(x))d\mu(y)$ converges as $r$ approaches $0$ to a constant multiple of the Laplace-Beltrami operator evaluated at $f$ (See \cite{buragol}). Hence we should expect that for $x\in R^n,$ under its Euclidean metric with Lebesgue measure, that $\beta(x)=2.$ We show this to be the case. 

Moreover, by a variant of the well known Faber-Krahn inequality, if $B=B_R(0)$ is the ball of radius $R$ about the origin in $\mathbb{R}^n$ and $\lambda_1(B)$ is the first positive eigenvalue of (minus) the Dirichlet Laplacian on $B$, then 
\[\lambda_1(B)\geq CR^{-2},\] where $C$ is a constant independent of $R$. 

As a preliminary step towards form convergence estimates, we will establish the following as an analog of the Faber-Krahn inequality. Under the time scaling condition $E_\beta,$ if $\lambda_{1,r}(B)$ is the bottom of the spectrum of $\mathscr{L}_r$ with Dirichlet boundary conditions on a ball $B=B_R(x)$, then \[\lambda_{1,r}(B) \geq CR^{-\beta(x)},\] where $C$ is independent of $r, R,$ and $x$ provided $r$ and $R$ are small enough.

Additionally, under the assumption that the measure $\mu$ is doubling, we show that $\beta(x)\geq 2$ for all $x\in X.$

Next we construct a Dirichlet form limit of approximate random walks via $\Gamma$-convergence techniques in the spirit of the paper \cite{sturm}. 

We use as approximating forms \[\mathscr{E}_r(f)=\langle f, \mathscr{L}_r f\rangle_{\mu_r}=\int_X \frac{1}{r^{\beta(x)}\mu(B_r(x))}\int_{B_r(x))}(f(y)-f(x))^2d\mu(y)d\mu(x),\] where $d\mu_r(x)=q_r(x)d\mu(x)$ is the equilibrium measure for the continuous time walk $(X^{(r)}_t)_{t\geq 0}.$ Under variable Ahlfors regularity and time regularity, this measure is comparable to $\mu$ with constants independent of $r$ provided $r$ is small enough.

Our approach is notable in that our time scaling is allowed to be spatially dependent. Moreover, we use the exit time functions as test functions instead of the traditional approach of using harmonic functions. This approach has the advantage of not relying on an elliptic Harnack inequality. Lastly, we show via a probabilistic argument that under time regularity, a subsequence of approximating walks converges weakly to a limit that has continuous paths almost surely.

In \cite{peabel} and \cite{pearson} a Laplacian was constructed on a ultrametric Cantor set using a non-commutative geometric approach. Ideas for how to generalize this procedure to an arbitrary compact metric space were proposed in \cite{palmer}. However, it was recognized that the approach did not work for many many examples and that another exponent was needed for a Dirac operator. The original impetus for beginning this line of research was to find the appropriate exponent for the Dirac operator, with the ultimate goal of understand the conditions needed on a metric space to construct a strongly local, regular Dirichlet form. It is my hope that the results and ideas in this paper may be even a small step toward this goal.

\subsection{Related Work}

Similar exit time scaling exponents and power law scaling conditions on such exponents in various contexts have been considered by a myriad of other authors. 

There is a notion of a walk dimension found in the literature on fractal graphs. In \cite{telcs1}, a local exponent $d_W(x)$ is defined for a random walk on an infinite graph $G$ as follows. If $x$ is a vertex and $N>0$ an integer, let $E_N(x)$ be the expected number of steps needed for a random walk on $G$ starting at $x$ to reach of vertex of graph distance more than $N$ away from $x.$ In other words $E_N(x)$ is the expected exit time of the walk from the ``graph ball" of graph distance (or ``chemical distance") $N$ about $x.$ Then set $d_W(x):=\limsup_{N\rightarrow \infty} \frac{\log(E_N(x))}{\log(N)}.$ In the literature on random walks on infinite graphs, scaling 
conditions of the exit time $E_R(x)$ from a graph ball of graph radius 
$R$ about $x$ of the form $E_R(x)\asymp R^\beta$ have been considered  
\cite{telcs2}, \cite{grigortelcs}, \cite{barlowesc}. It is clear that if a graph satisfies 
such a condition then $\beta$ must be $d_W$ as defined above.  Barlow has shown in \cite{barlowesc} that if a graph satisfies such an exit time condition and an additional volume scaling condition analogous to Ahlfors regularity, then it 
must be the case that $2\leq d_W \leq 1+\alpha,$ where $\alpha$ is a dimension arising from a volume scaling condition of graph balls analogous to Ahlfors regularity. In \cite{telcseins} conditions are given for the so called Einstein relation to hold connecting resistance growth and volume growth on annuli to mean exit time growth on graph metric balls. The monograph \cite{art} presents an excellent exposition and overview of the general theory of random walks on infinite graphs. Moreover, Theorems 7.7 and 7.8 presented here follow in part ideas presented in the proofs for the graph case in Lemma 2.2 and 2.3 in \cite{art}.

On many fractals such as the Sierpi\'nski gasket and carpet it is known that the fractal may be represented by an infinite graph. On the Sierpi\'nski gasket one may compute explicitly the mean exit time from graph metric balls, and one finds that $\beta=\frac{\log5}{\log 2}.$ For the Sierpi\'nski carpet it is known that it must satisfy an exit time scaling condition for some power of $\beta.$ However the exact value of $\beta$ in this case is unknown.

In the setting of metric measure Dirichlet spaces, a walk dimension $\beta$ has also appeared in certain sub-diffusive heat kernel estimates on various fractals and infinite fractal graphs. For a wide class of fractals, including the Sierpi\'nski gasket and carpet, for which diffusion processes are known to exist, heat kernel estimates of the form 
\[p_t(x,y)\asymp c_1t^{-\frac{\alpha}{\beta}}\exp\left({-c_2\left(\frac{d(x,y)^{\beta}}{t}\right)^{\frac{1}{\beta-1}}}\right)\]
have been shown to hold \cite{barlowgasket}, \cite{barlowcarpet}, \cite{fitz}, \cite{grigor2012two}. It is known that such heat kernel estimates imply that the underlying mesaure is Ahlfors regular with Hausdorff dimension $\alpha$ \cite{grigorlau}, thus providing us with another motiviation for its study. The exponent $\beta$ appearing in these estimates is called the walk dimension \cite{barlowcarpet1}. Moreover, in the setting of metric measure Dirichlet spaces, consequences of a mean exit time scaling condition similar to $E_\beta$ is considered in \cite{grigor2012two}. Such conditions, together with volume doubling and an elliptic Harnack inequality, have been shown to imply the existence of a heat kernel along with certain heat kernel estimates \cite{grigor2012two}. Additionally, an adjusted Poincar\'{e} inequality involving the mean exit time has been proposed in \cite{bass2013} and \cite{barlowcarpet1}. Such Poincar\'{e} inequalities or resistance estimates together with an elliptic Harnack inequality often play an important role in the proofs of the existence of diffusion processes (See \cite{kusuokazhou}, \cite{barlowres}.) 

In the setting of Riemannian geometry, the function giving the mean exit time from a ball starting at a given point is known as the torsion function, and the integral of the mean exit time function is known as the torsional rigidity \cite{vand}. 

It is known that the domains of many diffusions on fractals are a type of Besov-Lipschitz function space \cite{jon}, \cite{kumagai}. For $\sigma>0, r>0,$ let \[\mathscr{E}_{r,\sigma}(f)=\int_X\frac{1}{r^\sigma\mu(B_r(x))}\int_{B_r(x)}(f(y)-f(x))^2d\mu(y)d\mu(x).\] Then let \[W^\sigma(X,\mu):=\{f\in L^2(X,\mu)\;|\;\sup_{r>0}\mathscr{E}_{r,\sigma}(f)<\infty\}.\] Then $W^\sigma(X,\mu)$ is a Banach space with norm $\|f\|^2_{W^\sigma(X,\mu)}=\|f\|^2_{L^2(X,\mu)}+ \sup_{r>0}\mathscr{E}_r(f).$ A potential theoretic definition of a walk dimension has been given as a critical exponent $\beta^*$, obtained by varying $\sigma$, where $W^\sigma(X,\mu)$ changes behavior to containing only constant functions \cite{grigorlau}, \cite{sturm}. Moreover, in \cite{grigorlau}, conditions were given for $\beta^*$ to equal $\beta,$ provided $\beta$ may be defined from heat kernel estimates. In \cite{gu} it was proven that $\beta^*$ is a Lipschitz invariant among metric measure spaces with an Ahlfors regular measure.

Recently, a proposal for a method to define $\beta^*$ without reference to diffusion was proposed by Grigor'yan \cite{grigorwalk}. The method, applied there to the Sierpi\'nski gasket, involves the procedure of forming a weighted hyperbolic graph induced from the graph approximations to the space and seeing the original space as a Gromov hyperbolic boundary (See, for instance,  \cite{lau} and \cite{Piaggio} for more on this method). A random walk on the hyperbolic graph induces a non-local form on the boundary whose domain is another type of Besov-Lipschitz space. Again, $\beta^*$ is seen as a critical exponent where the space changes to have sufficiently many non-constant functions. The hyperbolic graph approximation allows one to examine these functions in terms of the random walk on the hyperbolic graph.
 
The local Hausdorff dimension was defined in \cite{Loc}. A curve with continuously varying local dimension was considered, somewhat informally, in \cite{nottale}. A variable dimensional Koch curve and a local Hausdorff measure were defined in \cite{Sob}. Also, variable Ahlfors $Q(\cdot)\mhyphen$regular measures were considered in \cite{Sob}.  For $Q=d$ constant, it is known that $d=\dim_H(X)$ and that if $\mu$ is any other Ahlfors $d\mhyphen$regular measure, then $H^d\asymp \mu,$ where $H^d$ is the Hausdorff measure at dimension $d$ \cite{Hein}. 

In \cite{burago2013graph} and \cite{buragol} the authors used a discrete approximation with weighted $\epsilon\mhyphen$net graphs and a continuous approximation with a given Borel measure of full support, respectively, to create approximate Dirichlet forms on a compact metric space. In \cite{sturm} variational ($\Gamma\mhyphen$)convergence was used to study limits of certain approximating forms defined both through approximating graphs and by means of a given Borel measure of full support. Additionally, in \cite{sturm}, sufficient conditions were given for a ($\Gamma\mhyphen$)limit of such approximating forms to generate a non-trivial diffusion process.


\chapter{Preliminaries}
In this chapter we cover many of the mathematical preliminaries, aside from basic analysis, point-set topology, and operator theory, that will be used in the sequel. 

\subsubsection{Notation}
We adopt the following notational conventions. 

By $\mathbb{N}$ we mean the set of all non-negative integers, and by $\mathbb{Z}^+$ we mean $\mathbb{N}\setminus\{0\}$. 

We shall use positive and non-negative synonymously. In particular, non-negative operator or non-negative quadratic form means the same as positive operator or positive quadratic form, respectively. If we wish to exclude zero, we shall use the terminology ``strictly positive".   

Let $\overline{\mathbb{R}}=\mathbb{R}\cup \{-\infty,\infty\}$ denote the extended real numbers. We extend the partial order $\leq$ on $\mathbb{R}$ to $\overline{\mathbb{R}}$ by setting $-\infty\leq c\leq \infty$ for all $c\in \mathbb{R}.$ The topology on $\overline{\mathbb{R}}$ is the order topology generated by a basis of sets of the form $(c,\infty]$ and $[-\infty,c)$ for $c\in \mathbb{R}.$ Note that this topology makes $\overline{\mathbb{R}}$ homeomorphic to $[0,1];$ hence it is compact. 

Let $A\subset \overline{\mathbb{R}}.$ By $A_{+}$ we mean the set of all non-negative elements of $A.$ In particular $\mathbb{Z}_+=\mathbb{N}\supsetneq \mathbb{Z}^+.$ 
If $A$ is a non-empty and closed, by $\inf_A$ and $\sup_A$ we mean that the infimum or supremum, respectively, is restricted to the partially ordered set $A$ with ordering $\leq$ from $\overline{\mathbb{R}}.$ That is, in particular, $\inf_A\emptyset = \sup_{x\in A}x$ and $\sup_A\emptyset = \inf_{x\in A} x.$   By $\inf$ or $\sup$ without subscript we shall mean either $\inf_{\overline{\mathbb{R}}}$ and $\sup_{\overline{\mathbb{R}}},$ respectively, however in some cases $\inf_{\overline{\mathbb{R}}_+}$ and $\sup_{\overline{\mathbb{R}}_+},$ respectively, when appropriate and if there is little risk of confusion.

We adopt the conventions that the empty sum is $0$ and the empty product is $1.$ Moreover, if $\mathscr{A}$ is a collection of subsets of a set $S,$ then $\cup\mathscr{A}:=\cup_{A\in \mathscr{A}} A$ and $\cap\mathscr{A}=\cap_{A\in\mathscr{A}} A.$ In particular $\cup\emptyset =\emptyset$ and $\cap \emptyset=S.$ 

If $A$ is a countable set, by $\sharp A$ we mean the cardinality of $A;$ where if $A$ is a countably infinite set, we write $\sharp A = \infty.$ If $A$ is any set, by $\sharp A\leq \infty$ we mean $A$ is a countable set. Additionally, if $A$ is a set, we denote the power set of $A$ by $\mathscr{P}(A).$ 

If $A$ is a set, both $\chi_A$ and $1_A$ denote the characteristic function of $A.$ However, we will reserve the notation $1_A$ for when $A$ is a member of a $\sigma$-algebra of a probability space under discussion.

If $f$ and $g$ are extended real valued functions, we write $f\asymp g$ if there exists a $C>0$ such that $\frac{1}{C}f\leq g\leq Cf.$ 

For $\mathscr{H}$ a Hilbert space, we write $\langle f,g \rangle,$ or $\langle f, g \rangle_{\mathscr{H}}$ if there is a possible ambiguity, for the inner product of $f,g\in\mathscr{H}$. We assume all vector spaces to be over the real numbers unless otherwise specified. By $\mathscr{B}(\mathscr{H})$ we mean the space of bounded linear operators on $\mathscr{H}.$

If $(X,d)$ is a metric space, topological notions such as closure of subsets of $X$ shall be considered, unless otherwise specified, with respect to the metric topology on $X.$ Let $(X,d)$ a metric space. If $A\subset X,$ by the diameter of $A,$ written $\diam(A)$ or $|A|,$ we mean $\sup_{\overline{\mathbb{R}}_+}\{d(x,y)\;|\;x,y\in A\}.$ If $A, B\subset X,$ by the distance from $A$ to $B$ we mean $\dist(A,B):=\inf_{\overline{\mathbb{R}}_+}\{d(x,y)\;|\;x\in A,\; y\in B\}.$ For $x\in X,\; A\subset X$ we write $\dist(x, A)$ for $\dist(\{x\},A).$ If $r>0$, $x\in X,$ by $B_r(x)$ we mean the open ball about (or centered at) $x,$ $B_r(x)=\{y\in X\;|\;d(x,y)<r\}$. For $r>0, x\in X,$ by $B_r[x]$ we mean the closed ball about $x,$ $B_r[x]=\{y\in X\;|\;d(x,y)\leq r\}.$ Note in general that the closure of $B_r(x),\; \overline{B_r(x)},$ may be properly contained in $B_r[x].$ If $B$ is an open ball about $x,$ by the radius of $B$ we mean the number $r:=\inf\{r'>0 \;|\; B_{r'}(x)= B\}.$ However, if $X$ is connected and $0<r<\diam(x),$ then the radius of $B_r(x)$ is $r$ for any $x.$

For $X$ a topological space, we let $C(X)$ stand for the real valued continuous functions on $X.$

\section{Basic Geometric Measure Theory}
In this section, we recall various basic concepts of Geometric Measure Theory that will be needed in the sequel, including the definition of Hausdorff dimension and the construction of the $s$-dimensional Hausdorff measure. See also \cite{Dever1}. Throughout let $(X,d)$ be a metric space.
\subsection{Metric Outer Measures and the Carath\'{e}odory Construction}

\begin{definition}
An outer measure on $X$ is a function $\mu^*:\mathscr{P}(X)\rightarrow [0,\infty]$ with $\mu^*(\emptyset)=0$ such that 
\begin{equation*} \begin{split} &\mbox{(\textbf{a})\;if\;} A,B\subset X \mbox{\;with\;} A\subset B, \mbox{\;then\;} \mu^*(A)\leq \mu^*(B);\\ &\mbox{(\textbf{b})\;and if\;} 
\mathscr{C} \subset \mathscr{P}(X) \mbox{\;is countable, then\;} 
\mu^*(\cup \mathscr{C})\leq \sum_{A\in\mathscr{C}} \mu^*(A).\end{split}\end{equation*}\end{definition} Condition (a) is called monotonicity, and condition (b) is called countable subadditivity. Recall that a $\sigma$-algebra of subsets of $X$ is a subset $\mathscr{M}\subset \mathscr{P}(X)$ that is closed under countable unions and complementation. Note that since the empty set is countable, any algebra or $\sigma$-algebra contains $\cup\emptyset =\emptyset$. If $\mathscr{C}\subset \mathscr{P}(X),$ then $\sigma(\mathscr{C})$, called the $\sigma$-algebra generated by $\mathscr{C}$, is the smallest $\sigma$-algebra of subsets of $X$ containing $\mathscr{C}.$ Formally, $\sigma(A)$ is the intersection of all $\sigma$-algebras of subsets of $X$ containing $\mathscr{C}.$ 

A subset of $\mathscr{P}(X)$ is called disjoint if the intersection of any two (distinct) elements is empty. 
Recall further that a (positive) measure $\mu$ on $\mathscr{M}$ is a function $\mu:\mathscr{M}\rightarrow [0,\infty]$ such that if $\mathscr{C}\subset \mathscr{P}(X)$ is countable and disjoint, then $\mu(\cup\mathscr{C})= \sum_{A\in\mathscr{C}}\mu(A).$

The Borel sigma algebra is the smallest $\sigma\mhyphen$algebra containing the open sets of $X.$ Elements of the Borel $\sigma\mhyphen$algebra are called Borel sets. A Borel measure is a measure defined on the $\sigma\mhyphen$algebra of Borel sets. By a measure on $X,$ we mean, unless otherwise specified, a measure on the Borel $\sigma$-algebra of $X.$ A measure $\mu$ on a $\sigma\mhyphen$algebra $\mathscr{M}$ is called complete if for all $N\in \mathscr{M}$ with $\mu(N)=0,$ $\mathscr{P}(N)\subset \mathscr{M}$.

The following Carath\'{e}odory construction of measures from outer measures is essential to the subject. The second half of the following proposition is known as the Carath\'{e}odory Extension Theorem. The proof, following \cite{Folland}, may be found in Appendix A.
\begin{proposition} \label{Car} If $\mu^*$ is an outer measure on $X,$ then if \[\mathscr{M^*}:=\{A\subset X\;|\;\forall E\subset X\;[\;\mu^*(E)=\mu^*(E\cap A)+\mu^*(E\cap A^c)\;]\;\},\] $\mathscr{M^*}$ is a $\sigma\mhyphen$algebra and $\mu^*|_{\mathscr{M^*}}$ is a complete measure.  
\end{proposition}

Sets $A,B\subset X$ are called positively separated if $\dist(A,B)>0.$ \begin{definition} An outer measure $\mu^*$ on $X$ is called a metric outer measure if for all $A,B\subset X$ with $A,B$ positively separated, $\mu^*(A\cup B)=\mu^*(A)+\mu^*(B).$ \end{definition} Metric outer measures provide a convenient way to construct Borel measures, as may be seen by the following well known proposition. The proof, found in Appendix A, follows closely the one found in \cite{Falc1}.

\begin{proposition} \label{prop2.2}
If $\mu^*$ is a metric outer measure on a metric space $X,$ then $\mathscr{M}^*$ contains the $\sigma\mhyphen$algebra of Borel sets. In particular, $\mu^*$ may be restricted to a Borel measure. 
\end{proposition}

Let $\mathscr{B}$ be the collection of all open balls in $X,$ where we consider $\emptyset$ and $X$ open balls, $\emptyset = B_0(x)$ and $X=B_\infty(x)$ for 
any $x\in X.$ By a covering class we mean a collection $\mathscr{A}\subset \mathscr{P}(X)$ with $\emptyset, X \in \mathscr{A}.$ We will primarily work with the covering classes $\mathscr{C}:=\mathscr{P}(X)$ and $\mathscr{B}.$ For $\mathscr{A}$ a covering class and $A\subset X$, let \[\mathscr{A}_\delta(A) := \{\mathscr{U} \in \mathscr{P}(\mathscr{A}) \;|\;\mathscr{U} \mbox{\;at 
most countable},\; |U|\leq \delta \mbox{\;for}\; U\in \mathscr{U},\; A\subset \cup 
\mathscr{U}\}.\] 

For $\tau:\mathscr{C}\rightarrow [0,\infty]$ with $\tau(\emptyset)=0,$
let $\mu^* _{\tau, \delta} (A):=\inf \{ \sum_{U\in \mathscr{U}}\tau(U)\;|\; 
\mathscr{U}\in \mathscr{C}_\delta(A)\}$ and $\mu^* _{\tau}(A)=\sup_{\delta>0} \mu^* _{\tau, \delta} 
(A).$ 
The following Proposition, see also \cite{Fed} or \cite{Dever1}, provides a convenient method to construct metric outer measures, and hence Borel measures, from suitable functions defined on covering classes. While the proof is straightforward, we include it in Appendix A for completeness.
\begin{proposition}  \label{outer} $\mu^* _{\tau}$ is a metric outer measure. \end{proposition}

We may restrict to any covering class $\mathscr{A}$ containing $\emptyset$ and $X$ by setting $\tau(U)=\infty$ for $U\in \mathscr{C}\setminus \mathscr{A}.$

It then follows by Proposition \ref{prop2.2} that the $\mu^* _{\tau}$ measurable sets contain the Borel $\sigma\mhyphen$algebra. Let $\mu_\tau$ be the restriction of $\mu^* _{\tau}$ to the Borel sigma algebra. Then by Proposition \ref{Car}, $\mu_\tau$ is a Borel measure on $X.$

\subsection{Hausdorff Measure and Dimension}
\begin{definition} For $s\geq 0$ let $H^s$ be the measure obtained from the choice $\tau(U)=|U|^s$ for $U\neq \emptyset$ and $\tau(\emptyset)=0.$ If $U$ is non-empty and $|U|=0$ we adopt the convention $|U|^0=1$. $H^s$ is called the $s$-dimensional Hausdorff measure. Let $\lambda^s$ be the measure obtained by restricting $\tau$ to the smaller covering class $\mathscr{B}$ of open balls. Concretely, let $\lambda^s$ be the measure obtained by setting $\tau(B)=|B|^s$ for $B$ a non-empty open ball, $\tau(\emptyset)=0,$ and $\tau(U)=\infty$ otherwise. We call $\lambda^s$ the $s$-dimensional open spherical measure \cite{Fed}, \cite{Dever1}. \end{definition}

\begin{definition} Following \cite{Dever1}, we call Borel measures $\mu,\nu$ on $X$ \textit{strongly equivalent}, written $\mu \asymp \nu$, if there exists a constant $C>0$ such that for every Borel set $E$, \[\frac{1}{C}\nu(E)\leq \mu(E)\leq C\nu(E).\] \end{definition}

The proof of the following lemma may also be found in \cite{Dever1}.
\begin{lemma} For any $s\geq 0,$ $\lambda^s \asymp H^s.$ 
\end{lemma}
\begin{proof} It is clear that $H^s \leq \lambda^s.$ Conversely, let $A$ be a Borel set. We may assume $A\neq \emptyset.$ Also, we may assume $H^s(A)<\infty,$ since 
otherwise the reverse inequality is clear. If $s=0$ then $H^0$ is the counting measure. So let $H^0(A)=n<\infty.$ Let $x_1,...,x_n$ be an enumeration of the elements of $A.$ Let $r$ be the minimum distance between distinct elements of $A.$ For $0<\delta<r$ let $B_i = 
B_{\frac{\delta}{2}}(x_i)$ for each $i.$ Then $(B_i)_{i=1}^n \in \mathscr{B}_\delta(A)$ 
and $\lambda^{0,*}_\delta(A)\leq n=H^0(A)$. Hence $\lambda^0(A) \leq H^0(A).$ So we may 
assume $s>0.$  Let $\epsilon>0.$ Let $\delta>0$ and let $\mathscr{U}\in 
\mathscr{C}_\delta(A)$ with $\sum_{U\in \mathscr{U}} |U|^s <\infty.$ We may assume $U\neq 
\emptyset$ for each $U\in \mathscr{U}.$ Choose $x_U\in U$ for each $U\in \mathscr{U}.$ 
Let $\mathscr{U}_0=\{U\in \mathscr{U}\;|\;|U|=0\}, \mathscr{U}_1:=\{U\in \mathscr{U}\;|
\;|U|>0\}.$ Then for each $U\in \mathscr{U}_0$ choose $0<r_U<\delta$ such that 
$\sum_{U\in \mathscr{U}_0} r_U^s<\frac{\epsilon}{2^s}.$ For $U\in \mathscr{U}_1$ let $r_U = 2|U|.$ Then let $B_U:=B_{r_U}(x_U)$ for $U\in \mathscr{U}.$ It follows that 
$(B_U)_{U\in \mathscr{U}}\in \mathscr{B}_{4\delta}(A)$ and $\lambda^{s,*}_{4\delta}
(A)\leq \sum_{U\in \mathscr{U}}|B_U|^s \leq 2^s(\sum_{U\in \mathscr{U}_0} r_U^s + 
\sum_{U\in \mathscr{U}_1} r_U^s) \leq \epsilon + 4^s\sum_{U\in \mathscr{U}} |U|^s.$ Hence 
$\lambda^s(A)\leq 4^sH^s(A).$ \end{proof}

Let $X$ be a metric space and $A\subset X$.  Let $0\leq t<s.$ If $(U_i)_{i=1}^\infty \in \mathscr{C}_\delta(A)$ then $H^s_\delta(A)\leq \sum_i |U_i|^s \leq \delta^{s-t}\sum_i |U_i|^t.$ So $H^s_\delta(A) \leq \delta^{s-t} H_\delta^t(A)$ for all $\delta>0$.

Suppose $H^t(A)<\infty$. Then since $\delta^{s-t}\overrightarrow{_{_{\delta \rightarrow 0^+}}} 0,$ $H^s(A)=0.$ Similarly, if $H^s(A)>0$ and $t<s,$ $H^t(A)=\infty.$ It follows that we may make the following definition.
\begin{definition} We have \[{\sup}_{{\overline{\mathbb{R}}_+}}\{s\geq 0\;|\;H^s(A)=\infty\}={\inf}_{\overline{\mathbb{R}}_+}\{s\geq 0\;|\;H^s(A)=0\}.\] We denote the common number in $[0,\infty]$ by $\dim(A).$ It is called the Hausdorff dimension of $A.$\end{definition} Since $H^s\asymp \lambda^s,$ we also have \[\dim(A)= {\sup}_{{\overline{\mathbb{R}}_+}}\{s\geq 0\;|\;\lambda^s(A)=\infty\}={\inf}_{{\overline{\mathbb{R}}_+}}\{s\geq 0\;|\;\lambda^s(A)=0\}.\]

\begin{lemma} \label{mon} If $A\subset B \subset X$ then $\dim(A)\leq \dim(B).$\end{lemma}
\begin{proof} By monotonicity of measure, $H^s(A)\leq H^s(B)$. So $H^s(B)=0$ implies 
$H^s(A)=0.$ Therefore $\dim(A)=\inf\{s\geq 0\;|\;H^s(A)=0\}\leq \inf\{s\geq 0\;|
\;H^s(B)=0\}=\dim(B).$\end{proof} 

\subsection{Metric Measure Spaces}
\begin{definition}By a \textit{metric measure space} we mean a triple $(X,d,\mu)$ consisting of a metric space $(X,d)$ and a finite, positive Borel measure $\mu$ of full support\footnote{Full support for a Borel measure means $\mu(B_r(x))>0$ for all $x\in X, r>0.$}.\end{definition} By a compact metric measure space we mean that the underlying metric space $(X,d)$ is compact.

\begin{definition}
Let $Q>0$ be a constant\footnote{In chapter 3 we will allow $Q$ to be variable.}. A Borel measure is called \textit{Ahlfors regular} of dimension $Q$ if there exists a $C>0$ such that for all $x$ and every $0<r<\diam(X)/2,$
\[\frac{1}{C}r^Q\leq \mu(B_r(x))\leq Cr^Q.\]\end{definition} It is well known that if $X$ possesses an Ahlfors regular measure $\mu$ of dimension $Q,$ then $\dim_H(X)=Q$ and $\mu\asymp H^Q$ (see e.g. \cite{Piaggio} \cite{Hein}). We will generalize this result to variable Hausdorff dimension in Chapter 2. 

\begin{definition}A Borel measure $\mu$ on a metric space is called \textit{doubling} if there exists a constant $C>0$ such that for all $x\in X$ and all $0<r<\diam(X),$ \[0<\mu(B_{2r}(x))\leq C\mu(B_r(x)).\]\end{definition} In particular, if $\mu$ is Ahlfors regular of dimension $Q,$ then $\mu$ is doubling since there exists a constant $C>0$ such that $0<(1/C)2^Qr^Q\leq \mu(B_{2r}(x)C\leq 2^Qr^Q\leq C^22^Q\mu(B_r(x)).$

In particular, a metric space with a doubling measure $\mu$ is a metric measure space\footnote{Metric spaces with a doubling measure $\mu$ are often called in the literature, especially in relation to harmonic analysis, \textit{spaces of homogeneous type}. See, for example, \cite{semmes}.}.

\section{Graph Approximation}
In this section we discuss methods of approximating a compact metric space by graphs. Our main tools will be epsilon nets and dyadic cubes. We discuss these topics below. First we define graphs.

\begin{definition} By a (undirected) graph, we mean a pair $G:=(V(G),E(G)),$ where $V(G)$ is a finite non-empty set and $E(G)\subset V(G)\times V(G)$ such that if $(x,y)\in E(G)$ then $(y,x)\in E(G)$. The set $V(G)$ is called the set of vertices, and the set $E(G)$ is called the set of edges.\end{definition} If $(x,y)\in E(G)$, we write $x\sim y,$ or $x\sim_G y$ in case of possible ambiguity. For $x\in V(G),$ the set $V(G)_x:=\{y\in V(G)\;|\;y\sim x\}$ is called the set of neighbors of $x,$ and the \textit{degree} of $x,$ written $\deg(x),$ or $\deg_G(x)$ in case of possible ambiguity, is the number \[\deg(x)=\sharp V(G)_x.\]

A path in $G$ is a finite list $\gamma=(x_k)_{k=0}^{n-1}$ of elements of $V(G)$ such that $n\in \mathbb{N}$ and for all $k<n,$ $x_k\sim x_{k+1}.$ Since if $n=0$ the condition is vacuous, single vertices $(x_0)$ are paths. The length of a path, written $\ell(\gamma),$ is one less than the number of vertices in the list $\gamma$. For example the length of a single vertex path is $0$. If $\gamma=(x_k)_{k=0}^{n-1},$ we say $\gamma$ is a path connecting $x_0$ and $x_{n-1}$, which we abbreviate with symbols as $\gamma:x_0\rightarrow x_{n-1}.$ We then define $d_G:V(G)\times V(G)\rightarrow \overline{\mathbb{R}}_+$, called the graph distance, by
\[d_G(x,y):={\inf}_{\overline{\mathbb{R}}_+}\{\ell(\gamma)\;|\;\gamma:x\rightarrow y\},\] where $x,y\in G.$ Since single vertices are paths, $d_G(x,y)=0$ if and only if $x=y.$ Clearly $d_G$ is symmetric, since any path $\gamma=(x_k)_{k=0}^{n-1}$ connecting $x=x_0$ to $y=x_{n-1}$ has a reverse $\bar{\gamma}:=(x_{n-1-k})_{k=0}^{n-1},$ connecting $y$ to $x.$ Moreover, if $x,y,z\in V(G)$ and $\gamma_1=(x_k)_{k=0}^{n-1}$ is a path connecting $x$ to $y$ and $\gamma_2=(y_k)_{k=0}^{m-1}$ is a path connecting $y$ to $z$, then let $x_{n-1+k}:=y_k$ for $0\leq k<m.$ Then $\gamma:=(x_k)_{k=0}^{n+m-2}$ is a path connecting $x$ to $z$ of length $\ell(\gamma_1)+\ell(\gamma_2).$ Hence $d_G(x,z)\leq \ell(\gamma_1)+\ell(\gamma_2).$ Since $\gamma_1,$ $\gamma_2$ are arbitrary, it follows $d_G(x,z)\leq d_G(x,y)+d_G(y,z).$ Hence $d_G$ is a distance.\footnote{By distance, we mean it satisfies all properties of a metric, except that it may take on the value $+\infty.$} However, $d_G$ is not necessarily a metric since we may have $d_G(x,y)=\infty$ if there are no paths connecting $x$ to $y,$ as $\inf_{\overline{\mathbb{R}}_+}\emptyset=\infty.$ In the case that for any $x,y\in G,$ there exist at least one path connecting $x$ to $y$, we say that $G$ is connected. In such a case, $d_G$ is finite, and is thus a metric.

\subsection{Epsilon Nets}
Let $(X,d)$ be a metric space. We will be interested in ways to approximate $X$ with graphs, defined from maximally $\epsilon$-separated sets called $\epsilon$-nets.

\begin{definition}
Let $\epsilon>0.$ An $\epsilon$-net in $X$ is a set $\mathcal{N}\subset X$ such that \begin{equation*}
\begin{split}
\mbox{(a)}&\;\cup_{x\in \mathcal{N}}B_{\epsilon}(x)=X, \mbox{\;and}\\
\mbox{(b)}&\;B_\epsilon(x)\cap \mathcal{N}=\{x\}\;\mbox{for all\;}x\in \mathcal{N}.
\end{split}
\end{equation*}
\end{definition}
Call a set $A\subset X$ $\epsilon$-separated if $d(x,y)\geq \epsilon$ for all $x,y\in A$ with $x\neq y.$ Then suppose $\mathcal{N}$ is $\epsilon$-separated and is maximal in the sense that if $\mathcal{M}$ is another $\epsilon$-separated set containing $\mathcal{N}$ then $\mathcal{N}=\mathcal{M}.$ Then clearly $\mathcal{N}$ satisfies (b). Also $\mathcal{N}$ satisfies (a) since if a $y$ could be chosen outside of the union of all of the epsilon balls over points in $\mathcal{N},$ then $\mathcal{N}\cup\{y\}$ would be $\epsilon$-separated, a contradiction. Hence $\mathcal{N}$ is an $\epsilon$-net. Similarly, if $\mathcal{N}$ is an $\epsilon$-net, then it is maximally $\epsilon$-separated, since for $\mathcal{N}\subset \mathcal{M}$, if $y\in \mathcal{M}\setminus \mathcal{N},$ then $\mathcal{M}$ cannot be $\epsilon$-separated since $d(x,y)<\epsilon$ for some point $x\in \mathcal{N}$ by (a). 

We now show that there always exist $\epsilon$-separated sets, that they may be refined, and that they are finite if $X$ is compact. 
\begin{proposition} \label{netexistance} Suppose $(X,d)$ is a metric space and $\epsilon>0.$ Let $A\subset X$  be $\epsilon$-separated. Then there exists an $\epsilon$-net $\mathcal{N}$ containing $A.$ If $X$ is compact, then $\mathcal{N}$ is finite.
\end{proposition}

\begin{proof} We apply Zorn's lemma. Let $\mathscr{S}:=\{\mathcal{M}\subset X\;|\;\mathcal{M}\;\epsilon\mbox{-separated with\;} A\subset \mathcal{M}\}.$ Then $\mathscr{S}$ is partially ordered under set inclusion $\subset.$ Note $\mathscr{S}\neq \emptyset$ as it contains $A.$ Let $\mathscr{C}$ be a totally ordered subset of $\mathscr{S}.$ Let $\mathcal{M}=\cup\mathscr{C}.$ Then clearly $A\subset \mathcal{M}\subset X.$ Moreover, if $x,y\in \mathcal{M},$ then let $\mathcal{M}_1,\mathcal{M}_2\in\mathscr{C}$ with $x\in \mathcal{M}_1,$ $y\in \mathcal{M}_2.$ Then since $\mathscr{C}$ is totally ordered, we may assume $\mathcal{M}_1\subset \mathcal{M}_2.$ Then $x,y\in \mathcal{M}_2,$ and so $d(x,y)\geq \epsilon.$ Hence $\mathcal{M}\in \mathscr{S}$ and $\mathcal{M}'\subset \mathcal{M}$ for all $\mathcal{M}'\in\mathscr{C}.$ By Zorn's lemma, it follows that $\mathscr{S}$ contains maximal elements. By the previous discussion, any such maximal element $\mathcal{N}$ is an $\epsilon$-net containing $A.$

If $X$ is compact, then since $\cup_{x\in \mathcal{N}} B_\epsilon(x)$ is an open cover of $X,$ there exists a finite $\mathcal{N}'\subset \mathcal{N}$ with $\cup_{x\in \mathcal{N}'}B_\epsilon(x)=X.$ Hence $\mathcal{N}'$ is an $\epsilon$-net. Since $\mathcal{N}$ is also $\epsilon$-separated with $\mathcal{N}'\subset \mathcal{N}$, it follows $\mathcal{N}=\mathcal{N}'$ is finite.
\end{proof}

For the remainder of this section we assume $(X,d)$ is a compact metric space. Given numbers $\epsilon,\delta>0$, we define a class $G(X,\epsilon,\delta)$ of graphs with vertices in $X$ as follows. If $G$ is a graph with $V(G)\subset X$, we say $G$ is a member of $G(X,\epsilon,\delta)$ if $G$ is connected, $V(G)$ is an $\epsilon$-net, and $V(G)_x\subset B_{\delta}(x)\setminus\{x\}$ for all $x\in V(G).$  We will see that $G(X,\epsilon,\delta)$ is non-empty for all $0<\epsilon\leq \delta/2.$

Let $\epsilon>0$ and $\mathcal{N}$ an $\epsilon$-net. \begin{definition}For $\rho\geq 2,$ we define a ``proximity graph" $G^p_{\rho\epsilon}(\mathcal{N})$ induced by $\mathcal{N}$ as follows. We let $\mathcal{N}$ be the vertex set with edge relation defined by \[(x,y)\in E(G^p_{\rho\epsilon}(\mathcal{N})) \;\mbox{if and only if\;}d(x,y)<\rho\epsilon.\]\end{definition} A similar approach is taken in \cite{burago2013graph}. Note each vertex has a ``loop" edge $(x,x)$. If we wish to remove self edges, we may instead impose the edge admittance condition $0<d(x,y)<\rho\epsilon.$

\begin{definition} For $\eta\geq 1,$ we define a ``covering graph" $G^c_{\eta\epsilon}(\mathcal{N})$ to have vertex set $\mathcal{N}$ and edge relation defined by \[(x,y)\in E(G^c_{\eta\epsilon}(\mathcal{N}))\mbox{\;if and only if\;}B_{\eta\epsilon}(x)\cap B_{\eta\epsilon}(y)\neq \emptyset.\]\end{definition} A similar approach is taken in \cite{sturm}. Note, again that this definition allows loops. For a ``loopless" version, we would add to the edge admittance condition that $x\neq y.$
See the figure below. 
\vspace*{1 in}
\begin{figure}[H]
\centering
\includegraphics[scale=.5]{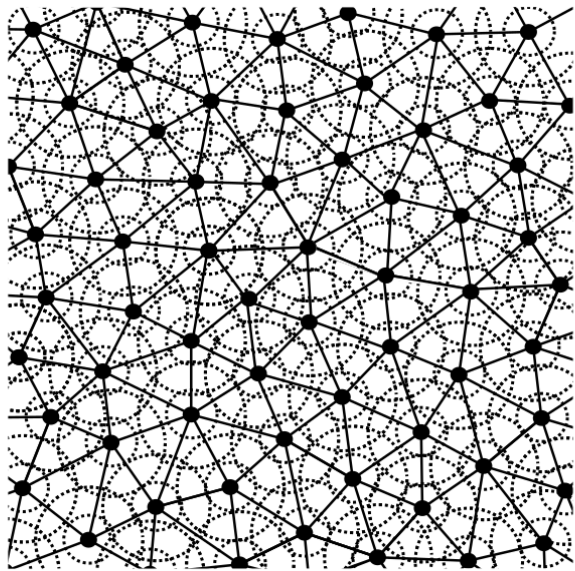}
\caption{A loopless covering graph in the plane.}
\end{figure}

\vspace*{1 in}
We now show that the topological connectedness of $X$ is closely related to the connectedness of approximating graphs. 

But first we need the following lemma.
\begin{lemma} (Lebesgue number lemma) Suppose $(X,d)$ is a compact metric space. Given any open cover of $X,$ there exists a number $\delta>0$ such that any ball of radius $\delta>0$ is contained in some member of the cover.
\end{lemma}
\begin{proof} Suppose $\mathscr{U}$ is an open cover. By compactness, let $(U_i)_{i=1}^n$ be a finite subcover. Note each $U_i^c$ is compact. Hence $\varphi_i:=\dist(\cdot, U_i^c)$ is continuous for each $i.$ Let $\varphi:=\max_i \varphi_i.$ Then $\varphi$ is continuous. Note $\varphi(x)>0$ for all $x.$ Since $X$ is compact, $\varphi$ takes on an absolute minimum value $2\delta>0.$ If $B_\delta(x)\cap U_i^c \neq \emptyset$ for all $i,$ then $\varphi(x)\leq \delta,$ a contradiction. Hence $B_\delta(x)\subset U_i$ for some $i.$
\end{proof}

\begin{proposition} If $X$ is connected, then for any $\epsilon>0$, $\epsilon$-net $\mathcal{N}$ and $\eta\geq 1,$ both $G^p_{2\eta\epsilon}(\mathcal{N})$ and $G^c_{\eta\epsilon}(\mathcal{N})$ are connected. Moreover, $X$ is connected if for all $\epsilon>0$ and any $\epsilon$-net $\mathcal{N}$, either $G^p_{2\eta\epsilon}(\mathcal{N})$ or $G^c_{\eta\epsilon}(\mathcal{N})$ are connected.
\end{proposition}

\begin{proof}
Suppose $X$ is connected. Let $\mathcal{N}$ be an $\epsilon$-net in $X.$ Let $G=(E(G),V(G))$ be the $\eta\epsilon$-covering graph on $\mathcal{N}.$ Note we only need to show connectedness of $G$ since $2\eta\epsilon$-proximity graph contains all of the edges of the $\eta\epsilon$-covering graph. Let $x_0\in \mathcal{N}.$ Let $A$ be the set of $x\in \mathcal{N}$ that are connected to $x_0$ by a path in the graph $G.$ Let $U=\cup_{x\in A} B_{\epsilon}(x).$ Note $x_0\in U$ and $U$ is open. Suppose $A\neq \mathcal{N}.$ Let $V=\cup_{y\in \mathcal{N}\setminus A} B_{\epsilon}(y).$ Then $V$ is non-empty and open. Moreover, since $\mathcal{N}$ is an $\epsilon$-net, $U\cup V=X.$ Also, if $B_{\epsilon}(x)\cap B_{\epsilon}(y)\neq \emptyset$ for some distinct $x,y\in \mathcal{N},$ then $x\sim y$ since $\eta\geq 1$. Therefore $U$ and $V$ are disjoint, a contradiction. Hence $A=\mathcal{N}$ and $G$ is connected.

Conversely, suppose for any $\epsilon>0$ and any $\epsilon$-net $\mathcal{N}$ that either the $2\epsilon$-proximity graph or the $\epsilon$-covering graph are connected, but $X$ is not connected. We may assume only the $2\epsilon$-proximity graph is connected, since it contains at least the edges of the $\epsilon$-covering graph for any $\epsilon$. Let $U,V$ be open, disjoint, and non-empty in $X$ with $U\cup V=X.$ By the Lebesgue number lemma, there exists a $\delta>0$ such that for any $x\in X$, either $B_\delta(x)\subset U$ or $B_\delta(x)\subset V.$ Let $x_0\in U$ and $y_0\in V.$ Since $B_\delta(x_0)\subset U,$ $d(x_0,y_0)\geq \delta$. Then, by Proposition \ref{netexistance}, we may choose a $\delta/2$-net $\mathcal{N}$ in $X$ containing $x_0$ and $y_0.$ Since the $\delta$-proximity graph on $\mathcal{N}$ is connected, we may choose a path $\gamma=(x_k)_{k=0}^{n-1}$ in the graph with $x_{n-1}=y_0.$ Let $k$ be the largest integer with $x_k\in U.$ Then $x_{k+1}\in V$ and $d(x_k,x_{k+1})<\delta.$ But $B_{\delta}(x_k)\subset U$ by choice of $\delta,$ a contradiction. Therefore $X$ is connected.
\end{proof}

\begin{definition} A metric space $(X,d)$ is called \textit{doubling} if there exists a constant $D>0$ such that for any $x\in X$ and for any $r>0$, $B_{2r}(x)$ can be covered by at most $D$ balls of radius $r.$\end{definition} Recall a Borel measure $\mu$ on  a metric space $(X,d)$ is doubling if there exists a $C>0$ such that for all $r>0$ and $x\in X$, $0<\mu(B_{2r}(x))\leq C\mu(B_r(x))<\infty$. 

If $X$ is compact and has a doubling measure, then $X$ is doubling. Indeed, let $x\in X$ and $r>0.$ Let $(x_i)_{i=1}^n$ be an $r$-net in $B_{2r}(x).$ Then since $B_{2r}(x)\subset B_{4r}(x_i)$ for every $i$ and since $\mu$ is doubling, $\mu(B_{2r}(x))\leq C^3\mu(B_{r/2}(x_i)).$ But since the $B_{r/2}(x_i)$ are 
disjoint, $\frac{n}{C^3}\mu(B_{2r}(x))\leq \mu(\cup_{i=1}^n B_{r/2}(x_i)) \leq \mu(B_{4r}(x)) \leq 
C\mu(B_{2r}(x)).$ Hence $n$ is at most $C^4.$ It follows that $X$ is doubling with $D=C^4.$ 
In fact, it was shown in \cite{volb} that the converse is also true. In more generality, it is known that the equivalence between the existence of doubling measures and the doubling condition holds for complete metric spaces \cite{luuk}\cite{kaenmaki}.

The following proposition shows that a uniform upper bound on the degree for all proximity graphs is equivalent to the doubling condition. Note that since $\deg_{G^c_{\rho\epsilon/2}(\mathcal{N})}\leq \deg_{G^p_{\rho\epsilon}(\mathcal{N})},$ the ``only if" part of the following proposition also holds for the covering graph $G^c_{\eta\epsilon}(\mathcal{N})$ with $\eta=\rho/2.$

\begin{proposition}
Let $\rho\geq 2.$ Then $(X,d)$ is doubling if and only if there exists a constant $M$ such that $\deg_{G^p_{\rho\epsilon}(\mathcal{N})}\leq M$ for all $\epsilon>0$ and any $\epsilon$-net $\mathcal{N}$. 
\end{proposition}
\begin{proof}

Suppose $X$ is doubling with doubling constant $D.$ Let $\mathcal{N}$ be an $\epsilon$-net. Let $x\in \mathcal{N}.$ Let $B:=B_{\rho\epsilon}(x)$ and $\mathcal{N}_x=\{y\in \mathcal{N}\;|\;y\sim x\}.$ Note that $\mathcal{N}_x\subset B_{\rho\epsilon}(x).$ Let $n$ be a positive integer with $\rho\leq 2^{n-1}.$ Then by the doubling condition we may choose $(B_{\frac{\rho\epsilon}{2^n}}(x_i))_{i=1}^M$ to be a cover of $B_{\rho\epsilon}(x)$ with $M\leq D^{n}.$  If $y,z\in \mathcal{N}_x$ and $y,z\in B_{\rho\epsilon/2^n}(x_i),$ then $d(y,z)<\rho\epsilon/2^{n-1}\leq\epsilon,$ so that $y=z.$ It follows by the pigeonhole principle that $\deg_{G^p_{\rho\epsilon}(\mathcal{N})}(x)\leq M.$

Conversely, suppose $M>0$ such that for any $\epsilon$-net $\mathcal{N},$ $\deg_{G^p_{\rho\epsilon}(\mathcal{N})}(x)\leq M$ for all $x\in \mathcal{N}.$ Let $r>0, x\in X.$ Let $\mathcal{N}$ be an $r/\rho$-net containing $x.$ Then if $z\in B_r(x)\cap \mathcal{N},$  then also $z\sim x.$ Hence $B_r(x)$ may be covered by at most $M$ balls of radius $r/\rho\leq r/2.$ Since $r>0$ and $x\in X$ were arbitrary, the doubling condition holds.\end{proof}

\subsection{Tilings and Dyadic Cubes}
Throughout this section let $(X,d)$ be a compact metric space. 
\begin{definition}By an \textit{$\epsilon$-tiling} of $X$, we mean a finite collection $(T_x)_{x\in \mathcal{N}}$ of Borel subsets of $X$ such that $\mathcal{N}$ is an $\epsilon$-net, $(T(x))_{x\in\mathcal{N}}$ is a partition of $X,$ and $B_{\epsilon/2}(x)\subset T(x)\subset B_{\epsilon}(x)$ for all $x\in \mathcal{N}.$ \end{definition}

\begin{lemma} An $\epsilon$-tiling of $X$ exists.
\end{lemma}
\begin{proof}
Let $\mathcal{N}$ be an $\epsilon$-net. For $x\in \mathcal{N}$ define the open and closed \textit{Voronoi cells} $U(x)$ and $V(x),$ respectively, by $U(x)=\{y\in X\;|\;d(x,y)< d(x,z)\;\mbox{for all\;} z\in \mathcal{N}\setminus\{x\}\}$ and $V(x)=\{y\in X\;|\;d(x,y)\leq d(x,z)\;\mbox{for all\;} z\in \mathcal{N}\}$. Then clearly the $(U(x))_{x\in \mathcal{N}}$ are disjoint and $\cup_{x\in \mathcal{N}}V(x)=X.$ Note further, that since the $(B_{\epsilon/2}(x))_{x\in\mathcal{N}}$ are disjoint, we have $B_{\epsilon/2}(x)\subset U(x)$ for all $x\in \mathcal{N}.$ Similarly, since $\cup_{x\in \mathcal{N}} B_{\epsilon}(x)=X,$ if $x\in \mathcal{N}$ and $y\in X$ with $d(x,y)\geq \epsilon,$ then $y\in B_{\epsilon}(z)$ for some $z\in\mathcal{N}.$ It follows $y\notin V(x).$ Hence $V(x)\subset B_{\epsilon}(x).$ Since $\mathcal{N}$ is finite, as $X$ is compact, label $\mathcal{N}=\{x_1,x_2,...,x_n\},$ where $n=\sharp \mathcal{N}.$ Then for $x_k\in \mathcal{N},$ let 
\[T(x_k):=\{y\in X\;|\;k=\min\{j\in I\;|\;y\in T(x_j)\}\}.\] Then $T(x_1)=V(x_1), T(x_2)=V(x_2)\setminus V(x_1),...,T(x_n)=V(x_n)\setminus \cup_{j<n} V(x_j).$ Note, we have $U(x_k)\subset T(x_k)\subset V(x_k)$ for all $k\in I.$ It follows that $B_{\epsilon/2}(x)\subset T(x)\subset B_{\epsilon}(x)$ for all $x\in \mathcal{N}.$
\end{proof}
\vspace*{1 in}

\begin{figure}[H]
\centering
\includegraphics[scale=.38]{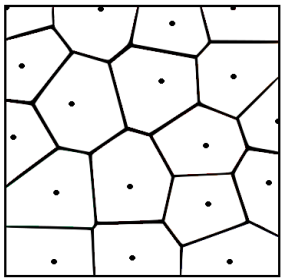}
\caption{A Vornoi tiling of a plane region.}
\end{figure}

Although we will not have much use for it in the sequel, we recall the proof of the main theorem in \cite{kaenmaki} on the existence of a decomposition of a compact\footnote{In \cite{kaenmaki} the proof is valid for an arbitrary doubling space. The proof of the version below in the compact case does not make use of a doubling condition.} metric space by a sequence of nested ``dyadic cubes."\footnote{The name ``dyadic cubes" is meant to reflect the usual nested dyadic cube decomposition of $\mathbb{R}^n$ by sets of the form $\prod_{i=1}^n{[j_i/2^{j_0}, (j_i+1)/2^{j_0})},$ where $(j_i)_{i=0}^\infty\subset \mathbb{Z}^{n+1}.$} A construction of a nested dyadic decomposition, except for a set of measure zero, of a doubling metric measure space by open sets, was given earlier by \cite{christ}. The proof we provide is fairly simple, however the proof in \cite{christ} has the advantage of providing control on the measure of the ``cubes" near the boundaries. The construction may be used to define what is called a ``Michon tree'' in \cite{pearson}. With suitable weights, the original space may be recovered as a type of Gromov hyperbolic boundary of a certain graph defined from the tree. See \cite{Piaggio} for details. 

\begin{theorem} \label{dyadic}(Theorem 2.1 from \cite{kaenmaki}) Suppose $(X,d)$ is a compact metric space and $(\epsilon_k)_{k=0}^\infty$ is a sequence of positive numbers strictly decreasing to zero. Then for each $k\in \mathbb{N}$ there exists an $\epsilon_k$-net $\mathcal{N}_k=(x_{k,i})_{i=1}^{M_k}$ and a  partition $(Q_{k,i})_{i=1}^{M_k}$ of $X$ by Borel sets with $x_{k,i}\in Q_{k,i}$ such that the following properties hold: \begin{equation*}
\begin{split}
\mbox{(1)}&\; \mathcal{N}_{k}\subset \mathcal{N}_{k+1}\mbox{\;for all\;}k;\\
\mbox{(2)}&\; \mbox{For\;}k_1<k_2\mbox{\;and all\;} i_2,\; \sharp\{i_1\;|\;Q_{k_1,i_1}\cap Q_{k_2,i_2}\neq \emptyset\}=1; \mbox{\;and}\\
\mbox{(3)}&\;\mbox{If\;}0<r<1/3\;\mbox{and\;} \epsilon_k=r^k\;\mbox{for all\;}k, B_{(\frac{1-3r}{2-2r})r^k}(x_{k,i})\subset Q_{k,i} \subset B_{(\frac{1}{1-r})r^k}[x_{k,i}]\mbox{\;for all\;} k,i.
\end{split}
\end{equation*}
\end{theorem}
 
\begin{proof} Let $x_0\in X.$ Let $\mathcal{N}_0$ be a $\epsilon_0$-net in $X$ containing $x_0.$ Then having chosen an $\epsilon_j$ net $\mathcal{N}_j$ for each $j\leq k$ with $\mathcal{N}_j\subset \mathcal{N}_{j+1}$ for $j<k,$ extend $\mathcal{N}_k$ to an $\epsilon_{k+1}$-net $\mathcal{N}_{k+1},$ which is possible by Proposition \ref{netexistance}. Note each $\mathcal{N}_k$ has a finite number, say $M_k,$ of elements since $X$ is compact. Let $\mathcal{N}_k=(x_{k,i})_{i\in I_k},$ where $I_k=\{i\in \mathbb{Z}^+\;|\;i\leq M_k\}.$ Let $P:=\coprod_{k\in\mathbb{N}} I_k=\cup_{k\in\mathbb{N}}(\{k\}\times I_k).$ For each $k\in \mathbb{N}$, define $\varphi_{k}:I_{k+1}\rightarrow I_k$ by 
\[\varphi_k(i)=\min\{i'\in I_k\;|\;d(x_{k+1,i},x_{k,i'})=\min_{q\in I_k} d(x_{k+1,i},x_{k,q})\}.\]Then we define a partial order relation $\leq$ on $P$ as follows.  If $k\in \mathbb{N}$ and $i\in I_{k+1},$ set \[(k+1,i)\leq (k,\phi_k(i)).\] We then extend to the smallest partial order including these relations. That is we set $(k,i)\leq (k,i)$ for each $(k,i)\in P.$ Then if $p,q\in P,$ we set $p\leq q$ if there exists an $n\in \mathbb{Z}^+$ and $p_1,p_2,...,p_n\in P$ with $p_1=p,$ $p_n=q,$ and $p_i\leq p_{i+1}$ for $1\leq i <n.$ Then $\leq$ is a partial order on $P.$ Let $I_{-1}:=\{1\}$ and $Q_{-1,1}:=X.$ Then define $\varphi_{-1}:I_0\rightarrow I_{-1}$ to be constant equal to $1.$ 

For $k\in \mathbb{N},$ we define $(Q_{k,i})_{i\in I_k}$ inductively by \[Q_{k,i}=Q_{k-1,\varphi_{k-1}(i)}\cap \overline{\{x_{l,j}|\;(l,j)\leq (k,i)\}}\setminus \cup_{j<i} Q_{k,j}.\] 

We first show for each $k\in \mathbb{N}$ that $(Q_{k,i})_{i\in I_k}$ is a partition of $X$ by Borel sets. Let $A_{k,i}:=\overline{\{x_{l,j}\;|\;(l,j)\leq (k,i)\}}.$ We proceed by induction. For $k=0,$ since $\cup_{i\in I_0} A_{0,i}$ is dense, $X=\overline{\cup_{i\in I_0}A_{0,i}}\subset \cup_{i\in I_0} \overline{A_{0,i}},$ where the last inclusion holds since $I_0$ is finite so that $\cup_{i\in I_0} \overline{A_{0,i}}$ is closed. Then the $(Q_{0,i})_{i\in I_0}$ are the disjointization\footnote{If $(A_k)_{k=1}^\infty$ is a collection of subsets of a set $S,$ then if $B_k=A_k\setminus (\cup_{j<k} A_j)=A_k\setminus (\cup_{j<k} B_j)$ for $k\in\mathbb{Z}^+,$ (where $A_1=B_1$ since $\cup\emptyset=\emptyset$) then the $(Bk)_{k=1}^\infty$ are disjoint with $\cup_{k=1}^n B_k =\cup_{k=1}^n A_k.$} of the Borel sets $(\overline{A_{0,i}})_{i\in I_0};$ hence they form a Borel partition. Now let $k\in \mathbb{Z}^+$, and suppose that $(Q_{k-1,i})_{i\in I_{k-1}}$ is a Borel partition of $X.$ Since by the induction hypothesis $Q_{k-1,i}$ is Borel for each $i\in I_{k-1}$, $\overline{A_{i,k}}$ is Borel for each $i\in I_k,$ and the $(Q_{i,k})_{i\in I_k}$ are the result of a disjointization procedure applied to countably many Borel sets, it follows $(Q_{i,k})_{i\in I_k}$ is a disjoint collection of Borel sets. It remains only to show that their union is $X.$  Since $\mathcal{N}_{k-1}\subset \mathcal{N}_k,$ if $x_{k-1,j}\in \mathcal{N}_{k-1}$ is labeled $x_{k,i}\in \mathcal{N}_k,$ then $j=\varphi_{k-1}(i).$ Hence $\varphi_{k-1}^{-1}(k-1,j)\neq \emptyset$ for all $j\in I_{k-1}.$ It follows that the sets $(\varphi_{k-1}^{-1}(k-1,j))_{j\in I_{k-1}}$ partition $I_k.$ Hence, since $\cup_{j\in I_{k-1}} Q_{k-1,j}=X$ by the induction hypothesis, 
\begin{equation*}
\begin{split}
&\cup_{i\in I_{k}}(Q_{k-1,\varphi_{k-1}(i)}\cap \overline{A_{k,i}})=\cup_{j\in I_{k-1}}\cup_{i\in \varphi_{k-1}^{-1}(j)}(Q_{k-1,j}\cap \overline{A_{k,i}})\\
&=\cup_{j\in I_{k-1}}\left (Q_{k-1,j}\cap (\cup_{i\in \varphi^{-1}(k-1,j)}\overline{A_{k,i}})\right )\supset \cup_{j\in I_{k-1}} Q_{k-1,j}=X.
\end{split}
\end{equation*} 
Therefore the $(Q_{k,i})_{i\in I_k}$ form a Borel partition of $X.$

We now show that the $Q_{i,k}$ satisfy properties (1) and (2). Property (1) holds by construction. Suppose $k_1,k_2\in \mathbb{N}$ with $k_1<k_2$. Define $\varphi_{k_2,k_1}:=\varphi_{k_2-1}\circ\varphi_{k_2-2}\circ...\circ\varphi_{k_1}.$ Let $i_2\in I_{k_2}.$ Then by construction $Q_{k_2,i_2}\subset Q_{k_1,\varphi_{k_2,k_1}(i_2)}.$ Since the $(Q_{k_2,j})_{j\in I_{k_2}}$ are disjoint, $Q_{k_2,i_2}$ non-trivially intersects exactly one $Q_{k_2,j}$ for $j\in I_{k_2}.$ Hence (2) holds. 

Now suppose $0<r<1/3$ and $\epsilon_k=r^k$ for all $k.$ We show $(3)$ holds. Note that if $(k_2,i_2)\leq (k_1,i_1)$ for $k_2\geq k_1$ then $i_1=
\varphi_{k_2,k_1}(i_2).$ It follows that the $(A_{k,i})_{i\in I_k}$ are disjoint. For $k=0,$ 
$A_{k,i}\subset Q_{k,i}$ for all $i\in I_0.$ Suppose $A_{k,i}\subset Q_{k,i}$ for all $i\in I_k$ 
where $k\in \mathbb{N}.$ Then if $i\in I_{k+1},$ $A_{k,\phi_k(i)}\subset Q_{k,\phi_k(i)}.$ Hence 
$A_{k+1,i}\subset Q_{k,\phi_{k(i)}}.$ Since the $(A_{k+1,j})_{j\in I_{k+1}}$ are disjoint, it 
follows that $A_{k+1,i}\subset Q_{k+1,i}.$ Hence, by induction, $A_{k,i}$ subset $Q_{k,i}$ for all 
$(k,i)\in P.$ For $(k,i)\in P$ and $\epsilon>0,$ suppose $B_\epsilon(x_{k,i})\cap \mathcal{N}\subset A_{k,i}.$ Then for any $j\in I_k,$ if $B_\epsilon(x_{k,i})\cap Q_{k,j}\neq \emptyset,$ since $A_{k,j}$ is dense in $Q_{k,j},$ we have $\emptyset \neq B_{\epsilon}(x_{k,i})\cap A_{k,j} \subset B_{\epsilon}(x_{k,i})\cap \mathcal{N}\subset A_{k,i}.$ It follows that $i=j.$ So $B_\epsilon(x_{k,i})\subset Q_{k,i}.$ Now if $(k',i')\notin A_{k,i},$ we may assume $k'>k$ as the $\mathcal{N}_j$ are nested, then \begin{equation*}
\begin{split}
r^k/2&<d(x_{k,i},x_{k,\varphi_{k',k}(i')})\leq d(x_{k,i},x_{k',i'})+\sum_{l=0}^{k'-k-1}d(x_{k+l+1,\varphi_{k',k+l+1}(i')},x_{k+l,\varphi_{k',k+l}(i')})\\
&\leq d(x_{k,i},x_{k',i'})+r^k(\frac{r}{1-r}).
\end{split}
\end{equation*}
Hence $B_{r^k(1/2-r/(1-r))}(x_{k,i})\subset Q_{k,i}.$ Conversely, since $Q_{k,i}\subset \overline{A_{k,i}},$ if $(k',i')\leq (k,i)$ with $k'\geq k,$ then since $i=\varphi_{k',k}(i'),$ $d(x_{k,i},x_{k',i'})\leq \sum_{l=0}^{k'-k-1} d(x_{k+l+1,\varphi_{k',k+l+1}(i')},x_{k+l,\varphi_{k',k+l}(i')})\leq r^k(\frac{1}{1-r}).$ Hence $Q_{k,i}\subset B_{r^k/(1-r)}[x_{k,i}].$
\end{proof}

In fact, if $1/3\leq r<1,$ then as remarked in \cite{kaenmaki}, (3) also holds but with a constants $c_1,c_2>0$ such that 
$B_{c_1r^k}(x_{k,i})\subset Q_{k,i} \subset B_{c_2r^k}[x_{k,i}]\mbox{\;for all\;} k,i.$ Indeed, if $1/3\leq r<1,$ then let $j=\min\{i\in \mathbb{Z}^+\;|\;r^i<1/3\}.$ Note first we still have $c_2=\frac{1}{1-r}$ since only $0<r<1$ was needed for that bound. Let $(\mathscr{N}_k)_{k=0}^\infty$ be an increasing sequence of $r^k$-nets. Then $(\mathscr{N}_{jk})_{k=0}^\infty$ is an increasing sequence or $(r^j)^k$-nets. However, $r^{j\lceil k/j \rceil}\leq r^k,$ where $\lceil \cdot \rceil$ is the ceiling function. Hence, by (3) applied to $\epsilon_k=r^{jk},$ we have, noting that $N_k\subset N_{j\lceil k/j \rceil},$ $B_{(\frac{1-3r^j}{2-2r^j})r^{j\lceil k/j \rceil}}(x_{k,i})\subset Q_{k,i}.$ Then, as $j\lceil k/j \rceil \leq j(k/j+1)=k+j,$ we may take $c_1=(\frac{1-3r^j}{2-2r^j})r^j.$

Assuming the doubling condition, M. Christ in \cite{christ} proves the following version of the dyadic cube decomposition. Notably there is control of the measure of the neighborhoods the boundaries of the cubes. The proof, found in \cite{christ}, is rather lengthy and is omitted here.

\begin{proposition} (M. Christ in \cite{christ}) Suppose $X$ is doubling with doubling measure $\mu.$ Then there exists a $\delta\in (0,1)$, constants $\nu, C_1,C_2,C_3>0,$ and a collection $(U_{k,i})_{k\in\mathbb{Z},i\in I_k}$ of open subsets of $X$ such that 
\begin{equation*}
\begin{split}
\mbox{(1)}&\; \mbox{The\;}(U_{k,i})_{i\in I_k}\;\mbox{are disjoint and\;}\mu(\left(\cup_i U_{k,i}\right)^c)=0\mbox{\;for all\;}k;\\
\mbox{(2)}&\; \mbox{For\;}k_1<k_2\mbox{\;and all\;} i_2,\; \sharp\{i_1\;|\;U_{k_1,i_1}\cap U_{k_2,i_2}\neq \emptyset\}=1; \\
\mbox{(3)}&\;\mbox{For all\;}k\;\mbox{there exists a\;}\delta^k \mbox{-net\;}(x_{k,i})_{i\in I_k}\mbox{\;such that for all\;}i\in I_k\\& B_{C_1\delta^k}(x_{k,i})\subset U_{k,i} \subset B_{C_2\delta^k}(x_{k,i});\mbox{\;and}\\
\mbox{(4)}&\; \mbox{For all \;}k,i,\;\mu(\{x\in U_{k,i}\;|\;\dist(x,U_{k,i}^c)\leq t\delta^k\})\leq C_3t^\eta\mu(U_{k,i}).
\end{split}
\end{equation*}
\end{proposition}
Note that (4) implies that $\mu(\partial U_{k,i})=0$ for all $k,i.$ Indeed, let $N_i=\{j\in I_k\;|\;\overline{U_{j,k}}\cap \overline{U_{i,k}}\neq\emptyset\}.$ The doubling condition and (3) ensure that the set $N_i$ is finite. Then $\mu(\partial U_{k,i})\leq \inf_{t>0} \mu(\cup_{j\in N_i}\{x\in U_{k,i}\;|\;\dist(x,U_{k,i}^c)\leq t\delta^k\})\leq \inf_{t>0}C_3\sharp(N_i)t^{\eta}\mu(X)=0.$  

Additionally, we note in passing that if $X$ is Ahlfors regular of dimension $\alpha$, then it was proven in \cite{Gigante} that there exists constants $C_1,C_2>0$ such that for any positive integer $N,$ $X$ may be partitioned into $N$ sets of equal measure, each set containing a ball of radius $C_1N^{-1/\alpha}$ and contained in a ball 
or radius $C_2N^{-1/\alpha}.$  The proof uses the dyadic cube decomposition of $X.$ The constants $C_1,C_2$ depend on the constants $c_1,c_2$ in the paragraph following Theorem \ref{dyadic} and the constants in the definition of Ahlfors regularity. The proof may be found in \cite{Gigante}.
\section{Probability}

In this section we summarize aspects of probability theory that will be used in the sequel, emphasizing rudiments of the theory of random walks on a metric space. 

\subsection{General Theory}

A measurable space is a pair $(\Omega,\mathscr{E})$ consisting of a non-empty set $\Omega$ and a $\sigma$-algebra $\mathscr{E}$ of subsets of $\Omega.$ Given a measurable space $(\Omega,\mathscr{E}),$ a positive measure $\mathbb{P}$ on $\mathscr{E}$ with $\mathbb{P}(E)=1$ is called a probability measure. The triple $(\Omega, \mathscr{E},\mathbb{P})$ is called a probability space. The set $\Omega$ is often called a sample space, with $\mathscr{E}$ called the collection of events. 

If $\Omega$ is a topological space, by $\mathscr{B}_\Omega$ we mean the Borel $\sigma$-algebra generated by the topology of $\Omega.$ If $(\Omega_i,\mathscr{M}_i)_{i\in I}$ is a collection 
of measurable spaces, the product $\sigma$-algebra $\otimes_{i\in I}
\mathscr{M}_i$ of the $(\mathscr{M}_i)_{i \in I}$ is the smallest $\sigma$-algebra of subsets of $\prod_{i\in I} \Omega_i$ making each projection 
$\pi_j :\prod_{i\in I} \Omega_i\rightarrow \Omega_j,$ defined by $x\mapsto x(j)=x_j$ for $x\in \prod_{i\in I}\Omega_i,$ $
(\otimes_{i\in I}\mathscr{M}_i, \mathscr{M}_j)$-measurable for $j\in I.$

Given another measurable space $(\Omega',\mathscr{E}')$, a $\Omega'$-valued random variable $\mathcal{X}$ on $\Omega$ is an ($(\mathscr{E},\mathscr{E}')$-) measurable function $\mathcal{X}:\Omega\rightarrow \Omega'.$ 
If $\mathbb{P}$ is a probability measure on $\Omega,$ the distribution of ${\mathcal{X}}:\Omega\rightarrow \Omega'$ is the probability measure $\mathbb{P}_{{\mathcal{X}}}$ on $\Omega'$ defined by $\mathbb{P}_{{\mathcal{X}}}(A)=\mathbb{P}({\mathcal{X}}^{-1}(A))=\mathbb{P}({\mathcal{X}}\in A)$ for $A\in\mathscr{E}'.$  In the case that $(\Omega',\mathscr{E}',\mu)$ is a $\sigma$-finite measure space and $\mathbb{P}_{\mathcal{X}} \ll \mu,$ the Radon-Nikodym derivative $d\mathbb{P}_{\mathcal{X}}/{d\mu}$ is called a ($\mu$-) density of $X.$ Part of the probabilistic point of view is to define random variables by their distribution. One example, for which we will have particular use, is an exponential random variable of mean $\delta>0.$ It takes values in $[0,\infty)$ and its density with respect to Lebesgue measure $dt$ is $t\mapsto \frac{1}{\delta}e^{-t/\delta}$. Hence its distribution is defined by $A\mapsto \int_A \frac{e^{-t/\delta}}{\delta}dt,$ where $A\subset [0,\infty)$ is Borel measurable. 

Given a probability measure $\mathbb{P}$ on $\Omega$ and a $\mathbb{P}$-integrable $\overline{\mathbb{R}}$-valued random variable ${\mathcal{X}},$ its ($\mathbb{P}$-) expectation, $\mathbb{E}({\mathcal{X}}),$ is the integral $\int_\Omega {\mathcal{X}}
(\omega)d\mathbb{P}(\omega).$ For example, if $\mathcal{X}$ is exponentially distributed with mean $\delta,$ then, changing variables, $
\mathbb{E}({\mathcal{X}})=\int_\Omega \mathcal{X}(\omega)d\mathbb{P}(\omega)=
\int_0^\infty t\frac{d \mathbb{P}_{\mathcal{X}}}{dt}dt=\int_0^\infty \frac{t}{\delta}e^{-t/ \delta}dt=\delta$. If $\mathcal{X}$ is a random variable on $\Omega$ and 
$A\in\mathscr{E},$ then by $\mathbb{E}({\mathcal{X}},A)$ we mean $\int_A {\mathcal{X}}
(\omega)d\mathbb{P}(\omega).$ 

For ${\mathcal{X}}$ a random variable on $(\Omega,\mathbb{P})$ with $\mathbb{E}(|{\mathcal{X}}|)<\infty$, let ${\mathcal{X}}\mathbb{P}$ be the signed measure $A\mapsto \mathbb{E}({\mathcal{X}},A).$ Note if $\mathscr{A}$ is a sigma-algebra contained in $\mathbb{P},$ we 
have ${\mathcal{X}}\mathbb{P}|_{\mathscr{A}}\ll \mathbb{P}|_{\mathscr{A}}.$  The conditional expectation $\mathbb{E}
({\mathcal{X}}|\mathscr{A})$ is the Radon-Nikodym derivative $d(\mathcal{X}
\mathbb{P})/d(\mathbb{P}|_\mathscr{A}).$ 

\begin{definition}
A Polish space is a separable, completely metrizable topological space; which means that it possesses a countable basis, and there exists at least one metric that is complete and with metric topology equal to the original topology.\end{definition}

If $\Omega$ is a Polish space, let $C_b(\Omega)$ be the Banach space of bounded (real-valued) continuous functions on $E.$  Let $P(\Omega)$ be the space of Borel probability measures on $\Omega.$ Note $P(\Omega)$ is a subset of the dual of $C_b(\Omega).$ We endow $P(\Omega)$ with the the subspace weak-* topology. It can be shown, see \cite{stroock} for example, that $P(\Omega)$ is also Polish. If in addition, $\Omega$ is compact, it can also be shown that in that case $P(\Omega)$ is also compact. Convergence with respect to the topology of $P(\Omega)$ is called weak convergence of probability measures.

Before continuing, we state a few results on the regularity of Borel measures on metric spaces.

\begin{definition} A Borel measure $\nu$ is called outer regular if for every Borel set $A$, 
\[\nu(A)={\inf}_{\overline{\mathbb{R}}_+}\{\nu(U)\;|\;A\subset U,\; U \mbox{ open}\}.\] It is called inner regular if for any Borel set $A$, 
\[\nu(A)={\sup}_{\overline{\mathbb{R}}_+}\{\nu(F)\;|\;A\supset F,\;F \mbox{ compact}\}.\] The Borel measure $\nu$ is called regular if it is both outer and inner regular.\end{definition}
We state the following classical results on regularity of Borel measures. We recall the proofs in Appendix A for completeness, following \cite{van} and \cite{billingsley}.
\begin{proposition} Suppose $\nu$ is a finite Borel measure on a metric space $X$. Then $\nu$ is outer regular, and \[\nu(A)={\sup}_{\overline{\mathbb{R}}_+}\{\nu(F)\;|\;A\supset F,\;F \mbox{ closed}\}.\] If $X$ is $\sigma$-compact or if $X$ is separable and complete, then $\nu$ is also inner regular. \label{reg}
\end{proposition}  In particular, if $\mathbb{P}$ is a Borel probability measure on a Polish space $\Omega,$ then $\mathbb{P}$ is regular.

Suppose again $A,B\in\mathscr{E}$ with $\mathbb{P}(A)\neq 0.$ Then let $\sigma(A)=\{\emptyset,A,A^c,E\},$ that is $\sigma(A)$ is the smallest $\sigma$-algebra containing $A.$ Then $\mathbb{E}({1}_B|\sigma(A))=\mathbb{P}(B|A){1}_A+\mathbb{P}(B|A^c){1}_{A^c}.$ In particular, if $\omega\in A,$ $\mathbb{E}({1}_B|\sigma(A))(\omega)=\mathbb{P}(B|A).$ Moreover, for $C\in\sigma(A),$ \[\int_C\mathbb{E}({1}_B|\sigma(A))(\omega)d\mathbb{P}(\omega)=\mathbb{P}(C\cap B).\] Hence, we might generalize conditional probability as follows. Suppose $\mathscr{F}$ is a sub-$\sigma$-algebra of $\mathscr{E}.$ Then for $B\in\mathscr{E},\omega\in \Omega,$ we might interpret $\mathbb{E}({1}_B|\mathscr{F})(\omega),$ where we have chosen a particular element in the equivalence class, as the conditional probability of $B$ given $\omega$ with respect only to the information contained in $\mathscr{F}.$ Then if ${\phi}:\Omega\rightarrow S$ is a random variable, let $\mathscr{F}$ be the $\sigma$-algebra generated by ${\phi}.$ Then, since the measure $A\mapsto \int_{{\phi}^{-1}(A)}\mathbb{E}({1}_B|\mathscr{F})d\mathbb{P}(\omega)$ is absolutely continuous with respect to the distribution $\mathbb{P}_{{\phi}}$ of ${\phi},$ there is a measurable ${\psi}_B$ on $S$ with $\int_A{\psi}_B(s)d\mathbb{P}_{\phi}(s)=\mathbb{P}({{\phi}}^{-1}(A)\cap B).$ Informally, we would like to interpret ${\psi}_B(s)$ as a conditional probability of $B$ given ${\phi}=s.$  However, the problem is that we would like for each $s$ the map $B\mapsto {\psi}_B(s)$ to be a Borel probability measure. Since members of the equivalence classes for the given densities are only unique $\mathbb{P}$ a.e., it is not immediately clear how to construct such conditional probabilities. The following result of \cite{cond} shows that for any Hausdorff space supporting a regular Borel probability measure, conditional probabilities always exist. See Theorem 5.3 of \cite{kallenberg}.

\begin{proposition} \label{cond} Suppose $\Omega$ is a Hausdorff space with regular Borel probability measure $\mathbb{P}$. Suppose $(S,\Sigma)$ is a measurable space and ${\phi}:\Omega\rightarrow S$ is measurable. Then there is a function $\mathscr{B}_{\Omega}\times S\ni (A,s)\mapsto \mathbb{P}^s(A)\in [0,1]$ such that 
\begin{equation*}
\begin{split}
\mbox{\bf{(a)}}&\;A\mapsto \mathbb{P}^s(A)\;\mbox{is a regular probability measure for all\;}s\in S,\\
\mbox{\bf{(b)}}&\;s\mapsto \mathbb{P}^s(A) \mbox{\;is measurable for all\;} A\in \mathscr{B}_\Omega,\;\mbox{and}\\
\mbox{\bf{(c)}}&\int_{A}\mathbb{P}^s(B)d\mathbb{P}_{{\phi}}(s)=\mathbb{P}({\phi}^{-1}(A)\cap B)\;\mbox{for all\;} A \in \Sigma, B \in \mathscr{B}_\Omega.\end{split}
\end{equation*}

\end{proposition}

Let $\sigma(\phi)$ be the sub-$\sigma$-algebra of $\mathscr{B}_{\Omega}$ generated by $\phi.$ Then properties $(b)$ and $(c)$ imply that $s\mapsto \mathbb{P}^s(A)$ is equal to $\mathbb{E}(\chi_A|\sigma(\phi))$ almost everywhere.  By Lemma \ref{reg}, the above proposition applies, in particular, to a Borel probability measure on a Polish space. The function $\mathscr{B}_{\Omega}\times S\ni (A,s)\mapsto \mathbb{P}^s(A)\in [0,1]$ appearing in the proposition is called a regular conditional probability induced by ${\phi}.$

\subsection{Discrete Time Random Walks}

Much of the material in this section may be found in the paper \cite{Dever1} of the author. 

Recall that a metric measure space is a triple $(X,d,\mu)$ consisting of a metric space $(X,d)$ together with a non-negative Borel measure $\mu$ on $X$ of full support. 
For the remainder of the section, we assume $(X,d,\mu)$ is a metric measure space where $(X,d)$ is a compact metric space containing more than a single point. 

Suppose $p$ is a non-negative function in $L^2(X\times X,\mu\otimes\mu)$ with $\sup_{(x,y)\in X\times X}p(x,y)<\infty$, such that for all $x\in X,$ $\int p(x,y) d\mu(y)=1$ and $p(x,x)=0.$ Then for $f\in L^2(X,\mu),$ let \[Pf(x):=\int p(x,y)f(y)d\mu(y).\] Note that $P$ is a compact operator, as it is a Hilbert-Schmidt integral operator. We wish to use $p$ to define a discrete time walk. But first we need a way to create a measure on the space of discrete time paths $X^{\mathbb{N}}$ in $X$ using $p.$ 
We now state the following version of the Kolmogorov Extension Theorem, found on p.523 in \cite{Hitch}. We provide a proof of this classical result in the Appendix.

\begin{proposition} \label{ext} (Kolmogorov Extension Theorem)
Let $(W_t,\Sigma_t)_{t\in T}$ be a family of Polish spaces, and for each finite subset $F$ of $T,$ let $\mu_F$ be a probability measure on $\Omega_F=\prod_{t\in F} W_t$ with its product Borel $\sigma$-algebra $\Sigma_F.$ Assume the family $(\mu_F)$ satisfies the consistency condition that if $F'\subset F$ then $\mu_F\circ \pi_{F,F'}^{-1}=\mu_{F'},$ where $\pi_{F,F'}$ is the coordinate projection of $\Omega_{F}$ onto $\Omega_{F'}.$  Then there is a unique probability on the infinite product $\sigma$-algebra $\bigotimes_{t\in T}\Sigma_t$ that extends each $\mu_F.$
\end{proposition}

We apply the theorem as follows. Let $\nu$ be a Borel probability measure on $X.$ Let $T=\mathbb{N}$ and $(W_t,\Sigma_t)=(X,\mathscr{B}_{X})$ for each $t.$ For $F$ a finite subset of $T,$ let $t_1<t_2<...<t_n$ be a list of the elements of $F.$ Let $A\in\Sigma_F$. If $t_1>0,$ let \[\mu_F(A)=\int_X\int_A\prod_{i=1}^np(x_{i-1},x_i)d\mu^{\otimes n}(x_n,...,x_1)d\nu(x_0),\] where $\mu^{\otimes m}=\otimes_{i=1}^m\mu$ for $m\in \mathbb{Z}^+.$ If $t_1=0,$ let \[\mu_F(A)=\int_A\left(\prod_{i=2}^{n}p(x_{i-1},x_i)\right)d(\mu^{\otimes (n-1)}\otimes \nu)(x_n,...x_1).\]
Since $\int p(x,y)d\mu(y)=1$ and $\nu$ is a probability measure on $X,$ the family $(\mu_F)$ satisfies the consistency condition of Proposition \ref{ext}. Hence there exists a measure $\mathbb{P}^\nu$ on $\Omega:=X^{\mathbb{N}}$ extending each $\mu_F.$ Let $\mathbb{E}^\nu$ denote the $\mathbb{P}^\nu$ expectation. For $x\in X,$ let $\delta_x$ be the Borel measure $\delta_x(A)=\chi_A(x)$ for $A\in \mathscr{B}_X.$ Then let $\mathbb{P}^x:=\mathbb{P}^{\delta_x}$ and $\mathbb{E}^x:=\mathbb{P}^{\delta_x}$ for $x\in X.$ 
The following lemma will allow us to easily show the maps $x\mapsto \mathbb{P}^x(A)$ for $A\in \otimes_{t\in \mathbb{N}}\mathscr{B}_X$ are measurable. A monotone class of subsets of $S$ is a collection $\mathscr{C}$ of sets that is closed under monotone unions or intersections, i.e. if $(A_i)_{i=1}^\infty \subset \mathscr{C}$ with either $A_i\subset A_{i+1}$ for all $i$ or $A_i\supset A_{i+1}$ for all $i$, then $\cup_{i=1}^\infty A_i, \cap_{i=1}^\infty A_i \in \mathscr{C}.$ 
A proof of the following lemma may be found in the Appendix.
\begin{lemma} (Monotone class lemma) \label{monot} Suppose $\mathscr{C}$ is a monotone class on $S$ containing an algebra $\mathscr{A}$ of subsets of $S.$ Then $\sigma(\mathscr{A})\subset \mathscr{C}.$
\end{lemma}

We show the map $x\mapsto \mathbb{P}^x(A)$ is measurable for all $A\in \otimes_{t\in \mathbb{N}} \mathscr{B}_X.$ . Let \[\mathscr{C}=\{A \in \otimes_{t\in \mathbb{N}} \mathscr{B}_X\;|\;x\mapsto \mathbb{P}^x(A)\mbox{\;measurable}\}.\] Then let $\mathscr{A}$ be the algebra of cylinder sets $\pi_F^{-1}(\prod_{i=0}^n A_i)$ where $A_i\in \mathscr{B}_X$ for $i=0,...,n,$ $F\subset \mathbb{N}$ with $\sharp F=n+1,$ $n\in \mathbb{N}.$ By the consistency condition, we may assume $0\in F.$ Then $\mathbb{P}^{x}(A)=\chi_{A_0}(x)\int_{A_1\times...\times A_{n}}p(x,x_1)\prod_{i=2}^np(x_{i-1},x_i)d\mu^{\otimes n}(x_n,...,x_1),$ which clearly is measurable. Hence $\mathscr{C}$ contains the algebra $\mathscr{A}.$ Suppose $(A_k)_{k=1}^\infty \subset \mathscr{C}$ with $A_k\subset A_{k+1}$ for all $k.$ Then let $\phi_n(x)=\mathbb{P}^x(A_n).$ Then since $A_n \uparrow \cup_k A_k,$ by continuity of measure, $\phi_n(x)\rightarrow \mathbb{P}^x(\cup_k A_k)$ for all $x.$ Since each $\phi_n$ is measurable, $\cup_k A_k \in \mathscr{C}.$ Similarly, if $A_k\supset A_{k+1}$ for all $k,$ $\phi_n(x)\rightarrow \mathbb{P}^x(\cap_k A_k)$ by continuity of measure. Hence $\mathscr{C}$ is a monotone class containing $\mathscr{A}.$ Since $\mathscr{A}$ generates $\otimes_{t\in \mathbb{N}}\mathscr{B}_X,$ we have that the map $x\mapsto \mathbb{P}^x(A)$ is measurable for all $A\in \otimes_{t\in \mathbb{N}}\mathscr{B}_X,$ as desired.

Then for $k\in\mathbb{N},$ let $\bm{x}_k:=\pi_k:X^{\mathbb{N}}\rightarrow X$ be the projection onto the $k$th coordinate. Then, as the product $\sigma$-algebra makes all of the projections measurable, $\bm{x}_k$ is a random variable for each $k.$ The process $\bm{x}:=({\bm{x}}_k)_{k=1}^\infty$ is a discrete time random walk on $X.$

Then $\bm{x}$ is a time-homogeneous Markov process\footnote{The Markov property is often defined as $\mathbb{P}^x(\bm{x}_k\in A|(\bm{x}_i)_{i\leq j})=\mathbb{P}^x(\bm{x}_k\in A|\bm{x}_j)$. This has the disadvantages of requiring the time-homogenity property to be defined separately. Instead, we follow the definition of a homogeneous Markov process via transition probabilities given in \cite{rogers}.}, meaning it satisfies the (homogeneous) \textit{Chapman-Kolmogorov} equation, that is if for each $j\in \mathbb{N},$ if $p_k$ is the $k$-step transition probability $p_k^x(A)=\mathbb{P}^x(\bm{x}_k\in A)$ for $x\in X,$ $A\in \mathscr{B}_X,$ then for $j,k\in \mathbb{N}$ with $j<k,$ we have
\[\int_A p^y_{k-j}(A)dp^x_j(y).\] Indeed, this may easily be seen to hold from the equation \[p^{x_0}_n(A)=\left(\prod_{j=1}^{n-1}\int_X p(x_{j-1},x_j)\right) \int_A p(x_{n-1},x_n)d\mu(x_n)...d\mu(x_1),\] for $x_0\in X, n\in \mathbb{N}, A\in \mathscr{B}_X.$ Note each $p_k^x$ is a probability measure and for each $A\in \mathscr{B}_X,$ the map $x\mapsto p^x_k(A)$ is measurable. The $p^y_{j-k}(A)$ appearing in the Chapman-Kolmogorov equation may be interpreted as the ``transition probability from $y$ at time $j$ to an element of $A$ at time $k.$" The process is called time-homogeneous since this ``transition probability" depends only on $k-j.$

\begin{definition} A Borel measure $\nu$ on $X$ is called an equilibrium measure if for all $f\in C(X),$ \[\nu\circ P(f)=\nu(Pf)=\int Pf(x)d\nu(x)=\int \int p(x,y)f(y)d\mu(y)d\nu(x)=\int f(x)d\nu(x)=\nu(f).\]\end{definition} A particularly important example is as follows. Suppose that 
there exists a bounded measurable function $\phi:X\rightarrow [0,\infty)$ with $\int \phi d\mu =1$ such that \[\phi(x)p(x,y)=\phi(y)p(y,x).\] Then let a measure $\nu$ be defined by $d\nu = \phi d\mu.$ In this case, $\nu$ is an equilibrium measure, and the process is said to be $\nu$-symmetric.

For what follows let us fix a point $a\in X$ and an open or closed set $A$ containing $a.$ Define ${\tau}_A:=\inf_{\mathbb{N}} \{k\geq 0\;| \bm{x}_k\notin A\}.$ Then ${\tau}_A$ is called the first exit time from $A$ (or the first hitting time for $A^c).$ Let ${E}_A(x):=\mathbb{E}^x{\tau}_A$ denote the \textit{mean exit time} of $\bm{x}$ from $A$ starting at $x.$ For $\omega\in \Omega,$ we adopt the notation $\omega(k)$ for $\pi_k(\omega)$ when $k\in \mathbb{N}.$

\begin{proposition} \label{markov} For $x\in U,$ we have \[E_A(x)=1+\int p(x,y)E_A(y)d\mu(y).\] If $x\notin U,$ we have $E_A(x)=0.$
\end{proposition}

\begin{proof} The second claim is clear. If $x\notin A$ and $\omega$ is a path starting at $x,$ then $\omega(0)=x\notin U.$ Hence ${\tau}_A(\omega)=0.$ So $E_A(x)=\EM^x{\tau}_A =0.$

First, if $A\in\mathscr{B}_X,$ the map $x\mapsto \mathbb{P}^x(A)$ is measurable by the argument following Lemma \ref{monot}. Thus the map $E_A:x\mapsto \mathbb{E}^x({\tau}_A)$ is measurable, since $\tau_A$ is non-negative and measurable, and thus is a pointwise limit of simple functions. Now suppose $x\in A.$ Then $\mathbb{E}^x {\tau}_A = \mathbb{E}^x(1+({\tau}_A-1))=1+\mathbb{E}^x({\tau}_A-1).$ Hence we concentrate on the $\mathbb{E}^x({\tau}_A-1)$ term. Note that if $\omega$ is a path starting at $x$ and $\omega(1)\notin U,$ then $({\tau}_A-1)(\omega) =0.$ So when taking the expectation we may assume the path takes its first step in $A.$ For $k\geq 1$, $y\in A,$  let \[A_k(y):=\{\omega\;|\;\omega(0)=y,\;\omega(j)\in A\;\mbox{for $1\leq j\leq k-1$ and } \omega(k)\notin A\}.\] Since we assume the first step is in $A,$ $A_1=\emptyset.$ Set $x:=x_0.$ Then for $k\geq 2,$
\[\mathbb{P}^x(A_k(x))=\int_A p(x_0,x_1)\prod_{j=2}^{k-1} (\int_A p(x_{j-1},x_j))\int_{U^c} p(x_{k-1},x_k)d\mu(x_k)...d\mu(x_1).\] However, since $\bm{x}$ is time-homogeneous, for each $x_1\in A,$ \[\prod_{j=2}^{k-1} (\int_A p(x_{j-1},x_j))\int_{U^c} p(x_{k-1},x_k)d\mu(x_k)...d\mu(x_2)= \mathbb{P}^{x_1}(A_{k-1}).\] Therefore 
\begin{equation}
\begin{split}
\mathbb{E}^x({\tau}_A-1)&=\sum_{k=1}^\infty (k-1)\mathbb{P}^x(A_k(x))\\
&= \sum_{k=2}^\infty \int_A p(x_0,x_1)\mathbb{P}^{x_1}(A_{k-1}(x_1))d\mu(x_1) \\
&= \int_A p(x,y)(\sum_{k=1}^\infty \mathbb{P}^y (A_k(y))) d\mu(y)\\
&= \int_A p(x,y)E_A(y)d\mu(y).\\
\end{split}
\end{equation} The result then follows.
\end{proof}

\subsection{Epsilon-Approximate Random Walks}

\begin{definition}If $(X,d)$ is a compact, metric space, we will call a discrete time Markov process $\bm{x}=(\bm{x}_k)_{k=0}^\infty$ on a state space $S\subset X$ \textit{$\epsilon$-approximate} if 
\begin{equation*}
\begin{split}
\mbox{(a)\;}&S \mbox{\;is\;}\epsilon\mbox{-dense in\;}X,\;\mbox{that is\;}\cup_{x\in S} B_{\epsilon}(x)=X;\\
\mbox{(b)\;}&\mbox{The step size is at most\;}2\epsilon,\mbox{\;that is\;}\mathbb{P}^x(\bm{x}_1\in B_{2\epsilon}(x))=1\mbox{\;for all\;}x\in X; \mbox{\;and} \\
\mbox{(c)\;}&\mbox{The expected step size is at least\;}\epsilon/2, \mbox{that is\;}\mathbb{E}^x(d(x,\bm{x}_1))\geq \epsilon/2 \mbox{\;for all\;}x\in X.
\end{split}
\end{equation*}
\end{definition}
\begin{example} Suppose $\mathcal{N}$ is an $\epsilon$-net in $X.$ Let $G_{2\epsilon}$ be the proximity graph on $\mathcal{N}$ defined in section 1.2. Then for $x,y\in \mathcal{N},$ define a transition probability $p(x,y)=\frac{1}{\deg(x)}\chi_{B_{2\epsilon}(x)\setminus\{x\}}(y).$ Then the resulting discrete time process is $\epsilon$-approximate.
\end{example}

\begin{example} Suppose $X$ is a connected doubling space with doubling measure $\mu.$ For $0<r<r'$ and $x\in X,$ let  $A_{r,r'}(x)=B_{r'}(x)\setminus B_r[x].$ Note, since $X$ is connected, if $r'<\diam(X)$ then $A_{r,r'}(x)$ is non-empty and open. Suppose $\epsilon<\diam(X).$ Then define $p:X\times X\rightarrow \mathbb{R}^+$ by \[p(x,y)=\frac{1}{\mu(A_{\epsilon/2,\epsilon}(x))}\chi_{A_{\epsilon/2,\epsilon}(x)}(y).\] The $p$ is a $\mu$-transition kernel for an $\epsilon$-approximate walk.
\end{example}

\begin{example} Suppose $X$ is a doubling space with doubling measure $\mu$ such that for some $r>0$ and all $x\in X,$ the function $t\mapsto \mu(B_t(x))$ is convex on $[0,r).$ Then for $\epsilon<r,$ define $p(x,y)=\frac{1}{\mu(B_r(x))}\chi_{B_r(x)}(y).$ Consider the discrete time process $\bm{x}$ defined by the $\mu$-transition kernel $p.$ We show $\mathbb{E}^x(d(x,\bm{x}_1))\geq \epsilon/2 \mbox{\;for all\;}x\in X.$ By convexity, we have the bound $\int_0^\epsilon \mu(B_t(x))dt\leq \mu(B_\epsilon(x))\epsilon/2.$ Hence, integrating by parts 
\[\mathbb{E}^x(d(x,\bm{x}_1))=\int_0^\epsilon t\frac{d}{dt} \mathbb{P}^x(d(x,\bm{x}_1)\leq t)dt = \epsilon-\frac{1}{\mu(B_\epsilon(x))}\int_0^\epsilon \mu(B_t(x))dt\geq \epsilon-\epsilon/2=\epsilon/2.\]
It then follows that $\bm{x}$ is $\epsilon$-approximate. 
\end{example}
\subsection{Continuous Time Random Walks}

We will show how we may use the discrete time process $\bm{x}$ together with a suitable assignment $x\mapsto \tau(x)>0$ of local waiting times to induce a continuous time process. Intuitively, this process waits on average a time $\tau(x)$ at site $x\in X$ before jumping to a neighboring site according to the discrete time process $\bm{x}.$  We now make this precise.

For $\omega=(t_i,x_i)_{i=0}^\infty\in (\mathbb{R}_+\times X)^{\mathbb{N}},$ let $\omega(k)_1=t_k$ and $\omega(k)_2=x_k$ for $k\in \mathbb{N}.$ Then let $\Omega^*:=\{\omega\in (\mathbb{R}_+\times X)^{\mathbb{N}}\;|\;\omega(0)_1=0\;\}.$  Let $p_0$ be the map $\Omega^*\ni \omega \mapsto \omega(0)_2\in X.$ 

\begin{proposition} \label{kol}
Suppose $\tau:X\rightarrow (0,\infty)$ is measurable with both $\tau$ and $\frac{1}{\tau}$ bounded.  Then for each $x\in X,$ there exists a measure $\mathbb{P}^x$ defined on the product Borel $\sigma$-algebra $\Sigma$ of $\Omega^*$ such that if $A\in \Sigma$ then the map $x\mapsto \mathbb{P}^x(A)$ is measurable and such that for measurable cylinder sets of the form $A=\{\omega \in \Omega^*\;|\;\omega(j)\in A_j\times U_j, j=1,...,n\},$ we have that 
\begin{equation*}
\mathbb{P}^{x_0}(A)=\left(\prod_{j=1}^n \int_{A_j}\int_{U_j} \frac{e^{-t_{j}/ \tau(x_{j-1})}}{\tau(x_{j-1})}p(x_{j-1},x_{j}) \right ) dt_nd\mu(x_{n})...dt_1d\mu(x_{1}).
\end{equation*}
 
\end{proposition}
\begin{proof}
Let $W_0=X.$ For $k\in \mathbb{Z}^+,$ let $W_k=[0,\infty)\times X.$ Then since $[0,\infty)$ and $X$ are complete metric spaces, so is $[0,\infty)\times X.$ Let $\Sigma_k$ be the Borel $\sigma\mhyphen$ algebra on $W_k.$ Let $d\eta=dt\otimes d\mu$ be the product Borel measure of Lebesgue measure on $\mathbb{R}_+$ with $\mu$ on $X$. Let $\nu$ be a Borel probability measure on $X.$ For $F$ a finite subset of $\mathbb{N},$ let $k_1<k_2<...<k_n$ be a list of the elements of $F.$ Let $A\in\Sigma_F$. If $k_1>0,$ let \[\mu_F(A)=\int_X\int_A\prod_{i=1}^n \frac{e^{-{t_i}/\tau(x_{i-1})}}{\tau(x_{i-1})}p(x_{i-1},x_i)d\eta^{\otimes n}((t_n,x_n),...,(t_1,x_1))d\nu(x_0),\] where $\eta^{\otimes m}=\otimes_{i=1}^m\eta$ for $m\in \mathbb{Z}^+.$ If $k_1=0,$ let \[\mu_F(A)=\int_A\left(\prod_{i=2}^{n}\frac{e^{-{t_i}/\tau(x_{i-1})}}{\tau(x_{i-1})}p(x_{i-1},x_i)\right)d\eta^{\otimes n}((t_n,x_n),...,(t_1,x_1)).\]

Since $X\times [0,\infty)$ is a Polish space, we need only check the consistency condition. Let $F,F' \subset \mathbb{N}$ with $F\subset F'$. We may assume $0\notin F,$ $F=\{1,...,n\}\subset F'=\{1,...,m\}$ and $n<m.$  Then since 
\[\int_X\int_{[0,\infty)}\frac{e^{-t/ \tau(a})}{\tau(a)}p(a,x)dtd\mu(x)=1,\] for any $a\in X,$ if $A\subset \Sigma_{F},$ the consistency conditions hold. Hence by the Kolmogorov Extension Theorem there exists a measure on $\mathbb{P}^\nu$ on $\Omega^*$ extending the above definition. For $x\in X,$ let $\mathbb{P}^x=\mathbb{P}^{\delta_x}.$ Then by a completely analogous argument as the one following Lemma \ref{monot}, since $\Sigma$ is generated by the algebra of cylinder sets, the map $x\mapsto \mathbb{P}^x(A)$ is measurable for any $A\in \Sigma.$
\end{proof}

For $t\geq 0$, $x\in X,$ let $\hat{t}:\Omega^*\rightarrow \mathbb{N}$ be defined by $\hat{t}(\omega):=\sup\{k\;|\;\sum_{j=0}^k \omega(j)_1\leq t\}.$ Then define $(\bm{x}^*_t)_{t\in T}$ by $\bm{x}^*_t(\omega):=\omega(\hat{t}(\omega))_2$ for $t\in T,\;\omega\in\Omega^*.$ 

The process $\bm{x}^*=(\bm{x}^*(t))_{t\in T}$ is known as a pure jump Markov process. See   \cite{gihman},\cite{kallenberg} for more information on jump processes. Also, see Theorem 10.20 of \cite{kallenberg} for a proof that the process satisfies the Markov property.

The process $\bm{x}^*$ behaves like the process $\bm{x}$ at discrete ``jump times." that is if $\tau_k:=\inf\{t\geq 0\;|\;\hat{t}=k\}$  for $k\in \mathbb{N}$, then $\bm{x}^*_{\tau_k}$ has the same distribution as $\bm{x}_k$ for each $k\in \mathbb{N}.$  Indeed, $\hat{t}(\omega)=k$ if and only if $\sum_{j=0}^k \omega(j)_1\leq t<\sum_{j=0}^{k+1}\omega(j)_1.$ Hence $\tau_k(\omega)=\sum_{j=0}^k \omega(j)_1.$ So $\widehat{\tau_k(\omega)}=k.$ Thus $\mathbb{P}^x(\bm{x}^*_{\tau_k}\in A)=\mathbb{P}^x(\{\omega\in\Omega^*\;|\; \omega(k)_2\in A\})=\mathbb{P}(\bm{x}_k\in A).$ It follows that $\bm{x}^*_{\tau_k}$ has the same distribution as $\bm{x}_k$ for each $k\in \mathbb{N}.$ 

Let $\tau_{B}:=\inf\{t\geq 0\;|\;\bm{x}^*_t\notin B\}.$ For $x\in X,$ let $\mathbb{E}^x$ be the $\mathbb{P}^x$ expectation, and let $E_{B}(x):=\mathbb{E}^x \tau_{B}.$ 

\begin{proposition} \label{eq1} For $x\in B,$ we have $E_{B}(x)=\tau(x)+\int p(x,y)E_{B}(y)d\mu(y).$ If $x\notin B,$ we have $E_{B}(x)=0.$
\end{proposition}

\begin{proof} The second claim is clear. If $x\notin B$ and $\omega\in \Omega^*$ with $\omega(0)_2=x,$  then since $\hat{0}(\omega)=0,$ $\tau_{B}(\omega)=0.$ So $E_{B}(x)=\mathbb{E}^x\tau_{B} =0.$

Now suppose $x\in B.$ Let $\tau_1(\omega):=\inf\{t\geq 0\;|\;\hat{t}(\omega)=1\}=\omega(1)_1.$ Then $\tau_1$ has exponential distribution with mean $\tau(x).$ Hence $\mathbb{E}^x \tau_{B} = \mathbb{E}^x(\tau_1+(\tau_{B}-\tau_1))=\tau(x)+\mathbb{E}^x(\tau_{B}-\tau_1).$ We then concentrate on the $\mathbb{E}^x(\tau_{B}-\tau_1)$ term.   
\begin{equation*}
\begin{split}
\mathbb{P}^x(\tau_B-\tau_1\geq t)&=\mathbb{P}^x(x^*_{\tau_1+s}\in B \mbox{\;for\;} 0\leq s< t)\\
&= \int p(x,y)\mathbb{P}^y(x^*_s\in B \mbox{\;for\;} 0\leq s< t)d\mu(y) \\
&= \int p(x,y)\mathbb{P}^y(\tau_B\geq t)d\mu(y).\\
\end{split}
\end{equation*} Hence the distribution of $\tau_B-\tau_1$ is defined by $\mathbb{P}^x(\tau_B-\tau_1\in A)=\int p(x,y)\mathbb{P}^y(\tau_B\in A)d\mu(y).$ Therefore 
\[\mathbb{E}^x(\tau_B-\tau_1)=\int p(x,y)\mathbb{E}^y(\tau_B) d\mu(y).\] The result then follows.
\end{proof}

Note by construction, for each $\omega,$ the path $t\mapsto \bm{x}^*_t(\omega)$ is right continuous with left limits at each $t\in \mathbb{R}_+.$ Such paths are called \textit{c\`{a}dl\`{a}g}.\footnote{The term c\`{a}dl\`{a}g comes from the French phrase ``continue \`{a} droite, limite \`{a} gauche," which means ``continuous from the right, limit from the left."}  Let $\mathcal{D}$ denote the set of all c\`{a}dl\`{a}g paths $\gamma:[0,\infty)\rightarrow X.$ Note that $\mathcal{D}\subset X^{\mathbb{R}_+}.$ Let $\mathscr{F}$ be the subspace $\sigma$-algebra on $\mathcal{D}$ of the product Borel $\sigma$-algebra $\otimes_{t\in [0,\infty)} \mathscr{B}_X$ on $X^{\mathbb{R}_+},$ i.e. 
\[\mathscr{F}:=\{B\cap \mathcal{D}\;|\;B\in\otimes_{t\in \mathbb{R}_+} \mathscr{B}_X\}.\] 
Let $\iota:\Omega^*\rightarrow \mathcal{D}$ denote the map $\omega \mapsto \iota(\omega):t\mapsto \bm{x}^*_t(\omega).$ The map $\iota$ is $(\mathscr{F}^*, \mathscr{F})$ measurable, where $\mathscr{F}^*$ is the product Borel $\sigma$-algebra on $\Omega^*.$ Indeed, as the map $\omega\mapsto \bm{x}^*_t(\omega)$ is measurable for each $t,$ if $A\in \mathscr{B}_X,$ then $\iota^{-1}(\pi_t^{-1}(A)\cap\mathcal{D}) =(\bm{x}_t^*)^{-1}(A)\in\mathscr{F}^*.$ Then since the projections generate $\otimes_{t\in \mathbb{R}_+} \mathscr{B}_X,$ measurability follows. Hence, by taking the distribution of $\iota,$ we may consider the measure $\mathbb{P}^\nu$ for any initial distribution $
\nu$ as a measure on $\mathcal{D}.$ Thus we may think of the process $\bm{x}^*$ as being defined on $\mathcal{D}.$\footnote{With the product Borel $\sigma$-algebra, which coincides with the Borel $\sigma$-algebra of the Skorokhod J1 topology.}

For $t\geq 0$ and $x\in X,$ the transition measure $p_t^x$ is a Borel measure on $X$ defined by 
\[p_t^x(A)=\mathbb{P}^x(\bm{x}^*_t\in A).\] Then if $s\leq t,$ the Chapman-Kolmogorov equation
\[p_t^x(A)=\int p_{t-s}^y(A)dp_s(y)\] holds. Then define $P_t$ on $L^2(X,\mu)$ by \[P_tf(x)=\mathbb{E}^xf(\bm{x}^*_t)\] for $f\in L^2(X,\mu).$ By the Chapman-Kolmogorov equation, first on simple functions, then extending to $f\in L^2(X,\mu),$ we have $P_{t}=P_{t-s}\circ P_s$ for all $0\leq s\leq t<\infty.$

\subsection{Generator} Let $(\mathcal{B}, \|\cdot\|)$ be a Banach space. A strongly continuous (one parameter) semigroup in $\mathcal{B}$ is a collection $(P_t)_{t\in \mathbb{R}_+}$ of bounded linear operators on $\mathcal{B}$ such that 
\begin{equation*}
\begin{split}
\mbox{(a)}&\;P_0=I\;\mbox{and\;} P_t\circ P_s=P_{t+s}\;\mbox{for all\;}t,s\in\mathbb{R}_+\mbox{, and}\\
\mbox{(b)}&\;\lim_{t\rightarrow 0^+}\|P_tf-f\|=0\;\mbox{for all\;} f\in \mathcal{B}.
\end{split}
\end{equation*}

If $(P_t)_{t\geq 0}$ is a strongly continuous semigroup, the generator, $A$, of $(P_t)_{t\geq 0}$ is defined by $Af=\lim_{t\rightarrow 0^+} t^{-1}(Ptf-f),$ with domain $\mathscr{D}(A)$ consisting of $f\in \mathcal{B}$ such that the above limit exists in $\mathcal{B}.$ Then the generator $A$ is closed, $\mathscr{D}(A)$ is dense in $\mathcal{B}$, and $e^{tA}=\sum_{k=0}^{\infty}t^kA^k/k! = P_t.$ See \cite{engel}, \cite{fukushima}. 

A strongly continuous semigroup $(P_t)_{t\geq 0}$ is called contractive if $\|P_t\|\leq 1$ for all $t,$ that is $\|P_tf\|\leq \|f\|$ for all $f\in \mathcal{B}, t\geq 0.$  

Suppose $(X,d)$ is a compact metric space with a finite Borel measure $\mu.$ A strongly continuous semigroup $(P_t)_{t\geq 0}$ on $L^{\infty}(X,\mu)$ is called a \textit{Feller semigroup} if it is contractive, $P_t(C(X))\subset C(X)$ for all $t\geq 0,$ and $P_tf\geq 0$ for all $f\in C(X)$ with $f\geq 0.$ Feller semigroups satisfying $P_t1= 1$ for all $t$ may be used to generate time homogeneous Markov processes with c\`{a}dl\`{a}g sample paths. A sketch of this construction is as follows. See \cite{rogers} for details. For $F$ a finite subset of $\mathbb{R}_+,$ with $0\leq t_1<t_2<...<t_n,$ a list of the elements of $F,$ a product Borel measure $p^x_{F}$ on $X^F$ may be defined by Carath\'{e}odory extension form the premeasure defined on cylinder sets $A=\prod_{i=1}^n A_i$ for $A_i\in \mathscr{B}_X$ for each $i$, by 
\[p_{F}^x(A)=P_{t_1}(\chi_{A_1}(P_{t_2-t_1}(\chi_{A_2}(...P_{t_{n-1}-\sum_{i=1}^{n-2}t_i}(\chi_{A_{n-1}}(P_{t_n-\sum_{i=1}^n t_i}(\chi_{A_n})))...))))(x).\]
Since each $P_t$ is a positivity preserving contraction on $C(X)$ with $P_t1=1,$ each $p_F$ is a probability measure. The semigroup property then ensures that the Kolmogorov consistency conditions hold. Hence probability measures $\mathbb{P}^x$ may be defined on $X^{\mathbb{R}_+}$ with its product Borel $\sigma$-algebra. The definition of the cylinder measures ensures the Chapman-Kolmogorov equations hold. Lastly, the strong continuity ensures that the process may be modified to have c\`{a}dl\`{a}g paths. See \cite{rogers}.

A semigroup $(P_t)_{t\geq 0}$ on $L^2(X,\mu)$ is called \textit{Markovian} if $0\leq P_t|f|\leq 1$ $\mu$-a.e. for all $f\in L^\infty(X,\mu)$ with $\|f\|_{L^\infty(X,\mu)}\leq 1.$ Strongly continuous contractive Markovian semigroups are in one-to-one correspondence with Dirichlet forms. See \cite{fukushima}.

We end the section by calculating the generator for the process $\bm{x}^*.$ Recall, for $t\geq 0,$ $P_t$ is defined by $P_tf(x)=\mathbb{E}^xf(\bm{x}^*_t)$ for $f\in L^2(X,\mu), x\in X.$

\begin{proposition} \label{gen} Define $\mathcal{L}$ on $L^2(X,\mu)$ by \[\mathcal{L}f(x):=\int_X \frac{1}{\tau(x)}p(x,y)(f(x)-f(y))d\mu(y).\] Then for $f\in L^2(X,\mu),$ $\|\frac{1}{t}(P_tf-f)+{\mathcal{L}f\|_{L^2(X,\mu)}}_{\overrightarrow{t\rightarrow 0^+}} 0.$ In particular, $-\mathcal{L}$ is the generator for the $L^2(X,\mu)$ semigroup defined by $(P_t)_{t\geq 0}$ and $\mathscr{D}(\mathcal{L})=L^2(X,\mu).$
\end{proposition}
\begin{proof} 

Since $\sup_{x\in X, y\in X}p(x,y)<\infty$, $\tau$ and $\frac{1}{\tau}$ are bounded above uniformly, and $\mu$ is a finite measure, let $C>0$ such that $p(x,y)\leq C$ for $x,y\in X,$ $\mu(X)<C,$ and $\frac{1}{C}\leq \frac{1}{\tau(x)}\leq C$ for all $x.$ Let $f\in L^2(X,\mu),$ $t>0.$ Set $x_0=x.$ By the Cauchy-Schwartz inequality, $\|f\|_{L^1(X,\mu)}\leq C\|f\|_{L^2(X,\mu)}.$ Note that $\mathbb{E}^{x} f(\bm{x}^*_t)$ consists of an infinite sum from $k=0$ to $\infty$ of terms of the form \[\left( \prod_{i=1}^k \int_0^{t-\sum_{j=1}^{i-1} t_j} \int_X \frac{e^{\frac{-t_{i}}{\tau(x_{i-1})}}p(x_{i-1},x_i)}{\tau(x_{i-1})}\right)\int_{t-\sum_{j=1}^k t_j}^\infty  \frac{e^{\frac{-t_{k+1}}{\tau(x_k)}}f(x_k)}{\tau(x_k)}\left(\prod_{j=0}^{k-1}dt_{k+1-j}d\mu(x_{k-j})\right)dt_1.\] Define an operator $A_{t,k}$ such that $A_{t,k}f(x)$ gives the above $k$-th term evaluated at $x$ for any $f\in L^2(X,\mu).$ For $k\geq 2,$ the absolute value of the $k\mhyphen$th term is bounded above by $(ct)^k\|f\|_{1,\mu}\leq C(ct)^k\|f\|_{L^2(X,\mu)},$ for some $c>0$ depending only on $C$. Hence, dividing by $t$ and taking the $L^2$ norm using the uniform bound, we have $\frac{1}{t}\|A_{t,k}f\|_{L^2(X,\mu)}\leq C^2(ct)^{k-1}\|f\|_{L^2(X,\mu)}.$ Thus, for $0<t<\frac{1}{c},$ \[\sum_{k=2}^\infty  \frac{1}{t}\|A_{t,k}f\|_{L^2(X,\mu)}\leq \frac{C^2ct}{1-ct}\|f\|_{L^2(X,\mu)}\;\overrightarrow{_{_{t\rightarrow 0^+}}}\; 0.\] Hence we need only consider $\|\frac{1}{t}(A_{0,t}f+A_{1,t}f-f)+\mathcal{L}f\|_{L^2(X,\mu)}.$

Let us then consider the $k=0$ and $k=1$ terms separately. The term for $k=0$ is, recalling the convention that the empty product is $1$ and the empty sum is $0,$ given by
\[A_{0,t}f(x)=\int_t^\infty \frac{e^{-t_1/\tau(x)}f(x)}{\tau(x)}dt_1=e^{-t/\tau(x)}f(x).\] Define 
$\frac{1}{\tau}$ to be the multiplication operator $(\frac{1}{\tau})f(x)=\frac{1}{\tau(x)}f(x).$ 
Then \[\|\frac{1}{t}(A_{0,t}f-f)+\frac{1}{\tau}f\|^2_{L^2(X,\mu)}=\int |\frac{1}{t}(e^{-t/\tau(x)}-1)+\frac{1}{\tau(x)}|^2|f(x)|^2d\mu(x)_{\;\overrightarrow{t\rightarrow 0^+}}\; 0\]
by the dominated convergence theorem, since $|\frac{1}{t}(e^{-t/\tau(x)}-1)+\frac{1}{\tau(x)}|^2|
f(x)|^2\rightarrow 0$ as $t\rightarrow 0^+$ and $|\frac{1}{t}(e^{-t/\tau(x)}-1)+\frac{1}{\tau(x)}|\leq 2C$ for all $t>0, 
x\in X$ by the mean value theorem and the bound $\frac{1}{\tau(x)}\leq C.$

Now the term for $k=1$ is \begin{equation*}
\begin{split}
&A_{1,t}f(x)=\int_0^{t}\int_X\int_{t-t_1}^\infty \frac{e^{-t_1/\tau(x)}p(x,x_1)e^{-t_2/\tau(x_1)}f(x_1)}{\tau(x)\tau(x_1)}dt_2d\mu(x_1)dt_1 \\
&=\int_X\frac{p(x,x_1)}{\tau(x)}\int_0^{t} e^{-t_1/\tau(x)}e^{-(t-t_1)/\tau(x_1)}f(x_1)dt_1d\mu(x_1).
\end{split}
\end{equation*}
However, since \[\frac{e^{-t/\tau(x_1)}}{t}\int_0^t e^{-t_1(\tau(x)^{-1}-\tau(x_1)^{-1})}dt_1\; 
\overrightarrow{_{t\rightarrow 0^+}}\; 1\] for all $x_1\in X$ and since \[\left |\frac{p(x,x_1)}
{\tau(x)}\frac{1}{t}\int_0^t (e^{-{t}/\tau(x_1)}e^{-t_1(1/\tau(x)-1/\tau(x_1))}-1)f(x_1)dt_1\right |\leq C^2|f(x_1)|\] for all $t>0,$ $x,x_1\in X,$ by the dominated convergence theorem,
\[\int \frac{p(x,x_1)}{\tau(x)}\frac{1}{t}\int_0^t (e^{-{t}/\tau(x_1)}e^{-t_1(1/\tau(x)-1/
\tau(x_1))}-1)f(x_1)dt_1d\mu(x_1)_{\;\overrightarrow{t\rightarrow 0^+}}\; 0\] for all $x.$ Hence, by 
another application of the dominated convergence theorem, \begin{equation*}
\begin{split}
\|&\frac{1}{t}A_{1,t}f-\frac{1}{\tau}Pf\|^2_{L^2(X,
\mu)}=\\&\int \left |\int \frac{p(x,x_1)}{\tau(x)}\frac{1}{t}\int_0^t (e^{-{t}/\tau(x_1)}e^{-t_1(1/
\tau(x)-1/\tau(x_1))}-1)f(x_1)dt_1d\mu(x_1)\right|^2d\mu(x)_{\;\overrightarrow{t\rightarrow 0^+}}\; 0.\end{split}
\end{equation*}
Therefore, putting it all together, we have 
\begin{equation*}
\begin{split}
&\|\frac{1}{t}(A_{0,t}f+A_{1,t}f-f)+\mathcal{L}f\|_{L^2(X,\mu)}= \|\frac{1}{t}(A_{0,t}f+A_{1,t}f-f)+\frac{1}{\tau}f-\frac{1}{\tau}Pf\|_{L^2(X,\mu)}\\
&\leq \|\frac{1}{t}(A_{0,t}f-f)+\frac{1}{\tau}f\|^2_{L^2(X,\mu)}+{\|\frac{1}{t}A_{1,t}f-\frac{1}{\tau}Pf\|^2_{L^2(X,\mu)}}\;\overrightarrow{_{_{t\rightarrow 0^+}}}\;0.
\end{split}
\end{equation*}
Therefore $-\mathcal{L}=\frac{1}{\tau}(P-I)$ is the generator of $(P_t)_{t\geq 0}$ on $L^2(X,\mu).$
\end{proof}

\section{$\Gamma\mhyphen$ Convergence and Dirichlet Forms}
The origins of the theory of $\Gamma\mhyphen$ Convergence may be found in DeGiorgi \cite{giorgi} and DeGiorgi and Franzoni \cite{giorgi1}. The inspiration for our later use of the sequential compactness properties of $\Gamma\mhyphen$ convergence for Markovian forms came from its use by Sturm in \cite{sturm1} and by Kumagai and Sturm in \cite{sturm}. 

A Dirichlet form may be thought of as a generalization of the Dirichlet energy form $u\mapsto \frac{1}{2}\int_\Omega |\nabla u|^2dx$, for $u\in H^1_0(\Omega),$ $\Omega\subset \mathbb{R}^n$ a domain. An excellent reference on the theory of Dirichlet forms is \cite{fukushima}. The theory of regular Dirichlet forms is particularly rich, with theoretical underpinnings in the landmark papers \cite{beurling1}, \cite{beurling2} of Beurling and Deny, which led in part to a complete classification of regular Dirichlet forms known as the Beurling-Deny formula. See, e.g., \cite{mosco}, \cite{fukushima} for more information.   

We first develop the general theory of $\Gamma$-convergence before discussing Dirichlet forms and form convergence.

\subsection{$\Gamma$-Convergence}
The results given in this section are well known. We present them here for the convenience of the reader. In our exposition we follow primarily \cite{maso} and \cite{mosco}. 

Let $X$ be a topological space. For $x\in X$ let $\mathscr{N}(x)$ be the collection of open neighborhoods of $x.$ If $\mathscr{B}$ is a base for the topology of $X$ let $\mathscr{B}(x)$ denote the basic neighborhoods of $x$, that is the neighborhoods of $x$ in $\mathscr{B}.$

\begin{definition}
Let $(A_n)_{n=1}^\infty$ be a sequence of subsets of $X$. We define the Kuratowski lower and upper limits \cite{kuratowski} as follows. For the lower limit, let \[K\mhyphen\liminf_{n\rightarrow \infty} A_n :=\{x\in X\;|\;\forall U\in \mathscr{N}(x)\exists N\in \mathbb{Z}^+ \forall n\geq N\;[A_n\cap U\neq \emptyset]\}.\]
For the upper limit, let \[K\mhyphen\limsup_{n\rightarrow \infty} A_n :=\{x\in X\;|\;\forall U\in \mathscr{N}(x)\forall N\in \mathbb{Z}^+ \exists n\geq N\;[A_n\cap U\neq 
\emptyset]\}.\]
Then we have $K\mhyphen\liminf_{n\rightarrow \infty} A_n \subset K\mhyphen\limsup_{n \rightarrow 
\infty} A_n,$ and if they are equal we denote the common value by $K\mhyphen\lim_{n\rightarrow \infty} A_n.$
\end{definition}

\begin{rem}
It is clear that the values of $K\mhyphen\liminf_{n\rightarrow \infty}$ and $K\mhyphen\limsup_{n\rightarrow \infty}$ are unchanged if we replace everywhere $\mathscr{N}(x)$ with $\mathscr{B}(x).$
\end{rem}

\begin{lemma} If $(A_n)_{n=1}^\infty$ is a sequence of sets in $X$ then both $K\mhyphen\limsup_{n\rightarrow \infty}A_n$ and $K\mhyphen\limsup_{n\rightarrow \infty}A_n$ are closed sets. 
\end{lemma}
\begin{proof} Suppose $x\in (K\mhyphen\limsup_{n\rightarrow \infty}A_n)^c.$ Then there exists a $U\in \mathscr{N}(x)$ such that for some $N\in \mathbb{Z}^+$ and every $n\geq N$ we have $A_n\cap U=\emptyset.$ Since this condition holds for every $y\in U$ we have $U\subset (K\mhyphen\limsup_{n\rightarrow \infty}A_n)^c$. Similarly, if $x\in (K\mhyphen\liminf_{n\rightarrow \infty}A_n)^c$ then there exists a $U\in \mathscr{N}(x)$ such that for every $N\in \mathbb{Z}^+$ and some $n\geq N$ we have $A_n\cap U=\emptyset.$ Since this condition holds for every $y\in U$ we have $U\subset (K\mhyphen\liminf_{n\rightarrow \infty}A_n)^c$. The result follows.
\end{proof}
\begin{lemma} \label{subsequence}
Suppose $A=K\mhyphen\lim_{n\rightarrow \infty} A_n.$ Then for any subsequence $(A_{n_k})_{k=1}^\infty$ we have $A=K\mhyphen\lim_{k\rightarrow \infty} A_{n_k}.$
\end{lemma}

\begin{proof} Let $x\in K\mhyphen\liminf_{n\rightarrow \infty} A_n$ and $U\in \mathscr{N}(x).$ Then there exists an $N\in \mathbb{Z}^+$ such that for $n\geq N$ we have $U\cap 
A_n\neq \emptyset.$ But $n_k\geq k$ for all $k.$ Hence if $k\geq n$ then $n_k\geq N$ so $A_{n_k}\cap U\neq \emptyset.$ So $x  \in K\mhyphen\liminf_{k\rightarrow \infty} A_{n_k}.$ 
If $x\in K\mhyphen\limsup_{k\rightarrow \infty} A_{n_k}$ and $U\in \mathscr{N}(x)$ then if 
$N\in \mathbb{Z}^+$ then there exists a $k\geq N$ with $A_{n_k}\cap U\neq \emptyset.$ 
But again $n_k\geq k \geq N$. Hence $x\in K\mhyphen\limsup_{n\rightarrow \infty} A_n.$  
Hence we have the observation
\[K\mhyphen\liminf_{n\rightarrow \infty} A_n  \subset K\mhyphen\liminf_{k\rightarrow \infty} A_{n_k} 
\subset K\mhyphen\limsup_{k\rightarrow \infty} A_{n_k}\subset K\mhyphen\limsup_{n\rightarrow \infty} 
A_n,\] from which the result follows immediately.
\end{proof}

\begin{definition}
Let $f:X\rightarrow \overline{\mathbb{R}}.$ We say $f$ is lower semicontinuous if for every $c\in \mathbb{R},$
we have $f^{-1}((c,\infty])$ is open. 
\end{definition}

\begin{definition} If $f:X\rightarrow \overline{\mathbb{R}},$ the epigraph of $f,$ written $\epi(f),$ is the set of all $(x,t)\in X\times \overline{\mathbb{R}}$ with $f(x)\leq t.$
\end{definition}
\begin{proposition} If $f:X\rightarrow \overline{\mathbb{R}}$ then $f$ is lower semicontinuous if and only if $\epi(f)$ is closed.\end{proposition}
\begin{proof} Suppose $f$ is lower semicontinuous. Then for $t\in \mathbb{R}$ let $B_t=f^{-1}((t,\infty])\times [-\infty,t).$ Then $B_t$ is open in $X\times \overline{\mathbb{R}}.$ Let $B=\cup_{t\in \mathbb{R}} B_t.$ Again, $B$ is open. If $(x,y)\in B$ then $(x,y)\in B_t$ for some $t.$ So $f(x)>t>y.$ So $(x,y)\in (\epi(f))^c.$ Conversely, if $(x,y)\in (\epi(f))^c$ then $f(x)>y.$ Let $t\in \mathbb{R}$ with $f(x)>t>y.$ Then $(x,y)\in B_t.$ Hence $(\epi(f))^c=B$ is open. So $\epi(f)$ is closed.

Now suppose $\epi(f)$ is closed. Let $c\in \mathbb{R}.$ If $f^{-1}((c,\infty])$ is not open then there exists an $x\in X$ with $f(x)>c$ such that for all $U\in \mathscr{N}(x)$ there exists an $x_U\in U$ with $f(x_U)\leq c.$ Then $((x_U,c))_{U\in \mathscr{N}(x)}$ defines a net in $\epi(f)$ converging to $(x,c)$. By the closure of $\epi(f)$ we have $f(x)\leq c,$ a contradiction. Hence $f$ is lower semicontinuous.
\end{proof}

\begin{lemma} Suppose for each $n,$ $f_n:X\rightarrow \overline{\mathbb{R}}.$ Let $A_n=\epi(f_n)$, $A=K\mhyphen\liminf_{n\rightarrow \infty} A_n$ and $B=K\mhyphen\limsup_{n\rightarrow \infty}A_n.$ Then there exist unique lower semicontinuous functions $f,g:X\rightarrow \overline{\mathbb{R}}$ such that $f=\epi(A),g=\epi(B).$ Moreover $g\leq f.$
\end{lemma}
\begin{proof}  Note $(x,\infty)\in A_n$ for each $x\in X.$ It follows $(x,\infty)\in A,(x,\infty)\in B$ for each $x\in X.$ Hence $\{t\in \overline{\mathbb{R}}\;|\;(x,t)\in 
A\}\neq \emptyset$ and $\{t\in \overline{\mathbb{R}}\;|\;(x,t)\in B\}\neq \emptyset.$
We define $f(x)=\inf\{t\in \overline{\mathbb{R}}\;|\;(x,t)\in A\}$ and $g(x)=\inf\{t\in 
\overline{\mathbb{R}}\;|\;(x,t)\in B\}$ for $x\in X.$ Since $A,B$ are closed sets it 
follows $(x,f(x))\in A$ and $(x,g(x))\in B$ for $x\in X.$ So suppose $x,x'\in X,$ 
$f(x)<s<\infty,\; g(x')<t<\infty$. Then let $U\in\mathscr{N}(x),V\in \mathscr{N}(x'),$ 
and $\epsilon>0.$ Let $\eta>0$ so that $f(x)+\eta<s$ and $g(x')+\eta<t.$ Then since 
$[-\infty,f(x)+\eta)$ is a neighborhood of $f(x)$ there exists an $N$ such that for 
all $n\geq N$ we have $f_n(y)\leq f(x)+\eta\leq s$ for some $y\in U.$ So $A_n \cap 
(U\times (s-\epsilon,s+\epsilon))\neq \emptyset$ for $n\geq N.$ It follows that 
$(x,s)\in A.$ Similarly since $[-\infty,g(x')+\eta)$ is a neighborhood of $g(x'),$  if 
$N$ is a non-negative integer then there exists an $n\geq N$ such that $f_n(y)\leq 
g(x')+\eta\leq t$ for some $y\in V.$ So for all $N$ there exists an $n\geq N$ such 
that $A_n\cap (V\times (t-\epsilon,t+\epsilon))\neq \emptyset$. Hence $(x',t)\in B.$ 
It follows that $A=\epi(f)$ and $B=\epi(g).$ The uniqueness is immediate. The lower semicontinuity follows from the closure of $\epi(f)$ and $\epi(g).$ Since $\epi(f)\subset \epi(g)$ and $\epi(f)$ is closed, for any $x\in X$ we have $(x,f(x))\in \epi(g)$ so that $g(x)\leq f(x).$
\end{proof}
\begin{definition}
Let $(f_n)_{n=1}^\infty$, $f$ be functions from $X$ to $\overline{\mathbb{R}}.$ We define $\Gamma\mhyphen$ limits as follow.
\[\Gamma\mhyphen\limsup_{n\rightarrow \infty} f_n=f\; \mbox{if and only if\;} K\mhyphen\liminf_{n\rightarrow \infty} \epi(f_n)=\epi(f).\]
\[\Gamma\mhyphen\liminf_{n\rightarrow \infty} f_n=f\; \mbox{if and only if\;} K\mhyphen\limsup_{n\rightarrow \infty} \epi(f_n)=\epi(f).\]
\[\Gamma\mhyphen\lim_{n\rightarrow \infty} f_n=f\; \mbox{if and only if\;} K\mhyphen\lim_{n\rightarrow \infty} \epi(f_n)=\epi(f).\]
\end{definition}

\begin{proposition}  If $(f_n)_{n=1}^\infty$ is a sequence of functions from $X$ to $\overline{\mathbb{R}}$ then for $x\in X,$
\[(\Gamma\mhyphen\liminf_{n\rightarrow \infty} f_n)(x)=\sup_{U\in\mathscr{B}(x)}\liminf_{n\rightarrow \infty} \inf_{y\in U} f_n(y),\] and
\[(\Gamma\mhyphen\limsup_{n\rightarrow \infty} f_n)(x)=\sup_{U\in\mathscr{B}(x)}\limsup_{n\rightarrow \infty} \inf_{y\in U} f_n(y).\] In particular, $(\Gamma\mhyphen\liminf_{n\rightarrow \infty} f_n)(x)\leq (\Gamma\mhyphen\limsup_{n\rightarrow \infty} f_n)(x)$, and if $f=\Gamma\mhyphen\lim_{n\rightarrow \infty} f_n$ exists, then $(\Gamma\mhyphen\liminf_{n\rightarrow \infty} f_n)(x)= f(x)= (\Gamma\mhyphen\limsup_{n\rightarrow \infty} f_n)(x)$ for all $x\in X.$

\end{proposition}
\begin{proof} Let $f(x)=\sup_{U\in\mathscr{B}(x)}\liminf_{n\rightarrow \infty} \inf_{y\in U} f_n(y).$ Then we have
\begin{align*}
(x,t)\in \epi(f) &\iff \forall U \in \mathscr{B}(x)\forall N\;[\inf_{n\geq N}\inf_{y\in U} f_n(y)\leq t\;] \\
&\iff \forall U \in \mathscr{B}(x)\forall V\in \mathscr{N}(t)\forall N \exists n\geq 
N\exists y\in U\exists s\in V\;[f_n(y)\leq s]\\
&\iff \forall U\in \mathscr{N}(x,t)\forall N\exists n\geq N\;[\epi(f_n)\cap U\neq 
\emptyset]\\
&\iff (x,t)\in K\limsup_{n\rightarrow \infty} 
\epi(f_n)=\epi(\Gamma\mhyphen\liminf_{n\rightarrow \infty} f_n).
\end{align*}
Hence $f=\Gamma\mhyphen\liminf_{n\rightarrow \infty} f_n.$ 

Now let $g(x)=\sup_{U\in\mathscr{B}(x)}\limsup_{n\rightarrow \infty} \inf_{y\in U} f_n(y).$ Then we have
\begin{align*}
(x,t) \notin \epi(g) &\iff \exists U \in \mathscr{B}(x)\;\inf_N \sup_{n\geq N} \inf_{y\in U} f_n(y)>t \\
&\iff \exists U \in \mathscr{B}(x)\exists V\in \mathscr{N}(t)\forall N \exists n\geq 
N\forall y\in U\forall s\in V\;[f_n(y)> s]\\
&\iff \exists U\in \mathscr{N}(x,t)\forall N\exists n\geq N\;[\epi(f_n)\cap U= 
\emptyset]\\
&\iff (x,t)\notin K\mhyphen\liminf_{n\rightarrow \infty} 
\epi(f_n)=\epi(\Gamma\mhyphen\limsup_{n\rightarrow \infty} f_n).
\end{align*}
Hence $g=\Gamma\mhyphen\limsup_{n\rightarrow \infty} f_n.$ \end{proof}

\begin{proposition}  \label{gammachar} If $(f_n)_{n=1}^\infty$ is a sequence of extended real valued functions in $X$ with $f=\Glim f_n$ then for every $x\in X$ and for any sequence $x_n \convn x$ in $X$ we have 
\[f(x)\leq \liminf_{n\rightarrow \infty} f_n(x_n).\] Moreover, if in addition $X$ is first countable then for every $x\in X$ there exists a sequence $x_n\convn x$ in $X$ with 
\[f(x)=\lim_{n\rightarrow \infty}f_n(x_n).\]
\end{proposition}

\begin{proof} Let $x_n \convn x$. Let $U\in \mathscr{B}(x).$ Then there exists an $N$ such that for all $n\geq N$ we have $x_n\in U.$ So if $M\geq N,$ \[\inf_{n\geq M}\inf_{y\in U} f_n(y) \leq \inf_{n\geq M}f_n(x_n).\] Hence letting $M\rightarrow \infty$ yields
\[\liminf_{n\rightarrow \infty} \inf_{y\in U}f_n(y)\leq \liminf_{n\rightarrow \infty} f_n(x_n).\]
Since this holds for all $U\in \mathscr{B}(x),$ 
\[f(x)\leq \liminf_{n\rightarrow \infty} f_n(x_n).\]

Now suppose $X$ is first countable. Let $x\in X$. We may assume $f(x)<\infty.$ 

Since we may choose $\mathscr{B}(x)$ countable, let $\mathscr{B}(x)=(B_k)_{k=1}^\infty.$ Moreover, we may assume $B_{k+1}\subset B_k$ for each $k.$ 

Let $n_0=1.$ Suppose for $j\in \mathbb{Z}^+$, $(n_k)_{k=0}^{j-1}$ have been chosen with $n_{k-1}<n_l$ for $0<k\leq j-1$ and for each $0<k\leq j-1$ and each $l\geq n_k$ we have a $y^l_k \in B_k$ with $f_l(y_k^l)<f(x)+\frac{1}{k}.$ Then let $n_j>n_{j-1}$ such that 
\[\sup_{l\geq n_j} \inf_{y\in B_j} f_l(y) < f(x)+\frac{1}{j}.\]
For $l\geq n_j$ let $y_l^j \in B_j$ with 
\[f_l(y_j^l)<f(x)+\frac{1}{j}.\]

Then set 
\begin{displaymath}
   x_k = \left\{
     \begin{array}{lr}
       x\;,\;\;k<n_1\\
       y^k_j \;,\;n_j\leq k<n_{j+1}.\;
     \end{array}
   \right.
\end{displaymath} 
Then $x_k \convk x$ and \[f_k(x_k)<f(x)+\frac{1}{j}\;\;\mbox{for}\;\;k\geq n_j.\]
Hence $\sup_{k\geq n_j} f_k(x_k) \leq f(x)+\frac{1}{j}.$ So 
\[\limsup_{n\rightarrow \infty} f_n(x_n)\leq f(x).\]
However, we have just  shown that $f(x)\leq \liminf_{n\rightarrow \infty} f_n(x_n)$ so the result follows.

\end{proof}

The following proposition may be found in \cite{maso} as Theorem 8.5. 

\begin{proposition} \label{gamma}Suppose $X$ is second countable. Then every sequence $(f_n)_{n=1}^\infty$ of extended real valued functions on $X$ has a $\Gamma \mhyphen$convergent subsequence. 
\end{proposition} 

\begin{proof} Since $X$ is second countable we may find a countable neighborhood base $\mathscr{B}=(B_k)_{k=1}^\infty.$ Since $\overline{\mathbb{R}}$ is compact, for each $k$ there exists a strictly increasing $\eta_k:\mathbb{Z}^+\rightarrow \mathbb{Z}^+$ such that $\lim_{n\rightarrow \infty} \inf_{y\in U_k} f_{\eta_k(n)}(y)$ exists in $\overline{\mathbb{R}}$. Then let $\alpha_1=\eta_1$ and for $k> 1,$ $\alpha_k=\eta_k\circ\alpha_{k-1}.$ Then define $\alpha:\mathbb{Z}^+\rightarrow \mathbb{Z}^+$ by $\alpha(k)=\alpha_k(k).$ Then $\alpha$ is strictly increasing and $\lim_{n\rightarrow \infty} \inf_{y\in B_k} f_{\alpha(n)}(y)$ exists for every $k\in \mathbb{Z}^+.$ Then set \[f(x)=\sup_{U\in \mathscr{B}(x)} \lim_{n\rightarrow \infty} \inf_{y\in U} f_{\alpha(n)}(y).\] Then $\Gamma\mhyphen \lim_{n\rightarrow \infty} f_{\alpha(n)} = f.$
\end{proof}

Suppose $X$ is a (real) Banach space with dual $X^*.$ Recall the weak-* topology on $X^*$ is the smallest topology on $X^*$ making all of the evaluation maps ev$_x:X^*\rightarrow \mathbb{R},$ $\phi\mapsto \phi(x),$ continuous for $x\in X.$ Since $X$ is a Banach space, by the Hahn-Banach Theorem, it embeds isometrically into its double dual $X^{**}$ by the map $x\mapsto $ev$_x.$  The weak topology on $X$ is defined to be the weak-* topology on $X^{**}$ restricted to $X,$ identified with its image of under the evaluation embedding. We denote the dual pairing map $X^*\times X\rightarrow \mathbb{R}$ defined by $(\phi,x)\mapsto \phi(x)$ for $\phi\in X^*, x\in X,$ by $\langle \phi, x\rangle.$ In particular, if $(X, \langle \cdot, \cdot \rangle)$ is a Hilbert space, by the Riesz Representation Theorem, the map $x\mapsto \langle x, \cdot\rangle$ is an isometric isomorphism of $X$ with $X^*.$ A less detailed version of the following Lemma may be found in Proposition 8.7 of \cite{maso}.

\begin{lemma}\label{weak} Suppose $X$ is a Banach space with separable dual $X^*.$  Let $(\phi_i)_{i=1}^\infty$ be dense in $\{\phi\in X^*\;|\;\|\phi\|=1\}.$ Define a function $d$ on $X\times X$ by \[d(x,y)=\sum_{i=1}^\infty 2^{-i}|\langle \phi_i, x-y\rangle|.\] Then $d$ is a metric and $(X,d)$ is separable. Moreover, if $B$ is a bounded subset of $X$ then the weak subspace topology on $B$ corresponds with the $d$-metric topology on $B$. 
\end{lemma}
\begin{proof} Clearly $d$ is finite, non-negative valued, symmetric, and satisfies the triangle inequality. Since by the Hahn-Banach Theorem $X^*$ separates the points of $X,$ it follows that $d(x,y)=0$ if and only if $x=y.$ Hence $d$ is a metric. 

Let $B_M=\{y\in X\;|\;\|y-x\|\leq M\}$ for $M\in \mathbb{Z}^+.$ Let $x\in X.$ Suppose $(x_\alpha)$ is a net in $B_M$ with $x_\alpha\rightarrow x$ weakly. Let $\epsilon>0.$ Choose $N$ so large that $\sum_{j=N+1}^\infty 2^{-j}<\frac{\epsilon}{4M}.$ Then we may choose $\alpha_0$ so that for $\alpha\geq \alpha_0$ and $j\leq N,$ $|\phi_j(x_\alpha)-\phi_j(x)|<\frac{\epsilon}{2}.$ Then for $\alpha\geq \alpha_0$, $d(x_\alpha,x)\leq \sum_{j=1}^N \epsilon2^{-j}+\sum_{j=N+1}^\infty 2^{-j}\|x_\alpha-x\|<\epsilon.$ Conversely, suppose $(x_\alpha)$ is a net in $B_M$ with $d(x_\alpha,x)\rightarrow 0$. Then $\phi_j(x_\alpha)\rightarrow \phi_j(x)$ for all $j\in\mathbb{Z}^+.$ Since the multiples $c\phi_i$ for $c>0, i\in \mathbb{Z}^+,$ are dense in $X^*$, it follows $\phi(x_\alpha)\rightarrow \phi(x)$ for all $\phi\in X^*.$ Therefore the $d$-metric topology on $B_M$ and the weak subspace topology on $B_M$ are the same. Then since $(B_M,d)$ is separable and $X=\cup_{M=1}^\infty B_M,$ it follows that $(X,d)$ is separable.\end{proof}

Part of the following Proposition may be found in \cite{maso} as Corollary 8.12.

\begin{proposition} \label{weakcomp} Suppose $X$ is a separable Banach space with separable dual, $(f_n)_{n=1}^\infty$ is a collection of $\overline{\mathbb{R}}$ valued functions on $X$ with $f_n\geq \Psi$ for all $n,$ where $\Psi:X\rightarrow \overline{\mathbb{R}}$ satisfies $\lim_{\|x\|\rightarrow \infty} \psi(x)=\infty$ for all $x\in X.$ Then there exists a subsequence $(f_{n_k})$ such that $f_{n_k}$ $\Gamma$-converges in both the norm topology and weak topology of $X.$ 
\end{proposition}
\begin{proof} By Proposition \ref{gamma}, we may assume $f_n$ $\Gamma$-converges in the norm topology of $X.$ Let $f^{-}_d,f^{+}_d$ be the lower and upper $\Gamma$-limits under the $d-$topology, respectively, and let $f^{+}_w, f^{-}_w$ be the lower and upper $\Gamma$-limits under the weak topology, respectively. We show $f^+_d=f_w^+.$ For the proof that $f^-_d=f_w^{-}$, see Proposition 8.10 in cite \cite{maso}. Recall \[(\Gamma\mhyphen\limsup_{n\rightarrow \infty} f_n)(x)=\sup_{U\in\mathscr{B}(x)}\limsup_{n\rightarrow \infty} \inf_{y\in U} f_n(y).\] Let $x\in X.$ We prove $f^+_w(x)\leq f^+_d(x).$ The proof that $f^+_d(x)\leq f^+_w(x)$ is completely analogous. We may assume $f^+_d(x)<\infty.$ Suppose $t<f^+_w(x).$ Then there exists a weak open neighborhood $U$ of $x$ such that $t<  \limsup_{n\rightarrow \infty} \inf_{y\in U} f_n(y).$ Let $r\in \mathbb{R}$ such that $r>f^+_d(x).$ Let $B=\{y\in X\;|\;\Psi(y)\leq r\}.$ Then $B$ is norm bounded, by our assumptions on $\Psi.$ However, Lemma \ref{weak} shows that the $d$-topology on $B$ and the weak subspace topology on $B$ are the same. So there exists a $d$-open neighborhood $V$ of $x$ such that $V\cap B\subset U.$ Note $\inf_{y\in V\cap B}f_n(y)\geq \inf_{y\in U}f_n(y)$ for all $n.$ Hence $t<\limsup_{n\rightarrow \infty} \inf_{y\in V\cap B} f_n(y).$ Note that if $y\in B^c,$ then $r<\psi(y)\leq f_n(y)$ for all $n.$ Therefore  \[\inf_{y\in V}f_n(y)=\min(\inf_{y\in V\cap B}f_n(y),\inf_{y\in V\cap B^c}f_n(y))\geq \min(\inf_{y\in V\cap B}f_n(y), r).\] Hence \[f^+_d(x)\geq \limsup_{n\rightarrow \infty}\inf_{y\in V}f_n(y)\geq \min(\limsup_{n\rightarrow \infty} \inf_{y\in V\cap B} f_n(y),r)>\min(t,r).\] Letting $r\rightarrow \infty$ yields $f^+_d(x)>t.$ Since $t<f^+_w(x)$ was arbitrary, $f^+_w(x)\leq f^+_d(x),$ as desired. 

However, $(X,d)$ is separable. Therefore, by Proposition \ref{gamma},  there exists a subsequence $f_{n_k}$ such that $f_{n_k}$ $\Gamma$-converges in the $d$-topology. However, we have just show that this implies $\Gamma$-convergence in the weak topology as well. Since $f_n$ $\Gamma$-converges in the norm topology, by Lemma \ref{subsequence} any subsequence does as well. This proves the result.
\end{proof}

\subsection{Dirichlet Forms}
Our presentation in this section is inspired by \cite{mosco}. Let $\mathscr{H}$ be a real Hilbert space. 

\begin{definition} By a form on $\mathscr{H}$ we mean a subspace\footnote{Following \cite{mosco}, we do not require forms or Dirichlet forms to have dense domain.} $\mathscr{D}(Q)$ of $\mathscr{H}$ together with a positive definite symmetric bilinear mapping $Q:\mathscr{D}(Q)\times \mathscr{D}(Q)\rightarrow \mathbb{R}.$ 
\end{definition} 

Given a form $Q$ we define $Q_\triangle:\mathscr{H}\rightarrow \overline{\mathbb{R}}$ by  \begin{displaymath}
   Q_\triangle(f) = \left\{
     \begin{array}{lr}
       Q(f,f) & , f \in \mathscr{D}(Q)\\
       \infty & , f \notin \mathscr{D}(Q).
     \end{array}
   \right.
\end{displaymath} 
Then $\mathscr{D}(Q)=\{f\in \mathscr{H}\;|\;Q_\triangle(f)<\infty\}.$ Moreover, for $f,g\in \mathscr{D}(Q)$ we have the polarization identity 
\[Q(f,g)=\frac{1}{4}(Q_\triangle(f+g)-Q_\triangle(f-g)).\]

Note also that $\sqrt{Q_\triangle}$ is a seminorm on $\mathscr{D}(Q)$ and for $f,g\in \mathscr{D}(Q),$ \[|Q(f,g)|\leq \sqrt{Q_\triangle(f)}\sqrt{Q_\triangle(g)}.\]

\begin{lemma} A functional $F:\mathscr{H}\rightarrow \overline{\mathbb{R}}$ is of the form $Q_\triangle$ for some uniquely defined form $Q$ if and only if the following three properties hold:
\begin{align*}
1.& \;F(0)=0\;\mbox{and\;} F(f)\geq 0\;\mbox{for all\;} f\in \mathscr{H},\\ 2.&\;F(tf)=t^2F(f)\;\mbox{for all\;} f\;\mbox{with\;} F(f)<\infty\;\mbox{and all\;} t\in \mathbb{R}, \;\mbox{and} \\ 3.& \;F(f+g)+F(f-g)=2F(f)+2F(g)\; \mbox{for all\;} f,g\;\mbox{with\;} F(f)<\infty, F(g)<\infty.
\end{align*} 
\end{lemma}
\begin{proof} If $F=Q_\triangle$ for some form $Q$ then the conditions stated clearly hold for $F.$ Conversely suppose $F$ satisfies the three properties listed in the statement of the lemma. Then let $\mathscr{D}(Q)=\{f\in \mathscr{H}\;|\;F(f)<\infty\}.$ Then $0\in\mathscr{D}(Q).$ Let $f,g\in \mathscr{D}(Q),t\in \mathbb{R}.$ Then $F(tf+g)+F(tf-g)=2F(tf)+2F(g)=2t^2F(f)+2F(g)<\infty.$ Hence $f+tg\in\mathscr{D}(Q).$ So $\mathscr{D}(Q)$ is a subspace. Then for $f,g\in \mathscr{D}(Q)$ let \[Q(f,g)=\frac{1}{4}(F(f+g)-F(f-g)).\] Then since $F(-f)=F(f),$ $Q$ is symmetric and non-negative definite. Note $Q(-f,g)=\frac{1}{4}(F(-(f-g))-F(-(f+g)))=-\frac{1}{4}(F(f+g)-F(f-g))=-Q(f,g).$ Let $f_1,f_2,g\in \mathscr{D}(Q)$. Then 
\begin{align*}
&Q(f_1+f_2,g)=\frac{1}{4}(F(f_1+(f_2+g))-F(f_1+(f_2-g))\\
&= \frac{1}{4}(2F(f_1)+2F(f_2+g)-F(f_1-f_2-g)-(2F(f_1)+2F(f_2-g)-F(f_1-f_2+g))\\
&=Q(f_1-f_2,g)+2Q(f_2,g).
\end{align*}
Interchanging the roles of $f_1$ and $f_2$ yields $Q(f_1+f_2,g)=Q(f_2-f_1,g)+2Q(f_1,g)=-Q(f_1-f_2,g)+2Q(f_1,g).$ Adding this to the previous equation yields $Q(f_1+f_2,g)=Q(f_1,g)+Q(f_2,g).$ Then by induction we have $Q(nf,g)=nQ(f,g)$ for all $n\in \mathbb{N}$ and $f,g\in\mathscr{D}(Q).$ Let $f,g\in\mathscr{D}(Q)$ and $n\in  \mathbb{Z}^+.$ Then $\frac{f}{n}\in \mathscr{D}(Q)$ and $Q(f,g)=nQ(\frac{f}{n},g).$ Hence $Q(\frac{f}{n},g)=\frac{1}{n}Q(f,g).$ Since $Q(-f,g)=-Q(f,g)$ it follows that $Q(qf,g)=qQ(f,g)$ for all $q\in \mathbb{Q}.$ Now note $F(f \pm g)=F(f)+F(g)\pm 2Q(f,g)\geq 0$. Hence \[|Q(f,g)|\leq \frac{1}{2}(F(f)+F(g)).\] Then let $t\in \mathbb{R}.$ Let $q\in\mathbb{Q}$ so that $|t-q|< \max\{1,\frac{\epsilon}{F(f)+F(g)+1}\}.$ Then 
\begin{align*} |tQ(f,g)-Q(tf,g)|&= |tQ(f,g)-Q((t-q)f,g)-qQ(f,g)|\\
&=|(t-q)Q(f,g)-Q((t-q)f,g)| 
\leq |t-q||Q(f,g)|+|Q(t-q)f,g)| \\
&\leq \frac{|t-q|}{2}(F(f)+F(g))+\frac{1}{2}(|t-q|^2F(f)+F(g))\\ 
&\leq |t-q|(F(f)+F(g))<\epsilon.
\end{align*} Hence $Q(tf,g)=tQ(f,g).$ Therefore $Q$ is a form with $Q_\triangle = F.$
\end{proof}

\begin{rem}
If $Q$ is a form then $\|\cdot\|_Q:\mathscr{D}(Q)\rightarrow \mathbb{R}$ defined by $\|f\|_Q=\sqrt{Q_\triangle(f)+\|f\|^2_\mathscr{H}}$ is a norm on $\mathscr{D}(Q).$ Moreover it is generated by an inner product $\langle f,g \rangle_Q =Q(f,g)+\langle f,g \rangle$ for $f,g\in \mathscr{D}(Q).$
\end{rem}
\begin{definition} A form $Q$ is called closed if $(\mathscr{D}(Q),\|\cdot\|_Q)$ is complete.
\end{definition}
By the above remark $Q$ is closed if and only if $\mathscr{D}(Q)$ is a Hilbert space with inner product $\langle \cdot,\cdot\rangle_Q.$

The following result may be found in \cite{mosco}.
\begin{proposition} A form $Q$ is closed if and only if $Q_\triangle$ is lower semicontinuous.
\end{proposition}
\begin{proof} Suppose $Q_\triangle:\mathscr{H}\rightarrow \overline{\mathbb{R}}$ is lower semicontinuous. Then let $(f_n)_{n=1}^\infty$ be a $\|\cdot\|_Q \mhyphen$ Cauchy sequence in $\mathscr{D}(Q).$ Then $(f_n)_{n=1}^\infty$ is also $\|\cdot\|_{\mathscr{H}}$ Cauchy. Hence there exists an $f\in \mathscr{H}$ with $f_n\;\overrightarrow{_{_{n\rightarrow \infty}}} \;f$ in $\mathscr{H}$. Let $\epsilon>0.$ Then there exists an $N>0$ such that \[Q_\triangle(f_n-f_m)\leq \epsilon\; \mbox{\;for\;}\;n,m\geq N.\] However since $Q_\triangle$ is lower semicontinuous, $Q_\triangle^{-1}([-\infty,\epsilon])$ is closed and the sequence tail $(f_n-f_m)_{n=N}^\infty$ lies in $Q_\triangle^{-1}([-\infty,\epsilon])$ for each $m\geq N$. Hence $f-f_m\in Q_\triangle^{-1}([-\infty,\epsilon])$ for $m\geq N.$
Therefore \[Q_{\triangle}(f-f_m)\leq \epsilon \;\;\mbox{for\;\;} m\geq N.\] It follows that $Q_\triangle(f-f_m)\;\overrightarrow{_{_{m\rightarrow \infty}}}\; 0$. Then let $m$ so large that $Q_\triangle(f-f_m)\leq 1.$ Then $Q_\triangle(f)\leq 1+Q_\triangle(f_m)<\infty.$ Hence $f\in \mathscr{D}(Q)$. So $(\mathscr{D}(Q),\|\cdot\|_Q)$ is complete.

Conversely, suppose $(\mathscr{D}(Q),\|\cdot\|_Q)$ is complete but $Q_\triangle$ is not lower semicontinuous. Then there exist a $c\geq 0,$ an $f\in Q_\triangle^{-1}((c,\infty]),$ and a sequence $(f_n)_{n=1}^\infty \subset Q_\triangle^{-1}([0,c])$ with $f_n\;\overrightarrow{_{_{n\rightarrow \infty}}}\;f$ in $\mathscr{H}.$  Since $(f_n)_{n=1}^\infty$ converges in $\mathscr{H}$ it is bounded in $\mathscr{H}$. Since $(Q_\triangle(f_n))_{n=1}^\infty$ is also bounded, there exists an $M>0$ so that $\|f_n\|_Q\leq M$ for all $n\in \mathbb{Z}^+.$ Since $(\mathscr{D}(Q),\langle \cdot,\cdot\rangle_{Q})$ is a Hilbert space, by Alaoglu's Theorem the sequence $(f_n)_{n=1}^\infty$ has a weakly convergent subsequence $(f_{n_k})_{k=1}^\infty$ with $f_{n_k} \rightharpoonup h$ as $k\rightarrow \infty$ for some $h$ with $\|h\|_Q\leq M.$ For $g\in \mathscr{H}$ let $g^*$ denote the functional $h\mapsto \langle g,h \rangle_\mathscr{H} \in \mathscr{H}^*.$ Then the map $g\mapsto g^*$ is an isometric isomorphism. Let $g\in \mathscr{H}.$ Then for $h\in \mathscr{D}(Q),$  we have $|g^*(h)|\leq \|h\|_{\mathscr{H}}\leq \|h\|_Q$. Hence $g^*=\langle g,\cdot\rangle_\mathscr{H}\in \mathscr{D}(Q)^*.$ So \[g^*(f_{n_k})=\langle g, f_{n_k}\rangle\; \overrightarrow{_{_{k\rightarrow \infty}}}\;g^*(h)=\langle g, h\rangle_{\mathscr{H}}.\]
Therefore $(f_{n_k})_{k=1}^\infty$ converges weakly in $\mathscr{H}$ to $h.$ However, since $(f_{n_k})_{k=1}^\infty$ converges strongly in $\mathscr{H}$ it also converges 
weakly in $\mathscr{H}$ to $f.$  Therefore $h=f.$ However for each $k,$ $|\langle f, f_{n_k}\rangle_Q|\leq \|f\|_Q\|f_{n_k}\|_Q$ so that taking lower limits (noting $\|f\|_Q> c\geq 0$ by assumption) yields \[\|f\|
_Q\leq \liminf_{k\rightarrow \infty}\|f_{n_k}\|_Q.\] So that also $\|f\|^2_Q\leq 
\liminf_{k\rightarrow \infty} \|f_{n_k}\|^2_Q.$ Then since $\|f_{n_k}\|_{\mathscr{H}}\;
\overrightarrow{_{_{k\rightarrow \infty}}}\;\|f\|_{\mathscr{H}}$, using the definition 
of $\|\cdot\|^2_Q=Q_\triangle+\|\cdot\|^2_\mathscr{H}$ we have 
\[Q_\triangle(f)\leq \liminf_{k\rightarrow \infty} Q_\triangle(f_{n_k})\leq c,\]
a contradiction. Hence $Q_\triangle$ is lower semicontinuous.
\end{proof}

We now further specialize to the case that $\mathscr{H}=L^2(X,\mu)$ is a separable Hilbert space, where $X$ is a topological space and $\mu$ is a positive Borel measure on $X.$ Note in particular that $\mathscr{H}$ is second countable. 

\begin{definition} A form $Q$ on $\mathscr{H}=L^2(X,\mu)$ is called \textit{strongly Markovian} 
if for every $f\in \mathscr{D}(Q)$ the function $\underline{\overline{f}}=\min\{\max\{f,0\},1\}\in\mathscr{D}(Q)$ and $Q_\triangle(\underline{\overline{f}})\leq Q_\triangle(f).$ 
\end{definition}

\begin{definition} A form $Q$ on $\mathscr{H}=L^2(X,\mu)$ is called a \textit{Dirichlet form}\footnote{Note our definition differs from the one in \cite{fukushima} in that we do not require the domain to be dense. What we here call a Dirichlet form is often called in the literature a Dirichlet form in the wide sense.}   if it is closed and strongly Markovian. 
\end{definition}

\begin{rem} The definition of a \textit{Markovian} form is more complicated than the definition of what we call here a strongly Markovian form.  However, in the case that a form is closed the two agree. See \cite{fukushima} for details. Moreover, all forms we consider here will be strongly Markovian so we have no need to introduce the full Markovian condition.
\end{rem}

Now suppose $(X,d)$ is a compact metric space with a finite Borel measure $\mu$. Let $C(X)$ denote the continuous functions on $X.$ 

\begin{definition} A Dirchlet form $Q$ on $L^2(X,\mu)$ with domain $\mathscr{D}(Q)$ is called \textit{regular} if $\mathscr{D}(Q)\cap C(X)$ is dense both in $C(X),$ with respect to the uniform norm, and in $\mathscr{D}(Q)$ with respect to the norm $f\mapsto \|f\|_Q=(Q(f,f)+\|f\|^2_{L^2(X,\mu)})^{1/2}.$ \end{definition}
\begin{definition} A form $Q$ is called \textit{local} if for every $f,g\in \mathscr{D}(Q)$ with disjoint supports\footnote{The support of a function on $X$ is the closure of the inverse image under the function of $\overline{\mathbb{R}}\setminus\{0\}$.  Hence the condition is $\overline{\{x\in X\;|\;f(x)\neq 0\}}\cap \overline{\{x\in X\;|\;g(x)\neq 0\}}=\emptyset$.},  $Q(f,g)=0.$ \end{definition}

Any regular Dirichlet form may be used to define a symmetric Markov process with c\`{a}dl\`{a}g paths\footnote{In fact, it corresponds to a symmetric Hunt process, which has paths that are continuous from the right and quasi-continuous from the left. See \cite{fukushima}.}. Since $Q$ is a closed form (with a dense domain by the regularity condition, as $C(X)$ is dense in $L^2(X,\mu))$, there exists a positive self-adjoint operator $\mathcal{L}$ with domain contained in $\mathscr{D}(Q)$ such that $Q(f,f)=\langle \mathcal{L}f, f \rangle$ for $f\in \mathscr{D}(\mathcal{L}).$ The operator $-\mathcal{L}$ is then the generator for the process corresponding to $Q$. The process corresponding to a local, regular, Dirichlet form has sample paths that are continuous with probability one. For more details and proofs of these statements, see \cite{fukushima}. 

For processes that are not necessarily symmetric, such as non-symmetric jump processes, there are generalizations of the theory of symmetric Dirichlet forms to what are called non-symmetric Dirichlet forms and semi-Dirichlet forms. See \cite{oshima} for more information. 
\subsection{$\Gamma$-Convergence of Strongly Markovian Forms}

Again, assume that $\mathscr{H}=L^2(X,\mu)$ is a separable Hilbert space, where $X$ is a topological space and $\mu$ is a positive Borel measure on $X.$

\begin{theorem} \label{gammacomp} Let $(Q_n)_{n=1}^\infty$ be a sequence of forms on $L^2(X,\mu)$. Then there exists a closed form $Q$ on $L^2(X,\mu)$ and a subsequence $(Q_{n_k})_{k=1}^\infty$ such that $Q_\triangle = \Gamma \mhyphen \lim_{k\rightarrow \infty} (Q_{n_k})_\triangle.$ If in addition each $Q_n$ is strongly Markovian, then $Q$ is a Dirichlet form.\end{theorem}

\begin{proof}

Let $\mathscr{H}=L^2(X)$. Since $\mathscr{H}$ is separable there exists a function $F:\mathscr{H}\rightarrow \overline{\mathbb{R}}$ and a subsequence $(Q_{n_k})_{k=1}^\infty$ such that $F = \Gamma \mhyphen \lim_{k\rightarrow \infty} (Q_{n_k})_\triangle.$  For simplicity we assume $F = \Gamma \mhyphen \lim_{n\rightarrow \infty} (Q_{n})_\triangle.$ 

Let $f\in \mathscr{H}.$ Then there exists and $f_n\convn f$ with $0\leq 
(Q_n)_{\triangle}(f_n) \convn F(f).$ Hence $F(f)\geq 0.$ Then $0\leq F(0)\leq \liminfn 
(Q_n)_\triangle(0)=0.$ So $F(0)=0.$ For $f\in \mathscr{H}$ with $F(f)<\infty,t\in \mathbb{R},t\neq 0$ let 
$f_n\convn f$ with $(Q_n)_{\triangle}(f_n)\convn F(f).$ Then $tf_n \convn tf$ so 
\[F(tf)\leq \liminfn (Q_n)_{\triangle}(tf_n) = t^2F(f).\]
Now let $g_n\convn tf$ with $(Q_n)_{\triangle}(g_n)\convn F(tf).$ Then $\frac{g_n}{t}\convn f$ and so \[F(f)\leq \liminfn (Q_n)_{\triangle}(\frac{g_n}{t})=\frac{1}{t^2}F(ft).\] Hence $F(tf)=t^2F(f).$ 
Now let $f,g\in \mathscr{H}$ with $F(f)<\infty, F(g)<\infty.$ Then let $f_n\convn f,g_n\convn g$ with $(Q_n)_{\triangle}(f_n)\convn F(f),(Q_n)_{\triangle}(g_n)\convn F(g).$ Then $f_n+g_n\convn f+g, f_n-g_n\convn f-g,$ and so 
\begin{align*}
F(f+g)+F(f-g)&\leq \liminfn (Q_n)_{\triangle}(f_n+g_n) + \liminfn (Q_n)_{\triangle}(f_n-g_n)\\
&\leq \liminfn ((Q_n)_{\triangle}(f_n+g_n)+(Q_n)_{\triangle}(f_n-g_n))\\
&= \liminfn (2(Q_n)_{\triangle}(f_n)+2(Q_n)_{\triangle}(g_n))=2F(f)+2F(g).
\end{align*}

Now let $u_n\convn f+g, v_n\convn f-g$ with $(Q_n)_{\triangle}(u_n)\convn F(f+g)$ and 
$(Q_n)_{\triangle}(v_n)\convn F(f-g).$ Then $u_n+v_n\convn 2f$, $u_n-v_n\convn 2g,$ and 
so 
\begin{align*}
4F(f)+4F(g)&=F(2f)+F(2g)\leq \liminfn (Q_n)_{\triangle}(u_n+v_n) + \liminfn (Q_n)_{\triangle}(u_n-v_n)\\
&\leq \liminfn ((Q_n)_{\triangle}(u_n+v_n)+(Q_n)_{\triangle}(u_n-v_n))\\
&= \liminfn (2(Q_n)_{\triangle}(u_n)+2(Q_n)_{\triangle}(v_n))=2F(f+g)+2F(f-g).
\end{align*}
Hence $F(f+g)+F(f-g)=2F(f)+2F(g).$ It follows that $F=Q_\triangle$ for some form $Q.$ 

Since $F$ is a $\Gamma \mhyphen$ limit it follows that $\epi(F)$ is closed. Therefore $Q_\triangle$ is lower semicontinuos. Hence $Q$ is a closed form. 

We complete the proof by showing that if each $Q_n$ is strongly Markovian then $Q$ is strongly Markovian. 
Let $f\in \mathscr{D}(Q).$ Let $f_n\convn f$ with $(Q_n)_{\triangle}(f_n) \convn Q_\triangle (f).$ Then since
\[\int |\overline{\underline{f_n}}-\overline{\underline{f}}|^2 d\mu \leq \int |f_n-f|^2d\mu,\] it follows that $\overline{\underline{f_n}} \convn \overline{\underline{f}}.$

Hence \[Q_\triangle(\overline{\underline{f}})\leq \liminfn (Q_n)_{\triangle}(\overline{\underline{f_n}})\leq \liminfn (Q_n)_{\triangle}(f_n)=Q_{\triangle}(f).\]
Therefore $Q$ is a Dirichlet form.
\end{proof}

Suppose $(Q_n)_{n=1}^\infty$ is a sequence of forms and $Q$ is another form, defined on $L^2(X,\mu)$. Then, following \cite{mosco}, a sequence $Q_n$ is said to \textit{Mosco converge} to $Q$ if \begin{equation*}
\begin{split} \mbox{(a)}\;& \mbox{for every\;} f\in L^2(X,\mu)\; \mbox{and every sequence\;} (f_n)_{n=1}^\infty \mbox{\;in\;} L^2(X,\mu) \mbox{\;converging weakly to\;} f, \\&Q(f,f)\leq \liminf_{n\rightarrow \infty} Q(f_n,f_n), \mbox{\;and}\\
\mbox{(b)}\;&\mbox{for every\;} f\in L^2(X,\mu), \mbox{\;there exists a sequence\;} (f_n)_{n=1}^\infty \mbox{\;converging (strongly) to\;} f \mbox{\;with\;} \\&\limsup_{n\rightarrow \infty} Q(f_n,f_n)\leq Q(f,f).
\end{split}
\end{equation*}

In particular, by the characterization from Proposition \ref{gammachar}, a sequence $(Q_n)_{n=1}^\infty$ on the separable space $L^2(X,\mu)$ $\Gamma$-converges with respect to the strong and weak topologies on $L^2(X,\mu)$ if and only if it Mosco converges to the same limit.  We then have the following result. 
\begin{proposition} \label{mosco} Suppose $(Q_n)_{n=1}^\infty$ is a sequence of forms on $L^2(X,\mu)$ that are uniformly coercive, that is there exists a $c>0$ with $Q_m(f,f)\geq c\|f\|^2$ for all $f\in L^2(X,\mu).$ Then there exists a closed form $Q$ on $L^2(X,\mu)$ and a subsequence $Q_{n_k}$ such that $Q$ is the Mosco limit of the $Q_{n_k}$ on $L^2(X,\mu)$. Moreover, $Q(f,f)\geq c\|f\|^2_{L^2(X,\mu)}$ for all $f\in L^2(X,\mu).$ If in addition each $Q_n$ is strongly Markovian, then $Q$ is a Dirichlet form.
\end{proposition}
\begin{proof} Since $L^2(X,\mu)$ is a separable Hilbert space, its dual, being isomorphic to $L^2(X,\mu)$, is also separable. Let $\Psi(f)=c\|f\|^2.$ Then apply Proposition \ref{weakcomp} and Theorem \ref{gammacomp}. Let $f_n\rightarrow f$ in $L^2(X,\mu)$ such that $Q_n(f_n,f_n)\rightarrow Q(f,f).$ Since $\|f_n\|^2_{L^2(X,\mu)}\rightarrow \|f\|^2_{L^2(X,\mu)},$ we have $Q(f,f)\geq c\|f\|^2_{L^2(X,\mu)}.$ \end{proof}

\chapter{Local Dimension}
In this chapter we define a local Hausdorff dimension $\alpha.$ We then define a corresponding  local Haudorff measure and study variable Ahlfors regularity. We show that any variable Ahlfors regular measure of variable dimension $Q$ is strongly equivalent to the local Haudorff measure and that $Q=\alpha,$ in analogy with the constant dimensional theory. Finally, at the end we construct various examples, including a variable dimensional Sierpi\'nski carpet. This chapter may be also found in the work \cite{Dever1} of the author.

\section{Local Hausdorff Dimension and Measure}
Let $(X,d)$ be a metric space. Let $\mathscr{O}(X)$ be the collection of open subsets of $X.$ For $x\in X,$ let $\mathscr{N}(x)$ be the open neighborhoods of $x.$ 
\begin{definition}
Define the \textit{local (Hausdorff) dimension} $\alpha:X\rightarrow [0,\infty]$ by \[\alpha(x):=\inf \{\dim(U)\;|\; U\in 
\mathscr{N}(x)\}.\]   Also by \ref{mon}, \[\alpha(x)=\inf\{\dim(B_\epsilon(x))\; | \;\epsilon>0\}.\] 
\end{definition}
\begin{proposition} The local dimension $\alpha$ is upper semicontinuous. In particular, it is Borel measurable and bounded above. \label{semicont}
\end{proposition}
\begin{proof} Let $c\geq 0.$ If $c=0$ then $\alpha^{-1}([0,c))=\emptyset\in\mathscr{O}
(X)$. So let $c>0.$ Then suppose $x \in \alpha^{-1}([0,c)).$ Then there exists a $U\in 
\mathscr{N}(x)$ such that $\dim(U)<c.$ Then for $y\in U$, since $U$ is open there exists a $V\in \mathscr{N}(y)$ with $V\subset U.$ So $\alpha(y)\leq \dim(V)\leq \dim(U)<c.$ Hence $x\in U \subset \alpha^{-1}([0,c)).$ Therefore $\alpha^{-1}([0,c))$ is open. Note the sets $\alpha^{-1}([0,n))$ for $n\in \mathbb{Z}^+$ form an open cover of $X$. Since $X$ is compact there exists an $N\in \mathbb{Z}^+$ such that $X=\alpha^{-1}([0,N)).$ Hence $\alpha$ is bounded above.
 
\end{proof}

\begin{lemma} If $A$ is a Borel set and $0\leq s_1\leq s_2$ then $\lambda^{s_2}(A)\leq \lambda^{s_1}(A).$ 
\end{lemma}
\begin{proof} If $\mathscr{U}\in \mathscr{B}_\delta(A)$ then $|U|^{s_2}\leq |U|^{s_1}$ for $0<\delta<1.$ The result then follows from the definition of $\lambda^s.$
\end{proof}

For $U\subset X, U\neq \emptyset,$ let $\tau(U)=|U|^{\dim(U)}.$ Set $\tau(\emptyset)=0.$ Then $H_{\loc}:=\mu_\tau$ is called the \textit{local Hausdorff measure}. If we restrict $\tau$ to $\mathscr{B}$ then the measure $\lambda_{\loc}:=\mu_\tau$ is called the \textit{local open spherical measure}. 

\begin{lemma} \label{asym} If $\dim(X)<\infty$ then $H_{\loc} \asymp \lambda_{\loc}.$\end{lemma}
\begin{proof} Clearly $H_{\loc} \leq \lambda_{\loc}.$ Let $A$ be a Borel set. We may assume $H_{\loc}(A)<\infty$ and $A\neq \emptyset.$ Let $\epsilon>0,0<\delta<\frac{1}{4},$ $\mathscr{U}\in \mathscr{C}_\delta(A)$ with $U\neq \emptyset$ for $U\in \mathscr{U}$ and 
$\sum_{U\in \mathscr{U}}|U|^{\dim(U)}<\infty.$ Let $x_U\in U$ for each $U\in \mathscr{U}.
$ If $|U|=0$ then $U=\{x_U\}$ is a singleton and so $\dim(U)=0.$ Then, by our convention, $|U|^{\dim(U)}=1.$ Hence there are at most finitely 
many $U\in \mathscr{U}$ with $|U|=0.$ Let $\mathscr{U}_0$ be the collection of such 
$U\in\mathscr{U}.$ Let $m:=\min_{U\in \mathscr{U}_0}\alpha(x_U).$ let $0<r<\delta/2$ 
such that $(2r)^m\leq 1.$ Then, for $U\in \mathscr{U}_0,$ $U\subset B_r(x_U)$ and, since 
$0<2r<\delta<1$ and $m\leq \alpha(x_U)\leq \dim(B_r(x_U)),$ $|B_r(x_U)|
^{\dim(B_r(x_U))}\leq (2r)^m \leq 1.$ For $U\in \mathscr{U}_0$ set $r_U:=r$. Let 
$\mathscr{U}_1$ be the collection of $U\in \mathscr{U}$ with $|U|>0.$ For $U\in 
\mathscr{U}_1$ let $r_U:=2|U|.$ Then for $U\in \mathscr{U}$ let $B_U:=B_{r_U}(x_U).$ Then 
$U\subset B_U$, $(B_U)_{U\in \mathscr{U}} \in \mathscr{B}_{4\delta}(A),$ and, since $|B_U|\leq 4|U|\leq 4\delta<1$ 
and $\dim(U)\leq \dim(B_U) \leq \dim(X)<\infty$ for $U \in \mathscr{U},$ $\sum_{U\in 
\mathscr{U}} |B_U|^{\dim(B_U)} \leq \sum_{U\in \mathscr{U}_0} 1 + 
\sum_{U \in \mathscr{U}_1} (4|U|)^{\dim(U)} \leq 4^{\dim(X)} \sum_{U \in \mathscr{U}}|U|^{\dim(U)}.$ 
Hence $\lambda_{\loc} \leq 4^{\dim(X)}H_{\loc}.$
\end{proof}

\begin{proposition}
If $d_0$ is the dimension of $X,$ then $H^{d_0}\ll  H_{\loc}.$ 
\end{proposition}
\begin{proof} 
Let $\phi(U)=|U|^{d_0}$, $\tau(U)=|U|^{\dim(U)}$ where $\phi(\emptyset)=0, \tau(\emptyset)=0.$ By definition, since $\dim(U)\leq d_0$ for all $U,$ if $0<\delta<1$ then for any set $A,$ 
\[\mu^*_{\phi,\delta}(A)\leq \mu^*_{\tau,\delta}(A).\] 

Suppose $N\subset X$ is Borel measurable with $H_{\loc}(N)=0.$ Let $\epsilon>0.$ Then for all $\delta>0$ there exists a $\mathscr{U}_\delta \in \mathscr{C}_\delta(N)$ such that $\sum_{U\in \mathscr{U}_\delta}|U|^{\dim(U)}<\epsilon.$ But since for  $U\in \mathscr{U}_\delta,$ $\{U\}\in \mathscr{C}_\delta(U),$ and since $\mu^*_{\phi,\delta}$ is an outer measure, for $0<\delta<1$ we have
\[\mu^*_{\phi,\delta}(N)\leq \sum_{U\in \mathscr{U}_\delta}\mu^*_{\phi,\delta}(U)\leq 
\sum_{U\in \mathscr{U}_\delta}\mu^*_{\tau,\delta}(U)\leq \sum_{U\in \mathscr{U}_\delta}|U|^{\dim(U)}<\epsilon.\] Hence $H^{d_0}(N)\leq \epsilon.$  Since $\epsilon>0$ was arbitrary, $H^{d_0}(N)=0.$
\end{proof}

The following two propositions relate the local dimension to the global dimension.

\begin{proposition} \label{3.6} Suppose $X$ is separable with Hausdorff dimension $d_0.$ Let $A:=\alpha^{-1}([0,d_0)).$ Then $H^{d_0}(A)=0.$ In particular, $H^{d_0}(X)=H^{d_0}(\alpha^{-1}(\{d_0\}))$.
\end{proposition}
\begin{proof}
$A$ is open since $\alpha$ is upper semicontinuous. We may assume $A$ is non-empty. For $x\in A$ let $U_x\in \mathscr{N}(x)$ with $\dim(U_x)<d_0$ and $U_x\subset A.$ Then the $U_x$ form an open cover of $A.$ Since $X$ has a countable basis, there exists a countable open cover $(U_k)$ of $A$ with the property that for all $x$ there exists a $k$ with $x\in U_k \subset U_x.$ In particular $\dim(U_k)\leq \dim(U_x)<d_0.$ Let $A_j:=\cup_{k\leq j} U_k.$ Then, since $H^s(A_j)\leq \sum_{k\leq j} H^s(U_k)$ for any $s\geq 0,$ $\dim(A_j)\leq \max_{k\leq j} \dim(U_k)<d_0.$  So $H^{d_0}(A_j)=0$ for all $j.$ But by continuity of measure, $H^{d_0}(A)=\sup_j H^{d_0}(A_j)=0.$ Since $\alpha\leq d_0$ the other result follow immediately. 
\end{proof}

\begin{proposition} \label{dimmax}
Let $X$ be a separable metric space. Then $\dim(X)=\sup_{x\in X}\alpha(x).$ Moreover, if $X$ is compact then the supremum is attained. 
\end{proposition}

\begin{proof}
Clearly $\sup_{x\in X}\alpha(x) \leq \dim(X).$ Conversely, let $\epsilon>0.$ For $x\in X$ let $U_x\in \mathscr{N}(x)$ such that $\dim(U_x)\leq \alpha(x)+\frac{\epsilon}{2}.$ Then the $(U_x)_{x\in X}$ form an open cover of $X.$ Since $X$ is separable it is Lindel\"{o}f. So let $(U_{x_i})_{i\in \mathbb{Z}_{+}}$ be a countable subcover. Then if $\sup_{i\in \mathbb{Z}^+}\dim(U_{x_i})=\infty$ then also $\dim(X)=\infty.$ Else if $\sup_{i\in \mathbb{Z}^+}\dim(U_{x_i})<t<\infty$ then $H^t(U_{x_i})=0$ for all $i$ and so $H^t(X)\leq \sum_{i=1}^\infty H^t(U_{x_i})=0.$ So $\dim(X)\leq t.$ Hence $\dim(X)=\sup_{i\in \mathbb{Z}^+}\dim(U_{x_i}).$ Then choose $j\in \mathbb{Z}^+$ such that $\dim(X)\leq \dim(U_{x_j})+\frac{\epsilon}{2}.$ Then $\dim(X)\leq \alpha(x_j)+\epsilon\leq \sup_{x\in X}\alpha(x)+\epsilon.$ Hence $\dim(X)=\sup_{x\in X}\alpha(x).$

Now suppose $X$ is compact. Let $d_0:=\dim(X).$ It remains to show that there exists some $x\in X$ with $\alpha(x)=d_0.$ Suppose not. Then clearly $d_0>0.$ Let $m\geq 1$ be an integer such that $\frac{1}{m}<d_0.$ Then the sets $U_n:=\alpha^{-1}[0,d_0-\frac{1}{n})$ for $n\geq m$ form an open cover of $X.$ By compactness there exists a finite subcover. So there exists an $N>0$ such that $X=\alpha^{-1}[0,d_0-\frac{1}{N}).$ Hence $\sup_{x\in X} \alpha(x)\leq d_0-\frac{1}{N}<d_0,$ a contradiction. \end{proof}

\section{Variable Ahlfors Regularity}

Recall that a metric space $(X,d)$ is Ahlfors regular of exponent $s>0$ if there exists a 
Borel measure $\mu$ on $X$ and a $C>1 $ such that for all $x\in X$ and all $0<r\leq 
\diam(X),$ \[\frac{1}{C}r^{s}\leq \mu(B_r(x))\leq Cr^{s}.\] It is well known that if $X$ 
supports such a measure $\mu$  then $s$ is the Hausdorff dimension of $X$ and $\mu$ is 
strongly equivalent to the Hausdorff measure $H^{s}$ \cite{Hein},\cite{Fed}.

In this section we generalize this result on a compact metric space. 

\begin{definition} If $Q:X\rightarrow (0,\infty)$ is a bounded function, then a measure $\nu$ is called \textit{(variable) Ahlfors $Q\mhyphen$regular} if there exists a constant $C>1$ so that \[\frac{1}{C} \nu(B_r(x))\leq  r^{Q(x)}\leq C\nu(B_r(x))\] for all $0<r\leq \frac{\diam(X)}{2}$ and $x\in X$\cite{Sob}.\end{definition}

 We show that if $X$ is compact and supports such a measure $\nu$ then $Q$ is the local Hausdorff dimension and $\nu$ is strongly equivalent to the local Hausdorff measure $H_{\loc}.$
Our presentation in this section is strongly influenced by \cite{Sob}.

Suppose $(X,d)$ is compact. For $Q:X\rightarrow [0,\infty)$, define $Q^-,Q^+, Q^c:\mathscr{B}\rightarrow [0,\infty)$ by $Q^-(U)=\inf_{x\in U} Q(x), Q^+(U)=\sup_{x\in U} Q(x).$ For arbitrary $Q:X\rightarrow [0,\infty)$ define $Q_c:\mathscr{B}\rightarrow 
[0,\infty)$ by $Q_c(B_r(x))=Q(x).$ Then for $\tilde{Q}:\mathscr{B}\rightarrow [0,\infty)$ with $Q^- \leq \tilde{Q}\leq Q^+,$ let $\lambda^{\tilde{Q}}:=\mu_\tau$, where $\tau$ is restricted to $\mathscr{B}$ and defined by $\tau(B):=|B|^{\tilde{Q}(B)}, \tau(\emptyset)=0.$

\begin{definition}For $Q:X\rightarrow [0,\infty)$, we call $X$ \textit{$Q\mhyphen$amenable} if $0<\lambda^{Q_c}(B)<\infty$ for every non-empty open ball $B$ of finite radius in $X.$\footnote{In the case of $X$ compact this is equivalent to $\lambda^{Q_c}$ being finite with full support.} \end{definition} 

\begin{proposition} \label{amen} If $X$ is $Q\mhyphen$amenable with $Q$ continuous then $\alpha(x)=Q(x)$ for all $x\in X.$ 
\end{proposition}
\begin{proof} Let $B=B_r(x)$ be a non-empty open ball, $q^-:=Q^-(B), q^+:=Q^+(B), d_0:=\dim(B).$ Let $B':=B_{\frac{r}{2}}(x)$ Then if $d_0<q^-, \dim(B')\leq d_0< q^-$ so $\lambda^{q^-}(B')=0.$ Let $0<\delta<\min\{\frac{r}{8}, 1\}.$ Let $\mathscr{U}\in \mathscr{B}_\delta(B')$ and $U\in \mathscr{U}.$ We may assume $U\cap B' \neq \emptyset$, since otherwise $\mathscr{U}$ may be improved by removing such a $U.$ Say $U=B_{r_U}(x_U).$
Moreover, we may assumue $r_U\leq 2\delta.$ Indeed, if $|U|=0$ and $r_U>\delta$ then $U=B_{\delta}(x_U)$ so we may take $r_U=\delta$ in that case. If $|U|>0$ then if $r_U>2|U|$ then $B_{r_U}(x_U) = B_{2|U|}(x_U)$ so we may take $r_U = 2|U|\leq 2\delta.$ Say $w\in U\cap B'.$ Then if $z \in U,$ 
$d(z,x)\leq d(z,x_U)+d(x_U,w)+d(w,x)\leq 2r_U+\frac{r}{2} \leq 4\delta+\frac{r}{2}<r.$ So $U\subset B.$ Hence $q^- \leq Q^c(U)$ and since $|U|\leq 1,$ $|U|^{q^-}\geq |U|^{Q^c(U)}$. So $\lambda^{Q_c}(B')=0,$ a contradiction. If $d_0>q^+$ then since $B_r(x)=\cup_{n=1}^\infty B_{r-\frac{1}{n}}(x),$ it is straightforward, using countable subadditivity of the measures $H^s$ and the definition of Hausdorff dimension, to verify that $\dim(B_r(x)) = \sup_{n\geq 1} \dim(B_{r-\frac{1}{n}}(x))$. Since $d_0>q^+,$ let $N$ so that $\dim(B_{r-\frac{1}{N}}(x))>q^+.$ Let $r'=r-\frac{1}{N}$ and $B'=B_{r'}(x).$ Then $\lambda^{q^+}(B') = \infty.$ Let $0<\delta<\frac{1}{4N}.$ Let $\mathscr{U}\in \mathscr{B}_\delta(B')$ and $U\in \mathscr{U}.$ We may assume $U\cap B' \neq \emptyset,$ say $w\in U\cap B'$. As before we may assume $r_U\leq 2\delta$ and so if $z\in U$ then $\rho(z,x)\leq 2r_u+r'\leq 4\delta+r'<r.$ So $U\subset B.$ Hence $Q_c(U)\leq q^+$. Since $|U|<1,$ $|U|^{q^+}\leq |U|^{Q_c(U)}.$  Hence $\lambda^{Q_c}(B')=\infty,$ a contradiction. Hence $Q^-(B)\leq \dim(B) \leq Q^+(B)$ for every non-empty open ball $B.$ 
The result then follows since $Q$ is continuous.  
\end{proof}

A Borel measure $\nu$ on a metric space $X$ is said to have local dimension $d_\nu(x)$ at $x$ if $\lim_{r\rightarrow 0^+} \frac{\log(\nu(B_r(x)))}{\log(r)}=d_\nu(x)$. Since the limit may not exist, we may also consider upper and lower local dimensions at $x$ by replacing the limit with an upper or lower limit, respectively. 

It can be immediately observed that if $\nu$ is Ahlfors $Q\mhyphen$regular then $d_\nu(x)=Q(x)$ for all $x$.

\begin{definition}
A function $p$ on a metric space $(X,d)$ is \textit{log-H{\"o}lder continuous} if 
there exists a $C>0$ such that $|p(x)-p(y)|\leq \frac{-C}{\log(d(x,y))}$ for all $x,y$ with $0<d(x,y)<\frac{1}{2}.$
\end{definition}
The following is may be found in \cite{Sob} (Proposition 3.1).

\begin{lemma} For $X$ compact, if $Q:X\rightarrow (0,\infty)$ log-H{\"o}lder continuous, $\tilde{Q}:\mathscr{B}\rightarrow [0,\infty)$ with $Q^-\leq \tilde{Q} \leq Q^+,$ then $\lambda^{Q^+}\asymp \lambda^{Q^-} \asymp \lambda^{\tilde{Q}}.$
\end{lemma}
\begin{proof}
Let $U$ open with $0<|U|<\frac{1}{2}.$ 
Then for $x,y\in U,$ $|Q(x)-Q(y)|\leq \frac{-
C}{\log(|U|)}.$ Hence $0\leq 
\log(|U|)(Q^-(U)-Q^+(U))\leq C.$ Then $|U|^{Q^+(U)}\leq |U|^{Q^-(U)} \leq e^C|U|^{Q^+(U)}.$ The result follows.
\end{proof}
The following lemma may be found in \cite{Sob} (Lemma 2.1).
\begin{lemma} If $\nu$ is Ahlfors $Q\mhyphen$regular then $Q$ is log-H\"older continuous. \label{log} \end{lemma}
\begin{proof} By Ahlfors regularity, there exists a constant $D>1$ such that $\nu(B_r(x))\leq Dr^{Q(x)}$ and $r^{Q(x)}\leq D\nu(B_r(x))$ for all $0<r\leq\diam(X)/2,\;x\in X.$ Suppose $x,y\in X$ with $0<r:=d(x,y)<\frac{1}{2}.$ Say $Q(x)\geq Q(y).$ Since $Q$ is bounded, let $R<\infty$ be an upper bound for $Q,$ and let $C$ be defined by $e^C:=2^R D^2.$ Then since $B_r(y)\subset B_{2r}(x),$ $r^{Q(y)}\leq D\nu(B_r(y))\leq D\nu(B_{2r}(x))\leq D^2 2^R r^{Q(x)} = e^Cr^{Q(x)}.$ Hence $d(x,y)^{|Q(y)-Q(x)|}\geq e^{-C}.$ So $|Q(x)-Q(y)|\leq \frac{-C}{\log(d(x,y)}.$
\end{proof}

Hence, in particular, if $\nu$ is Ahlfors $Q\mhyphen$regular then $Q$ is continuous.

\begin{proposition} \label{equiv} If $\nu$ is a finite Ahlfors $Q\mhyphen$regular Borel measure on a separable metric space then $\nu\asymp \lambda^{Q_c}.$
\end{proposition}

\begin{proof} Since $Q$ is bounded, let $R>0$ be an upper bound for $Q$. Let $C>0$ be a constant such that $\nu(B_r(x))\leq Cr^{Q(x)}$ and $r^{Q(x)}\leq C\nu(B_r(x))$ for all $x\in X$ and $0<r\leq \diam(X).$ 
Then let $A\subset X$ be Borel measurable. For $\delta>0$ let $(B_i)_{i\in I} \in 
\mathscr{B}_\delta(A).$ Say $B_i=B_{r_i}(x_i).$ Let 
\[r_i^\prime=\sup\{d(x_i,y)\;|\;y\in B_i\}.\] Then $B_i\subset \{y\;|\;
d(x_i,y)\leq r_i^\prime\}=B_{r_i'}[x].$ Note, by Ahlfors regularity, $\nu(B_{r_i'+1}(x_i))<\infty$. So by continuity of measure \[\nu(B_{r_i'}[x_i])=\inf_{n\geq 1} \nu(B_{r_i'+\frac{1}{n}}(x_i))\leq C\inf_{n\geq 1} (r_i'+\frac{1}{n})^{Q(x_i)}=Cr_i'^{Q(x_i)}.\] It follows that \[\nu(A)\leq \sum_{i\in I}\nu(B_i)\leq 
C\sum_i r_i'^{Q(x_i)} \leq C\sum_i |B_i|^{Q(x_i)}.\] Hence \[\nu(A)\leq C\lambda_\delta^{Q_c}
(A)\leq C \lambda^{Q_c}(A).\]

Let $A$ be open. Let $\delta>0.$  Since $X$ is separable, let $(B_i)_{i\in I}\in \mathscr{B}_{\frac{\delta}
{10}}(A)$ be such that $\cup B_i = A$ and $I$ is at most countable. Then by the 
Vitali Covering Lemma, there exists a disjoint sub-collection $(B_j)_{j\in J}, J\subset I,$ of the 
$(B_i)_{i\in I}$ such that $A\subset \cup_{j\in J} 5B_j.$ Let $Q_j=Q(x_j),$ where $x_j$ is 
the center of $B_j.$ Let $r_j$ be the radius of $B_j.$ Then \[\lambda_\delta^{Q_c}(A)\leq 
\sum_j |5B_j|^{Q_j}\leq 10^R\sum_jr_j^{Q_j} \leq 10^R C\sum_j 
\nu(B_j)\leq 10^R C\nu(A).\] Since $\delta>0$ was arbitrary, \[\lambda^{Q_c}(A)\leq C 10^R \nu(A).\]
Let $B\subset X$ Borel measurable. Since $\nu$ is regular, for $\epsilon>0,$ let $A\subset X$ 
open with $B\subset A$ such that $\nu(B)\geq \nu(A)-\epsilon.$ Then \[\lambda^{Q_c}(B)\leq \lambda^{Q_c}(A)\leq 10^RC(\nu(B)+\epsilon).\] Since $\epsilon>0$ is arbitrary, \[\lambda^{Q_c}(B)\leq 10^RC\nu(B).\]
\end{proof}
\begin{theorem} \label{main} Let $X$ be a compact metric space. If $\nu$ is an Ahlfors $Q\mhyphen$regular Borel measure on $X$ then $Q=\alpha$ and $\nu \asymp H_{\loc}.$ 
\end{theorem}

\begin{proof} Since $X$ is compact, $\nu$ is finite. Hence by the previous proposition \ref{equiv}, $\nu\asymp \lambda^{Q_c}.$ Hence $X$ is $Q$ amenable, and $Q$ is continuous, as it is log-H\"older continuous. By \ref{amen}, $Q=\alpha.$ 

Let $B$ be a non-empty open ball. Since $B$ is open and the definition of $\alpha$ is local, we may apply \ref{dimmax} to $B$. Hence \[\dim(B)=\sup_{x\in B} \alpha(x) = Q^+(B).\] Therefore $\lambda^{Q^+}= \lambda_{\loc}.$ Finally, since $\lambda_{\loc}\asymp H_{\loc}$ by \ref{asym} and since $\asymp$ is transitive, we have  $\nu\asymp H_{\loc}.$
\end{proof}

Hence if a compact space admits an Ahlfors $Q\mhyphen$regular Borel measure then $\alpha=Q$ and that measure is strongly equivalent to the local measure.

\section{Constructions}

We now apply the preceding  mathematical developments to several examples. The first, a Koch curve with variable local dimension, may be found in \cite{Sob}. The remaining examples, a Sierpi\'nski gasket of variable local dimension, a Sierpi\'nski carpet of variable local dimension, and a Vicsek tree of variable local dimension, have not been considered, at least to our knowledge, in the literature before. 

\begin{subsection}{A variable dimensional Koch curve}
 Note that this first example is not new. It is a particular case of Koch curve constructed and analyzed in \cite{Sob}. Moreover, the idea for a Koch curve of variable dimension may be found in \cite{nottale}, although the notion of dimension used there is not precise. 

 Given $0<\theta_1<\theta_2<\frac{\pi}{2},$ we construct a compact metric space $K$ and a continuous bijective map $\phi:[0,1]\rightarrow K$ such that if $x\in K$ and $\phi(t)=x$ then \[\alpha(x)=\frac{2\log(2)}{\log(2+2\cos(\theta_1+t(\theta_2-\theta_1)))}.\] Then, for example (see figure 2), with $\theta_1=\frac{\pi}{36}=5^\circ,\theta_2=\frac{4\pi}{9}=80^\circ,$ the local dimension $\alpha$ variously continuously from approximately $1.001$ to approximately $1.625$. 
 
It also holds that the local Hausdorff measure $H_{\loc}$ is Ahlfors $\alpha\mhyphen$regular on $K$ and
$0<H_{\loc}(K)<\infty.$ Moreover, if $c\geq 0$ is any constant then $H^c(K)$ is either $0$ or $\infty.$ Hence the local Hausdorff measure is in a sense the ``correct" measure for $K$ \cite{Sob}.

The construction is a particular case of the construction given in \cite{Sob}. A reader interested in further details and proofs should consult \cite{Sob}. 

By a generator of parameters $L$ and $\theta$ we mean the following. Label the segments $1$ through $4.$ Put segment $1$ along the positive $x\mhyphen$axis with one end at the origin. Then connect segment $2$ at an angle of $\theta$ with the positive $x\mhyphen$ axis to the endpoint of segment $1$ at $(L,0)$. Then the other endpoint of segment $2$ will lie at $(L+L\cos(\theta),L\sin(\theta)).$ Then connect one endpoint of segment $3$ to the point $(L+L\cos(\theta),L\sin(\theta))$ in such a way that its second endpoint lies at $(L+2L\cos(\theta),0)$ on the positive $x\mhyphen$axis.  Then place segment $4$ on the positive $x\mhyphen$axis with its endpoints at $(L+2L\cos(\theta),0)$ and $(2L+2L\cos(\theta),0)$. We use coordinates only to specify the construction. The element may be translated and rotated freely. Note that if the overall length $2L+2L\cos(\theta)$ is set to $1, $ then $L$ and $\theta$ are related by 
$L=\frac{1}{2+2\cos(\theta)}.$ 
\vspace*{1 in}
\begin{figure}[H]
\centering
\includegraphics[scale=.25]{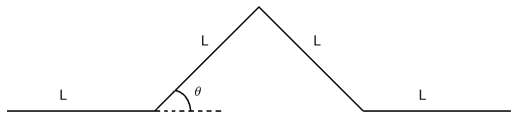}
\caption{A generator of parameters $L$ and $\theta.$}
\end{figure}
\vspace*{1 in}

We then construct the curve recursively as follows. Let $M>0$ and $\theta_1,\theta_2$ lower and upper angles with $0<\theta_1< \theta_2<\frac{\pi}{2}.$ For stage $0$ let $K_0$ be the segment connecting $(0,0)$ and $(M,0).$ At stage $1$, form $K_1$ by replacing $K_0$ with the generator of parameters $L$ and $\frac{\theta_1+\theta_2}{2}$ where $L=\frac{M}{2+2\cos((\theta_1+\theta_2)/{2})}.$ Then suppose we have constructed $K_0,K_1,...,K_n$ for any $M>0$ and $0<\theta_1<\theta_2<\frac{\pi}{2}$ where $K_n$ is made of 4 versions of $K_{n-1}$ each of overall length $L$, where the $i\mhyphen$th, for $i=1,2,3,4$ labeled from left to right, has lower and upper angles $\theta_1+(i-1)\frac{(\theta_2-\theta_1)}{4}$ and $\theta_1+i\frac{(\theta_2-\theta_1)}{4}$, respectively. Then, for $M>0,0<\theta_1<\theta_2<\frac{\pi}{2},$ form $K_{n+1}$ by replacing the $i\mhyphen$th version of $K_{n-1}$ in $K_n$ with a $K_n$ of overall length $L$ with lower and upper angles $\theta_1+(i-1)\frac{(\theta_2-\theta_1)}{4}$ and $\theta_1+i\frac{(\theta_2-\theta_1)}{4}$, respectively. Hence by induction we may construct $K_n$ for any $n$ and any overall length $M>0$ and angles of interpolation $0<\theta_1<\theta_2<\frac{\pi}{2}.$ 

We will now fix $M=1$ and the endpoints of $K_n$ at $(0,0)$ and $(1,0).$ 
Then note $K_n$ is made up of $4^n$ segments and $\phi_n$ is the map that is the bijective, piecewise continuous map that is constant speed from $\frac{i-1}{4^k}$ to $\frac{i+1}{4^k}$ connecting the endpoints of the $i\mhyphen$th segment making up stage $n,$ with $\phi_n(0)=(0,0)$ and $\phi_n(1)=(1,0).$ Then, as the sequence of $\phi_n$ is Cauchy in $C([0,1],\mathbb{R}^2),$ it converges to a continuous $\phi.$ We let $K=\phi[0,1].$ Moreover, the sequence of $K_n$ is Cauchy in the Hausdorff metric on compact subsets of $[0,1]\times [0,1],$ which is known to be complete. Hence we may also define $K$ as the limit of the $K_n$ in the sense of Hausdorff convergence. 
\vspace*{1 in}

\begin{figure}[H] \label{fig2}
\centering
\includegraphics[scale=.25]{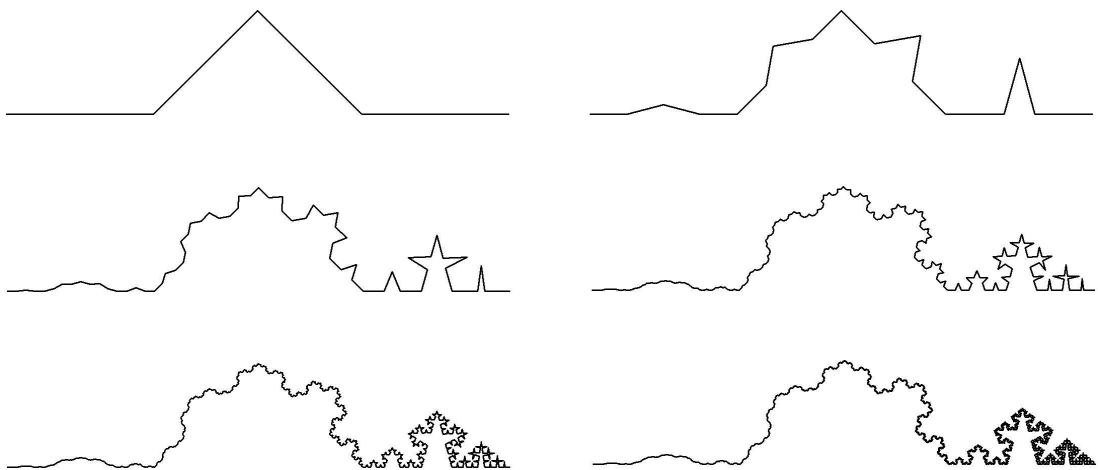}
\caption{Stages 1 through 6 of a variable dimensional Koch curve where $\theta_1=5^\circ$ and $\theta_2=80^\circ.$}
\end{figure}
\newpage
We state the following proposition without proof since it follows from Theorem 3.4 of \cite{Sob}, as the function $s:[0,1]\rightarrow (1/4,1/2)$ defined by $s(t)=\frac{1}{2+2\cos(\theta_1+t(\theta_2-\theta_1))}$ is Lipschitz.

\begin{proposition} Let $Q(x)=\frac{2\log(2)}{\log(2+2\cos(\theta_1+\phi^{-1}(x)(\theta_2-\theta_1)))}$. Then $\lambda^{Q_c}$ is an Ahlfors $Q\mhyphen$regular measure for $K.$ 
\end{proposition}

\begin{corollary} $Q$ is the local dimension and $H_{\loc}$ is Ahlfors $\alpha\mhyphen$regular. In particular, $0<H_{\loc}(K)<\infty.$ Moreover, if $c\geq 0$ then $H^c(K)$ is $0$ or $\infty.$

\end{corollary}
\begin{proof}
The first claim follows immediately from Theorem \ref{main}. The last follows from Proposition \ref{3.6}.
\end{proof}
\end{subsection}
\subsection{A variable dimensional Sierpi\'nski gasket}

Given a length $L$ and a scaling parameter $r\in [0,1/2],$ one may construct a generator of parameters $L$ and $r$ by exercising from a filled equilateral triangle of side length $L$ everything but the the three corner equilateral triangles with side length $rL$,  as may be seen in Figure 3 below.

\vspace*{1 in}
\begin{figure}[H]
\centering
\includegraphics[scale=.25]{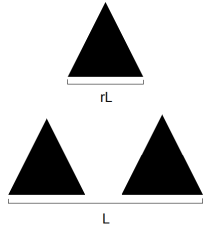}
\caption{A generator of parameters $L$ and $r.$}
\end{figure}
\vspace*{1 in}

We then construct the gasket recursively as follows. Let $L>0$ be the length parameter and $r_1,r_2$ lower and upper ratios with $0\leq r_1< r_2\leq 1/2.$ For stage $0$ let $K_0$ be the filled in equilateral triangle in $\mathbb{R}^2$ of side length $L$ with vertices at $(0,0), (L,0),$ and $(L/2,\sqrt{3}L/2)$. At stage $1$, form $K_1$ by replacing $K_0$ with the generator of parameters $L$ and $\frac{r_1+r_2}{2}$. Then suppose we have constructed $K_0,K_1,...,K_n$ for any $L>0$ and $0\leq r_1< r_2\leq\frac{1}{2}$ where $K_n$ is made of 3 versions of $K_{n-1}$ each of overall length $L(\frac{r_1+r_2}{2})$. Label these versions by a parameter $i$ for $i=1,2,3$ where version $1$ occupies the bottom left triangle and version $2$ is on the bottom to the right and version $3$ is on the top. Then, for $L>0,0\leq r_1< r_2\leq \frac{1}{2},$ form $K_{n+1}$ by replacing the $i\mhyphen$th version of $K_{n-1}$ in $K_n$ with a $K_n$ of length $\frac{r_1+r_2}{2}L$ and lower and upper ratios $r_1+(i-1)\frac{(r_2-r_1)}{3}$ and $r_1+i\frac{(r_2-r_1)}{3}$. Hence by induction we may construct $K_n$ for any $n$ and any overall length $L>0$ and ratios $0\leq r_1< r_2\leq \frac{1}{2}.$ 
Note that since $K_{n+1}\subset K_n$ for each $n$ and each $K_n$ is a closed subset of a compact set in $\mathbb{R}^2$, we may define the compact gasket $K$ by $K=\cap_{n=0}^\infty K_n.$

The following figure (Figure 4) was constructed with $r_1=.4$ and $r_2=.5.$ 
\vspace*{1 in}
\begin{figure}[H]
\centering
\includegraphics[scale=.55]{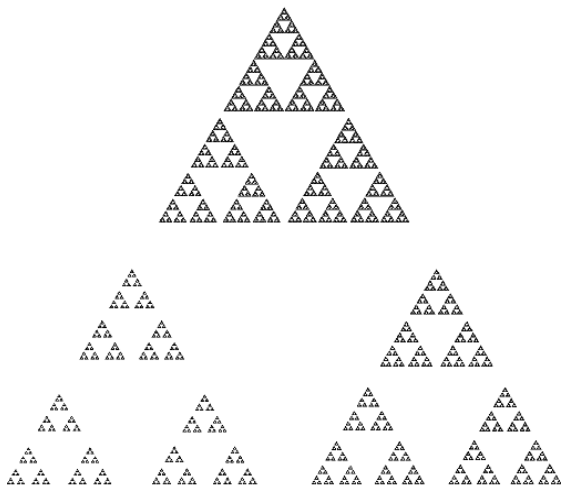}
\caption{A variable dimensional gasket with $r_1=.4, r_2=.5.$}
\end{figure}
\newpage
\subsection{A variable dimensional Sierpi\'nski carpet}
Given a base length $b$, a height $h$, and a scaling parameter $r\in [0,1],$ one may construct a generator of parameters $b,h,$ and $r$ by exercising from a filled rectangle of base $b$ and height $h$ the center open rectangle of base $br$ and height $hr.$ See the figure below.
\vspace*{1 in}
\begin{figure}[H]
\centering
\includegraphics[scale=.3]{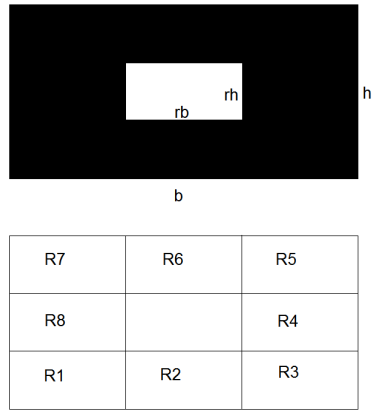}
\caption{(top) A generator of parameters $b,h,$ and $r$ and (bottom) its decomposition into sub-rectangles.}
\end{figure}
\vspace*{1 in}
As illustrated in the figure, we may decompose the generator into $8$ sub-rectangles $R1,R2,...,R_8.$ Rectangles $R1,R3,R5,$ and $R7$ have base $\frac{b-rb}{2}$ and height $\frac{h-rh}{2}.$ Rectangles $R2$ and $R6$ have base $rb$ and height $\frac{h-rh}{2}.$ Rectangle $R4$ and $R8$ have base $\frac{b-rb}{2}$ and height $rh.$

We then construct the carpet recursively as follows. Let $b, h>0$ be base and height parameters and $r_1,r_2$ lower and upper ratios with $0\leq r_1< r_2\leq 1.$ For stage $0$ let $K_0$ be the filled in rectangle in $\mathbb{R}^2$ with vertices at $(0,0), (b,0), (b,h),$ and $(0,h)$. At stage $1$, form $K_1$ by replacing $K_0$ with the generator of parameters $b, h,$ and $\frac{r_1+r_2}{2}$. Then suppose we have constructed $K_0,K_1,...,K_n$ for any $L>0$ and $0\leq r_1< r_2\leq 1$ where $K_n$ is made of 8 versions of $K_{n-1}$ where version $i$ sits in the spot for sub-rectangle $Ri$ of the generator for $i=1,2...,8.$ Then, for $L>0,0\leq r_1< r_2\leq 1,$ form $K_{n+1}$ by replacing the $i\mhyphen$th version of $K_{n-1}$ in $K_n$ with a $K_n$ of base and height equal to the base and height of sub-rectangle $Ri$ of $K_1$ and lower and upper ratios $r_1+(i-1)\frac{(r_2-r_1)}{8}$ and $r_1+i\frac{(r_2-r_1)}{8}$, respectively, for $i=1,...,8$. Hence by induction we may construct $K_n$ for any $n$ and any base $b$, height $h,$ and ratios $0\leq r_1< r_2\leq 1.$ 
Note that since $K_{n+1}\subset K_n$ for each $n$ and each $K_n$ is a closed subset of a compact set in $\mathbb{R}^2$, we may define the compact carpet $K$ by $K=\cap_{n=0}^\infty K_n.$

The following figure (Figure 6) was constructed with $b=h$ and $r_1=1/6,r_2=1/2.$
\vspace*{1 in}
\begin{figure}[H]
\centering
\includegraphics[scale=.75]{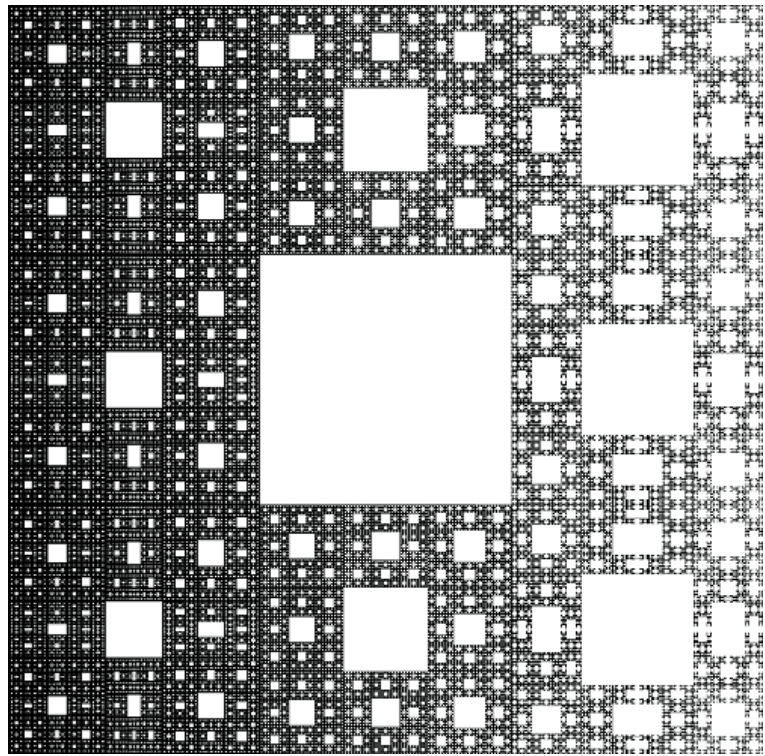}
\caption{A variable dimensional carpet constructed from $b=h$ and $r_1=1/6, r_2=1/2$.}
\end{figure}
\newpage
\subsection{A variable dimensional Vicsek tree}
Given a length $L$ and a scaling ratio $r\in [0, 1],$ one may construct a generator with parameters $L$ and $r$ as follows. Given a filled in square of side length $L$, drawing a square in the middle of side length $Lr$ induces a decomposition of the initial square into $9$ sub-rectangles. Label the rectangles from bottom left $R1,R2,R3$ on the bottom row, $R4,R5,R6$ on the middle row, and $R7,R8,R9$ on the top row. Then remove the interiors of $R2,R4, R6,$ and $R8.$ We emphasize that the ambient space is homeomorphic to $[0,L]\times [0, L]$, and the interior operation is taken with respect to the relative topology of this space. Hence the interior of the outer rectangles that are removed includes part of the outer ``boundary" of the original square; so the result is a union of 5 disjoint closed squares. See the figure below.\footnote{See also \cite{telcs2006} for the example of a weighted infinite Vicsek tree graph. } 
\vspace*{1 in}
\begin{figure}[H]
\centering
\includegraphics[scale=.22]{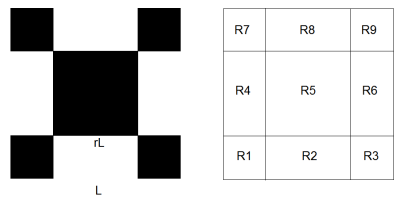}
\caption{(left) A generator of parameters $L$ and $r$ and (right) labeling.}
\end{figure}
\vspace*{1 in}
We then may perform the construction of the tree as follows. Given a length $L$ and ratios $r_1, r_2$ with $0\leq r_1< r_2\leq 1,$ let $K_0$ be the square $[0,L]\times [0,L]$. Then let $K_1$ be the result of replacing $K_0$ with the generator with parameters $L$ and $\frac{1}{2}(r_1+r_2).$  Having constructed $K_0, K_1,,,.K_n$ for some $n\geq 1,$ where $K_n$ is made up of $5$ copies of $K_{n-1}$ in rectangles $R1, R3, R5, R7,$ and $R9,$ 
we construct $K_{n+1}$ as follows. Replace $R1$ and $R7$ with copies of $K_n$ of length $\frac{1}{2}(L-rL)$ and ratios $r_1$ and $r_1+\frac{1}{3}(r_2-r_1).$ Replace $R5$ with a $K_n$ of length $rL$ and ratios $r_1+\frac{1}{3}(r_2-r_1)$ and $r_1+\frac{2}{3}(r_2-r_1).$ Replace $R3$ and $R9$ with copies of $K_n$ of length $\frac{1}{2}(L-rL)$ and ratios $r_1+\frac{2}{3}(r_2-r_1)$ and $r_2.$
Then let $K=\cap_{n=0}^\infty K_n$. 
\vspace*{1 in}
\begin{figure}[H]
\centering
\includegraphics[scale=.5]{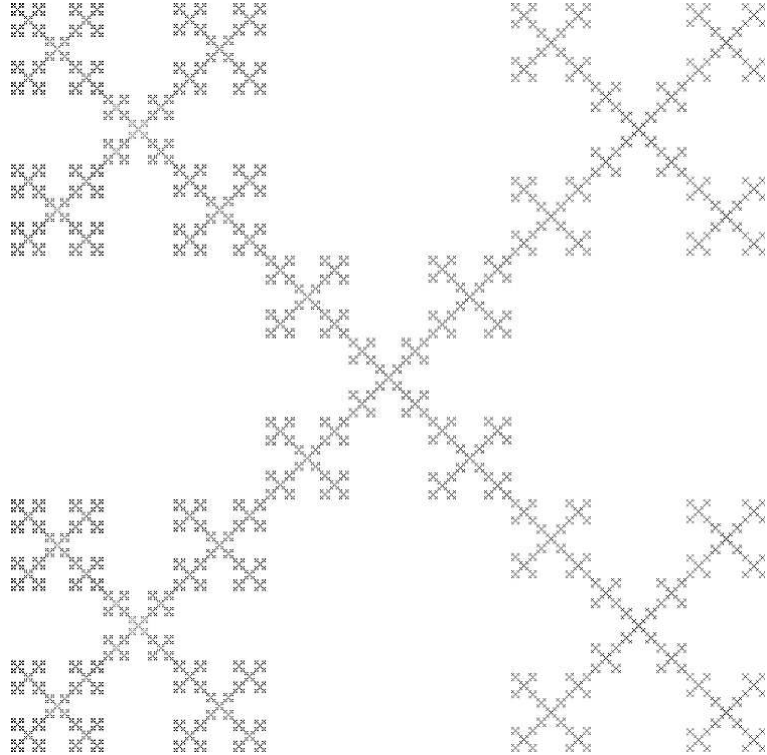}
\caption{Stage $6$ of a variable dimensional Vicsek tree constructed with $r_1=.25, r_2=.4$.}
\end{figure}

\chapter{The Local Walk Dimension}
In this chapter we provide a definition of a localized time scaling exponent $\beta$ on a compact metric space. The definition depends on the particular choice of approximate random walks, as defined in section 1.3. However, we strongly emphasize that the essential idea of defining $\beta$ as a critical exponent using mean exit times may be applied in much greater generality. Once a manner of approximate random walks is decided upon, a local walk dimension $\beta$ may be defined in a manner appropriate to that setting. Throughout this chapter, we assume $X$ is a compact metric space. 

\section{Discrete Approach}

\subsection{Uniform and Partially Symmetrized Walks on $\epsilon$-Nets}

Suppose $G$ is a graph. For $X\in V(G),$ let $V(G)_x=\{y\in V(G)\;|\;y\sim x\}$ denote the neighbors of $x$ in $G.$ A vertex $x$ is called isolated if it has no neighbors, that is $V(G)_x=\emptyset.$ Then define the \textit{uniform random walk} on $V(G)$ (see, for instance, \cite{art}) by the transition probabilities $p(G)_{x,y}$, for $x,y\in V(G),$ to jump from $x$ to $y$ in one time step, given by \begin{displaymath}
   P(G)_{x,y} = \left\{
     \begin{array}{lr}
       \frac{1}{\deg(x)}\chi_{\{z\in V(G)\;|\;z\sim x\}}(y) \;,& x\;\mbox{not isolated,} \\
       \delta_{x,y} \;,&  x\;\mbox{isolated.} 
     \end{array}
   \right.
\end{displaymath} 

Alternatively, given a graph $G$ with, we may define the \textit{partially symmetrized random walk} on $G$ as follows. First we assume $G$ has no isolated vertices. If it does and $x$ is such a vertex, we add a self edge $(x,x)$ at $x$. Then for $x\in V(G),$ let \[a(x):=\sum_{y\sim x}\frac{1}{\deg(x)}\left(1+\frac{\deg(x)}{\deg(y)}\right).\] Then we define the partially symmetrized walk by 
\[P(G)_{x,y}=a(x)^{-1}\left(\frac{1}{\deg(x)}\left(1+\frac{\deg(x)}{\deg(y)}\right)\right)\chi_{\{z\in V(G)\;|\;z\sim x\}}(y),\] for $x,y\in V(G)\footnote{Clearly if $G$ has constant degree then the uniform and partially symmetrized walks are equivalent.}\footnote{We call the walk ``partially symmetrized" since it is not symmetric with respect to the counting measure. However, it is is symmetric with respect to the measure $x\mapsto a_r(x),$ which is strongly equivalent to the counting measure given a uniform bound $\frac{1}{C}\leq \frac{\deg(x)}{\deg(y)}\leq C$ for some $C>0$ and all $x,y\in V(G).$}$

Either approach defines a discrete time Markov process $(\bm{x}(G)_{k})_{k\in \mathbb{Z}^+}$ with finite state space $V(G)$. Given $x\in V(G),$ let $\mathbb{P}^x$ be the probability defined on all paths starting at $x$, and let $\mathbb{E}^x$ be the expectation with respect to $\mathbb{P}^x.$ See section 2.3 for details.

Recall from that given an $\epsilon$-net $\mathcal{N}$ in a compact metric space $X,$ we may define an approximating graph with vertex set $\mathcal{N}$ in a number of ways, for example the proximity or covering graphs discussed in section 2.1. For $\mathcal{N}$ an $\epsilon$-net, let $G_\epsilon(\mathcal{N})$ be the $\epsilon$-covering graph on the vertex set $\mathcal{N}$, defined by the edge relation  $(x,y)\in E(G_\epsilon(\mathcal{N}))$ if and only if $x,y\in \mathcal{N}$ and $B_{\epsilon}(x)\cap B_{\epsilon}(y)\neq \emptyset.$ Note that $\deg(x)\geq 1$ for all $x\in \mathcal{N}$ since each $(x,x)$ is an edge.\footnote{We could also use the ``loopless covering graph``, but we add a self edge only at isolated vertices.}

Now suppose $\mathcal{N}$ is an $\epsilon$-net and $\bm{x}(G_\epsilon(\mathcal{N})_k)_{k=1}^\infty$ is the associated uniform or symmetrized random walk on $\mathcal{N}.$ Suppose we have chosen to use either the uniform or partially symmetrized random walk for all such nets $\mathcal{N}.$ Given a set $A\subset X,$ we let \[\tau(\mathcal{N},\epsilon)_A:=\inf\{k\in\mathbb{Z}^+\;|\;\bm{x}(G_\epsilon(\mathcal{N}))_k\notin \mathcal{N}\cap A\},\] where we recall the convention that $\inf \emptyset = \infty.$

Then we define, as in section 2.3, the \textit{mean exit time} (see also \cite{art}) from $A$ starting at $x\in \mathcal{N}$, denoted $E_{\mathcal{N}, \epsilon}(x, A),$ as follows\[E_{\mathcal{N}, \epsilon}(x, A):=\mathbb{E}^x\tau(\mathcal{N},\epsilon)_A.\] We let \[E^+_{\mathcal{N}, \epsilon}(A):=\max_{x\in \mathcal{N}}E_{\mathcal{N}, \epsilon}(x, A)=\max_{x\in \mathcal{N}\cap A}E_{\mathcal{N}, \epsilon}(x, A).\]

Given $A\subset X$ and $\gamma\geq 0,$ let 
\[\omega_{\gamma}(A):=\inf_{\delta>0}\sup\{E^+_{\mathcal{N}, \epsilon}(A)\epsilon^\gamma\;|\;\mathcal{N}\;\mbox{is an\;}\epsilon\mhyphen\mbox{net with\;}0<\epsilon<\delta\}.\]
\begin{proposition} For any $A\subset X,$ \[{\sup}_{\overline{\mathbb{R}}_+}\{\gamma\geq 0\;|\;\omega_{\gamma}(A)=\infty\}={\inf}_{\overline{\mathbb{R}}_+}\{\gamma\geq 0\;|\;\omega_{\gamma}(A)=0\}.\] We call the common value $\beta(A).$ Then if $A'\subset A$ we have $\beta(A')\leq \beta(A).$\end{proposition}

\begin{proof}
Suppose $0\leq \gamma'<\gamma.$ Then \begin{equation*}
\begin{split}
&\sup\{E^+_{\mathcal{N}, \epsilon}(A)\epsilon^\gamma\;|\;\mathcal{N}\;\mbox{an\;}\epsilon\mhyphen\mbox{net with\;}0<\epsilon<\delta\}\\&\leq \delta^{\gamma-\gamma'}\sup\{E^+_{\mathcal{N}, \epsilon}(A)\epsilon^{\gamma'}\;|\;\mathcal{N}\;\mbox{an\;}\epsilon\mhyphen\mbox{net with\;}0<\epsilon<\delta\},
\end{split} \end{equation*} for any $\delta>0.$ Hence if $\omega_{\gamma'}(A)<\infty,$ then $\omega_{\gamma}(A)=0;$ and if $\omega_{\gamma}(A)>0,$ then $\omega_{\gamma'}(A)=\infty.$ Let $\beta(A):=\sup\{\gamma\geq 0\;|\;\omega_{\gamma}(A)=\infty\}.$ If $0\leq\gamma\leq \beta(A)\leq \infty,$ then $\omega_{\gamma}(A)=\infty$. Hence $\inf\{\gamma\geq 0\;|\;\omega_{\gamma}(A)=0\}\geq \beta(A).$ If $\inf\{\gamma\geq 0\;|\;\omega_{\gamma}(A)=0\}> \beta(A),$ then there exists a $\gamma>\beta(A)$ with $\omega_{\gamma}(A)>0.$ But then if $\beta(A)<\gamma'<\gamma,$ then $\omega_{\gamma'}(A)=\infty$, a contradiction. It follows that $\beta(A)=\inf\{\gamma\geq 0\;|\;\omega_{\gamma}(A)=0\}$.

Now suppose $A'\subset A.$ Let $\mathcal{N}$ be an $\epsilon$-net. Clearly for any infinite $G_\epsilon(\mathcal{N})$ path $\omega$, $\tau(\mathcal{N},\epsilon)_{A'}(\omega)\leq \tau(\mathcal{N},\epsilon)_{A}(\omega).$ Hence $E_{\mathcal{N},\epsilon}(x, A')\leq E_{\mathcal{N},\epsilon}(x,A)$ for all $x\in \mathcal{N}.$ It follows, for any $\gamma\geq 0,$ that $\omega_{\gamma}(A')\leq \omega_{\gamma}(A).$ Hence \[\{\gamma\geq 0\;|\;\omega_{\gamma}(A')= \infty\}\subset\{\gamma\geq 0\;|\;\omega_{\gamma}(A)=\infty\}.\] Therefore $\beta(A')\leq \beta(A),$ as desired.
\end{proof}

We now define the local exponent $\beta:X\rightarrow [0,\infty]$ as follows.
\begin{definition} For $x\in X$ let \[\beta(x):=\lim_{r\rightarrow 0^+}\beta(B_r(x))=\inf_{r>0}\beta(B_r(x)).\]
\end{definition}

We now give a number of remarks. First, note that for each $A\subset X$ we have defined $\beta(A)$ by means of a ``$\limsup$" along $\epsilon$-nets. It may not be surprising that we also could have chosen to use a ``$\liminf$". That is, for $\gamma\geq 0$ we may define \[\omega^{-}_{\gamma}(A):=\sup_{\delta>0}\inf\{E^+_{\mathcal{N},A}\epsilon^\gamma\;|\;\mathcal{N}\;\mbox{is an\;}\epsilon\mhyphen\mbox{net with\;}0<\epsilon<\delta\}.\] Then we may define 
\[\beta^{-}(A):=\sup\{\gamma\geq 0\;|\;\omega^{-}_{\gamma}(A)=\infty\}=\inf\{\gamma\geq 0\;|\;\omega^{-}_{\gamma}(A)=0\}.\]
Then as before one may show that this is well defined and that if $A'\subset A$ then $\beta^{-}(A')\leq \beta^{-}(A).$ Hence we have that \[\beta^{-}(x):=\lim_{r\rightarrow 0^+}\beta^{-}(B_r(x))\] exists. Then since $\omega^{-}_\gamma(A)\leq\omega_{\gamma}(A),$ we have that $\beta^{-}(A)\leq \beta(A).$ Hence \[\beta^{-}\leq \beta.\] 

We will focus, however, on $\beta.$ One justification for this choice is as follows. Intuitively, $\beta$ may be seen as a localized ``walk packing dimension." That is, if one has a continuous curve $\gamma:[0,T]\rightarrow X$ and $\gamma^*$ is its image in $X$, the packing pre-measure of $\gamma^*$ at a given dimension $d$ is defined as a $\limsup$ as $\delta\rightarrow 0^+$ of terms of the form $\sum|B_i|^d$ where the $B_i$ are a collection of disjoint balls with centers in $\gamma^*$ and diameters less than $\delta.$ Given an $\epsilon<\delta$ and an $\epsilon$-net on $\gamma^*$ the balls of radius $\epsilon/2$ with centers elements of the net are then an allowed packing. Since the diameter of each ball of radius $\epsilon/2$ is bounded above by $\epsilon$ and below by $\epsilon/2$ one might consider, as an approximation, terms of the form $\sharp \mathcal{N}\epsilon^d$ where $\sharp \mathcal{N}$ is the number of elements of $\mathcal{N}$. If $B$ is a ball with center $x_0$ and $\gamma:[0,\infty)\rightarrow X$ is a continuous ``sample path" with $\gamma(0)=x_0,$ set $T=\inf\{t\geq 0\;|\;\gamma(t)\notin B\}$. Then $\gamma|_{[0,T]}^*$ is the image of the path in $B$. If $\mathcal{N}$ is an $\epsilon$-net on $\gamma|_{[0,T]}^*$ containing $x_0$, and if one considers $\mathcal{N}$ as an approximation to the continuous sample 
path, then the the $(\sharp \mathcal{N}-1)\mhyphen$th step along the approximation is the last step before exiting $B.$ Therefore, one may intuitively regard $\beta$ as a kind of local walk packing dimension. In that interpretation, the use of $\limsup$ is entirely natural.

In practice, such as is the case with many of the standard examples of self-similar fractals such as the Sierpi\'nski gasket or carpet, arising from the construction one may be given a sequence $\epsilon_k$ of positive numbers decreasing to $0$ and for each $k$ an $\epsilon_k$-net $\mathcal{N}_k$ in $X$. Then one might wish to consider a weaker definition of a $\beta$ defined only as a $\limsup$ of weighted mean exit times along the chosen sequence of nets $(\mathcal{N}_k)_{k=1}^\infty.$

\begin{example} Consider a variable dimensional Koch curve $K$ considered in section 3.3. Let $\alpha:X\rightarrow [0,\infty)$ be the local dimension. Let $\mu$ be the local Hausdorff measure as defined in section 3.1. Then $\alpha$ is continuous and $\mu$ is Ahlfors $\alpha$-regular. Let $C\geq 1$ such that for any $x\in K$ and any $0<r<1/2$, 
$\frac{1}{C}r^{\alpha(x)}\leq \mu(B_r(x))\leq Cr^{\alpha(x)}.$ Let $B=B_r(x)$ with $0<r<1/2.$ Then let $\mathcal{N}$ be an $\epsilon$-net in $K$ with $0<\epsilon<r.$ We use $\mu$ to estimate the number of points in $\mathcal{N}\cap B.$ Let $\alpha^{-}(B)=\inf_{y\in B}\alpha(y)$ and $\alpha^+(B)=\sup_{y\in B}\alpha(y).$ Note that 
\[\frac{1}{C}\sharp(\mathcal{N}\cap B)(\frac{\epsilon}{2})^{\alpha^+(B)}\leq \sum_{y\in \mathcal{N}\cap B}\mu(B_{\epsilon/2}(y))\leq \mu(B)\leq \sum_{y\in \mathcal{N}\cap B}\mu(B_{\epsilon}(y))\leq C\sharp(\mathcal{N}\cap B)\epsilon^{\alpha^{-}(B)}.\] 
Hence, using the Ahlfors regularity, 
\[\frac{r^{\alpha(x)}}{C^2\epsilon^{\alpha^{-}(B)}}\leq \sharp(\mathcal{N}\cap B) \leq \frac{2^{\alpha^+(B)}C^2r^{\alpha(x)}}{\epsilon^{\alpha^+(B)}}.\]

Let $n>2$ an integer and $G$ be the path graph with $n$ vertices and self-loops. That is $V=\{0,1,...n\}$ and $(j,k)\in E$ if and only if $|j-k|\leq 1.$ Then with the standard random walk on $G,$ let $E_n(x)$ be the expected number of steps needed for a walker starting at $x\in V$ to reach $0$ or $n.$ Then $E_n(k)=\frac{3}{2}n(n-k).$ This follows since for $0<k<n$ with $k\in V,$ $\frac{2}{3}E_n(k)=1+\frac{1}{3}E_n(k-1)+\frac{1}{3}E_n(k+1)$, and $E_n(0)=E_n(n)=0.$ Hence \[\frac{3n(n-1)}{4}\leq \max_{k\in V}E_n(k)\leq \frac{3n(n+1)}{4}.\]

Suppose given an $\epsilon$-net $\mathcal{N}$ in $K$ we define the graph $G(\mathcal{N})$ as a covering graph defined by the edge relation $x\sim y$ if $B_{\epsilon}(x)\cap B_{\epsilon}(y)\neq \emptyset.$ Then the portion of the graph $G(\mathcal{N})$ in $B$ is a path graph. Hence there exists a constant $D>1$, independent of $r$, such that 
\[\frac{r^{2\alpha(x)}}{D\epsilon^{2\alpha^{-}(B)}}\leq E^+_{\mathcal{N}, B} \leq \frac{Dr^{2\alpha(x)}}{\epsilon^{2\alpha^+(B)}}.\]
It follows that $2\alpha^{-}(B_r(x))\leq \beta(B_r(x))\leq 2\alpha^+(B_r(x)).$ Since $\alpha$ is continuous, letting $r\rightarrow 0^+$ yields \[\beta(x)=2\alpha(x).\] Hence $\beta$ is variable as well. 
This example also illustrates a substantial difference with the definition of $\beta$ often defined for infinite graphs. Barlow has shown that such an exponent is bounded below by $2$ and above by $\alpha+1$ \cite{barlowesc}. However, in the above example, $\beta=2\alpha>\alpha+1$. Intuitively, such a discrepancy comes from the fact that the graph distance does not well approximate the metric on $K.$

\end{example}

\subsection{Discrete Walks on Tilings in a Metric Measure Space}
If we are given the additional structure of a measure, we may define ``natural" approximate discrete walks with more general edge weights by using the measure. Recall from section 2.1, that by a metric measure space we mean a triple $(X,d,\mu)$ consisting of a metric space $(X,d)$ together with a non-negative Borel measure $\mu$ on $X$ of full support. For example, if $(X,d)$ is variable Ahlfors regular, $\mu$ could be the local Hausdorff measure. 

Recall an $\epsilon$-tiling is a Borel partition $T=(T_x)_{x\in \mathcal{N}(T)}$ of $X$, where $\mathcal{N}(T)$ is an $\epsilon$-net, and $B_{\epsilon/2}(x)\subset T_x\subset B_\epsilon(x)$ for $x\in \mathcal{N}(T).$ Suppose $(X,d,\mu)$ is a metric measure space. Then if $T$ is an $\epsilon$-tiling, $\mu(T_x)>0$ for all $x\in \mathcal{N}(T)$ since $T_x$ contains a ball and $\mu$ has full support. Note the assignment $x\mapsto \mu_{T}(x):=\mu(T_x)$ is a measure on $\mathcal{N}(T).$   Then given the covering graph $G_{\epsilon}(\mathcal{N}(T))$, let $V_x:=\sum_{y\sim x} \mu_{T}(y)$ for $x\in \mathcal{N}(T).$ Define a $\mu$-uniform walk $P(G)_{x,y},$ for $x,y\in \mathcal{N}(T)$ by
\begin{displaymath}
   P(G)_{x,y} = \left\{
     \begin{array}{lr}
       \frac{1}{V_x}\chi_{\{z\in \mathcal{N}(T)\;|\;z\sim x\}}(y) \;,& x\;\mbox{not isolated,} \\
       \delta_{x,y} \;,&  x\;\mbox{isolated.} 
     \end{array}
   \right.
\end{displaymath} 

Alternatively, given a covering graph $G_{\epsilon}(\mathcal{N}(T))$, we may define the a partially symmetrized  walk as follows. First, if $G_{\epsilon}(\mathcal{N}(T))$ has an isolated vertex $x,$ we add a self edge $(x,x)$ at $x$. Then for $x\in V(G),$ let \[a(x):=\sum_{y\sim x}\frac{1}{V_x}\left(1+\frac{V_x}{V_y}\right).\] Then we define the partially symmetrized walk by 
\[P(G)_{x,y}=a(x)^{-1}\left(\frac{1}{V_x}\left(1+\frac{V_x}{V_y}\right)\right)\chi_{\{z\in \mathcal{N}(T)\;|\;z\sim x\}}(y),\] for $x,y\in \mathcal{N}(T).$

Then $P(G)$ is a transition operator on $L^2(\mathcal{N},\mu_{T})$ since $\sum_{y\in \mathcal{N}(T)}P(G)_{x,y}\mu_{T}(y)=1$ for all $x\in \mathcal{N}.$ Let $\bm{x}(G_\epsilon(T)_k)_{k=1}^\infty$ be the associated random walk on $\mathcal{N}(T).$ 

Suppose we have chosen to use either the $\mu$-uniform or partially symmetrized walks for all such tilings $T.$ Given a set $A\subset X,$ we let \[\tau(T,\epsilon)_A:=\inf\{k\in\mathbb{Z}^+\;|\;\bm{x}(G_\epsilon(\mathcal{N}(T)))_k\notin \mathcal{N}(T)\cap A\}.\] As before, if $A\subset X,$ we let $E_{T, \epsilon}(x, A):=\mathbb{E}^x\tau(T,\epsilon)_A.$ and $E^+_{T, \epsilon}(A):=\max_{x\in \mathcal{N}(T)}E_{T, \epsilon}(x, A).$

Analogously as before, if $A\subset X$ and $\gamma\geq 0,$ let 
\[\omega_{\gamma}(A):=\inf_{\delta>0}\sup\{E^+_{T, \epsilon}(A)\epsilon^\gamma\;|\;T\;\mbox{is an\;}\epsilon\mhyphen\mbox{tiling with\;}0<\epsilon<\delta\}.\]
Then, as before, \[{\sup}_{\overline{\mathbb{R}}_+}\{\gamma\geq 0\;|\;\omega_{\gamma}(A)=\infty\}={\inf}_{\overline{\mathbb{R}}_+}\{\gamma\geq 0\;|\;\omega_{\gamma}(A)=0\}.\] We call the common value $\beta(A).$ Then if $A'\subset A$ we have $\beta(A')\leq \beta(A).$

Hence, as before, we may define a local exponent $\beta$ by $\beta(x):=\lim_{r\rightarrow 0^+}\beta(B_r(x))$ for $x\in X.$

One of the benefits of using tilings in an approach to discrete approximate walks is that it permits use of the \textit{averaging method}, which allows approximate probabilistic Laplacians to be considered on all of $L^2(X,\mu)$ instead of finite dimensional Hilbert spaces. If $T$ is an $\epsilon$-tiling, define the averaging operator $\av_{T}:L^2(X,\mu)\rightarrow L^2(\mathcal{N}(T),\mu_T)$ by $\av_T(f)(z)=\hat{f}(z):=\frac{1}{\mu(T_z)}\int_{T_z}f(x)d\mu(x)$ for $f\in L^2(X,\mu),\;z\in\mathcal{N}(T).$  Thus if $A$ is an operator on $L^2(\mathcal{N}(T),\mu_T)$, one may transfer $A$ to an operator $\hat{A}$ on $L^2(X,\mu)$ by $\hat{A}f=A\hat{f}$ for $f\in L^2(X,\mu).$  
 
\section{Continuous Approach}

Suppose $(X,d,\mu)$ is metric measure space where $(X,d)$ is a connected, compact metric space containing more than a single point.

\begin{lemma} \label{contmeas} 	Let $r>0.$ The map $x\mapsto \mu(B_r(x))$ is lower semi-continuous. In particular, there exists a $c>0$ such that $\mu(B_r(x))>c$ for all $x\in X.$
\end{lemma}
\begin{proof} Let $\phi(x)=\mu(B_r(x))$ for $x\in X.$ Let $t\in\mathbb{R}$ and suppose $x\in \phi^{-1}(t,\infty).$ Let $\epsilon>0$ with $\epsilon<\phi(x)-t.$ By continuity of measure, there exists a $\delta>0$ so that $\mu(B_{r}(x)-\mu(B_{r-\delta}(x))<\epsilon.$ Let $y\in X$ with $d(x,y)<\delta.$ Suppose $z\in X$ with $d(x,z)<r-\delta.$  Then $d(z,y)<d(z,x)+d(x,y)<r.$ Hence $B_{r-\delta}(x)\subset B_r(y).$ Therefore $\mu(B_r(y))\geq \mu(B_{r-\delta}(x))>\mu(B_r(x))-\epsilon>t.$ Hence $\phi^{-1}(t,\infty)$ is open.

As for the second claim, since $\mu$ has full support, $\phi(x)>0$ for all $x.$ Hence $X=\cup_{n=1}^\infty \phi^{-1}(1/n,\infty).$ Since $X$ is compact, there exists a finite $N$ so that $X=\cup_{n=1}^N \phi^{-1}(1/n,\infty).$ Thus $\phi(x)>1/N$ for all $x\in X.$
\end{proof}

For $r>0,$ $x,y\in X$ let $u_r(x,y):=\frac{1}{\mu(B_r(x))}\chi_{B_r(x)}(y)$. For $f\in L^2(X,\mu),$ let 
\[(U_r f)(x):=\frac{1}{\mu(B_r(x))}\int_{B_r(x)} f(y)d\mu(y).\]  We then may consider the associated discrete time random walk generated by the Markov kernel $U_r$. In \cite{olliv} (with $r$ replaced by $\epsilon$), this walk is called the ``$\epsilon$-step random walk".

\subsection{Partially Symmetrized Walk}
However, for reasons that will become clear later, we will work primarily with the following ``partially symmetrized" version which we define as follows. For $r>0$ and $x\in X,$ let $v_r(x)=\mu(B_r(x))$ and \[a_r(x):=\frac{1}{\mu(B_r(x)}\int_{B_r(x)}\left(1+\frac{v_r(x)}{v_r(y)}\right)d\mu(y).\] Then for $r>0$ and $x,y\in X,$ define 
\[w_r(x,y):=a_r(x)^{-1}\frac{1}{\mu(B_r(x))}\chi_{B_{r}(x)}(y)\left(1+\frac{v_r(x)}{v_r(y)}\right).\]

By the previous lemma and since $\mu$ is a finite measure, there exists a $C_1>0$ so that $\frac{1}{C_1}\leq v_r(x) \leq C_1$ for all $x\in X.$ It follows that $\frac{1}{C_1^2}\leq \frac{v_r(x)}{v_r(y)}\leq C_1^2$ for all $x,y\in X.$ Hence there exists a $C>0$ so that  $p_r(x,y)\leq C$ for all $x,y\in X.$ As in section 2.3, define a Markov kernel $W_r$ by $dW_r(x,\cdot)=w_r(x,\cdot)d\mu$ for $x\in X.$ We also denote the corresponding Markov operator by $W_r.$ That is for $f\in L^2(X,\mu),$ \[W_r(f)(x)=\int w_r(x,y)d\mu(y)=a_r(x)^{-1}\frac{1}{\mu(B_r(x))}\int_{B_r(x)}\left(1+\frac{v_r(x)}{v_r(y)}\right)f(y)d\mu(y).\] Then let $\bm{y}(r)=(\bm{y}(r)_k)_{k=0}^\infty$ be the associated discrete time random walk induced by $W_r.$ It is a discrete time pure jump Markov process\footnote{Note if $v_r$ is constant, as in the case of $\mathbb{R}^n$ with Lebesgue measure, we have $W_r=U_r.$}.

Let $R>0$ and $x_0\in X.$ Let $B:=B_R[x_0].$  Then let $\tau_{B,r}:=\inf\{k\;|\;\bm{y}(r)_k\notin B\}.$ Then let $E_{B,r}(x)=\mathbb{E}^x\tau_{B,r}.$ We will often omit the $B$ in the notation, writing $E_r(x)$ instead of $E_{B,r}(x).$

Then we have $E_r(x)=0$ for $x\in B^c$. For $x\in B,$ by Proposition \ref{markov} we have 
\[E_r(x)=1+\int_{B_r(x)}w_r(x,y)E_r(y)d\mu(y).\] 

\begin{example} \label{euclid}
Let $B$ the closed ball of radius $R$ about the origin in $\mathbb{R}^n$ under the Euclidean norm $|\cdot|$. Then let $X$ be a compact subset of $\mathbb{R}^n$ containing $B_{R+1}(x).$ Let $x\in B^\circ.$ Let $r$ be small enough so that $B_r(x)\subset B.$ Then using the process defined by uniform jumps in a ball of radius $r$ according to the Lebesgue measure, we have \[E_r(x)=\left(\frac{n+2}{n}\right)\frac{R^2-|x|^2}{r^2}.\]
\end{example}

Let \[E^+_{r,B}:=\sup_{y\in B} E_{r,B}(y).\]
For $\beta>0$ let \[T_\beta(B):=\limsup_{r\rightarrow 0^+} E^+_{r,B} r^\beta.\] 
\begin{proposition} There exists a unique $\beta(B)\in [0,\infty]$ defined by \[\beta(B):=\sup\{\beta\;|T_\beta(B)=\infty\}=\inf\{\beta\;|\;T_\beta(B)=0\}\] with the property that if $\gamma<\beta(B)$ then $T_\gamma (B) = \infty,$ and if $\gamma>\beta(B)$ then $T_\gamma(B)=0.$ Moreover, if $B'=B_{R'}[x_0]$ with $R'<R$ then $\beta(B')\leq \beta(B).$
\end{proposition}
\begin{proof} 
If $0\leq \gamma<\beta$ then $E_{r,B}r^\beta = r^{\beta-\gamma} E_{r,B}r^{\gamma}.$ Hence if $T_\beta(B)>0$ then $T_\gamma(B)=\infty,$ and if $T_\gamma(B)<\infty$ then $T_\beta(B)=0.$ Monotonicity is clear from $E_{r,B'}\leq E_{r,B}.$
\end{proof}
Then let $\beta(x):=\inf_{R>0} \beta(B_R[x]).$ We call $\beta$ the \textit{local time exponent.} 

\begin{proposition} \label{betasemi} The local time exponent $\beta$ is upper semicontinuous. In particular $\beta$ is Borel measurable and bounded above. Moreover, we have $\beta\geq 1.$\label{semicont1}
\begin{proof} The proof of the first statements follows the same steps as Proposition \ref{semicont} and is omitted. To show 
$\beta\geq 1,$ suppose $x\in X$ and $R>0.$ Let $0<r<R.$ Let 
$B=B_R(x).$ Then we must have that $E_{B,r}(x)$ is greater than or equal to the smallest $k$ such that $x_0,x_1,...,x_k\in X$ 
with $x_0=x,$ $d(x_i,x_{i+1})\leq r$ for $i=0,...,k,$ $x_i\in B$ 
for $i=1,...,k-1,$ and $x_k\notin B$. Then $R\leq d(x_0,x_k)\leq 
\sum_{i=0}^{k-1}d(x_i,x_{i+1})\leq kr.$ Hence $E_{B,r}(x)\geq 
\frac{R}{r}.$ It then follows that $\beta(x)\geq 1.$ \end{proof}
\end{proposition}
\begin{example}
Consider a closed ball $B$ of radius $R$ about the origin in $\mathbb{R}^n$ as in Example \ref{euclid}. Then for $x\in B_R(0),$ $\beta(x)=2.$ This follows immediately from the formula in Example \ref{euclid}. Note that we also have that $T_2(B_R(x))\asymp R^2.$ 
\end{example}

\section{Variable Time Regularity}
Recall that the Hausdorff measure with constant dimension is used to calculate the (constant) Hausdorff dimension of open balls which in turn are used to determine the local dimension. Then, in a process of re-normalization, the local dimension is used in a new measure that takes into account the finer local properties given by the local dimension. Guided by this analogy, we re-normalize the walks $\bm{y}(r)$ to take into account the local time scaling.

However, in order to preserve approximate symmetry, it is necessary to first modify $w_r$ as follows. For $r>0$ and $x\in X,$ let $\tau_r(x):=r^{\beta(x)}$ and \[d_r(x):=\mu(B_r(x))r^{\beta(x)}=v_r(x)\tau_r(x).\] Then for $r>0$, $x\in X,$ let \[q_r(x):=\frac{1}{\mu(B_r(x))}\int_{B_r(x)}\left(1+\frac{d_r(y)}{d_r(x)}\right)d\mu(y).\] Then for $r>0$ and $x,y\in X,$ set 
\[p_r(x,y):=q_r(x)^{-1}\frac{1}{\mu(B_r(x))}\chi_{B_r(x)}(y)\left(1+\frac{d_r(x)}{d_r(y)}\right).\] Then we have a corresponding Markov operator $P_r$ defined for $f\in L^2(X,\mu)$ by \[P_rf(x)=\int p_r(x,y)f(y)d\mu(y).\] Observe that if $\beta$ is constant, then $P_r=W_r.$ Note by Proposition \ref{betasemi}, for $r>0$ fixed, we have that there exists a $C_1>0$ such that $\frac{1}{C_1}\leq \frac{d_r(x)}{d_r(y)}\leq C_1.$ It follows that there exists a $C>0$ such that for all $x,y\in X,$ $p_r(x,y)\leq C.$ In particular $P_r$ is bounded and defines a pure jump discrete time Markov process. 

\subsection{Time Renormalization}

  Let $\Omega=(\mathbb{R}_+\times X)^{\mathbb{Z}^+\setminus\{0\}}.$ For each $x_0\in X, r>0$, applying Proposition \ref{kol} with $\lambda(x)=\tau_r(x)$  yields a probability 
measure $\mathbb{P}^{x_0}_r$ on $\Omega$ defined on ``cylinder sets" as follows. For $A_1,A_2,...,A_n\subset \mathbb{R}_+$ measurable and $U_1,U_2,...,U_n\subset X$ measurable,
\begin{equation*} 
\begin{split}
&\mathbb{P}^{x_0}_r(\{\omega\in (\mathbb{R}_+\times X)^{\mathbb{Z}^+}\;|\;\omega(i)\in A_i\times U_i\;\mbox{for}\;i=1,...,n\})
\\ &= \left(\prod_{i=1}^n \int_{A_i\times U_i}  \frac{e^{-t_i/\tau_r(x_{i-1})}}{\tau_r(x_{i-1})}p_r(x_{i-1},x_i)\right)dt_nd\mu(x_n)...dt_1d\mu(x_1).
\end{split}
\end{equation*}
Then, by Proposition \ref{kol}, there exists a probability measure on $\Omega$ extending the above definition.
Let $\Omega_{x_0}=\{\omega:\mathbb{Z}^+\rightarrow [0,\infty)\times X\;|\; \omega(0)=(0,x_0)\}.$ Then $\mathbb{P}^{x_0}_r$ may be considered as a measure on $\Omega_{x_0}.$ For $t\geq 0,$ define $\hat{t}$ by \[\hat{t}(\omega)=\sup\{k\geq 0\;|\;\sum_{j=1}^{k} \omega(j)_1 \leq t\}.\] Then set $\bm{x}(r)_t(\omega):=\omega(\hat{t})_2.$ 

Since by Proposition \ref{semicont1}, there exists an $M>1$ 
such that $1\leq \beta \leq M$, we have that there exists a $C_1>0$ 
such that $C_1\leq r^{\beta(y)}$ for all $y\in X.$ By Lemma \ref{contmeas}, 
there exists $C_2, C_3>0$ such that $C_2\leq \mu(B_r(y))\leq C_3$ for all 
$y\in X.$ By Proposition \ref{gen}, we then have the following.

\begin{proposition} \label{generator} $(\bm{x}(r)_t)_{t\geq 0}$ is a continuous time Markov process with 
generator defined for $f\in L^2(X,\mu)$ and $x\in X,$ by \[-\frac{d}{dt}\Bigm |_{t=0} \mathbb{E}^{x} f(\bm{x}(r)_t) = \mathscr{L}_rf(x):=\frac{1}{r^{\beta(x)}}\int_{B_r(x)}p_r(x,y)(f(x)-f(y))d\mu(y).\]
\end{proposition}

Let \[\tau_{r,B}:=\inf\{t\;|\;\bm{x}(r)_t\notin B\}.\]
We now define \textit{exit time functions}. 
For $x\in X,$ let \[\phi_{r,B}(x):=\mathbb{E}^x\tau_{r,B}.\]

\begin{proposition} \label{equation} Let $B$ a non-empty ball and $r>0.$ We have $\mathscr{L}_r\phi_{r,B}(x)=1$ for $x\in B,$ and $\phi_{r,B}(x)=0$ for $x\notin B.$ 
\end{proposition}
\begin{proof} By Proposition \ref{eq1} we have that $r^{\beta(x)}=(I-P_r)\phi_{r,B}(x)$ for $x\in B.$ Hence, by Proposition \ref{generator}, $\mathscr{L}_r\phi_{r,B}(x)=1$ for $x\in B.$ It follows immediately from the definition of $\tau_{r,B}$ that $\phi_{r,B}$ vanishes outside of $B.$
\end{proof}

Let $\phi^+_{r,B}:=\sup_{y\in B} \phi_{r,B}(y).$
Then let $\mathcal{T}^+(B):=\lim \sup_{r\rightarrow 0^+} \phi^+_{r,B}.$

\begin{definition} For $\beta(\cdot)$ the local time exponent, we say $(X,d,\mu)$ satisfies $(E_\beta)$, or that $(X,d,\mu)$ is \textit{(variable) time regular} with exponent $\beta$ if there exists a $C>0$ such that for all $x\in X, 0<r<\frac{\diam(X)}{2},$ we have 
\[\frac{1}{C}r^{\beta(x)}\leq \mathcal{T}(B_r[x])\leq C r^{\beta(x)}.\]
\end{definition}

\begin{example}
Let $B$ the open ball of radius $R$ about the origin in $\mathbb{R}^n$ under the euclidean norm $|\cdot|$. Then let $X$ be a compact subset of $\mathbb{R}^n$ containing $B_{R+1}(x).$ Let $x\in B.$ Let $r$ small enough so that $B_r(x)\subset B.$ Then we have \[\phi_r(x)=\left(\frac{n+2}{n}\right)(R^2-|x|^2),\] so that $X$ satisfies $E_2.$
\end{example}

Recall that a function $\varphi$ on a metric space $(X,\rho)$ is log-H{\"o}lder continuous if 
there exists a $C>0$ such that $|\varphi(x)-\varphi(y)|\leq \frac{-C}{\log(\rho(x,y))}$ for all $x,y$ with $0<\rho(x,y)<\frac{1}{2}.$

The following is an analog of Lemma \ref{log} for $\beta$ instead of $\alpha$. The proof is analogous to that of Lemma \ref{log}.
\begin{lemma}
If $X$ is variable time regular with exponent $\beta$ then $\beta$ is log-H{\"o}lder continuous.
\end{lemma}

\begin{proof} By time regularity, there exists a constant $D>1$ such that $\mathcal{T}(B_r[x])\leq Dr^{\beta(x)}$ and $r^{\beta(x)}\leq D\mathcal{T}(B_r[x])$ for all $0<r\leq\diam(X)/2,\;x\in X.$ Suppose $x,y\in X$ with $0<r:=d(x,y)<\frac{1}{2}.$ Say $\beta(x)\geq \beta(y).$ Since $\beta$ is bounded, let $R<\infty$ be an upper bound for $\beta,$ and let $C$ be defined by $e^C:=2^R D^2.$ Then since $B_r[y]\subset B_{2r}[x],$ we have $\phi_{\epsilon, B_{r}[y]}(z)\leq \phi_{\epsilon, B_{2r}[x]}(z)\leq \phi^+_{\epsilon, B_{2r}[x]}$ for all $\epsilon>0$ and $z\in B_{r}[y].$ So $\phi^+_{\epsilon, B_{r}[y]}   \leq \phi^+_{\epsilon, B_{2r}[x]}$ for all $\epsilon.$ Hence $\mathcal{T}(B_r[y])\leq \mathcal{T}(B_r[x]).$ Hence $r^{\beta(y)}\leq D\mathcal{T}(B_r(y))\leq D\mathcal{T}(B_{2r}(x))\leq D^2 2^R r^{\beta(x)} = e^Cr^{\beta(x)}.$ Thus $d(x,y)^{|\beta(y)-\beta(x)|}\geq e^{-C}.$ So $|\beta(x)-\beta(y)|\leq \frac{-C}{\log(d(x,y)}.$
\end{proof}

\section{Green's functions and Dirichlet spectrum}
Let $0<R<\diam(X)/2,\; x_0\in X,$ $B=B_R[x_0].$ For $r>0$ consider the random walk killed on exiting $B.$ Let $P_r^B$ be the Markov operator with kernel defined by the function
\[p^B_r(x,y)= \chi_{B\times B}(x,y) p_r(x,y)\] recalling that $p_r(x,y)=q_r(x)^{-1}\frac{1}{\mu(B_r(x))}\left(1+\frac{d_r(x)}{d_r(y)}\right)\chi_{B_r(x)}(y),$ where we have defined $q_r(x)=\frac{1}{\mu(B_r(x)}\int_{B_r(x)}\left(1+\frac{d_r(x)}{d_r(y)}\right),$ and $d_r(x)=\mu(B_r(x))r^{\beta(x)}=v_r(x)\tau_r(x)$ for $x\in X.$ We have $P_r$ is the corresponding Markov operator. Note then that $P^B_r=\chi_{B}P_B\chi_B.$ 

Observe that for the process defined by $P_r$, there is an equilibrium Borel measure $\rho_r$ with density \[\frac{d\rho_r}{d\mu}(x):=\frac{q_r(x)}{\tau_r(x)}=\int_{B_r(x)}\left(\frac{1}{d_r(x)}+\frac{1}{d_r(y)}\right)d\mu(y).\] Indeed, we have $P_r$ is $\rho_r$-symmetric, since 
\[\frac{q_r(x)}{\tau_r(x)}p_r(x,y)=\left(\frac{1}{d_r(x)}+\frac{1}{d_r(y)}\right)\chi_{B_r(x)}(y)=\frac{q_r(y)}{\tau_r(y)}p_r(y,x).\]

Similarly, for the process defined by $W_r,$ there is an equilibrium Borel measure $\nu_r$ with density \[\frac{d\nu_r}{d\mu}(x)=a_r(x)=\int_{B_r(x)}\left(\frac{1}{v_r(x)}+\frac{1}{v_r(y)}\right)d\mu(y).\]

For $f$ defined on $B$ we let $f\chi_B$ the function defined on all of $X$ by extending it to equal $0$ outside of $B.$ Hence if $f\in L^1(B,\nu_r),$ then 
\begin{align*}
\|P_r^Bf(x)\|_{L^1(B,\nu_r)}&=\int_B |P_r^Bf(x)|d\rho_r(x) \leq \int_B |P^B_r|f|(x)|d\rho_r(x) \\
&= \int \int \chi_B(x)\chi_B(y)p_r(x,y)\frac{q_r(x)}{\tau_r(x)}|f(y)|d\mu(y)d\mu(x)\\& = 
\int \int \chi_B(x)\chi_B(y)\frac{q_r(y)}{\tau_r(y)}p_r(y,x)|f(y)|d\mu(x)d\mu(y)\\
&\leq \int \chi_B(y)|f(y)|\frac{q_r(y)}{\tau_r(y)}d\mu(y)= \|f\|_{L^1(B,\rho_r)}.
\end{align*}
Therefore $\|P^B_r\|_{\mathscr{B}(L^1(B,\rho_r))}\leq 1.$ In particular its spectral radius, as an operator on $L^1(B,\rho_r),$ is at most one. 

Note also that we may consider $P^B_r$ as an operator on $L^2(B,\nu_r).$ In this capacity if $f,g\in L^2(B,\nu_r)$ then we have 
\begin{align*}
\langle f, P^B_rg \rangle_{L^2(B,\rho_r)} &= \int_B f(x)P^B_rg(x)\frac{q_r(x)}{\tau_r(x)}d\mu(x) \\
&= \int\int \chi_B(x)\chi_B(y)\frac{q_r(x)}{\tau_r(x)}p_r(x,y)g(y)f(x)d\mu(y)d\mu(x) \\
&= \int\int \chi_B(x)\chi_B(y)\frac{q_r(y)}{\tau_r(y)}p_r(y,x)g(y)f(x)d\mu(x)d\mu(y) \\
&= \int_B g(y)P^B_rf(y)\frac{q_r(y)}{\tau_r(y)}d\mu(y)= \langle P^B_rf, g \rangle_{L^2(B,\rho_r)}.
\end{align*}
Hence $P^B_r$ is self adjoint. Note that since $P^B_r$ has a square integrable kernel it is compact.

Let $L_r:=I-P_r.$ Note $I-P_r$ is self adjoint on $L^2(X,\rho_r).$ 
Let $\tau_r$ denote the multiplication operator by $x\mapsto \tau_r(x)=r^{\beta(x)}.$ Also let $q_r$ denote the multiplication operator by $x\mapsto q_r(x).$ Then $\mathscr{L}_r=\frac{1}{\tau_r}(I-P_r).$ Define a measure $\mu_r$ by \[d\mu_r(x)=q_r(x)d\mu(x),\] where $q_r(x)=\frac{1}{\mu(B_r(x))}\int_{B_r(x)}\left(1+\frac{d_r(x)}{d_r(y)}\right)d\mu(y).$ Note $d\mu_r=\tau_r(x)d\rho_r.$ Then for $f,g\in L^2(X,\mu_r),$ 
\begin{equation*}
\begin{split}
&\langle f,\mathscr{L}_r g\rangle_{\mu_r}=\langle f, \frac{1}{\tau_r}L_r g\rangle_{\mu_r}
=\langle f, \frac{q_r}{\tau_r}L_rg\rangle_{\mu}=\langle f, L_r g\rangle_{\rho_r}
\\&=\langle L_r f, g\rangle_{\rho_r}=\langle \frac{q_r}{\tau_r} L_r f, g\rangle_{\mu}
=\langle \frac{1}{\tau_r}L_r f, g\rangle_{\mu_r}.
\end{split}
\end{equation*}

Hence $\mathscr{L}_r$ is self-adjoint on $L^2(X,\mu_r).$ Note, since we have $\mu, \rho_r, \nu_r,$ and $\mu_r$ are mutually strongly equivalent to one another, we have $f\in L^2(X,\mu)$ if and only if $f$ is in the $L^2$ space of any of the other three measures. Then we have for $f\in L^2(X,\mu),$ by the Fubini-Tonelli Theorem and since $\chi_{B_r(x)}(y)=\chi_{B_r(y)}(x),$
\begin{equation}
\begin{split}
&\langle f, L_r f\rangle_{\rho_r} = \langle f, \mathscr{L}_r f\rangle_{\mu_r}= \int f(x)q_r(x)\mathscr{L}_r f(x)d\mu(x) \\
&=\int \int \chi_{B_r(x)}(y)\left(\frac{1}{d_r(x)}+\frac{1}{d_r(y)}\right)f(x)(f(x)-f(y))d\mu(y)d\mu(x)\\
&= \frac{1}{2}\int \int \chi_{B_r(x)}(y)\left(\frac{1}{d_r(x)}+\frac{1}{d_r(y)}\right)(f(x)-f(y))^2d\mu(y)d\mu(x)\\
&= \int \int \chi_{B_r(x)}(y)\frac{1}{d_r(x)}(f(x)-f(y))^2d\mu(y)d\mu(x)\\
&= \int \frac{1}{r^{\beta(x)}\mu(B_r(x))}\int_{B_r(x)}(f(x)-f(y))^2d\mu(y)d\mu(x).
\end{split}
\end{equation}

Similarly, $\langle f, (I-W_r)f\rangle_{\nu_r}=\int\frac{1}{\mu(B_r(x))}\int_{B_r(x)}(f(y)-f(x))^2d\mu(y)d\mu(x).$

For an operator $A$ on $L^2(X,\mu),$ we define $A^B$ on $L^2(B,\mu)$ by $\chi_BA\chi_B.$ In particular, we may consider $A^B$ as an operator on $L^2(X,\mu)$ as well.

For convenience we state without proof the following classic result about the spectrum of a compact operator on a Banach space.  For more information and proofs see, for instance, \cite{Simon}.

\begin{proposition} (Riesz-Schauder) 
If $A$ is compact on a Banach space then $\sigma(A)\setminus\{0\}$ is discrete and contains only eigenvalues of finite multiplicity. The eigenvalues may only accumulate at $0.$ 
\end{proposition}

Hence the non-zero spectrum of $P^B_r$ and $W^B_R$ consists entirely of (real) eigenvalues. Although we will focus on $P^B_r$, note that the results apply equally well to $W^B_r$ with the appropriate substitutions of $\rho_r$ for $\nu_r,$ etc. Suppose $\lambda$ is an $L^2(B,\rho_r)$ eigenvalue of $P^B_r$. Then if $f\in L^2(B,\rho_r)$ is a corresponding eigenvector, since $\rho_r$ is a finite measure, $f\in L^1(B,\rho_r)$ . So $\lambda$ is also in the spectrum of $P^B_r$ as an operator on $L^1(B,\rho_r)$. Hence $|\lambda|\leq 1.$ It follows that the $L^2(B,\rho_r)$ spectrum of $P^B_r$ is contained in $[-1,1].$ Moreover, since $P^B_r$ is normal, its spectral radius equals its operator norm (see Theorem 2.2.11 in \cite{Simon}).

We now show that neither $1$ nor $-1$ are in the spectrum. 

\begin{lemma} \label{Pbound} $\|P^B_r\|_{\mathscr{B}(L^2(B,\rho_r))}<1$ and $\|W^B_r\|_{\mathscr{B}(L^2(B,\nu_r))}<1.$
\end{lemma}
\begin{proof}
We prove this for $P^B_r$ only. The result for $W^B_r$ is completely analogous. We already know that $\sigma(P^B_r)\subset[-1,1]$. If $1,-1$ are in the spectrum then they are eigenvalues since $P^B_r$ is compact. 
Suppose $(I-P^B_r)f=0$ for some $f\in L^2(B,\rho_r).$ Then \begin{equation*}
\begin{split}
&\langle f, L^B_rf\rangle_{L^2(B,\rho_r)}=\langle f, \mathscr{L}^B_r f\rangle_{\rho_r}=\\&\int_X \int_{B_r(x)}\frac{1}{d_r(x)}|\chi_{B}(y)f(y)-\chi_{B}(x)f(x)|^2d\mu(y)d\mu(x)=0.
\end{split}
\end{equation*}

Hence, since $X$ is compact there exist finitely many balls $B_r(x_1),...,B_r(x_n)$ covering $X$ with $f$ $\mu\mhyphen$a.e. constant on each $B_r(x_i).$ Since $X$ is connected, the graph formed with vertices $x_1,...,x_n$ and edge relation $x_i\sim x_j$ if and 
only if $B_r(x_i)\cap B_r(x_j)$ is connected. It follows that $\chi_Bf$ is $\mu\mhyphen$a.e constant, where we may extend $\chi_Bf$ outside of $B$ by setting it equal to $0.$ Then since $B_R[x_0]$ is closed and we have assumed its complement is non-empty, its complement contains an open set. Since $\mu$ has full support, $\mu(B^c)>0.$ But $f\chi_B$ is identically $0$ on $B^c.$ It follows that $f\chi_B$ is $0$ $\mu\mhyphen$a.e. Note that $d_r$ is bounded below by a positive constant. Hence $f\chi_B$ is $0$ $\rho_r\mhyphen$a.e. In particular $f=0 \in L^2(B,\rho_r).$

Now suppose $(I+P^B_r)f=0$ for some $f\in L^2(B,\rho_r).$ Then 
\[\langle f, (I+P^B_r)f\rangle_{L^2(B,\rho_r)}=\int_X \frac{1}{d_r(x)}\int_{B_r(x)} |\chi_{B}(y)f(y)+\chi_{B}(x)f(x)|^2d\mu(y)d\mu(x)=0.\]  
Hence there exist $x_1,...,x_n$ in $X$ with $(B_r(x_j))_{j=1}^n$ an open cover of $X$ 
such that $\chi_Bf(y)=-\chi_Bf(x_i)$ for $\mu\mhyphen$a.e. $y\in B_r(x_i)$ for each $i=1,..,n.$ In particular $f\chi_B$ is $\mu\mhyphen$a.e constant on each $B_r(x_i)$. Suppose $i\neq j$. 
Then since the graph formed with vertices $x_1,...,x_n$ by the non-empty intersection relation is connected, it follows that there is a path in the graph connecting $x_i$ and $x_j$. However, if $x_k\sim x_i$ then $B_r(x_i)\cap B_r(x_k)$ is non-empty and open. Moreover, for $\mu\mhyphen a.e.$ $y$ in the intersection, $\chi_b(y)f(y)=-\chi_B(x_i)f(x_i)=-\chi_B(x_k)f(x_k).$ Hence 
$\chi_B(x_i)f(x_i)=\chi_B(x_k)f(x_k).$ Continuing in this way along the path leads to $\chi_B(x_i)f(x_i)=\chi_B(x_j)f(x_j).$ Therefore $\chi_Bf$ is constant on $x_1,...,x_n.$ It follows by the connectivity of the induced graph that $\chi_Bf$ is 
$\mu\mhyphen a.e.$ constant. This implies that it is $\rho_r\mhyphen$a.e. constant. Since $B^c$ has non-zero measure, $f$ is $\rho_r\mhyphen$a.e. equal to $0.$

Therefore, since the spectrum away from $0$ is discrete, there exists an $\epsilon>0$ so that $\sigma(P^B_r)\subset [-1+\epsilon,1-\epsilon].$ In particular $\|P^B_r\|<1.$
\end{proof}
Now consider the following Dirichlet problem on $B,$
\[\mathscr{L}_ru(x)=f(x),\;\;\mbox{for}\;x\in B,\;u,f\in L^2(X,\mu),\;u(x)=0\;\;\mbox{for\;} x\in B^c.\]
We show that we can find a Green function $G^B_r\in L^1(X\times X, \;\mu\otimes \mu)$ solving the Dirichlet problem with 
\[u(x)=\int G^B_r(x,y)f(y)d\mu(y).\] Recall $d\nu_r(x)/d\mu=r^{\beta(x)}\mu(B_r(x));$ and $\mathscr{L}^B_r:L^2(B,\nu_r)\rightarrow L^2(B,\nu_r)$ is defined by 
\[\mathscr{L}^B_rf(x)=\frac{1}{r^{\beta(x)}}(I-P^B_r)f(x).\]

\begin{theorem} We have that $(L^B_r)^{-1}$ and $(W^B_r)^{-1}$ are bounded operators on $L^2(B,\mu)$ possessing integral kernels, $K_r^B$, $H_r^B,$ respectively, in $L^1(B\times B,\;\mu \otimes \mu)$. Moreover, $(L^B_r)^{-1}$ and $(W^B_r)^{-1}$ are bounded positive operators on $L^2(B,\rho_r)$ $L^2(B,\nu_r),$ respectively, and the functions \[g_r^B(x,y):=\frac{K^B_r(x,y)\tau_r(y)}{q_r(y)}\] and
\[h_r^B(x,y)=\frac{H^b_r(x,y)}{a_r(y)}\] are symmetric integral kernels for $(L^B_r)^{-1}$ in $L^2(B,\rho_r)$ and $(W^B_r)^{-1}$ in $L^2(B,\nu_r),$ respectively. Also, $(\mathscr{L}^B_r)^{-1}$ is a bounded operator on $L^2(B,\mu)$ which has an integral kernel $G^B_r\in L^1(B\times B,\;\mu \otimes \mu)$. Moreover, $(\mathscr{L}^B_r)^{-1}$ is a bounded positive operator on $L^2(B,\mu_r)$, and the function \[g_r^B(x,y)=\frac{G^B_r(x,y)}{q_r(y)}=\frac{K^B_r(x,y)\tau_r(y)}{q_r(y)}\] is a symmetric integral kernel for $(\mathscr{L}^B_r)^{-1}$ in $L^2(B,\mu_r).$ 
\end{theorem}

\begin{proof} First note that since $r$ is fixed and the maps $x\mapsto r^{\beta(x)}$ and $x\mapsto \mu(B_r(x))$ are positive and uniformly bounded both above and away from $0$, we have $\mu\asymp \mu_r \asymp \nu_r\asymp \rho_r$ So $L^2$ convergence in any of the four measures implies convergence in all of them.

We prove the result for $L^B_r.$ The result for $W^B_r$ is analogous. Lemma \ref{Pbound} implies that the Neumann series $\sum_{k=0}^\infty (P^B_r)^k$ converges in the operator norm topology to the bounded operator $(I-P^B_r)^{-1}.$
For each $j,$ let $K^B_{r,j}(\cdot,\cdot)$ be the integral kernel for $(P^B_r)^j.$ 

Set \[K^B_r(x,y):=\sum_{j=0}^\infty k^B_{r,j}(x,y).\] This is well defined since it is a sum of non-negative functions. 

Let $f\in L^2(B,\rho_r).$ Then let $h=(I-P^B_r)^{-1}f\in L^2(B,\rho_r).$ Let $h_n=\sum_{k=0}^n (P^B_r)^kf.$ Then by the convergence of the Neumann series, 
\[h_n\convn h\;\;\mbox{in\;} L^2(B,\rho_r).\] Note also that for each $x\in B$ and $n\in \mathbb{Z}^+$, 
\[\sum_{j=0}^n |\int K^B_{r,j}(x,y)f(y)d\mu(y)|\leq \sum_{j=0}^n (P^B_{r})^j|f|(x)\leq (I-P^B)^{-1}|f|(x).\] It follows for $\rho_r$ almost all $x\in B$ that 
\[\sum_{j=0}^\infty \int K^B_{r,j}(x,y)f(y)d\mu(y)\;\;\mbox{converges absolutely.}\]
Let $q$ denote the pointwise limit of this series. Then since $h_n\convn h$ in $L^2(B,\rho_r),$ there exists a subsequence $h_{n_j}$ such that \[h_{n_j} \convk h \;\;\mbox{pointwise}\;\rho_r\mhyphen\mbox{a.e.}.\] It follows that $h=q$ a.e. But by monotone convergence, \[K^B_r(x,\cdot)|f(\cdot)|\in L^1(B,\rho_r)\;\;\mbox{for}\;\rho_r\mhyphen\mbox{a.e\;} x\in B.\]
Since this function dominates $|\sum_{j=0}^n K^B_{r,j}(x,\cdot)f(\cdot)|$ for each $n,$ by the dominated convergence theorem it follows that 
\[(I-P^B_{r})^{-1}f(x)=\int K^B_r(x,y)f(y)d\mu(y).\]

Note further that $K^B_r\in L^1(\mu \otimes\mu).$ Indeed, since $\mu(B_r(x))$ is bounded above by a constant independent of $x,$ $\int K^B_r d(\mu\otimes \mu)=\|(I-P^B_r)^{-1}\chi_B \|
_{L^1(X,\mu)}=\|(I-P^B_r)^{-1}\chi_B\|_{L^1(B,\mu)}\leq C\|(I-P^B_r)^{-1}\chi_B\|
_{L^2(B,\rho_r)}<\infty$ for some constant $C>0.$

Note $L^B_r$ is self-adjoint on $L^2(B,\rho_r).$ Hence $(L^B_r)^{-1}$ is also self-adjoint. Therefore the integral kernel with respect to $\rho_r$ is symmetric. 
However, $(L^B_r)^{-1}f(x)=\int \frac{K_r^B(x,y)\tau_r(y)}{q_r(y)}f(y)d\rho_r(y).$

Therefore let $g_r^B(x,y)=\frac{K_r^B(x,y)\tau_r(y)}{q_r(y)}.$ Then $g^B_r$ is symmetric. Moreover, since 
\[\langle f, L^B_r f\rangle_{L^2(B,\rho_r)} = \int_X \int_{B_r(x)}\frac{1}{d_r(x)}|\chi_B(y)f(y)-\chi_B(x)f(x)|^2d\mu(y)d\mu(x)\geq 0,\]  $\mathscr{L}^B_r$ is a positive operator on $L^2(X,\rho_r).$

Set \[G^B_r(x,y)=K^B_r(x,y)r^{\beta(y)}=K^B_r(x,y)\tau_r(y).\]

We set $G^B_r(x,y)=0$ for $x\notin B$ or $y\notin B.$ Then if $\mathscr{L}_ru(x)=f(x)$ where $f=0$ outside of $B,$ then $\mathscr{L}^B_ru(x)=\chi_B(x)f(x).$

Hence \[(I-P^B_r)^{-1}(\chi_Bf)(x)=r^{\beta(x)}u(x).\]
Therefore, by our previous considerations, 

\[\chi_B(x) f(x) = \int K^B_r(x,y)r^{\beta(y)}u(y)d\mu(y)=\int G^B_r(x,y)u(y)d\mu(y).\] Since $K^B_r\in L^1(B\times B,\mu\otimes \mu),$ we also have $G^B_r\in L^1(B\times B,\mu\otimes \mu).$

Note $\mathscr{L}^B_r$ is self-adjoint on $L^2(B,\mu_r).$ Hence $(\mathscr{L}^B_r)^{-1}$ is also self-adjoint. Therefore the integral kernel with respect to $\mu_r$ is symmetric. 
However, \[(\mathscr{L}^B_r)^{-1}f(x)=\int \frac{G_r^B(x,y)}{q_r(y)}f(y)d\mu_r(y).\]

Therefore $\frac{G_r^B(x,y)}{q_r(y)}=\frac{K^B_r(x,y)\tau_r(y)}{q_r(y)}=g^B_r(x,y)$ is symmetric. Moreover, since 
\[\langle f, \mathscr{L}^B_r f\rangle_{L^2(B,\mu_r)} = \int_X \frac{1}{d_r(x)}\int_{B_r(x)}|\chi_B(y)f(y)-\chi_B(x)f(x)|^2d\mu(y)d\mu(x)\geq 0,\] $\mathscr{L}^B_r$ is a positive operator. Since $(I-P^B_r)^{-1}$ is bounded and $x\mapsto r^{\beta(x)}$ is bounded below, it follows that $(\mathscr{L}^B_r)^{-1}$ is also bounded. 
\end{proof}

\begin{corollary} $\phi_{r,B}(x)=\int_B G^B_r(x,y)d\mu(y)=\int_B g^B_r(x,y)d\mu_r(y).$
\end{corollary}
\begin{proof} This follows from the exit time equation $\mathscr{L}^B_r\phi_{r,B}=\chi_B.$
\end{proof}
\subsection{Faber-Krahn Inequality}
The following proof is inspired by the elegant Lemma 2.3 in Telcs \cite{art}.
\begin{theorem} \label{main1} There exists an $r_0>0$ so that if $\sigma^B_r$ is the bottom of the spectrum of $W^B_r$ on $L^2(B,\nu_r)$ then \[\sigma^B_r\geq \frac{c}{E^+_{r,B}}\] for all $0<r<r_0$ and some constant $c$ independent of $r,$ $R.$ and $x_0.$ 

Suppose $(E_\beta)$ holds. Then there exists an $r_0>0$ so that if $\lambda^B_r$ is the bottom of the spectrum for $\mathscr{L}^B_r$ on $L^2(B,\mu_r)$ then  \[\lambda^B_r\geq \frac{c}{R^{\beta(x_0)}}\] for all $0<r<r_0$ and some constant $c>0$ independent of $r,R,$ and $x_0.$
\end{theorem}
\begin{proof} 
We prove the second claim. The first follows from a similar argument using the symmetric Green function $h^B_r$ for $W^B_r$ since $\int h^B_r(x,y)d\nu_r(y)=E_{r,B}(x).$

If $\frac{1}{\lambda^B_r}$ is in the spectrum of $(\mathscr{L}^B_r)^{-1}$ we may choose a sequence $(f_n)_{n=1}^\infty$ with each $f_n$ non-zero and in $L^2(B,\mu_r)$ with 
\[\|((\mathscr{L}^B_r)^{-1}-\frac{1}{\lambda^B_r})f_n\|_{L^2(B,\mu_r)} \convn 0.\]
Hence also \[\|((\mathscr{L}^B_r)^{-1}-\frac{1}{\lambda^B_r})f_n\|_{L^1(B,\mu_r)} \convn 0.\]
We may assume each $\|f_n\|_{L^1(B,\mu_r)}=1.$ Let $\epsilon>0.$ Then let $f_n$ with 
\[\frac{1}{\lambda^B_r}=\|\frac{1}{\lambda^B_r}f_n\|_{L^1(B,\mu_r)}\leq \|(\mathscr{L}^B_r)^{-1}f_n\|+\epsilon.\]
Since $\limsup_{r \rightarrow 0^+} \phi^+_{r,B}\leq c_1R^{\beta(x_0)}$ for some $c_1>0$ independent of $x$ and $B,$ if $0<\epsilon<R^{\beta(x_0)},$ then there exists an $r_0>0$ so that if $0<r<r_0$ then $\phi^+_{r,B}\leq c_1R^{\beta(x_0)}+\epsilon<(c_1+1)R^{\beta(x_0)}=cR^{\beta(x_0)}.$ Then for $0<r<r_0,$ 
\begin{align*}
\|(\mathscr{L}^B_r)^{-1}f_n\|_{L^1(B,\mu_r)} &\leq \int_B\int_B g^B_r(x,y)|f_n(y)|d\mu_r(y)d\mu_r(x) \\ &= \int_B\int_B g^B_r(y,x)|f_n(y)|d\mu_r(x)d\mu_r(y) \\&= \int_B\phi_{r,B}(y)|f_n(y)|d\mu_r(y) \leq \phi_{r,B}^+ \|f_n\|_{L^1(B,\mu_r)}\\ &= \phi_{r,B}^+ \leq cR^{\beta(x_0)}.
\end{align*}

Hence $(\lambda^B_r)^{-1} \leq cR^{\beta(x_0)}+\epsilon.$ Since $\epsilon>0$ was arbitrary we have \[\frac{1}{\lambda^B_r} \leq cR^{\beta(x_0)}.\]
\end{proof}
We remark that since $\mu_r\asymp \mu$,  $\lambda-\mathscr{L}^B_r$ is invertible in $\mathscr{B}(L^2(B,\mu))$ if and only if it is invertible in $\mathscr{B}(L^2(B,\mu_r)).$ Hence the spectrum of $\mathscr{L}^B_r$ is the same on $L^2(B,\mu_r)$ and $L^2(B,\mu).$ Therefore the inequality for the bottom of the spectrum given above also holds for $L^2(B,\mu).$
\begin{corollary} \label{maincor}
There exists an $r_0>0$ such that for all $0<r<r_0$ and $f\in L^2(B,\mu_r)$ we have 
\[\langle f, \mathscr{L}^B_rf\rangle_{L^2(B,\mu_r)}\geq cR^{-\beta(x_0)}\|f\|_{L^2(B,\mu)}\] for some constant $c>0$ independent of $r,R,$ and $x_0.$
\end{corollary}
\begin{proof}
By the min-max principle $\lambda^B_r=\inf_{\|f\|_{L^2(B,\mu_r)}=1}\langle  f,\mathscr{L}^B_r f\rangle_{L^2(X,\mu_r)}.$ However $d\mu_r(x)=q_r(x)d\mu(x)$ and $q_r(x)=\frac{1}{\mu(B_r(x))}\int_{B_r(x)}\left(1+\frac{d_r(x)}{d_r(y)}\right)d\mu(y)\geq 1$ for all $x.$ It follows that $\|f\|_{L^2(X,\mu_r)}\geq \|f\|_{L^2(X,\mu)}.$
\end{proof}
Part of the proof of the following lemma essentially follows the idea of the proof of Lemma 2.2 in Telcs \cite{telcs1}. See also the proof of Theorem 4.8 in \cite{grigorlau}.
\begin{theorem} Suppose $\mu$ is doubling and $X$ is variable time regular of exponent $\beta(\cdot)$. Then $\beta(x)\geq 2$ for all $x\in X.$ 
\end{theorem}

\begin{proof}  Let $x\in X,$ $0<R<\diam(X)/2,$ and $0<r<1.$ Let $\psi_{R,r,x}(y):=(\frac{R-d(x,y)}{r})\chi_{B_R[x]}(y).$ Then if $d(y,z)<r,$ \[|\psi_{R,r,x}(y)-\psi_{R,r,x}(z)|^2\leq 1.\] By the min-max principle, 
\[\sigma^{B_R[x]}_r\leq \frac{\langle \psi_{R,r,x}, H^B_r \psi_{R,r,x}\rangle_{L^2(B_R[x],\nu_r)}}{\|\psi_{R,r,x}\|_{L^2(B_R[x],\nu_r)}}\leq \frac{\mu(B_R[x])}{\int_{B_{R/2}[x]}|\psi_{R,r,x}(y)|^2a_r(y)d\mu(y)}.\] Since $a_r\geq 1,$ $|\psi_{R,r,x}|^2a_r$ is bounded below by $R^2/4r^2$ on $B_{R/2}[x].$ Hence \[\sigma_r^{B_R(x)}\leq \frac{4r^2\mu(B_R[x])}{R^2\mu(B_{R/2}[x])}.\] However, since $\mu$ is doubling, there exists a $C_1>0$ such that for all $\rho>0, y\in X,$ 
\[0<\mu(B_{2\rho}(y))\leq C_1\mu(B_\rho(y))<\infty.\] Let $C=4C_1$. Then by Theorem \ref{main1}, there exists a $r_0>0$ and a $c>0$ such that $\sigma^{B_R(x)}_r\geq \frac{c}{E^+_{r,B_{R}(x)}}$ for all $0<r<r_0,0<R<\diam(X)/2, x\in X.$ Therefore, putting these inequalities together we have,
\[cr^{-2}\leq \frac{C}{R^2}E^+_{r,B_R[x]}.\] It then follows that 
\[\beta(B_R[x])\geq 2,\] since otherwise there would exist a $\gamma$ with $\beta(B_R[x])<\gamma<2.$ Then $\eta:=2-\gamma>0$ and \[\frac{c}{r^\eta}\leq \frac{C}{R^2}E^+_{r,B_R[x]}r^\gamma.\] Taking a $\limsup$ of both sides as $r\rightarrow 0^+$ then would imply a contradiction since the right hand side would be $0,$ by the definition of $\beta(B_R[x])$ as a critical exponent, and the left $\infty,$ since $\eta>0.$ Hence $\beta(B_R[x])\geq 2,$ as claimed.  Since $0<R<\diam(X)/2$ was arbitrary, \[\beta(x)\geq 2.\]
\end{proof}

\subsection{Additional Conditions}
We formulate additional conditions that may or may not hold on $(X,d,\mu).$ The first is a stronger versions of the time regularity condition. The second is a time comparison condition, analogous to the condition $(\bar{E})$ in \cite{grigortelcs}. It states that the maximum mean exit time $\phi_{r,B}^+$ from a ball $B_R(x)$ is comparable to $\phi_{r,B}(x)$ with a constant independent of of $B$ and $r$ for $r$ small enough. The last condition, which we call the ``thin boundary condition" is a constraint on the measure $\mu$ that spheres  of the form $\{y\in X\;|\;d(x,y)=r\}$ have $\mu$-measure $0$ provided $r$ is small enough. The second condition will not be used in the sequel, although we show that it implies continuity of the mean exit time functions $\phi_{r,B}$ for $r$ small enough.

Consider again the walk $(\bm{x}(r)_t)_{t\geq 0}$ on a compact, connected metric measure space $(X,d,\mu)$ and the mean exit time $\phi_{r,B}(x)$ of the walk from a ball $B$ starting at $x.$ 

\begin{definition} We say $(X,d,\mu)$ is \textit{strongly (variable) time regular} of exponent $\beta$ if there exists a $C>0,$ $r_0>0$, such that if $0<R<\diam(X)/2$, $0<r<r_0$, and $x\in X,$ then \[\frac{1}{C}R^{\beta(x)}\leq \phi^+_{r,B_R[x]}\leq CR^{\beta(x)}.\]
\end{definition} Clearly if $(X,d,\mu)$ is strongly time regular of exponent $\beta$ then it is time regular of exponent $\beta.$

\begin{definition} $(X,d,\mu)$ satisfies the \textit{time comparability condition} if there exist constants $0<R_0<\diam(X)/2$, $0<r_0<R_0$, $c>0,$ such that for all $0<r<r_0$, if $B=B_R[x]$ for $x\in X$ and $0<R<\diam(X)/2,$ then \[c\phi_{r,B}^+\leq \phi_{r,B}(x).\]
\end{definition}
\begin{proposition} \label{comparison} Suppose $(X,d,\mu)$ satisfies the time comparability condition and the variable time regularity condition. Then if $B=B_R[x_0]$ is a ball with $0<R<\diam(X)/2,$ there exist $c_1>0,\;r_1>0$ so that if $x\in B_{R/4}[x_0]$ and $0<r<r_1$, then $\phi_{r,B}(x)\geq c_1.$ 
\end{proposition}
\begin{proof} Let $x\in B_{r/4}[x_0].$ Then if $0<r<r_0,$ $B_{r/4}[x_0]\subset B_{r/2}[x]\subset B.$ Hence \[\phi_{r,B}(x)\geq \phi_{r,B_{R/2}[x]}(x)\geq c\phi^+_{r,B_{R/2}[x]}\geq c\phi^+_{r,B_{R/4}[x_0]}.\] However, by the time regularity condition there exists a $c_1>0$ and $0<r_1<r_0$ such that if $0<r<r_1$ then $\phi_{B_{R/4}[x_0]}^+\geq c_1.$
\end{proof}

It will be desirable for us to have a condition to ensures that the map $(r,x)\mapsto \mu(B_r(x))$ is continuous for $r$ small enough. Note, that there are examples of connected, compact, Ahlfors regular spaces in which the map $(r,x)\mapsto \mu(B_r(x))$ has discontinuities at all scales. See, for example, the following figure.

\begin{definition}
We will say that a Borel measure $\mu$ on $X$ satisfies the \textit{thin boundary condition} if there exists an $r_0>0$ such that for all $0<r\leq r_0,$ $\mu(B_r(x))=\mu(B_r[x]).$ 
\end{definition}

The thin boundary condition is equivalent to the spheres $\{y\in X\;|\;d(x,y)=r\}$ having $\mu$-measure $0$ for all $r<r_0$ and $x\in X.$ It is also equivalent to the existence of an $r_0>0$ such that the map $r \mapsto \mu(B_r(x))$ is continuous on $[0,r_0]$\footnote{A priori there are at most countably many discontinuities since the map $r\mapsto \mu(B_r(x))$ is non-decreasing.}  for all $x\in X.$
\vspace*{1 in}
\begin{figure}[H]
\centering
\includegraphics[scale=.2]{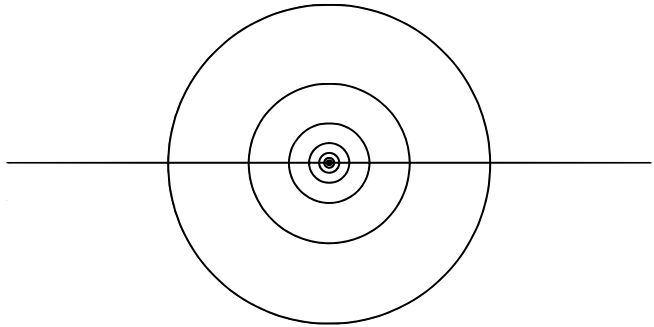}
\caption{Consider the segment $[-1,1]\times \{0\}$ in the plane together with the infinite union of concentric circles of the form $\sqrt{x^2+y^2}=\frac{1}{2^n}$ for $n\in\mathbb{Z}^+$ with the induced metric from the Euclidean plane. The space is then connected, compact, and Ahlfors regular with Ahlfors regular measure the $1$-dimensional Hausdorff measure $\mu$. However, the map $((x,y),r)\mapsto \mu(B_r(x))$ is discontinuous at $((0,0),\frac{1}{2^n})$ for all $n\in \mathbb{Z}^+.$}
\end{figure}
\vspace*{1 in}
The thin boundary condition is sufficient to prove the joint continuity of $\mu(B_r(x))$ in $r$ and $x$ for $r\leq r_0.$ 
It is needed for the proof of the following lemma.
 
 \begin{lemma} \label{contmeas1} Suppose $\mu$ satisfies the thin boundary condition with constant $r_0.$ Then for $r\leq r_0$ and $x\in X,$ $\lim_{y\rightarrow x}\mu(B_r(x)\triangle B_r(y))=0.$ Moreover, the function $(r,x)\mapsto \mu(B_r(x))$ is continuous on $[0,r_0]\times X,$ where $r_0>0$ is the constant from condition (c). In particular, $(r,x)\mapsto \mu(B_r(x))$ is uniformly continuous on $[0,r_0]\times X.$
 \end{lemma}
 \begin{proof} We adopt the notation $B_s(x)=\emptyset$ if $s\leq 0.$ Let $0\leq r\leq r_0\; x\in X.$ Let $\epsilon>0.$ Since $\mu(S_r(x))=0$ by the boundary condition (c); by continuity of measure there exists a $\delta>0$ such that $\mu(B_{r+\delta}(x))-\mu(B_{r-\delta}(x))<\epsilon.$ Suppose $y\in X$ with $d(x,y)<\delta/2$ and $r'\in[0,r]$ with $|r-r'|<\delta/2.$ Then if $z\in B_{r-\delta}(x)$ then $d(z,y)\leq d(z,x)+d(x,y)<r-\delta+\delta/2=r-\delta/2<r'.$ Hence $B_{r-\delta}(x)\subset B_{r}(x)\cap B_{r'}(y).$ Now if $z\in B_{r'}(y)$ then $d(z,x)\leq d(z,y)+d(y,x)<r'+\delta/2<\delta+r.$ So $B_r(x)\cup B_{r'}(y)\subset B_{r+\delta}(x).$ Then \begin{equation*}
 \begin{split}
 &\left|\mu(B_r(x))-\mu(B_{r'}(y))\right|\leq \mu(B_r(x))-\mu(B_r(x)\cap B_{r'}(y))+\mu(B_{r'}(y))-\mu(B_{r'}(y)\cap B_r(x))\\&=\mu(B_r(x)\triangle B_{r'}(y))=\mu((B_r(x)\cap B_{r'}(y)^c) \cup (B_{r'}(y)\cap B_{r}(x)^c))\\&= \mu((B_r(x)\cup B_{r'}(y))\cap (B_r(x)\cap B_{r'}(y))^c)=\mu(B_r(x)\cup B_{r'}(y))-\mu(B_r(x)\cap B_{r'}(y))\\&
 \leq \mu(B_{r+\delta}(x))-\mu(B_{r-\delta}(x))<\epsilon.
 \end{split}
 \end{equation*} Since $[0,r_0]\times X$ is compact as a product of compact spaces, it follows that the map $(r,x)\mapsto \mu(B_r(x))$ is uniformly continuous on $[0,r_0]\times X.$
\end{proof}

\begin{lemma} Suppose $\mu$ is variable Ahlfors regular and $(X,d,\mu)$ is time regular of exponent $\beta.$ Then for any $r_1>0,$ the maps $(r,x)\mapsto r^{\alpha(x)}$ and $(r,x)\mapsto r^{\beta(x)}$ are uniformly continuous on $[0,r_1]\times X.$\end{lemma}
\begin{proof} Let $r_1>0$ and $\psi(r,s)=r^s$ for $r,s\in [0,\infty)\times (0,\infty).$ Then $\psi$ is continuous. However, by Ahlfors regularity and time regularity, $\alpha$ and $\beta$ are positive and continuous. Hence $\psi(\cdot, \alpha(\cdot))$ and $\psi(\cdot,\beta(\cdot))$ are continuous on $[0,r_1]\times X.$ The uniform continuity follows by compactness of $[0,r_1]\times X.$
\end{proof}

The thin boundary condition allows us to show that the approximate exit time functions $\phi_{r,B}$ are continuous on $B.$

\begin{lemma} Suppose $\mu$ satisfies the thin boundary condition. Then there exists an $r_0>$ such that if $0<r<r_0$ and if $f$ if a bounded measurable function, then the map $x\mapsto A_rf(x):=\frac{1}{\mu(B_r(x))}f(y)d\mu(y)$ is continuous. 
\end{lemma}
\begin{proof} By the thin boundary condition we may choose $r_0>0$ such that for some $C>0$ and $0<r<r_0,$ the conclusions of Lemma \ref{contmeas1} are satisfied. Since $f$ is bounded, let $C>0$ so that $|f|\leq C.$ Let $\epsilon>0.$ Since the map $z\mapsto \mu(B_r(z))$ is continuous and positive, by compactness there exists a $c>0$ so that $c\leq \mu(B_r(z))$ for all $z\in X.$ Moreover, by the proof of Lemma \ref{contmeas1} there exists a $\delta>0$ so that if $d(x,y)<\delta$ then $|\mu(B_r(x))-\mu(B_r(y))|<\frac{c\epsilon}{2C}$ and $|\mu(B_r(x)\triangle B_r(y))|<\frac{c\epsilon}{2C}.$ Then for $d(x,y)<\delta,$  \begin{equation*}
\begin{split}
& |A_rf(x)-A_rf(y)| \leq \\&
\left|\left(1-\frac{\mu(B_r(x))}{\mu(B_r(y))}\right)\frac{1}{\mu(B_r(x))}\int_{B_r(x)}|f(w)|d\mu(w)\right| +\left|\frac{1}{\mu(B_r(y))}\int_{B_r(x)\triangle B_r(y)}|f(w)|d\mu(w)\right|\\&\leq \left( \left|\frac{\mu(B_r(y))-\mu(B_r(x))}{\mu(B_r(y))}\right|+\left|\frac{\mu(B_r(x)\triangle B_r(y))}{\mu(B_r(y))}\right|\right)C<\epsilon. \end{split}
\end{equation*}
\end{proof}

\begin{proposition}
\label{exitcont} Suppose $B=B_R[x_0]$ for $x_0\in X$ and $0<R<\diam(X)/2.$ Suppose $(X,d,\mu)$ is variable Ahlfors regular, satisfies the variable time regularity condition, and $\mu$ satisfies the thin boundary condition. Then there exists an $r_0>0$ such that for all $0<r<r_0,$ $\phi_{r,B}$ is continuous on $B.$ 
\end{proposition}
\begin{proof}
Let $r_0>0$ and $C_1>0$ so that for $0<r<r_0,$ the conclusions of Lemma \ref{contmeas1} are satisfied and $\phi_{r,B}^+\leq C_2.$ For $x\in B,$ since $\mathscr{L}_r\phi_{r,B}(x)=1,$ $\phi_{r,B}(x)=\tau_r(x)-P_r\phi_{r,B}(x).$ Since the map $x\mapsto \tau_r(x)$ is continuous by the time regularity condition, we need only prove continuity of the map $x\mapsto P_r\phi_{r,B}(x)$ on $B.$ For $x\in B,$ using the notation $A_r$ for the averaging operator of the previous lemma, $P^B_r\phi_{r,B}(x)=q_r(x)^{-1}A_r\phi_{r,B}(x)+q_r(x)^{-1}d_r(x)A_r(\frac{\phi_{r,B}}{d_r})(x).$ However, since $x\mapsto d_r(x)$ is continuous and positive and $X$ is compact, $\frac{1}{d_r}$ is bounded. Hence $\phi_{r,B}$ and $\frac{\phi_{r,B}}{d_r}$ are bounded. Then by the previous lemma, $q_r$, $A_r\phi_{r,B},$ and $A_r(\frac{\phi_{r,B}}{d_r}$ are continuous. Since $q_r\geq 1,$ also $q_r^{-1}$ is continuous. The result then follows since $P^B_r\phi_{r,B}$ is a sum of a finite products of continuous functions. \end{proof}

\chapter{Convergence of Approximating Forms}
In this chapter we apply the theory of $\Gamma\mhyphen$ convergence outlined in section 2.4 in an attempt to construct a non-trivial limiting Dirichlet form under the conditions of variable Ahlfors regularity and variable time regularity. 

\subsubsection{Kumagai-Sturm Construction}
In \cite{sturm}, the authors consider a locally compact, separable metric space $(X,d)$ with a non-negative Radon measure $\mu$ with the property that $\mu(B_r(x))>0$ for all $x\in X,r>0.$ For each $r>0$ they fix a measurable non-negative function $k_r:X\times X\rightarrow \mathbb{R}$ with the property that $k_r$ vanishes outside the $r$-diagonal $\{(x,y)\in X\times X\;|\;d(x,y)\leq r\}$ and satisfies $\sup_{x\in K} \int_K(k_r(x,y)+k_r(y,x))d\mu(y)<\infty$ for each $K$ compact.

Then a Dirichlet form $\mathscr{E}_r$ is defined by $\mathscr{E}_r(u)=\int_X \int_X |
u(x)-u(y)|^2 k_r(x,y)d\mu(x)d\mu(y)$ with $\mathscr{D}(\mathscr{E}_r):=\{u\in L^2(\mu)\;|\;\mathscr{E}_r(u)<\infty\}.$

They fix a sequence $(r_n)$ of positive numbers decreasing to $0$ and set \[\mathscr{E}^*(u):=\limsup_{n\rightarrow \infty} \mathscr{E}_{r_n}(u)\] and \[\mathscr{F}^*:=\{u\in C_0(X)\;|\;\mathscr{E}^*(u)<\infty\}.\]

They then use the sequential compactness properties of $\Gamma$-limits to prove the following.

\begin{theorem} (Kumagai-Sturm \cite{sturm})
Suppose $\mathscr{F}^*$ is dense in $C_0(X)$. Then for a suitable subsequence  
${r_n}_k$ of $r_n$ the following $\Gamma-$limit exists for each $u\in L^2(\mu)$
\[ \mathscr{E}_0(u):=\lim_{\alpha \rightarrow 0^+} \limsup_k \inf_{v\in L^2,\|u-v\|\leq 
\alpha} \mathscr{E}_{{r_n}_k}(u).\] Moreover, $\mathscr{E}_0$ is dominated by $\mathscr{E}^*$ and can be extended to a regular Dirichlet form $(\mathscr{E},
\mathscr{F})$ on $L^2(\mu)$ with core $\mathscr{F}^*.$ Hence there exists a $\mu$-reversible strong Markov process associated with $(\mathscr{E},\mathscr{F}).$
\end{theorem}

The full theorem in \cite{sturm} also gives another criterion that implies strong locality of the limit. This condition is connected to the ``measure contraction property" on length spaces and curvature type estimates (see \cite{sturm3} \cite{sturm2}).

Moreover, in \cite{sturm} they make a remark (Remark 1 (b)) that the condition in the definition of $\mathscr{F}^*$ may be replaced by 
the weaker condition $\Gamma$-$\limsup_k \mathscr{E}_{{r_n}_k}(u)<\infty.$ However, in \cite{sturm}, they only consider time scaling functions of the form $h(r)$ independent of $x.$ In such a case the generator of the approximating Dirichlet forms with respect to $\mu$ is computed, but with a non-zero probability of the walker to stay at $x.$ In our case, the use of partial symmetrization instead of a full symmetrization allows the elimination of the discrete portion of the transition kernel. The key observation that allows for the use of partial symmetrization (that is symmetrization with respect to $\mu_r$ instead of $\mu)$ is that variable Ahlfors regularity and time regularity imply, due to the resulting log-H{\"o}lder continuity of the exponents, an upper bound for $\frac{d_r(x)}{d_r(y)}\chi_{B_r(x)}(y)$ independent of $r,$ provided $r$ is small enough. This in turn implies that $\mu_r$ and $\mu$ are strongly equivalent with constants independent of $r$ for $r$ small enough. Moreover, since $\beta$ is variable, in our approach the natural time scaling function depends both on the scale $r$ and the position $x.$

For our applications, where $X$ is compact with local measure $\mu$ and has variable time regularity with exponent $\beta,$ a candidate for $k_r$ is 
\[k_r(x,y):=\frac{1}{\mu(B_r(x))r^{\beta(x)}}\chi_{B_r(x)}(y).\]
This leads to the approximate strongly Markovian form \[\mathscr{E}_r(f):=\int_X \frac{1}{\mu(B_r(x))}\int_{B_r(x)}\left(\frac{u(x)-u(y)}{r^{\beta(x)/2}}\right)^2d\mu(y)d\mu(x),\] for $f\in L^2(X,\mu).$ However, as we have seen, this form is just $\langle f, \mathscr{L}_r f\rangle_{\mu_r}.$

\section{Construction of a Limiting Dirichlet Form}
Suppose $(X,d)$ is a connected, compact measure space such that \begin{equation*}
\begin{split} \mbox{(a)}&\; X \mbox{\;has a variable Ahlfors regular measure\;}\mu\;\mbox{of exponent\;}\alpha;\mbox{\;and}\\
\mbox{(b)}&\; X \mbox{\;is variable time regular of exponent\;}\beta;\mbox{\;and}\\
\mbox{(c)}&\; X \mbox{\;satisifes the time comparability condition.\;}
\end{split}
\end{equation*}

For $f\in L^2(X,\mu),$ let \[\mathscr{E}_r(f)=\langle f,\mathscr{L}_r f\rangle_{\mu_r}=\int_{X}\frac{1}{\mu(B_r(x))r^{\beta(x)}}\int_{B_r(x)}(f(y)-f(x))^2d\mu(y)d\mu(x).\] Then for each $r>0,$ $\mathscr{E}_r$ defines a strongly Markovian form, as defined in section 2.4.

For $B$ a ball, we let $\mathscr{L}_r^B=\chi_B\mathscr{L}_r\chi_B$. Then we let $\mathscr{E}_r^B(f)=\langle f, \mathscr{L}^B_r f\rangle_{\mu_r}=\mathscr{E}_r(\chi_Bf)$ for $f\in L^2(X,\mu).$ Note also we may consider $\mathscr{E}^B_r$ as a form on $L^2(B,\mu).$

Fix an open ball $B=B_R(x_0)$ with $0<R<\diam(X)/2.$ Since the map $r\mapsto \mu(B_r(x_0))$ is non-decreasing, it has at most countable many jump discontinuities. By perturbing $R$ slightly if needed we may assume $\mu(B_R(x_0))=\mu(B_R[x_0])$.

\begin{theorem} Suppose the variable time regularity and time comparability conditions hold. Let $(\epsilon_n)_{n=1}^\infty$ be a sequence of positive numbers decreasing to $0.$ Suppose $\mathcal{E}_r=\mathscr{E}^B_r$ for all $r.$ Then there exists a subsequence $(\epsilon_{n_k})_{k=1}^\infty$ such that the sequence $(\mathcal{E}_{\epsilon_{n_k}})_{k=1}^\infty$ Mosco converges, that is $\Gamma$-converges in both the norm and weak topologies of $L^2(B,\mu)$, to a Dirichlet form $\mathcal{E}.$ Then $\mathcal{E}$ is non-trivial and contains an algebra of functions separating points.
\end{theorem}

\begin{proof}
Note since $\mu(B_R(x_0))=\mu(B_R[x_0]),$ we have $\chi_{B}=\chi_{B_R[x_0]}.$ Hence $\mathcal{L}^B=\mathcal{L}^{B_R[x_0]}.$ Then by the Corollary \ref{maincor} of the Faber-Krahn inequality, Theorem \ref{main1}, we have that there exists a $C>0$ and $r_0>0$ so that $\mathcal{E}_r(f)\geq C\|f\|^2_{L^2(B,\mu)}$ for all $f\in L^2(B,\mu)$ and $0<r<r_0.$

Let $(\epsilon_n)_{n=1}^\infty$ be a sequence of positive numbers decreasing to $0.$ By starting a sequence at a large enough index we may assume each $\epsilon_n\in (0,r_0).$ Then since the forms $(\mathcal{E}_{\epsilon_n})_{n=1}^\infty$ are uniformly coercive, by Proposition \ref{mosco} there exists a subsequence $(s_n)_{n=1}^\infty$ of $(\epsilon_n)_{n=1}^\infty$ with $\mathcal{E}_{n_k}$ $\Gamma$-converging as $k\rightarrow \infty$ to a Dirichlet form $\mathcal{E}$ with respect to both the norm and weak topologies of $L^2(B,\mu).$ Moreover, by Proposition \ref{mosco} we have $\mathcal{E}(f)\geq C\|f\|^2_{L^2(B,\mu)}$ for all $f\in L^2(B,\mu).$

Since $X$ is compact, it is separable. Let $D$ be a countable dense set. Then let $\mathscr{C}=\{B_q[y]\;|\;y\in D,\;q\in\mathbb{Q}^+\}.$ Let $\mathscr{B}=\{C\in \mathscr{C}\;|\;C\subset B\}.$ Say $\mathscr{B}=(B_k)_{k=1}^\infty$ with $B_k=B_{r_k}[x_k]\subset B$. Then $\cup_{k=1}^\infty B_{r_k/4}(x_k)=B$. Indeed, for $x\in B$, let $\delta=\dist(x,B^c)>0.$ Then since $\mathcal{D}$ is dense, let $y\in \mathcal{D}\cap B_{\delta/6}(x)$. Then $\dist(y,B^c)\geq 5\delta/6.$ Then let $q$ rational with $\delta/6<q<5\delta/24.$ Then $B_{4q}[y]\subset B,$ and by construction $x\in B_{q}(y), B_{4q}[y]\in \mathscr{B}.$  Moreover, if $x,y\in B$ with $x\neq y$ then there exist $j, k$ with $x\in B_j,\; y\in B_k$ and $B_j\cap B_k=\emptyset.$

Then for $\epsilon>0$ let $\phi_{k,\epsilon}(x)$ be the mean exit time function on $B_k$ starting at $x.$ The time scaling assumption $\limsup_{r\rightarrow 0^+} \phi_{k,r}^+ \leq Cr_k^\beta$ implies that the functions $\phi_{k,\epsilon}$ are uniformly bounded above by the same constant for $\epsilon$ small enough. Hence, by Alaoglu's Theorem, for each $k$ and any sequence $(t_n)$ of positive numbers decreasing to $0$, there exists a subsequence of the $\phi_{k, t_n}$ converging weakly in $L^2(X,\mu)$ to some limit.

For $k=1$ choose a subsequence $(t_{1,n})_{n=1}^\infty$ of $(s_n)_{n=1}^\infty$ such that $\phi_{1,t_{1,m}} \;
\overrightarrow{_{m\rightarrow \infty}}\; \phi_1$ weakly in $L^2(X,\mu)$. Having chosen $t_1,...,t_n$ with $(t_{j+1,n})
_{n=1}^\infty$ a subsequence of $(t_{j,n})_{n=1}^\infty$ with $\phi_{j,t_{j,m}}\overrightarrow{_{m\rightarrow \infty}}\; \phi_j$ weakly in $L^2(X,\mu)$ for $j\leq n,$ since $\phi_{n+1,r}$ is uniformly bounded above for $r$ small, we 
may choose a subsequence $(t_{n,m})_{m=1}^\infty$ such that $\phi_{n+1,t_{n+1,m}} \;
\overrightarrow{_{m\rightarrow \infty}} \;\phi_{n+1}$ weakly in $L^2(X,\mu)$. Hence for each $j\in\mathbb{Z}^+$ we have a function $\phi_j\in L^2(X,\mu)$ and decreasing sequence $(t_{j,n})_{n=1}^\infty$ of positive numbers decreasing to $0$ with the property that $(t_{j+1,n})_{n=1}^\infty$ is a subsequence of $(t_{j,n})_{n=1}^\infty$ and $\phi_{j,t_{j,m}}\; \overrightarrow{_{m\rightarrow \infty}}\; \phi_j$ weakly in $L^2(X,\mu).$ By the diagonal argument, $\phi_{j,t_{n,n}}\; \overrightarrow{_{n\rightarrow \infty}} \;\phi_j$ weakly in $L^2(X,\mu)$ for all $j\in \mathbb{Z}^+.$ Let $t_n=t_{n,n}$ for $n\in \mathbb{Z}^+.$

\begin{lemma} \label{weak1} Suppose $(f_n)_{n=1}^\infty$ is a sequence of non-negative\footnote{The ordering $\leq$ on $L^2(X,\mu)$ is defined by $f\leq g$ if and only if $\{x\in X\;|\;f(x)>g(x)\}$ is a $\mu$-null set.} functions in $L^2(X,\mu)$ such that $\supp(f_n)\subset A$ for all $n,$ where $A$ is a closed set. If $f_n\rightarrow f$ weakly then $f$ is non-negative and $\supp f\subset A.$
\end{lemma}
\begin{proof} We first show $f\geq 0.$ For $k\in \mathbb{N}$, let $N_k=\{x\in X\;|\;f(x)<-1/k\}.$ Then $0\leq \langle f_n,\chi_{A_k}\rangle \overrightarrow{_{n\rightarrow \infty}}\;\langle f, \chi_{A_k}\rangle <\frac{-1}{k}\mu(A_k).$ Hence $\mu(A_k)=0$ for all $k.$ Thus $\mu(\{x\;|\;f(x)<0\})=0.$ Hence $f\geq 0$ in $L^2(X,\mu)$.
 
 Note if $g\in L^2(X,\mu),$ $\langle f_n\chi_{A^c},g\rangle = \langle f_n,\chi_{A^c}g\rangle\rightarrow \langle f, \chi_{A^c}g\rangle = \langle \chi_{A^c}f,g\rangle.$ Hence $\chi_{A^c}f_n\rightarrow \chi_{A^c}f$ weakly. But each $\chi_{A^c}f_n=0$ in $L^2(X,\mu).$ So the previous result implies $f\chi_{A^c}$ is both $\geq 0$ and $\leq 0$ in $L^2(X,\mu).$ Hence it is $0.$ So $\supp(f)\subset A.$
\end{proof}

\begin{lemma} \label{meascomp} There exists a constant $C>0$ such that for all $0<r<\frac{1}{2}$ with $r<\diam(X)/2,$ if $f$ is a non-negative measurable function on $X$, then $\int f d\mu_r \leq C\int f d\mu.$ 
\end{lemma}
\begin{proof}
Since $\mu$ is variable Ahlfors regular, there exists a constant $C_1>0$ such that for all $x$ and all $0<r<\diam(X)/2,$ $\frac{1}{C_1}r^{\alpha(x)}\leq \mu(B_r(x))\leq C_1r^{\alpha(x)}.$ By variable Ahlfors regularity and time regularity, both $\alpha$ and $\beta$ are log-H{\"o}lder continuous. By log-H{\"o}lder continuity, there exists a $C_2>0$ so that for $x,y\in X$ with $0<d(x,y)<\frac{1}{2},$ $|\alpha(x)-\alpha(y)|\leq \frac{-C_2}{\log(d(x,y))}$ and $|\beta(x)-\beta(y)|\leq \frac{-C_2}{\log(d(x,y))}$. Suppose $0<r<\frac{1}{2}$ and $r<\diam(X)/2.$ If $0<d(x,y)<r,$ then \[r^{\alpha(y)-\alpha(x)}\geq r^{|\alpha(y)-\alpha(x)|}\geq r^{-C_2/\log(d(x,y))}.\] Hence \[\frac{r^{\alpha(x)}}{r^{\alpha(y)}}=r^{\alpha(x)-\alpha(y)}\leq r^{C_2/\log(d(x,y))}\leq r^{C_2/\log(r)}=r^{C_2}.\] Similarly $\frac{r^{\beta(x)}}{r^{\beta(y)}}\leq e^{C_2}$ for $0<d(x,y)<r.$ Hence for $x\in X,$ \begin{equation*}
\begin{split} q_r(x)&=\frac{1}{\mu(B_r(x))}\int_{B_r(x)}\left(1+\frac{\mu(B_r(x))r^{\beta(x)}}{\mu(B_r(y))r^{\beta(y)}}\right)d\mu(y)\\&\leq \frac{1}{\mu(B_r(x))}\int_{B_r(x)}\left(1+C_1^2\frac{r^{\alpha(x)}r^{\beta(x)}}{r^{\alpha(y)}r^{\beta(y)}}\right)d\mu(y)\\
&\leq \frac{1}{\mu(B_r(x))}\int_{B_r(x)}\left(1+C_1^2e^{2C_2}\right)d\mu(y)=1+C_1^2e^{2C_2}:=C.
\end{split}
\end{equation*} Then if $f$ is a non-negative measurable function, \[\int f(x)d\mu_r(x)=\int q_r(x)f(x)d\mu(x)\leq C\int f(x)d\mu(x)\] for all $0<r<\frac{1}{2}$ and $r<\frac{\diam(X)}{2}.$ 
\end{proof}
Since $(t_n)_{n=1}^\infty$ is a subsequence of $(s_n)_{n=1}^\infty,$ by Lemma \ref{subsequence}, $\mathcal{E}=\Glim\mathcal{E}_{t_n}$
in both the norm and weak topologies of $L^2(B,\mu).$ Let $\mathcal{D}(\mathcal{E}):=\{f\in L^2(B,\mu)\;|\;\mathcal{E}(f)<\infty\;\}.$  Since $\phi_{k,t_n}$ converges weakly to $\phi_k,$ we have by Proposition \ref{gammachar} that $\mathcal{E}(\phi_k)\leq \liminf_{n\rightarrow \infty} \mathcal{E}(\phi_{k,t_n}).$ 
Note for each $k$, the exit time functions $\phi_{r,B_k}$ are non-negative, supported in $B_k$ and satisfy $\mathscr{L}^B_r\phi_{r,B_k}=\chi_{B_k}$. By Lemma \ref{weak1}, also $\phi_k$ is non-negative and supported in $B_k.$ The exit time assumption $\limsup_{r\rightarrow 0^+} \phi_{k,r}^+ \leq Cr_k^\beta$ also implies there exists a constant $C_1>0$ and an $N>0$ so that for all $n\geq N,$ $\phi_{k,t_n}\leq C_1$ and $\phi_{k}\leq C_1.$ We may assume $N$ is large enough so that $0<t_n<1/2$ for $n\geq N.$ Then, if $n\geq N,$ by Lemma \ref{meascomp}, there exists a $C_2>0$ such that $\int \phi_{k,t_n}d\mu_{t_n}\leq C_2\int \phi_k d\mu\leq C_1C_2\mu(X).$ Hence for $n\geq N$
\[\langle \phi_{k,t_n}, \mathscr{L}^B_{t_n}\phi_{k,t_n}\rangle_{\mu_{t_n}}=\int_B \phi_{k,t_n}(x)d\mu(x)\leq C_1C_2\mu(X)<\infty.\] It follows that \[\mathcal{E}(\phi_k)\leq \liminf_{n\rightarrow \infty} \mathcal{E}_{t_n}(\phi_{k,t_n})<\infty.\]Thus each $\phi_{k}\in \mathcal{D}(\mathcal{E}).$

The time comparison condition implies that for $r$ small enough, $\phi_{r,B_{r_k}}(x)\geq c_1>0$ for all $x\in B_{r_k/4}(x_k)$ and $r$ small enough. In particular, since $\cup_{k}B_{r_k/4}(x_k)=X,$ the time comparison condition implies that the $\phi_k$ separate points. It is known that domains of Dirichlet forms on locally compact spaces with a Borel measure of full support are always closed under finite products of $L^\infty$ functions. See, e.g. \cite{fukushima} Theorem 1.4.2, for a proof. Hence $\mathcal{D}(\mathcal{E})$ contains an algebra of functions separating points. Also, Lemma \ref{weak1} implies that $\phi_j\geq c_1$ on $B_{r_j/4}(x_j)$. So $\|\phi_j\|_{L^2(B,\mu)}^2>0.$ Hence, since $\mathcal{E}(\phi_j)\geq C\|\phi_j\|_{L^2(B,\mu)}^2,$ $\mathcal{E}(\phi_j)>0$ for all $j.$ Therefore $\mathcal{E}$ is non-trivial. 
\end{proof}

\begin{remark} It is not known wheter the form is regular or if the domain is dense in $L^2(X,\mu).$ In \cite{hinz} the authors provide a method to transfer a Dirichlet form to a regular Dirichlet form on an associated space via a Gelfand transform.  Moreover, if the form is regular, one may take the ``diffusion part" in the Beurling-Deny representation formula to obtain a regular, strongly-local form.
\end{remark}

\section{Probabilistic Convergence of Approximating Walks}

In this section we prove tightness of the approximate continuous time walks $(\bm{x}(r))_{r>0}$ and show any subsequential weak limit has continuous paths almost surely. We begin by recalling essential facts about weak convergence and tightness.

Suppose $S$ is a Polish space. Let $P(S)$ be the space of all Borel probability measures on $S$. Then $P(S)$ may be given a metric topology such that it is complete and separable (See \cite{stroock} for a proof.). Sequential convergence with respect to this metric topology is called weak convergence, and is characterized as follows. If $(\mu_n)_{n=1}^\infty$ is a sequence of Borel probability measures on $S$ and $\mu$ is another Borel probability measure on $S,$ then the sequence $\mu_n$ converges weakly to $\mu$ if for every bounded, continuous function $f$ on $S,$ $\int f d\mu_n\rightarrow \int f d\mu.$

Suppose $\mu_n\rightarrow \mu$ weakly. Let $U\subset S$ be open. Let $\epsilon>0.$ By Proposition \ref{reg}, there exists a $K$ subset $U$ with $K$ closed and $\mu(U)\leq \mu(K)+\epsilon.$ By the Tietze extension theorem, there exists a continuous function $f$ with $0\leq f\leq 1$ such that $f=1$ on $K$ and $f=0$ on $U^c$. Then $\mu(U)-\epsilon\leq \mu(K)\leq \int fd\mu =\liminf f_nd\mu \leq \liminf \mu_n(U).$ Since this holds for all $\epsilon>0,$ $\mu(U)\leq \liminf \mu_n(U).$

\begin{definition}
A subset $\Sigma\subset P(S)$ is called \textit{tight} if for all $\epsilon>0,$ there exists a compact $K\subset S$ with $\mu(K)>1-\epsilon$ for all $\mu\in P(S).$ \end{definition}

The following theorem, known as Prohorov's theorem, characterizes relatively compact subsets of $P(S).$ For a proof see \cite{billingsley}. 
\begin{proposition} (Prohorov) A set $\Sigma\subset P(S)$ is relatively compact in $P(S)$ if and only if it is tight. \end{proposition} 

Let $X$ be a compact metric space. Then let $\mathcal{D}$ be the space of c{\`a}dl{\`a}g paths $\gamma:\mathbb{R}_+\rightarrow X$. 
For an interval of the form $[0,t],$ and $\delta>0,$ let $\mathscr{P}_\delta(t)$ be the collection of all $(t_i)_{i=0}^n$ for $n\in \mathbb{N}$ with $0=t_0<t_1<...t_{n-1}<t\leq t_n$ with $\min_{0\leq i<n}(t_{i+1}-t_i)> \delta.$

The following proposition may be found in \cite{kallenberg}. See also \cite{ethier}.
\begin{proposition} There exists a topology on $\mathcal{D},$ called the Skorohod J1 topology, such that $\mathcal{D}$ is a complete, separable metric space and the Borel topology of $\mathcal{D}$ is generated by the projection maps $\pi_t$ for $t\geq 0.$ Moreover, if $I$ is a dense subset of $\mathbb{R}_+$ then a set $A\subset \mathcal{D}$ is relatively compact if and only if\footnote{We do not need the condition that $\pi_t^{-1}(A)$ is relatively compact in $X$ for all $t\in I$ since it holds automatically as $X$ is compact.} \[\lim_{h\rightarrow 0^+}\sup_{\gamma\in A} \tilde{\omega}(\gamma,t,h)=0\] for all $t\in I,\;t>0,$ where \[\tilde{\omega}(\gamma,t,h):=\inf\{\max_{0\leq i< n}\sup_{r,s\in [t_{i},t_{i+1})}d(\gamma(r),\gamma(s))\;|\;(t_i)_{i=1}^n\in \mathscr{P}_h(t)\}.\]\end{proposition}

Let $\mathcal{C}$ be the space of continuous $X$-valued paths on $[0,\infty)$ with the topology of uniform convergence. Then $\mathcal{C}$ is a complete, separable metric space. For $\gamma\in \mathcal{C},$ $h>0$, let $\omega(\gamma,h)=\sup\{d(\gamma(s),\gamma(t))\;|\;|s-t|\leq h\}.$ The corresponding relative compactness theorem for $\mathcal{C}$ involving the uniform modulus of continuity $\omega$ on $\mathcal{C},$ which follows directly from the usual Arzel\'a-Ascoli theorem, is as follows. A set $A\subset \mathcal{C}$ is relatively compact in $\mathcal{C}$ if and only if $\limsup_{h\rightarrow 0^+}\sup_{\gamma\in A} \omega(\gamma,h)=0$.\footnote{Again, the ``pointwise boundedness" assumption is trivially satisfied since $X$ has finite diameter.}

Suppose $(\mathbb{P}_n)$ is a sequence of probability measures on $\mathcal{D}.$ We then have, see \cite{billingsley}, that the sequence is tight if for all $T>0,$ all $\epsilon>0$, and all $\eta>0,$ there exists a $0<\delta<1$ and an integer $N>0$ such that 
$\mathbb{P}_n\{\gamma\in\mathcal{D}\;|\;\tilde{\omega}(\gamma,T,\delta)\geq \epsilon\}\leq \eta$ for all $n\geq N.$ Indeed, suppose this condition holds. Note first that we may suppose it holds for $N=1$ by sufficiently decreasing $\delta$ if needed. Let $\epsilon>0.$ For each $k,$ choose $\delta_{k}<\delta_{k-1}$ with $\mathbb{P}_n\{\gamma\in\mathcal{D}\;|\;\tilde{\omega}(\gamma,k,\delta_{k})\geq \frac{1}{k}\}\leq \frac{\epsilon}{2^k}$ for all $n\in \mathbb{Z}^+.$ Then let $A_k=\{\gamma\in\mathcal{D}\;|\;\tilde{\omega}(\gamma,k,\delta_{k})\geq \frac{1}{k}\}$ for each $k\in \mathbb{Z}^+.$ Let $K=\overline{\cap_{k=1}^\infty A_k^c}.$ Then $\mathbb{P}_n(K^c)\leq \sum_{k=1}^\infty \mathbb{P}_n(A_k)\leq \epsilon$ for all $n.$ Moreover, if $\gamma\in \cap_{k=1}^\infty A_k^c$ and $T>0$, then for $k\geq T,$ $0<h<\delta_k,$ $\tilde{\omega}(\gamma,T,h)\leq\tilde{\omega}(\gamma,T,\delta_k)\leq \tilde{\omega}(\gamma,k,\delta_k)<\frac{1}{k}.$ Hence $\sup_{\gamma\in K} \tilde{\omega}(\gamma,T, h)\leq \frac{1}{k}$ for for $0<h<\delta_k$ and $k\geq T.$ It follows $\lim_{h\rightarrow 0^+}\sup_{\gamma \in \cap_k A_k^c} \tilde{\omega}(\gamma, T, h)=0$ for all $T>0.$ Thus $K$ is compact. Hence the sequence $(\mathbb{P}_n)$ is tight.

A version of the following proposition may be found in Theorem 8.3 of \cite{billingsley}. 
\begin{proposition}  Suppose that  $(\mathbb{P}_n)_{n=1}^\infty$ is a sequence of probability measures on $\mathcal{D}$ such that for each $\epsilon>0,\;\eta>0,$ there exists a $0<\delta<1$ and an integer $N>0$ such that \[\frac{1}{\delta}\mathbb{P}_n(\{\gamma\in\mathcal{D}\;|\;\sup_{t\leq s\leq t+\delta}d(\gamma(s),\gamma(t))\geq \epsilon\})\leq \eta\] for all $n\geq N$ and $t\geq 0.$ Then the sequence $(\mathbb{P}_n)$ is tight. Moreover, if $\mathbb{P}_{n_k}\rightarrow \mathbb{P}$ is a subsequential weak limit, then $\mathbb{P}(\mathcal{C})=1.$ 
\end{proposition}
\begin{proof} Suppose the stated condition holds. Let $\epsilon$, $\eta>0,$ and $T>0.$ Choose $\delta\in (0,1)$ such that 
$\frac{1}{\delta}\mathbb{P}_n\{\gamma\in\mathcal{D}\;|\;\sup_{t\leq s\leq t+\delta}d(\gamma(s),\gamma(t))\geq \epsilon\}\leq \frac{\eta}{T}.$ Let \[A_{3\epsilon,\delta, T}:=\{\gamma\in \mathcal{D}\;|\;\sup_{t\leq T}\sup_{t\leq s\leq t+\delta} d(\gamma(s), \gamma(t))> 3\epsilon\}.\] For $t>0$ let $B_{\epsilon, \delta, t}:=\{\gamma\in \mathcal{D}\;|\;\sup_{t\leq s\leq t+\delta} d(\gamma(s),\gamma(t))\geq \epsilon\}.$ Let $\gamma\in \mathcal{D}.$ Suppose for all $j\leq k$ an integer, with $k$ an integer satisfying $k\delta\leq T< k(\delta+1),$ that $\gamma\notin B_{\epsilon, \delta, j\delta}.$ Then if $t\leq T$ and $t\leq s\leq t+\delta,$ then either $s,t$ are both in the same interval of the form $[j\delta, (j+1)\delta]$ or $t\in [(j-1)\delta, j\delta],$ $s\in [j\delta, (j+1)\delta]$. In the first case $d(\gamma(s),\gamma(t))\leq d(\gamma(s),\gamma(j\delta))+\d(\gamma(j\delta), \gamma(t))<2\epsilon$. In the second case $d(\gamma(s),\gamma(t))\leq d(\gamma(s),\gamma(j\delta))+\d(\gamma(j\delta), \gamma((j-1)\delta))+d(\gamma((j-1)\delta), \gamma(t))<3\epsilon.$ It follows that $\gamma\notin A_{3\epsilon, \delta, T}.$ Therefore $A_{3\epsilon, \delta, T}\subset \cup_{j=0}^k B_{\epsilon, \delta, j\delta}.$ Hence \[\mathbb{P}_n(A_{3\epsilon, \delta, T})\leq \frac{k\eta{\delta}}{T}\leq\eta\] for all $n\in \mathbb{Z}^+.$ Now suppose $\tilde{\omega}(\gamma, T, \delta/2)\geq 4\epsilon.$ Then there exists a $0\leq j\leq k,$ where $k\delta\leq T<(k+1)\delta$ and $\sup_{s,t\in [j\delta, (j+1)\delta)}d(\gamma(s),\gamma(t))\geq 3\epsilon.$ Hence $\gamma\in A_{\epsilon, T}.$ It follows that \[\mathbb{P}_n(\{\gamma\in \mathcal{D}\;|\; \tilde{\omega}(\gamma, T,\delta/2)\geq 4\epsilon\})\leq \eta\] for all $n\in \mathbb{Z}^+.$

It follows that the sequence is tight. By Prohorov's theorem, let $(\mathbb{P}_{n_k})$ be a subsequence converging weakly to some limit $\mathbb{P}$ on $\mathcal{D}.$ Fix $T>0.$ 
Following \cite{ethier}, let $\Lambda$ denote the set of strictly increasing, Lipschitz continuous functions $\lambda,$ mapping $[0,\infty)$ onto $[0,\infty)$ such that $\|\log(\lambda')\|_{L^\infty(\mathbb{R}_+)}<\infty.$ Then we have, see \cite{ethier}, a sequence $(\gamma_n)$ in $\mathcal{D}$ converges with respect to the J1 topology to $\gamma\in \mathcal{D}$ if and only if there exists a collection $(\lambda_n)_{n=1}^\infty\subset \Lambda$ such that \[\lim_{n\rightarrow \infty} \|\log(\lambda_n')\|_{L^\infty(\mathbb{R}_+)}=0\] and \[\lim_{n\rightarrow \infty} \sup_{0\leq t\leq T} d(\gamma_n(\lambda_n(t)),\gamma(t))=0\] for all $T>0.$ Let $T>0.$ Suppose $\gamma\in A_{\epsilon, \delta/2, T}.$ Suppose for some sequence $(\gamma_n)$ in $\mathcal{D}$ converging with respect to the J1 topology to $\gamma,$ for all $N$ there exists an $n\geq N$ such that $\gamma_n\in A_{\epsilon, \delta, T}^c=\{\gamma'\in \mathcal{D}\;|\;\sup_{0\leq t\leq T}\sup_{t\leq s\leq t+\delta} d(\gamma'(s),\gamma'(t))<\epsilon\}.$ Choose $(\lambda_n)\subset \lambda$ with $\lim_{n\rightarrow \infty} \|\log(\lambda_n')\|_{\infty}=0$ and $\lim_{n\rightarrow \infty} \sup_{0\leq t\leq T} d(\gamma_n(\lambda_n(t)),\gamma(t))=0$ for all $t\leq T.$ Choose $s,t$ with $t\leq T,$ $|s-t|\leq\delta/2$ and $d(\gamma(s),\gamma(t))\geq \frac{3\epsilon}{2}$.  Then 
\[\frac{3\epsilon}{2}\leq d(\gamma(s),\gamma(t))\leq d(\gamma(t),\gamma_n(\lambda_n(t)))+d(\gamma_n(\lambda_n(t)),\gamma_n(\lambda_n(s)))+d(\gamma_n(\lambda_n(s)),\gamma(s)).\] Choose $N$ and $n\geq N$ such that $\sup_{0\leq u\leq T+1} d(\gamma_n(\lambda_n(u)),\gamma(u))<\frac{\epsilon}{4},$ $\|\log(\lambda_n')\|_{\infty}<\log(2),$ and $\omega(\gamma_n,\delta)<\epsilon.$ Then for $u$ between $s$ and $t,$ $|\log(\lambda_n'(u))|<\log(2).$ Hence $0<\lambda_n'(u)<2.$ By the mean value theorem, $|\lambda_n(s)-\lambda_n(t)|\leq 2|s-t|\leq \delta.$ However, $\omega(\gamma_n,\delta)<\epsilon.$ It follows that $d(\gamma_n(\lambda_n(t)),\gamma_n(\lambda_n(s)))<\epsilon.$ Thus $\frac{3\epsilon}{2}\leq d(\gamma(s),\gamma(t))<\frac{\epsilon}{2}+\epsilon=\frac{3\epsilon}{2},$ a contradiction. It follows that $\gamma$ is in the interior of $A_{\epsilon, \delta, T}$. Hence, by weak convergence,
\[\mathbb{P}(A_{\epsilon, \delta/2, T})\leq \liminf_{n\rightarrow \infty} \mathbb{P}_n(A_{\epsilon, \delta, T}).\] Then for $k\in \mathbb{Z}^+,$ choose $\delta_k\in (0,1)$ such that $\mathbb{P}_n(A_{\frac{1}{k},\delta_k, T})\leq \frac{1}{k}.$ Let $A_T:=\cup_{n=1}^\infty \cap_{k\geq n} A_{\frac{1}{k},\delta_k, T}.$ Then $\mathbb{P}(A_T)=0$ for all $T>0.$ Let $A=\cup_{j=1}^\infty A_j.$ Then we have $\mathbb{P}(A)=0.$ Suppose $\gamma\in \mathcal{D}\setminus \mathcal{C}.$ Then let $t$ be a discontinuity of $\gamma$ and $j\geq t+1.$ Then there exists a $k$ such that $\sup_{s,\;|s-t|<\delta_k}d(\gamma(s),\gamma(t))\geq \frac{1}{k}.$ Hence $\gamma \in A.$ Thus $\mathbb{P}(\mathcal{D}\setminus \mathcal{C})=0.$
\end{proof}

We then apply this tightness criterion to the continuous time pure jump Markov processes $(\bm{x}(r))_{r>0}.$ Let $r_k$ be a sequence of positive numbers decreasing to $0.$ Let $\mathbb{P}^{x}_k$ be the measure on $\mathcal{D}$ induced by $\bm{x}(r_k)$ with $\bm{x}(r_k)(0)=x.$ Fix $x_0\in X.$ Note then that $\{\gamma\in \mathcal{D}\;|\;\sup_{t\leq s\leq t+\delta} d(\gamma(s),\gamma(t))\geq \epsilon\}\subset \{\gamma\in\mathcal{D}\;|\;\inf\{s\;|\;\gamma(s+t)\notin B_{\epsilon/2}(\gamma(t))\}\leq \delta\}.$ Hence by the strong Markov property, 
\[\mathbb{P}^{x_0}_k(\sup_{t\leq s\leq t+\delta}d(\gamma(s),\gamma(t))\geq \epsilon)\leq \sup_{x\in X}\mathbb{P}^x_k(\tau_{B_{\epsilon/2}(x)}\leq\delta).\] 

The proof of the following lemma is an adaptation of the proof of Lemma 3.16 in \cite{barlowdiff}.
\begin{lemma} (See also Barlow \cite{barlowdiff}.)
Suppose $(X, d,\mu)$ is variable Ahlfors regular of exponent $\alpha,$ \textit{strongly} variable time regular of exponent $\beta,$ and satisfies the time comparability condition. Then there exist constants $r_0>0,$ $p_0\in (0,1)$ and $c>0$ so that if $0<r<r_0,$ $t>0$, $x\in X$, and $0<R<\diam(X)/2$, then \[\mathbb{P}_r^x(\tau_{B_R[x]}\leq t)\leq p_0+ctR^{-\beta(x_0)}.\]
\end{lemma}
\begin{proof} By strong time regularity and time comparability, there exist constants $C>1,$ $r_0>0$ such that if $0<r<r_0$ and $x\in X$ then \[\frac{1}{C}R^{\beta(x)}\leq \phi^+_{r,B_R(x)}\leq CR^{\beta(x)}\] and \[\frac{1}{C}\phi_{r, B_R[x]}^+\leq \phi_{r,B_R(x)}(x).\] 

Following Barlow in \cite{barlowdiff}, let $\tau=\tau_{B_R[x]}.$ Then $\tau\leq t+1_{(\tau>t)}(\tau-t).$ Hence $\mathbb{E}^x_r\tau\leq t+\mathbb{P}^x_r(\tau>t)\mathbb{E}^{\bm{x}_r(t)}(\tau-t)\leq t+ \mathbb{P}^x_r(\tau>t)\sup_{y\in B_R[x]}\mathbb{E}^y\tau.$ But $\sup_{y\in B_R[x]}\mathbb{E}^y_r\tau=\phi_{r,B_R[x]}^+.$ So let $0<r<r_0.$ Then \[\frac{1}{C^2}R^{\beta(x)}\leq \frac{1}{C}\phi_{r,B_R[x]}(x)\leq\mathbb{E}^x_r\tau=\phi_{r,B_R[x]}(x)\leq t+C\mathbb{P}^x_r(\tau\geq t)R^{\beta(x)}.\]
Hence \[\mathbb{P}^x_r(\tau\leq t)\leq \left(1-\frac{1}{C^3}\right)+\frac{1}{C}tR^{-\beta(x)},\] so that we take $p_0=1-\frac{1}{C^3}\in (0,1)$ and $c=1/C>0.$ \end{proof}

\begin{corollary} There exists a $p\in (0,1),$ $\delta>0$, and an $r_0>0$ so that if $0<r<r_0$ and $x\in X,$ then for all $0<R<\diam(X)/2,$
\[\mathbb{P}^x_r(\tau_{B_R[x]}\leq \delta R^{\beta(x)})\leq p.\]
\end{corollary}
\begin{proof} For $0<r<r_0,$ $x\in X,$ $\delta>0,$ and $0<R<R_0,$ by the previous lemma,
\[\mathbb{P}^x_r(\tau_{B_R[x]}\leq \delta R^{\beta(x)})\leq p_0+c\delta.\] So choose $0<\delta<\frac{1-p_0}{2c}$ and $p=p_0+\frac{1-p_0}{2}.$ \end{proof}

The following proposition is known as Hoeffding's inequality. See Theorem 1 of \cite{hoeffding} for a proof. 
\begin{proposition} (Hoeffding's inequality)
Suppose $T_1,T_2,...,T_n$ are independent random variables with $0\leq T_i\leq 1$ for all $i.$ If $\overline{T}=\frac{1}{n}\left(\sum_{i=1}^n T_i\right)$ and $\mu=\mathbb{E}(\overline{T}),$ then for $0<t<1-\mu,$
\[\mathbb{P}(\overline{T}-\mu \geq t)\leq e^{-2nt^2}.\]
\end{proposition}

We now prove an exponential bound on the tail of exit time. Or approach adapts an argument given by Barlow in \cite{barlowcarpet1}. Fix a ball $B=B_R[x_0]$ with $0<R<\min\{1,\diam(X)/2\}.$ Let $p,$ $\delta,$ and $r_0$ as in the previous corollary. Let $0<r<R.$ Note since $0<r<1$ and $\beta(x)\leq \beta(B)$ for all $x\in B,$ we have $r^{\beta(B)}\leq r^{\beta(x)}.$ Hence the previous corollary implies \[\mathbb{P}^x_\epsilon(\tau_B\leq \delta r^{\beta(B)})\leq p\] for all $x\in B$ and $0<\epsilon<r_0.$ Let $N=N(r)=\lfloor \frac{R}{r} \rfloor.$  Set $0<\epsilon<\min\{r_0,r/2\}.$ Let $\tau_1=\tau_{B_{r-\epsilon}[x_0]}.$ For $1\leq j<N,$ having defined $\tau_j,$ define $\tau_{j+1}$ by \[\tau_{j+1}(\gamma):=\inf\{t\geq\tau_{j}(\gamma)\;|\;\gamma(t)\notin B_{r-\epsilon}[\gamma(\tau_j(\gamma))]\}.\] Then the variables $\tau_1,\tau_2-\tau_1,...,\tau_{N}-\tau_{N-1}$ are independent. Moreover, since the ``jump size" is at most $\epsilon,$ the the first jump leaving $B_{r-\epsilon}[x]$ lands within $B_r[x].$ Hence a walk exiting $B_R[x_0]$ starting at $x_0$ must leave at least $N=\lfloor N/r\rfloor$ balls of radius $r.$ Set $\tau_0=0.$ Then, if $\gamma$ is a path with $\gamma(0)=x_0,$ \[\tau_{B_R[x_0]}(\gamma)\geq \sum_{i=1}^N \tau_i(\gamma)-\tau_{i-1}(\gamma).\]  For $t=1,...,n$ let \[A_i:=\{\gamma\in \mathcal{D}\;|\;\gamma(0)=x_0\;,\tau_i(\gamma)-\tau_{i-1}(\gamma)\leq \delta (r-\epsilon)^{\beta(B)}\}.\] Define a Bernoulli random variable $T_i$ by $T_i:=1_{A_i}$, where we are using the probability measure $\mathbb{P}_{\epsilon}^{x_0}$ on $\mathcal{D}.$ Then by the corollary,
\[\mathbb{E}(T_i)\leq p\] for all $i.$ Let $T:=\sum_{i=1}^n T_i$ and $\overline{T}:=\frac{T}{N}.$ It follows that $\mathbb{E}
(\overline{T})\leq p.$ Let $A$ be the event that $T\geq \left(\frac{p+1}{2}\right) N,$ which is the same as the event that $
\overline{T}\geq \frac{p+1}{2}.$ But $\overline{T}\geq \frac{p+1}{2}$ implies, since $\mathbb{E}(\overline{T})\leq p,$ that $
\overline{T}-\mathbb{E}(\overline{T})\geq \frac{1-p}{2}.$ Since $\frac{1-p}{2}<1-\mathbb{E}(\overline{T})=1-p,$ by Hoeffding's 
inequality we have 
\[\mathbb{P}^{x_0}_\epsilon\left(T\geq \left(\frac{p+1}{2}\right) N\right)\leq e^{-2N(\frac{1-p}{2})^2}=e^{-2Nc^2},\] where $c=(1-p)/2.$
If $\gamma\in A^c$ then for at least $\lceil\left(\frac{1-p}{2}\right)N\rceil\geq \left(\frac{1-p}{2}\right)N $ of the $i$ we 
have that $\tau_i(\gamma)-\tau_{i-1}(\gamma)>\delta (r-\epsilon)^{\beta(B)}.$ Thus $\tau_{B_R[x_0]}(\gamma)> \left(\frac{1-p}
{2}\right)N \delta (r-\epsilon)^{\beta(B)}. $  It follows that $\tau_{B_R[x_0]}(\gamma)\leq \left(\frac{1-p}{2}\right)N \delta 
(r-\epsilon)^{\beta(B)}=cNr^{\beta(B)}$ with $\gamma(0)=x_0$ implies that $\gamma\in A.$ Thus \begin{equation}\mathbb{P}^{x_0}_{\epsilon}
\left(\tau_{B}\leq cN\delta (r-\epsilon)^{\beta(B)}\right)\leq e^{-2Nc^2},\end{equation} where $N=\lfloor \frac{R}{r}\rfloor.$ The only 
conditions needed are that $0<r<R<\min\{1,\diam(X)/2\}$ and $0<\epsilon<\min\{r_0,r/2\}.$ 

Hence we may prove the following theorem. 
\begin{theorem} Suppose $(X,d,\mu)$ is a compact, connected variable Ahlfors regular metric measure space satisfying the 
strong variable time regularity and time comparability conditions. Then given a sequence of positive numbers $(r_k)$ decreasing 
to $0$ there exists a subsequence $(r_{k_j})$ such that the sequence of continuous time processes $(\bm{x}({r_k}_j))$ starting at $x_0$ 
converges weakly to a continuous time process $\bm{x}$ with continuous paths in the sense that $\mathbb{P}^{x_0}
(\mathcal{C})=1.$
\end{theorem}
\begin{proof} By the previous discussion it suffices to show that for any $R>0$ with $R<\min\{1,\diam(X)/2\}$ and any $\eta>0,$ there exists a $t\in (0,1)$ and $\epsilon_0>0$ such that for all $0<\epsilon<\epsilon_0,$ $\mathbb{P}^{x_0}_\epsilon(\tau_{B_R[x_0]}\leq t)\leq \eta t.$ Now fix $r=4\epsilon.$ Then if $N=\lfloor \frac{R}{4\epsilon} \rfloor$ choose $0<\epsilon_1<r_0$ such that $\frac{e^{-2Nc^2}}{\eta}\leq t.$ Note this is still true for $0<\epsilon<\epsilon_1.$ Then since $\beta(B)\geq \beta(x_0)\geq 2,$ we may choose $0<\epsilon_0<\epsilon_1$ such that $cN\delta (3\epsilon)^{\beta(B)}\leq \frac{9cR\delta}{4}\epsilon^{\beta(B)-1}<1.$ Again note that this still holds for $0<\epsilon<\epsilon_0.$ So let $t=cN\delta (3\epsilon)^{\beta(B)}.$ Then by our previous inequality (5.1), for $0<\epsilon<\epsilon_0,$
\[\mathbb{P}^{x_0}_{\epsilon} \left(\tau_{B}\leq cN\delta (r-\epsilon)^{\beta(B)}\right)=\mathbb{P}^{x_0}_{\epsilon}
(\tau_{B}\leq t)\leq e^{-2Nc^2}\leq \eta t.\] \end{proof}




\begin{appendices}

\chapter{Additional Proofs}
We provide proofs of many of the stated classical results from Chapter 2. The proofs provided are not original; however they are included for completeness. 
\subsection{Basic Measure Theory}

\subsubsection{Proposition \ref{Car} }

\begin{proof} Since $\mu^*$ is subadditive, $A\in \mathscr{M}^*$ if and only if $\mu^*(E)\geq \mu^*(E\cap A)+\mu^*(E\cap A^c)$ for all $E\subset X.$ Then $\emptyset \in \mathscr{M}^*$, since if $E\subset X,$ $\mu^*(E\cap \emptyset)+\mu^*(E\cap \emptyset^c) = \mu^*(\emptyset)+\mu^*(E)=\mu^*(E).$ By symmetry of the definition of measurability, $A\in \mathscr{M}^*$ if and only if $A^c\in \mathscr{M}^*.$ Let $A,B\in \mathscr{M}^*.$ If $E\subset X,$ since $A$ and $B^c$ are measurable, $\mu^*(E)\geq \mu^*(E\cap A)+\mu^*(E\cap A^c)$ and $\mu^*(E \cap A^c) \geq \mu^*(E\cap A^c\cap B^c)+\mu^*(E\cap A^c \cap B).$ Moreover, since $E\cap(A\cup B)\subset (E\cap A) \cup (E\cap A^c \cap B ),$ $\mu^*(E\cap(A\cup B))\leq \mu^*(E\cap A)+\mu^*(E\cap A^c \cap B^c)+\mu^*(E\cap A^c \cap B).$ Putting it together, $\mu^*(E)\geq \mu^*(E\cap(A\cup B))+\mu^*(E\cap A^c \cap B^c).$ So $A\cup B\in \mathscr{M}^*.$ Hence $\mathscr{M}^*$ is an algebra. Let $(A_i)_{i=1}^\infty \subset \mathscr{M}^*$ disjoint. Suppose for all $E\subset X,$ $\mu^*(E\cap \cup_{i=1}^{n} A_i)=\sum_{i=1}^n \mu^*(E\cap A_i).$ Then since $A_{n+1}\in \mathscr{M}^*$ and since $\cup_{i=1}^n A_i = (\cup_{i=1}^{n+1} A_i) \cap A_{n+1},$ $A_{n+1} = (\cup_{i=1}^{n+1} A_i) \cap A_{n+1}, \mu^*(E\cap(\cup_{i=1}^{n+1}A_i))=\mu^*(E\cap A_{n+1}) + \mu^*(E\cap (\cup_{i=1}^n A_i)) = \sum_{i=1}^{n+1} \mu^*(A_i).$ Hence, by induction, for all $n$ $\mu^*(E\cap (\cup_{i=1}^n A_i)) = \sum_{i=1}^n \mu^*(E\cap A_i).$ Since $\mathscr{M}^*$ is an algebra, $\cup_{i=1}^n A_i\in \mathscr{M}^*$ for each $n.$ Then for all $n,$ since $(\cap_{i=1}^n A_i)^c \subset (\cap_{i=1}^\infty A_i)^c,$   $\mu^*(E)=\mu^*(E\cap(\cup_{i=1}^n A_i))+\mu^*(E\cap (\cup_{1=i}^n A_i)^c) \geq \sum_{i=1}^n \mu^*(E\cap A_i)+\mu^*(E\cap(\cup_{i=1}^\infty A_i)^c).$ Letting $n\rightarrow \infty$ and applying countable subadditivity, $\mu^*(E)\geq \sum_{i=1}^\infty \mu^*(E\cap A_i)+\mu^*(E\cap(\cup_{i=1}^\infty A_i)^c)\geq \mu^*(E\cap (\cup_{i=1}^\infty A_i))+\mu^*(E\cap (\cup_{i=1}^\infty A_i)^c).$ Hence $\cup_{i=1}^\infty A_i$ in $\mathscr{M}^*.$ Setting $E=\cup_{i=1}^\infty A_i,$ the above inequality yields $\mu^*(\cup_{i=1}^\infty A_i) \geq \sum_{i=1}^\infty \mu^*(A_i).$ So $\mu^*(\cup_{i=1}^\infty A_i)=\sum_{i=1}^\infty \mu^*(A_i).$ Let $(B_i)_{i=1}^\infty \subset \mathscr{M}^*.$ Let $B_0=\emptyset.$ Then set $A_n=B_n\setminus (\cup_{i=0}^{n-1}B_i).$ Then $(A_i)_{i=1}^\infty \subset \mathscr{M}^*$ is disjoint. Hence $\cup_{i=1}^\infty B_i = \cup_{i=1}^\infty A_i \in \mathscr{M}^*.$ It follows that $\mathscr{M}^*$ is a $\sigma$-algebra and $\mu^*|_{\mathscr{M}^*}$ is a measure. Suppose $\mu^*(A)=0.$ Then for $E\subset X,$ $E\cap A\subset A$ so $\mu^*(E\cap A)=0.$ Hence $\mu^*(E\cap A)+\mu^*(E\cap A^c) = \mu^*(E\cap A^c) \leq \mu^*(E).$ So $A\in \mathscr{M}^*.$ Let $N\in \mathscr{M}^*$ with $\mu^*(N)=0.$ Then by monotonicity if $A\subset N, \mu^*(A)=0.$ So $A\in \mathscr{M}^*.$ Hence $\mu|_{\mathscr{M}^*}$ is a complete measure. \end{proof}

\subsubsection{Proposition \ref{prop2.2} }

\begin{proof} Suppose $(A_j)_{j=1}^\infty \subset X$ increasing, $A=\cup_j A_j,$ and $\dist(A_j,A\setminus A_{j+1})>0$ for all $j\geq 1.$ Let $A_0=\emptyset, B_j=A_j\setminus(\cup_{i<j} A_i) = A_j\setminus A_{j-1}.$ Then $(B_j)_{j=1}^\infty$ is disjoint, $\cup_{i=1}^n B_i = \cup_{i=1}^n A_i = A_n.$ Since $\cup_{i<j} B_i = A_{j-1}, B_j \subset A\setminus A_{j-1},$ $\dist(\cup_{i<j} B_i, B_j)>0$ for $j>1.$ Hence, by induction, $\mu^*(\cup_{i=1}^n B_i) = \sum_{i=1}^n \mu^*(B_i)$ for all $n.$ Hence $\lim_{n\rightarrow \infty} \mu^*(A_n) = \sum_{i=1}^\infty \mu^*(B_i) \geq \mu^*(A).$ So $\mu^*(A)=\lim_{n\rightarrow \infty} \mu^*(A_n).$ 

Let $U$ open, $E\subset X.$ Let $U_n:=\{y\in U\;|\;d(x,y)\geq \frac{1}{n} \mbox{ for all } x \in U^c\;\}.$ Then $(U_j)_{j=1}^\infty$ is increasing. Let $y\in U.$ Then since $U$ is open there exists an $n>0$ so that $B_{\frac{1}{n}}(y)\subset U.$ So if $x\in U^c$ then $d(x,y)\geq \frac{1}{n}.$ Hence $\cup_n U_n = U.$ Then $\dist(E\cap U_n,E\cap U^c)\geq \frac{1}{n}>0.$ So, since $(E\cap U_n)\cup (E\cap U^c) \subset E,$  $\mu^*(E)\geq \mu^*(E\cap U_n)+\mu^*(E\cap U^c).$ Let $A_n=E\cap U_n, A=E\cap U=\cup_n A_n.$ Then if $x\in A_n, y\in A\setminus A_{n+1},$  $y\in U, y\notin U_{n+1}.$ So $d(y,z)<\frac{1}{n+1}$ for some $z\in U^c.$ But $x\in A_n.$ So $d(x,z)\geq \frac{1}{n}.$ Hence $d(x,y)\geq d(x,z)-d(y,z)<\frac{1}{n}-\frac{1}{n+1}.$ So $\dist(A_n,A\setminus A_{n+1})\geq \frac{1}{n}-\frac{1}{n+1}>0.$ So $\mu^*(A)=\lim_{n\rightarrow \infty} \mu^*(A_n).$ Since $\mu^*(E)\geq \mu^*(A_n)+\mu^*(E\cap U^c)$ for all $n,$ letting $n\rightarrow \infty$ yields $\mu^*(E)\geq \mu^*(A)+\mu^*(E\cap U^c)=\mu^*(E\cap U)+\mu^*(E\cap U^c).$ Hence $U$ is $\mu^*$ measurable. So $\mathscr{M}^*$ contains a generating set for the Borel $\sigma$-algebra. Since $\mathscr{M}^*$ is a $\sigma$-algebra, it contains all of the Borel sets. 
\end{proof}

\subsubsection{Proposition \ref{outer} }
\begin{proof} For $\delta>0, \emptyset \in \mathscr{C}_\delta(\emptyset).$ Hence 
$\mu^* _{\tau}(\emptyset)=0.$ If $A\subset B$ then $\mathscr{C}_\delta(B)\subset 
\mathscr{C}_\delta(A).$ Hence $\mu^* _{\tau, \delta}(A)\leq \mu^* _{\tau, \delta}(B)$. So $\mu^* _{\tau}(A)\leq \mu^* _{\tau}(B)$. Let $(A_k)_{k=1}^\infty \subset \mathscr{P}(X).$ Let $\epsilon>0.$ Let $\delta>0$ so that $\mu^* _{\tau, \delta}(\cup_k A_k)+\frac{\epsilon}{2}>\mu^* _{\tau}(\cup_k A_k).$ Then for each $k$ let $(U_{k,j})_{j=1}^\infty \in \mathscr{C}_\delta(A_k)$ with $\sum_j \tau(U_{j,k})<\mu^* _{\tau, \delta}(A_k)+\frac{\epsilon}{2^{k+1}}.$ Then $(U_{j,k})_{j,k=1}^\infty \in \mathscr{C}_\delta(\cup_k A_k)$ and $\mu^* _{\tau, \delta}(\cup_k A_k)\leq \sum_{j,k}\tau(U_{j,k})\leq \sum_k \mu^* _{\tau, \delta}(A_k)+ \frac{\epsilon}{2}\leq \mu^* _{\tau}(A_k)+ \frac{\epsilon}{2}.$ Hence $\mu^* _{\tau}(\cup_k A_k)\leq \sum_k \mu^* _{\tau}(A_k)+ \epsilon.$ Since $\epsilon>0$ was arbitrary, $\mu^* _{\tau}(\cup_k A_k)\leq \sum_k \mu^* _{\tau}(A_k).$

For $E\subset X$ let $\mathscr{C}_\delta(E)^\prime:=\{\mathscr{U}\in \mathscr{C}_\delta(E)\;|\;\forall U\in \mathscr{U}\;[\;U\cap E\neq \emptyset\;]\}.$ Since $\mathscr{C}_\delta(E)^\prime \subset \mathscr{C}_\delta(E),$
$\mu^* _{\tau, \delta}(E)\leq \inf\{\sum_{U\in \mathscr{U}}\tau(U)\;|\; 
\mathscr{U}\in \mathscr{C}_\delta(E)^\prime\}.$ But for $\mathscr{U}\in \mathscr{C}_\delta(E),$ letting $\mathscr{U}^\prime:=\{U\in\mathscr{U}\;|\;U\cap E\neq \emptyset\;\},$ then $\mathscr{U}^\prime \in \mathscr{C}_\delta(E)^\prime$ and $\inf \{\sum_{U\in \mathscr{U}}\tau(U)\;|\; 
\mathscr{U}\in \mathscr{C}_\delta(E)^\prime\}\leq \sum_{U\in \mathscr{U}^\prime} \tau(U) \leq \sum_{U\in \mathscr{U}}\tau(U).$ Hence $\inf\{\sum_{U\in \mathscr{U}}\tau(U)\;|\; 
\mathscr{U}\in \mathscr{C}_\delta(E)^\prime\}=\mu^* _{\tau,\delta}(E)$ for each $\delta>0$. Now suppose $A,B\subset X$ with $\delta_0:=\frac{1}{2}$dist$(A,B)>0.$ Then note for $\delta<\delta_0, \mathscr{U}\in \mathscr{C}_\delta(A\cup B)^\prime$ if and only if $\mathscr{U}=\mathscr{U}_1\cup \mathscr{U}_2$ where $\mathscr{U}_1 \in \mathscr{C}_\delta(A)^\prime, \mathscr{U}_2 \in \mathscr{C}_\delta(B)^\prime$ and for $U_1\in \mathscr{U}_1,U_2\in \mathscr{U}_2,$ $U_1\cap U_2=\emptyset.$ It follows that $\mu^* _{\tau}(A\cup B)=\mu^* _{\tau}(A)+\mu^* _{\tau}(B).$ 
\end{proof}
\subsubsection{Proposition \ref{reg} }
\begin{proof} Let $\mathscr{M}:=\{A\;| \sup\{\nu(F)\;|\;A \supset F,\; F \mbox{ closed}\}=\inf\{\nu(U)\;|\;A\subset U,\; U \mbox{ open}\}\}.$ Since $X$ is both open and closed, $X\in \mathscr{M}.$ Suppose $A\in \mathscr{M}.$ Then for $\epsilon>0$ if $F\subset A\subset U$ with $F$ closed, $U$ open, and $\nu(U\setminus F)<\epsilon,$ then $U^c\subset A^c \subset F^c,$ $U^c$ closed, $F^c$ open, and $\nu(F^c\setminus U^c)=\nu(U\setminus F)<\epsilon.$ It follows that $A^c \in \mathscr{M}.$ Let $(A_n)\subset \mathscr{M}.$ Then for $\epsilon>0$ and for each $n,$ let $F_n\subset A_n \subset U_n$ with $F_n$ closed, $U_n$ open, and $\nu(U_n\setminus F_n)<\frac{\epsilon}{2^{n+1}}.$ Then let $A=\cup A_n.$ Let $N$ be sufficiently large so that $\nu(\cup_{n=1}^N F_n)>\nu(\cup_n F_n)-\frac{\epsilon}{2}.$ Then let $F=\cup_{n=1}^N F_n, U=\cup U_n.$ Then $F\subset A \subset U,$ $F$ is closed, $U$ is open, and $\nu(U\setminus \cup_n F_n) = \nu(U)-\nu(\cup_n F_n)<\frac{\epsilon}{2}.$ So $\nu(U\setminus F)=\nu(U)-\nu(F)<\epsilon.$ So $A\in \mathscr{M}$. Hence $\mathscr{M}$ is a $\sigma-$algebra. Let $A\subset X$ closed. Then $\nu(A)=\sup\{\nu(F)\;|\;X \supset F \mbox{ closed }\}.$ Let $U_n=B_{\frac{1}{n}}(A).$ Then $A\subset U_n,$ each $U_n$ is open, and $\cap_n U_n = A.$ So $\inf_n \nu(U_n) = \nu(A)$ by continuity of measure, since $\nu(X)<\infty.$ Hence $\nu(A)\leq \inf\{\nu(U)\;|\;A\subset U \mbox{ open }\}\}\leq \inf_n \nu(U_n)=\nu(A).$ Hence $\mathscr{M}$ contains all Borel sets. 

If $X$ is compact, then all closed subsets of $X$ are compact. Hence $\nu$ is regular when $X$ is compact. Now, if $X$ is $\sigma-$compact, $X=\cup X_n$ where $X_i\subset X_j$ for $j\geq i$ and each $X_i$ is compact. Let $A\subset X$ measurable. Let $\epsilon>0.$ Let $K_1\subset X_1\cap A$ compact with $\nu(A\cap X_1)<\nu(K_1)+\epsilon.$ Then we may choose $K_2\subset X_2\cap A$ with $K_1\subset K_2$ and $\nu(A\cap X_1)<\nu(K_1)+\frac{\epsilon}{2}.$ Continuing by induction, we may choose a sequence $K_n$ with $K_n\subset X_n \cap A,$ $K_n$ compact, $K_n\subset K_{n+1}$  and $\nu(A\cap X_n)<\nu(K_n)+\frac{\epsilon}{n}$ for each $n$. Then by continuity of measure, $\nu(A)=\sup_n \nu(A\cap X_n) \leq \sup _n \nu(K_n).$ The result follows for the case of $X$ $\sigma$-compact.

Now suppose $X$ is separable and complete. Then, if $A$ is a Borel set, since we may approximate $A$ arbitrarily by closed sets, we may assume $A$ is closed. But then, since $A$ is again separable and complete, we may assume $A=X.$ Let $(x_n)_{n=1}^\infty$ be dense in $X.$ Then let $B_{m,n}:=B_{1/m}(x_n)$ for each $m,n\in\mathbb{Z}^+.$ Then $(B_{m,n})_{n=1}^\infty$ covers $X$ for each $m.$ Let $\epsilon>0.$ Then we may choose a sequence $(i_n)_{n=1}^\infty$ of positive integers such that $\nu(\cup_{k=1}^{i_n} B_{n,k})>\nu(X)-\frac{1}{2^n}$ for $n\in\mathbb{Z}^+.$ Then let $A:=\cap_{n=1}^\infty \cup_{k\leq i_n} B_{n,k}.$ Note that $A$ is totally bounded; since if $\delta>0,$ then for $N$ large with $N^{-1}<\delta,$ $A\subset \cup_{k=1}^{i_N} B_{\delta}(x_k)$. Therefore $K:=\overline{A}$ is compact, as it is complete and totally bounded. However, \[\nu(K^c)\leq \nu(A^c)\leq \sum_{n=1}^\infty \nu((\cup_{k=1}^{i_n} B_{n,k})^c)\leq \sum_{n=1}^\infty \frac{\epsilon}{2^n}=\epsilon.\] Therefore, \[\nu(K)>\nu(X)-\epsilon,\] proving the desired result.
\end{proof}

\subsubsection{Lemma \ref{mon}}
\begin{proof} Let $\mathscr{C}':=\cap\{\mathscr{C}\subset \mathscr{P}(S)\;|\;\mathscr{C}\;\mbox{a monotone class containing\;} \mathscr{A}\}.$ Then it is easy to see that $\mathscr{C}'$ is a monotone class containing $\mathscr{A}.$ It suffices to show that $\mathscr{C}'$ is a $\sigma$-algebra. It suffices then to show that it is an algebra, since any countable union is a monotone union of finite unions, similarly for intersections. Let $B\in \mathscr{C}'$ and let $\mathscr{E}=\{A\in\mathscr{C}'\;|\;A\cap B, A^c \cap B\in\mathscr{C}'\}.$ Then $\mathscr{A}\subset \mathscr{E}.$ Let $(A_i)_{i=1}^\infty \subset \mathscr{E}.$ Suppose $A_k\subset A_{k+1}$ for all $k.$ Then $A_k\cap B\subset A_{k+1}\cap B$ for all $k.$ Hence, since each $A_k\cap B\in \mathscr{E}\subset \mathscr{C}',$  $\cup_{k} A_k\cap B = (\cup_k A_k)\cap B\in\mathscr{C}'.$ Similarly, $B\cap A_k^c\supset B\cap A_k^c$ for all $k$ and $B\cap A_k^c\in \mathscr{E}\subset \mathscr{C}'$ for each $k.$ Hence $\cap_k B\cap A_k^c=B\cap(\cap_k A_k^c)=B\cap(\cup_k A_k)^c\in\mathscr{C}'.$ So $\cup_kA_k\in \mathscr{C'}.$ Now suppose $A_k+1\supset A_k$ for all $k.$ Then $A_k^c\cap B\subset A_{k+1}^c\cap B$ with $A_k\cap B\in\mathscr{E}\subset \mathscr{C}'$ for all $k$. Hence $B\cap(\cap_k A_k)^c=B\cap (\cup_k A_k^c)=\cup_k (B\cap A_k^c)\in \mathscr{C}'.$ Similarly, since $A_k\cap B\supset A_{k+1}\cap B$ with $A_k\cap B\in\mathscr{C}'$ for all $k,$ $(\cap_k A_k) \cap B=\cap_k(A_k\cap B)\in \mathscr{C}'.$ So $\cap_k A_k\in \mathscr{E}.$ Hence $\mathscr{E}$ is a monotone class containing $\mathscr{C}'.$ So $\mathscr{C}'\subset \mathscr{E}.$ Hence $A\cap B$, $A\cap B^c\in\mathscr{C}'$ for all $A, B\in \mathscr{C}'.$ Since $X\in \mathscr{A}\subset \mathscr{C}',$ $\mathscr{C}'$ contains finite intersections and complements. It follows that $\mathscr{C}'$ is an algebra. 
\end{proof}

\subsection{Kolmogorov Extension Theorem}

Our proof of the Kolmogorov Extension Theorem follows \cite{rogers}. The proof is based around the Carath\'{e}odory Extension Theorem and the inner regularity of a Borel measure on a Polish space. We first state and prove the Carath\'{e}odory Extension Theorem, following \cite{Folland}.

\subsubsection{Carath\'{e}odory Extension Theorem}

For $X$ a non-empty set, an algebra of subsets of $X$ is a collection $\mathscr{A}\subset \mathscr{P}(X)$ that is closed under finite unions and complementation. If $\mathscr{A}$ is an algebra of subsets of $X,$ a function $\mu_0:\mathscr{A}\rightarrow [0,\infty]$ is called a pre-measure if $\mathscr{C}\subset \mathscr{A}$ is countable and disjoint with $\cup\mathscr{C}\in\mathscr{A}$, then $\mu_0(\cup\mathscr{C})= \sum_{A\in\mathscr{C}}\mu_0(A).$

\begin{proposition} (Carath\'{e}odory Extension Theorem) \label{cara} Suppose $\mu_0$ is a pre-measure on an algebra $\mathscr{A}$ of subsets of $X.$ Then the map $\sigma(\mathscr{A})\ni A \mapsto \inf_{[0,\infty]}\{\sum_{U\in\mathscr{C}}\mu_0(U)\;|\;\mathscr{C}\subset \mathscr{A}, \sharp\mathscr{C}\leq \infty, A\subset \cup\mathscr{C}\}$ is a measure on $\sigma(\mathscr{A})$ extending $\mu_0.$ Moreover, if $\mu_0(X)<\infty,$ then $\mu$ is the unique measure on $\sigma(\mathscr{A})$ extending $\mu_0.$
\end{proposition}

\begin{proof}
Suppose $\mathscr{A}\subset \mathscr{P}(X)$ is an algebra and $\mu_0$ is a pre-measure on $X.$ Define $\mu^*$ on $\mathscr{P}(X)$ by $\mu^*(A):=\inf_{[0,\infty]}\{\sum_{U\in\mathscr{C}}\mu_0(U)\;|\;\mathscr{C}\subset \mathscr{A}, \sharp\mathscr{C}\leq \infty, A\subset \cup\mathscr{C}\}$ for $A\in \mathscr{P}(X).$ Since $\mu_0(\emptyset)=0$ and $\emptyset\in\mathscr{A}$ is a countable cover of $\emptyset$, it follows that $\mu^*(\emptyset)=0.$ If $A\subset B$ then since any cover of $B$ is also a cover of $A$, it follows $\mu^*(A)\leq \mu^*(B).$ If $A\in \mathscr{A},$ then since $\{A\}$ is a countable cover of $A,$ it follows that $\mu^*(A)\leq \mu_0(A).$ Suppose $(A_k)_{k=1}^\infty\subset \mathscr{A}$ is a countable cover of $A.$ For each $k$, let $B_k:=A_k\setminus(\cup_{j<k} A_j).$ Then each $B_k,$ and hence also $B_k\cap A$, is in $\mathscr{A}.$ Moreover, the $B_k$ are disjoint with union $\cup_k A_k \supset A.$ Hence $\mu_0(A)=\mu_0(\cup_k(B_k\cap A))=\sum_k\mu_0(B_k\cap A)\leq \sum_k\mu_0(B_k)\leq \sum_k \mu_0(A_k).$ So $\mu_0\leq \mu^*(A).$ Hence $\mu^*|{\mathscr{A}}=\mu_0.$  Now suppose $(A_i)_{i=1}^\infty\subset X.$ Then we may assume $\mu^*(A_i)<\infty$ for all $i.$ Let $\epsilon>0.$ Then for each $i$ let $\mathscr{C}_i\subset \mathscr{A}$ be a countable cover of $A_i$ with $\sum_{U\in \mathscr{C}_i} \mu_0(U)<\mu^*(A_i)+\frac{\epsilon}{2^i}.$ Then let $\mathscr{C}=\cup_i\mathscr{C}_i.$ Then $\mathscr{C}\subset \mathscr{A}$ is a countable cover of $\cup_i A_i.$ Hence $\mu^*(\cup_i A_i)\leq \sum_{U\in\mathscr{C}}\mu_0(U)\leq \sum_i\mu^*(A_i)+\epsilon.$ Since $\epsilon>0$ was arbitrary, $\mu^*(\cup A_i)\leq \sum_i \mu^*(A_i).$ Hence $\mu^*$ is an outer measure. 

Let $\mathscr{M}^*$ be the collection of $\mu^*$-measurable sets. Then by the previous result, it suffices to show that $\sigma(\mathscr{A})\subset \mathscr{M}^*.$ However, this is equivalent to showing that $\mathscr{A}\subset \mathscr{M}^*$ since $\sigma(\mathscr{A})$ is the smallest $\sigma$-algebra containing $\mathscr{A}.$
Let $A\in\mathscr{A}.$ Let $E\subset X.$ We need only show $\mu^*(E)\geq \mu^*(E\cap A)+\mu^*(E\cap A^c)$. Let $\epsilon>0.$ We may assume that $\mu^*(E)<\infty.$ Then let $\mathscr{C}\subset \mathscr{A}$ be a countable cover of $E$ with $\sum_{U\in\mathscr{C}}\mu_0(U)<\mu^*(E)+\epsilon.$ However, for each $U,$ since $U\cap A, U\cap A^c$ are in $\mathscr{A}$ and disjoint with union $U\in \mathscr{A},$ $\mu_0(U)=\mu_0(U\cap A)+\mu_0(U\cap A^c).$ Let $\mathscr{C}':=\{A\cap U\;|\;U\in\mathscr{C}\}, \mathscr{C}'':=\{A^c\cap U\;|\;U\in\mathscr{C}\}.$ Then $\mathscr{C}',\mathscr{C}''$ are countable subsets of $\mathscr{A}$ that cover $E\cap A$ and $E\cap A^c,$ respectively. It then follows that 
\[\mu^*(E\cap A)+\mu^*(E\cap A^c)\leq \sum_{u\in\mathscr{C}}(\mu_0(U\cap A)+\mu_0(U\cap A^c))=\sum_{U\in\mathscr{C}}\mu_0(U)<\mu^*(E)+\epsilon.\] Since $\epsilon>0$ was arbitrary, $A\in \mathscr{M}^*.$ Hence $\mu:=\mu^*|_{\sigma(\mathscr{A})}$ is a measure that extends $\mu_0.$ 

Finally, we show uniqueness in the case that $\mu_0(X)<\infty.$ Suppose $\nu$ is another measure on $\sigma(\mathscr{A})$ extending $\mu_0.$ Let $A\in \sigma(\mathscr{A})$. Then if $\mathscr{C}\subset \mathscr{A}$ is a countable cover of $A,$ $\nu(A)\leq \sum_{U\in\mathscr{C}}\nu(U)=\sum_{U\in\mathscr{C}}\mu_0(U).$ Hence $\nu(A)\leq \mu^*(A)=\mu(A).$ Let \[\mathscr{F}=\{A\in\mathscr{P}(X)\;|\;\mu(A)=\nu(A)\}.\] Then $\mathscr{A}\subset \mathscr{F}.$ Note since $\mu_0(X)=\mu(X)=\nu(X)<\infty,$ we have if $A\in \mathscr{F},$ $\mu(A^c)=\mu(X)-\mu(A)=\nu(X)-\nu(A)=\nu(A^c).$ So $\mathscr{F}$ is closed under complementation. Let $A,B\in \mathscr{F}.$ Then $\mu(A\cap B)\leq \nu(A\cap B)=\mu(A)+\mu(B)-\nu(A\cup B)\leq \mu(A)+\mu(B)-\mu(A\cup B)=\mu(A\cap B).$ Hence $A\cap B\in\mathscr{F}.$ So $\mathscr{F}$ is an algebra. Now let $(A_i)_{i=1}^\infty \subset \mathscr{F}.$ Then for each $n,$ $\cup_{j\leq n} A_j\in \mathscr{F}$ and by continuity of measure, $\mu(\cup_k A_k)=\lim_n \mu(\cup_{j\leq n} A_j)=\lim_n\nu(\cup_{j\leq n} A_j)=\nu(\cup_k A_k).$ Hence $\mathscr{F}$ is a $\sigma$-algebra containing $\mathscr{A}.$ Thus $\sigma(\mathscr{A})\subset \mathscr{F}$ and $\mu=\nu.$
\end{proof}

\subsubsection{Proof of Proposition \ref{ext}}
\begin{proof} We follow a similar approach to \cite{kolnotes}. Let $\mathscr{C}$ be the collection of all cylinder sets $\pi_F^{-1}(\prod_{i\in F} A_i),$ where $F\subset T$ is finite and $A_i\in \Sigma_i$ for all $i\in F.$ Then $\mathscr{C}$ is an algebra with $\sigma(\mathscr{C})=\otimes_{t\in T} \Sigma_t.$ Define $\mu_0$ on $\mathscr{C}$ by $\mu_0(\pi_F^{-1}(\prod_{i\in F} A_i))=\mu_F(\prod_{i\in F} A_i).$ If $F'$, $F\subset T$ are finite with $F'\subset F$ and $\pi_F^{-1}(\prod_{i\in F} A_i)=\pi_{F'}^{-1}(\prod_{i\in F'} A_i),$ then $\pi_{F}\pi_F^{-1}(\prod_{i\in F} A_i)=\prod_{i\in F}A_i=\pi_F\pi_{F'}^{-1}(\prod_{i\in F'} A_i)=\pi_{F,F'}^{-1}(\prod_{i\in F'} A_i)$ since $\pi_{F}$ is surjective and $\pi_F\pi_{F'}^{-1}=\pi_{F,F'}^{-1}$. Hence, by the compatibility conditions, \[\mu_F(\prod_{i\in F} A_i)=\mu_{F}\pi_{F,F'}^{-1}(\prod_{i\in F'} A_i)=\mu_{F'}(\prod_{i\in F'} A_i).\] Thus $\mu_0$ is well-defined. Since each $\mu_F$ is a measure, $\pi_F^{-1}(\prod_{i\in F}A_i)$ and $\pi_F^{-1}(\prod_{i\in F}B_i)$ are disjoint if and only if $\prod_{i\in F}A_i$ and  $\prod_{i\in F}B_i$ are disjoint, $\mu_0$ is finitely additive. Let $(E_n)_{n=1}^\infty$ be a disjoint collection of sets in $\mathscr{C}.$ Suppose $E:=\cup_{n=1}^\infty E_n\in \mathscr{C}.$ 

Thus, if we can show that $\mu_0(E)=\sum_{n=1}^\infty \mu_0(E_n),$ then $\mu_0$ is a premeasure, and the desired result follows by Proposition \ref{cara}. However, $\mu_0(\cup_{i=1}^n E_i)=\sum_{i=1}^n \mu_0(E_i)\leq \mu_0(E)$ for all $n.$ Hence, we only need show that for all $\epsilon>0,$ there exists an $n$ such that $\mu_0(\cup_{i=1}^n E_i)\geq \mu_0(E)-\epsilon.$ Suppose this does not hold. Then there exists an $\epsilon>0$ such that \[\mu_0(\cup_{i=1}^n E_i)+\epsilon< \mu_0(E)\;\mbox{for all\;} n\in \mathbb{Z}^+.\] By taking complements relative to $E,$ this implies \[\epsilon< \mu_0(\cap_{i=1}^n (E\cap E_i^c))\;\mbox{for all\;} n\in \mathbb{Z}^+.\] For each $n,$ let $A_n:=\cap_{i=1}^n (E\cap E_i^c).$ Then each $A_n\in\mathscr{C}$, $\mu_0(A_n)>\epsilon$ for all $n,$ and the sequence of $A_n$ is decreasing with $\cap_{n=1}^\infty A_n=\emptyset.$  For $n\in \mathbb{Z}^+,$ since $A_n\in \mathscr{C},$ there exists a finite $F_n\subset T$ and an $\tilde{A}_n\in \Sigma_{F_n}$ with $A_n=\pi_{F_n}^{-1}(\tilde{A}_n).$ Let $F=\cup_{n=1}^\infty F_n.$ Then $F$ is countable. Let $t_1< t_2<...$ be an enumeration of $F.$ We may assume that $F$ is infinite, since otherwise the result would follow from the countable additivity of the measure $\mu_F$. Without loss of generality, we may assume that there exists a subsequence $t_{n_1}<t_{n_2}<...$ of the $t_n$ such that $A_k=\pi_{T_k}^{-1}(B_k)$ for each $k,$ where $T_k:=F\cap [0, t_{n_k}]$ 
and $B_k\in\Sigma_{T_k}.$ Since $\Omega_n:=\prod_{t\in T_n}W_t$ is Polish, $\mu_n:=\mu_{T_n}$ is inner regular by Proposition \ref{reg}. Hence there exists a compact $K_n\in \Omega_n,$ $K_n\subset B_n$ with $\mu_n(B_n)<\mu_n(K_n)+\frac{\epsilon}{2^n}.$ Let $C_n:=\pi_{T_n}^{-1}(K_n)$ for each $n.$ Then let $D_n:=\cap_{j=1}^n C_j$ for each $n.$ 
Then the $D_n$ are decreasing with \[\cap_{n=1}^\infty D_n=\cap_{n=1}^\infty C_n=\emptyset,\] where the last equality follows from $C_n\subset A_n$ for all $n.$ Note $D_n=\pi_{T_n}^{-1}(\cap_{i=1}^n \pi_{T_n,T_i}^{-1}(K_i)),$ since $\pi_{T_n,T_i}\circ\pi_{T_n}=\pi_{T_i}$ for all $i\leq n.$ Moreover, each $\tilde{D}_n=\cap_{i=1}^n \pi_{T_n,T_i}^{-1}(K_i)$ is compact as a closed subset of $K_n=\pi_{T_n,T_n}^{-1}(K_n).$ 
Note further that $\tilde{D}_n\neq \emptyset.$ Indeed, by construction $\mu_0(A_n)-
\mu_0(C_n)<\frac{\epsilon}{2^n}$ for all $n.$ Hence \[\mu_0(A_n)-\mu_0(D_n)=
\mu_0(\cap_{i=1}^n A_i \cap (\cup_{i=1}^n C_i^c))\leq \mu_0(\cup_{i=1}^n A_i\cap C_i^c)
\leq \sum_{i=1}^n(\mu_0(A_i)-\mu_0(C_i))<\epsilon.\] Since $\mu_0(A_n)>\epsilon,$ we must 
have $\mu_0(D_n)=\mu_{T_n}(\tilde{D}_n)>0.$ Hence $\tilde{D}_n\neq \emptyset.$ 

Set $T_0:=\emptyset.$ Then for $j\in \mathbb{Z}^+$ and $n\in \mathbb{Z}^+$ with $n\geq j,$
\[\pi_{T_{j}\setminus T_{j-1}}(D_{n+1})=\pi_{T_{n+1},T_{j}\setminus T_{j-1}}(\tilde{D}_{n+1})\subset \pi_{T_n,T_j\setminus T_{j-1}}(\tilde{D}_n)=\pi_{T_j\setminus T_{j-1}}(D_n).\] Since each $\pi_{T_{n},T_j\setminus T_{j-1}}(\tilde{D}_{n})$ is a closed subset of $\pi_{T_j,T_j\setminus T_{j-1}} \tilde{D}_j,$ which is compact as the continuous image of a compact set, and $\pi_{T_n,T_j\setminus T_{j-1}}(\tilde{D}_n)=\cap_{i=j}^n \pi_{T_{i},T_j\setminus T_{j-1}}(\tilde{D}_{i})\neq \emptyset$ for each $n,$  it follows that \[\prod_{t\in T_j\setminus T_{j-1}}W_t\supset S_j:=\cap_{n=j}^\infty \pi_{T_{n},T_j\setminus T_{j-1}}(\tilde{D}_{n})\neq\emptyset\] by the finite intersection property.  Hence, by the axiom of (countable) choice, \[\emptyset \neq S:=\prod_{j=1}^\infty S_j\subset \Omega:=\prod_{t\in T}W_t.\] Then let $\omega\in S\subset \Omega.$ Let $n\in \mathbb{Z}^+$. Then \[\pi_{T_n}(\omega)=\prod_{i=1}^n\pi_{T_n,T_i\setminus T_{i-1}}(\omega)\in \prod_{i=1}^n\cap_{k=i}^\infty \pi_{T_k,T_i\setminus T_{i-1}}(\tilde{D}_k)\subset \prod_{i=1}^n\pi_{T_n,T_i\setminus T_{i-1}}(\tilde{D}_n)=\tilde{D}_n.\] Hence $\omega\in D_n$ for all $n.$ Thus \[\omega\in \cap_{n=1}^\infty D_n=\cap_{n=1}^\infty C_n,\] a contradiction. 

We have thus shown that $\mu_0$ is a premeasure on $\mathscr{C}$. The result now follows from the Carath\'{e}odory Extension Theorem. \end{proof}

\addtocontents{toc}{\protect\renewcommand{\protect\cftchappresnum}{\appendixname\space}}
\addtocontents{toc}{\protect\renewcommand{\protect\cftchapnumwidth}{7em}}


\end{appendices}

\bibliography{references}
\bibliographystyle{amsplain}

\addcontentsline{toc}{chapter}{References}  

\end{document}

misc notes

\chapter{Localization of Forms}
Recall a Dirichlet form $Q$ on a real Hilbert space $(\mathcal{H},\langle\cdot,\cdot\rangle)$ is a Markovian, closed, symmetric, positive form, that is there exists a subspace $\mathcal{D}(Q)\subset \mathcal{H}$ such that $Q$ is a symmetric, bilinear map $Q:\mathcal{D}(Q)\times \mathcal{D}(Q)\rightarrow {\mathbb{R}}$ with $Q(f):=Q(f,f)\geq 0$ for all $f\in\mathcal{D}(Q)$, $(\mathcal{D}(Q),\langle \cdot,\cdot\rangle_{Q})$ is a Hilbert space, where $\langle f,g\rangle_Q=Q(f,g)+\langle f,g\rangle,$ and if $f\in \mathcal{D}(Q)$ then $\overline{\underline{f}}=\min\{\max\{f,0\},1\}\in\mathcal{D}(Q)$ and the (strong) Markovian condition $Q(\overline{\underline{f}},\overline{\underline{f}})\leq Q(f,f)$ holds.  We do not require the domain to be dense.

Suppose $(X,d)$ is a compact metric space with a Borel measure $\mu$. We shall call a collection $\mathscr{C}$ of closed Borel subsets of $X$ refining if there exists a $\delta_0>0$ such that if $0<\delta<\delta',$ then the set $\mathscr{C}_{\delta}(X):=\{\mathscr{A}\subset \mathscr{P}(X)\;|\;\sharp \mathscr{A}\leq \infty,\; \diam(A)\leq \delta\;\mbox{for all\;}A\in \mathscr{A},\; \cup \mathscr{A}=X\}\neq \emptyset.$ Suppose that for each $A\in\mathscr{C}$ there exists a Dirichlet form $\mathcal{E}_A$ on $L^2(B,\mu).$ We may extend the definition of $\mathcal{E}_A$ as a quadratic form to all of $L^2(A)$ by setting $\mathcal{E}_A(f)=\infty$ for $f\notin \mathscr{D}(\mathcal{E}_A).$ Then if $f\in L^2(X)$ with $\supp f\subset A$, we may consider $f\in L^2(A)$ and define $\mathcal{E}_A(f)=\mathcal{E}_A(\chi_A f).$ For $f\in L^2(X)$ and $\delta>0$, let $\mathscr{C}_\delta(f):=\{(A,f_A)_{B\in\mathscr{A}}\;|\;\mathscr{A}\in\mathscr{C}_\delta(X),\;(f_A)_{A\in\mathscr{A}}\subset L^2(X),\;\supp(f_A)\subset A\mbox{\;for all\;}A\in\mathscr{A},\;\sum_{A\in\mathscr{A}}f_A=f\}.$

Then if $f\in L^2(X),$ define $\mathcal{E}(f)$ by 
\[\mathcal{E}(f):=\lim_{\delta\rightarrow 0^+}\inf\{\sum_{A\in\mathscr{A}} \mathcal{E}_{A}(f_A)\;|\;(A,f_A)_{A\in \mathscr{A}}\in \mathscr{C}_\delta(f)\}.\]

\begin{proposition} If $\alpha$ and $\beta$ are constant, then there exist constants $C>0,$ $r_0>0$ such that for all $0<r<r_0,$
\[\frac{1}{C}\mathscr{E}_r\leq \mathcal{E}_r\leq C\mathscr{E}_r.\]
\end{proposition}

\begin{lemma} There exists a $C>0$ and $r_1>0$ so that if $0<r<r_1$ and $f\in L^2(X,\mu)$, then $D_r^{-1/2}\mathscr{L}_rf\leq C\mathscr{L}_rD_r^{-1/2}f.$
\end{lemma}
Since $\alpha$ and $\beta$ are log-H{\"o}lder continuous, there exists a $C_1>0$ so that if $0<d(x,y)<\frac{1}{2}$ then $|\alpha(x)-\alpha(y)|\leq \frac{-C_1}{d(x,y)}$ and $|\beta(x)-\beta(y)|\leq \frac{-C_1}{d(x,y)}.$ Let $0<r<1/2$ and $x,\;y\in X$ with $0<d(x,y)<r.$ Then $r^{|\alpha(y)-\alpha(x)|}\geq r^{-C_1/\log(d(x,y))}.$ So $|r^{\alpha(y)-\alpha(x)}|=r^{-|\alpha(y)-\alpha(x)|}\leq r^{C_1/\log(d(x,y))}.$ However, $\log(d(x,y))<\log(r)$ so $C_1/\log(d(x,y))>C_1/\log(r)$ which implies $r^{C_1/\log(d(x,y))}\leq r^{C_1/\log(r)}=e^{C_1}.$ Hence, since we may also make the same argument with $\beta$ in place of $\alpha,$ $r^{\alpha(x)}\leq e^{C_1}r^{\alpha(y)},$ $r^{\beta(x)}\leq e^{C_1}r^{\beta(y)}$, $r^{\alpha(y)}\leq e^{C_1}r^{\alpha(x)}$, and $r^{\beta(y)}\leq e^{C_1}r^{\beta(x)}$ for all $0<r<1/2$ and all $x\;y$ with $0<d(x,y)<r.$ By variable Ahlfors regularity there exists a $C_2>0$ so that $\frac{1}{C_2}r^{\alpha(x)}\leq \mu(B_r(x))\leq C_2r^{\alpha(x)}$ for all $0<r<\diam(X)/2$ and $x\in X.$ Let $r_1>0$ and smaller than both $\diam(X)/2$ and $1/2.$ Then if $0<r<r_1$ and $x,\;y\in X$ with $0<d(x,y)<r$ then 
\[r^{\beta(x)}\mu(B_r(x))\leq C_1C_2r^{\beta(y)}r^{\alpha(x)}\leq C_1^2C_2r^{\beta(y)+\alpha(x)}\leq C_1^2C_2^2r^{\beta(y)}\mu(B_r(y)).\] Hence also \[\frac{1}{\sqrt{r^{\beta(y)}\mu(B_r(y))}}\leq \frac{C_1C_2}{\sqrt{r^{\beta(x)}\mu(B_r(x))}}\] for $0<r<r_1$ and $0<d(x,y)<r.$ Let $f\in L^2(X,\mu).$ Let $C:=C_1C_2$. It follows that \begin{equation*}
\begin{split}
&P_rD_r^{-1/2}f(x)=\frac{1}{\mu(B_r(x))}\int_{B_r(x)}\frac{f(y)d\mu(y)}{\sqrt{r^{\beta(y)}\mu(B_r(y))}}\\
&\leq \frac{C}{\sqrt{r^{\beta(x)}\mu(B_r(x))}\mu(B_r(x))}\int_{B_r(x)}f(y)d\mu(y)=CD_r^{-1/2}P_rf.
\end{split}
\end{equation*} 
However, $D_r^{-1/2}$ commutes with $\frac{1}{r^\beta}$. Hence 
\[D_r^{-1/2}\mathscr{L}_rf=\frac{1}{r^\beta}\left(D_r^{-1/2}f-D_r^{-1/2}P_rf \right)\leq \frac{1}{r^{\beta}}\left(D_r^{-1/2}f-\frac{1}{C}P_rD_r^{-1/2}f\right).\]

 Moreover, by the Faber-Frahn inequality there exists an $r_0>0$ and $c_3>0$ so that if $0<r<r_0$ then $\mathscr{E}_r(f\chi_{B_j})\geq c_3\|f\chi_{B_j}\|^2_{\nu_r}$ for all $f\in L^2(X,\mu).$ Let $f_n\rightarrow \phi_j$ in $L^2(X,\mu)$ with $\mathscr{E}(f_n)\rightarrow \mathscr{E}(\phi_j).$ Since $f_n\rightarrow f$ in $L^2(X,\mu)$ there exists a subsequence $f_{n_k}$ such that $f_{n_k}\rightarrow \phi_j$ pointwise. By Egorov's Theorem, there exists a constant $N>0$, a measurable subset $A\subset B_{r_j/4}(x_j)$ with $\mu(A)>0$ and $f_{n_k}(x)\geq \frac{c_1}{2}$ for $x\in A.$ Let $s_{k_n}:=q_n$ for $n\in \mathbb{Z}^+.$ Then 
\[\mathscr{E}(\phi_j)=\limsup_{n\rightarrow \infty} \mathscr{E}_n(f_n)

Hence, since $\|\phi_{s_k,B_j}\|^2\overrightarrow{_{_{j\rightarrow \infty}}} \;\|\phi_j\|^2,$ it follows $\|\phi_j\|^2\geq c_1^2\mu(B_j):=c_2>0$.

However, the generator for this form is substantially more complicated than $\mathscr{L}_r$ and the connection to probabilistic jump processes is obscured. Since for each $r>0$, we have constructed an approximate Dirichlet form $\mathscr{L}_r$; a possible solution would be to use the approximate symmetric forms
\[\mathscr{E}^r_{\mbox{eq}}(f)=\langle f, \mathscr{L}_r f \rangle_{L^2(X,\nu_r)} = {\frac{1}{Z_r}}\int_X\int_{B_r(x)}|f(y)-f(x)|^2d\mu(y)d\mu(x),\]
where where $\nu_r$ here is the normalized equilibrium measure $d\nu_r=\frac{1}{Z_r}\mu(B_r(x))r^{\beta(x)}d\mu,$ with $Z_r=\int \mu(B_r(x))r^{\beta(x)}d\mu(x).$ However, the use of the equilibrium measure $\frac{1}{Z_r}d\nu_r$ is not desirable for $\alpha$ and $\beta$ non-constant since it leads to a loss of the local information contained in $x\mapsto \mu(B_r(x))r^{\beta(x)}.$  

Note the form $\langle f, \mathscr{L}_r f \rangle$ is not symmetric. Let $\nu_r$ be defined by $d\nu_r=\frac{\mu(B_r(x))r^{\beta(x)}}{Z_r},$ where $Z_r$ is the normalization factor $Z_r:=\int_X\mu(B_r(x))r^{\beta(x)}d\mu.$ Then the form $\mathscr{E}_r$ defined by \[\mathscr{E}_r(f):= \langle f, \mathscr{L}_r f\rangle_{\nu_r}=\frac{1}{Z_r}\int_X\int_{B_r(x)}|f(y)-f(x)|^2d\mu(y)d\mu(x),\] for $f\in L^2(X,\mu),$ is symmetric. 

Let $(\epsilon_n)_{n=1}^\infty$ be a sequence of positive numbers decreasing to $0.$ 

Since $X$ is compact, it is separable. Let $(B_{r_k}(x_k))_{k=1}^\infty$ be a basis for the metric topology on $X$ with $0<r_k<\diam(X)/2$ for all $k.$ Let $B_k=B_{r_k}[x_k]$ for each $k.$ Then if $x,y\in X$ with $x\neq y$, there exist $j, k$ with $x\in B_j,\; y\in B_k$ and $B_j\cap B_k=\emptyset.$ 

Then for $\epsilon>0$ let $\phi_{k,\epsilon}(x)$ be the mean exit time of $(\bm{x}(\epsilon)_t)_{t\geq 0}$ from $B_k$ starting at $x.$ The exit time assumption $\limsup_{r\rightarrow 0^+} \phi_{k,r}^+ \leq Cr_k^\beta$ implies that the functions $\phi_{k,\epsilon}$ are uniformly bounded above by the same constant for $\epsilon$ small enough. 

We then make the following assumption. We assume that for each $j$ there exists a subsequence $(t_n)_{j=1}^\infty$ of $(\epsilon_n)_{n=1}^\infty$ such that $\lim_{n\rightarrow \infty} \phi_{B_k,t_n}(x)$ exists for all $x\in X.$ Note that the convergence is also in $L^2(X,\mu)$ since the sequence is uniformly bounded.

For $k=1$ choose a subsequence $(t_{1,n})_{n=1}^\infty$ of $(s_n)_{n=1}^\infty$ such that $\phi_{1,t_{1,m}} \;
\overrightarrow{_{m\rightarrow \infty}}\; \phi_1$ in $L^2(X,\mu)$. Having chosen $t_1,...,t_n$ with $(t_{j+1,n})
_{n=1}^\infty$ a subsequence of $(t_{j,n})_{n=1}^\infty$ with $\phi_{j,t_{j,m}}\overrightarrow{_{m\rightarrow \infty}}\; \phi_j$ weakly in $L^2(X,\mu)$ for $j\leq n,$ by assumption, we 
may choose a subsequence $(t_{n,m})_{m=1}^\infty$ such that $\phi_{n+1,t_{n+1,m}} \;
\overrightarrow{_{m\rightarrow \infty}} \;\phi_{n+1}$ in $L^2(X,\mu)$. Hence for each $j\in\mathbb{Z}^+$ we have a function $\phi_j\in L^2(X,\mu)$ and decreasing sequence $(t_{j,n})_{n=1}^\infty$ of positive numbers decreasing to $0$ with the property that $(t_{j+1,n})_{n=1}^\infty$ is a subsequence of $(t_{j,n})_{n=1}^\infty$ and $\phi_{j,t_{j,m}}\; \overrightarrow{_{m\rightarrow \infty}}\; \phi_j$ in $L^2(X,\mu).$ By the diagonal argument, $\phi_{j,t_{n,n}}\; \overrightarrow{_{n\rightarrow \infty}} \;\phi_j$ in $L^2(X,\mu)$ for all $j\in \mathbb{Z}^+.$ Let $s_n=t_{n,n}$ for $n\in \mathbb{Z}^+.$

To proceed we need the definition of the $\Gamma$-lower limit given in \cite{maso}. If $F_h$ is a sequence of extended real valued functions from a topological space with neighborhoods of $x$ $\mathscr{N}(x)$ then $\Gamma$-$\limsup_{h\rightarrow \infty} F_h(x) := \sup_{U\in \mathscr{N}(x)} \limsup_{h\rightarrow \infty} \inf_{y\in U} F_h(y).$ 

We have the following characterization given in \cite{maso} (Proposition 8.1) in the case that $X$ is first countable. If $F=\Gamma$-$\limsup_{h\rightarrow \infty} F_h$ then $F$ is characterized by the following properties: (1)
$F(x)\leq \limsup _{h\rightarrow \infty} F_h(x_h)$ for every $x\in X$ and every sequence $x_h\rightarrow x,$ and (2) for every $x\in X$ there exists a sequence $x_h\rightarrow x$ with $F(x) = \limsup_{h\rightarrow \infty} F_h(x_h).$

Note for each $r>0,$ $f\in L^2(X\,u)$ if and only if $f\in L^2(X,\nu_r).$ Then each $r,$ $\mathscr{E}_r$ is a strongly Markovian form on $L^2(X,\mu).$ In particular, there exists a Dirichlet form $\mathscr{E}$ on $L^2(X,\mu)$ and a subsequence ${s_k}_j$ of the $s_k$ such that 
$\mathscr{E}=\Glim \mathscr{E}_{{s_k}_n}.$
Set $\mathscr{E}_n=\mathscr{E}_{{s_k}_n}.$ Then 
\[\mathscr{E}=\Glim \mathscr{E}_n.\]

Explicitly, using the notation $Z_r=\int v_r(x)r^{\beta(x)}d\mu(x)=\langle v_r r^\beta\rangle_\mu,$ where $v_r(x)=\mu(B_r(x)),$ for $f\in L^2(\mu)$ we have 
\[\mathscr{E}(f)=\Glim {\frac{1}{\langle v_{{s_k}_n} {{s_k}_n}^\beta \rangle_\mu}}\int_X\int_{B_{{s_k}_n}(x)}|f(y)-f(x)|^2d\mu(y)d\mu(x).\]

Let \[\mathscr{D}(\mathscr{E})=\{f\in L^2(X,\mu)\;|\;\mathscr{E}(f)<\infty\}.\]

\begin{proposition} \label{algebra} Each $\phi_j\in \mathscr{D}(\mathscr{E})$. In particular $\mathscr{D}(\mathscr{E})$ contains an algebra of functions separating points.
\end{proposition}
\begin{proof}
Fix $j\geq 1$ and for each $n$ let $\varphi_n=\phi_{j,{s_k}_n}.$ Then 
\[\varphi_n \convn \phi_j\;\;\mbox{in\;} L^2(X,\mu).\]

Then \begin{equation*}
\begin{split} 
\mathscr{E}_n(\varphi_n)&= \langle\phi_{j,s_{k_n}},\mathscr{L}_{s_{k_n}}\phi_{j,s_{k_n}}\rangle_{\nu_{s_{k_n}}}=\langle \phi_{j,s_{k_n}}, \chi_{B_j}\rangle_{\nu_{s_{k_n}}}\\ &=\int_B \phi_{j,s_{k_n}}d\nu_{s_{k_n}}\leq \phi^{+}_{j,s_{k_n}}\leq C\diam(X)^{\beta{(x_j)}}.
\end{split}
\end{equation*}
It then follows $\limsup_{n\rightarrow \infty} \mathscr{E}_n(\varphi_n)<\infty.$ However by the characterization of $\Gamma$- limit superior given above, this implies $\Gamma$-$\limsup_{n\rightarrow \infty} \mathscr{E}_n(\varphi_n)<\infty$. Therefore 
\[\phi_j\in\mathscr{D}(\mathscr{E})\;\;\mbox{for all\;}\;j\in\mathbb{Z}_+.\]

Let $f_n\rightarrow \phi_j$ in $L^2(X,\mu)$ with $\mathscr{E}_n(f_n)\rightarrow \mathscr{E}(\phi_j).$ By Proposition \ref{comparison}, there exists an $r_1>0$ and $c_1>0$ such that if $0<r<r_1$ and $x\in B_{r_j/4}[x_j]$ then $\phi_{r,B_j}(x)\geq c_1.$ In particular $\phi_j(x_j)>c_1.$ Since $\phi_j$ is supported in $B_j$ and the $B_j$ separate points, it follows that the $\phi_j$ separate points. Hence $\mathscr{D}(\mathscr{E})$ contains the algebra generated by the $\phi_j$ since domains of Dirichlet forms are closed under products (See Theorem 1.4.2 of \cite{fukushima}.).
\end{proof}

\subsection{Convergence of Normalized Laplacians}

An alternative approach is inspired by spectral graph theory. If $G=(V,E)$ is a finite graph with degree matrix $D$ and adjacency matrix $A,$ then the operator $D-A$ is called the graph Laplacian. It is positive definite and symmetric. The operator $I-D^{-1}A$ is called the probabilistic Laplacian. It is the generator of the uniform random walk on the graph defined by the transition operator $P=D^{-1}A.$ It is not, in general, symmetric with respect to the counting measure. However, it is symmetric with respect to the equilibrium measure $x\mapsto \deg(x).$ The operator $D^{-1/2}(D-A)D^{-1/2}=D^{1/2}(I-P)D^{-1/2}$ is known as the \textit{normalized Laplacian}. It has been extensively studied by many authors, notably by Fan Chung (e.g. \cite{chung}). It is symmetric with respect to the counting measure. Let $c$ be the counting measure on $V$ and $\deg$ the ``degree measure'' $x\mapsto \deg(x)$ on $V.$ Then the map $D^{-1/2}:L^2(V,c)\mapsto L^2(V,d)$ is unitary, and it follows that the normalized Laplacian is unitarily equivalent to the probabilistic Laplacian.

Moreover, 
$D_r^{-1/2}:L^2(X,\mu)\rightarrow L^2(X,\nu_r),$ is unitary. Indeed, if $f\in L^2(X,\mu),$ then \[\|D_r^{-1/2}f\|^2_{L^2(X,
\nu_r)} = \int (d\nu_r/d\mu)^{-1}(f(x))^2d\nu_r(x)=\int (f(x))^2d\mu(x)=\|f\|^2_{L^2(X,\mu)}.\]  Let $\widehat{\mathscr{L}}_r:=D_r^{-1/2}
L_rD_r^{-1/2}=D_r^{1/2}\mathscr{L}_rD_r^{-1/2}.$ Moreover, in $L^2(X,\mu),$ \[\langle f, \widehat{\mathscr{L}}_rf\rangle =\langle D^{-1/2}
f, L_rD^{-1/2}f\rangle = \langle {L}_rD^{-1/2}f, D_r^{-1/2}f\rangle=\langle \widehat{\mathscr{L}}_rf, f\rangle.\] It follows that $
\widehat{\mathscr{L}}_r$ is $\mu$-symmetric and is unitarily equivalent to $\mathscr{L}_r.$ In particular, $\widehat{\mathscr{L}}_r$ has the same 
spectrum as $\mathscr{L}_r.$ 

We then consider the approximate forms \[\langle f, \widehat{\mathscr{L}}_r f\rangle=\int_X\int_{B_r(x)}\left|\frac{f(y)}{\sqrt{\mu(B_r(y))r^{\beta(y)}}}-\frac{f(x)}{\sqrt{\mu(B_r(x))r^{\beta(x)}}}\right|^2d\mu(y)d\mu(x)\] for $f\in L^2(X,\mu).$

Then, for $r>0$, let $D_r$ be the multiplication operator by $d_r.$ Note that $D_r$ is 
bounded with bounded inverse since $x\mapsto r^{\beta(x)}\mu(B_r(x))$ is bounded both above and away from $0.$ 
For $B$ a ball of radius $R$ with $0<R<\diam(X)/2,$ we let $\mathscr{L}_r^B=\chi_B\mathscr{L}\chi_B$ and $\widehat{\mathscr{L}}_r^B=\chi_B\widehat{\mathscr{L}}_r\chi_B.$ Then we let $\mathscr{E}_r^B=\langle f, \mathscr{L}^B_r f\rangle_{\nu_r}=\mathscr{E}_r(\chi_Bf)$ and $\widehat{\mathscr{E}}_r(f)=\langle f, \mathcal{L}^B_r f\rangle = \mathcal{E}_r(\chi_B f)$ for $f\in L^2(X,\mu).$

Since $\widehat{\mathscr{L}}_r^B=\chi_BD_r^{1/2}\mathscr{L}_rD_r^{-1/2}\chi_B,$ we have \[(\widehat{\mathscr{L}}_r^B)^{-1}=\chi_BD_r^{1/2}(\mathscr{L}_r)^{-1}D_r^{-1/2}\chi_B.\] In particular, $\widehat{\mathscr{L}}$ has a Green's function \[\widehat{G_r^B}(x,y)=\sqrt{r^{\beta(x)}\mu(B_r(x))}G^B_r(x,y)\frac{1}{\sqrt{r^{\beta(y)}\mu(B_r(y))}}.\]

We then define the modified exit time functions \[\widehat{\phi}_{r,B}:=\int \widehat{G_r^B}(x,y)d\mu(y).\] We then modify the variable time regularity and time comparability conditions by everywhere replacing $\phi_{r,B}$ with $\widehat{\phi}_{r,B}.$ 

Fix an open ball $B=B_R(x_0)$ with $0<R<\diam(X)/2.$ By perturbing $R$ slightly if needed we may assume $\mu(B_R(x_0))=\mu(B_R[x_0])$.

For $B$ a ball of radius $R$ with $0<R<\diam(X)/2,$ we let $\mathscr{L}_r^B=\chi_B\mathscr{L}\chi_B$ and $\widehat{\mathscr{L}}_r^B=\chi_B\widehat{\mathscr{L}}_r\chi_B.$ Then we let $\mathscr{E}_r^B=\langle f, \mathscr{L}^B_r f\rangle_{\nu_r}=\mathscr{E}_r(\chi_Bf)$ and $\widehat{\mathscr{E}}_r(f)=\langle f, \mathcal{L}^B_r f\rangle = \mathcal{E}_r(\chi_B f)$ for $f\in L^2(X,\mu).$

Since $\widehat{\mathscr{L}}_r^B=\chi_BD_r^{1/2}\mathscr{L}_rD_r^{-1/2}\chi_B,$ we have \[(\widehat{\mathscr{L}}_r^B)^{-1}=\chi_BD_r^{1/2}(\mathscr{L}_r)^{-1}D_r^{-1/2}\chi_B.\] In particular, $\widehat{\mathscr{L}}$ has a Green's function \[\widehat{G_r^B}(x,y)=\sqrt{r^{\beta(x)}\mu(B_r(x))}G^B_r(x,y)\frac{1}{\sqrt{r^{\beta(y)}\mu(B_r(y))}}.\]

We then define the modified exit time functions \[\widehat{\phi}_{r,B}:=\int \widehat{G_r^B}(x,y)d\mu(y).\] We then modify the variable time regularity and time comparability conditions by everywhere replacing $\phi_{r,B}$ with $\widehat{\phi}_{r,B}.$ 

Fix an open ball $B=B_R(x_0)$ with $0<R<\diam(X)/2.$ By perturbing $R$ slightly if needed we may assume $\mu(B_R(x_0))=\mu(B_R[x_0])$.

We now show that $\mathcal{D}(\mathcal{E})$ is dense in $L^2(X,\mu).$ The idea for the following argument, using the Kaplansky density theorem, came from the answer to the posted question in \cite{overflow}. Let $\mathcal{H}=L^2(X,\mu, \mathbb{C})$ be the space of complex valued $\mu$-square integrable functions on $X.$ Let $\mathcal{B}$ be the space of all bounded, linear with scalar field $\mathbb{C}$ operators on $\mathcal{H}.$ A set $\mathscr{A}\subset \mathcal{B}$ is called a complex sub-algebra of $\mathcal{B}$ if it is a linear subspace of $\mathcal{B}$ that is closed under composition. Th algebra $\mathscr{A}$ is called self-adjoint if $A^*\in \mathscr{A}$ for all $A\in \mathscr{A},$ where the adjoint is defined via the sesquilinear inner product on $\mathcal{H}$ defined by $\langle f, g\rangle=\int \overline{f(x)}g(x)d\mu(x)$ for $f,g\in \mathcal{H}.$ The null space of the algebra $\mathscr{A}$ is the set $\{f\in \mathcal{H}\;|\;Af=0\mbox{\;for all\;}A\in \mathscr{A}\}.$ The commutant of $\mathscr{A},$ written $\mathscr{A}'$ is the set of all $B\in \mathscr{B}$ with $AB=BA$ for all $A\in\mathscr{A}.$ Let $(B_\alpha)$ be a net in $\mathscr{A}'$ with $B_\alpha\rightarrow B$ in the weak operator topology. Then for $f,g\in \mathcal{H},$ $A\in\mathscr{A},$ $\langle B_\alpha f, A^*g\rangle =\langle AB_\alpha f, g\rangle =\langle B_\alpha A f, g\rangle \rightarrow \langle BAf,g\rangle = \langle BF, A^*g\rangle = \langle ABf,g\rangle.$ It follows $AB=BA$. Hence $\mathscr{A}'$ is closed in the weak operator topology. Since any element of $\mathscr{A}$ commutes with all elements of $\mathscr{A}'$, we have $\mathscr{A}\subset \mathscr{A}''.$ A proof of the following proposition may be found in \cite{arveson}. 
\begin{proposition} Suppose $\mathscr{A}$ is a self-adjoint complex sub-algebra of $\mathcal{B}$ with null-space $\{0\}$. Then $\mathscr{A}$ is dense in $\mathscr{A}''$ in both the strong and weak operator topologies. \end{proposition} 

If $\mathcal{C}$ is a subset of $\mathscr{B},$ by the closed unit ball in $C\mathcal{C}$ we mean $\{B\;|\;\|B\|\leq 1\}.$ The following proposition is known as the Kaplansky density theorem. See Theorem 1.2.2 of \cite{arveson} for a proof. 
\begin{proposition} (Kaplansky density) Let $\mathscr{A}$ be a self-adjoint complex sub-algebra of $\mathcal{B}$ and $\mathscr{A}_s$ the closure in the strong operator topology. Then every self-adjoint element of $\mathscr{A}_s$ in the closed unit ball of $\mathscr{A}_s$ is the strong operator limit of a sequence of elements in the  closed unit ball of $\mathscr{A}.$
\end{proposition}

With notation from the previous theorem, let $\mathcal{A}$ be the real sub-algebra of $L^\infty(X,\mu)$ generated by the exit time functions $(\phi_k).$ Then $\mathcal{A}\subset \mathcal{D}(\mathcal{E}).$ Then let $\mathscr{A}=\mathcal{A}+i\mathcal{A}=\{f+ig\;|\;f,g\in \mathcal{A}\}.$ Then for $h\in \mathscr{A},$ we may consider the multiplication operator, which we also denote by $h$ in $\mathcal{B},$ defined by $hf(x)=h(x)f(x)$ for all $f\in \mathcal{H},\; x\in X.$ Then $\mathscr{A}$ is a complex sub-algebra of $\mathcal{B}.$ Note that it is self-adjoint. Hence $\mathscr{A}$ is dense in $\mathscr{A}''$ in both the norm and weak operator topologies. We show that $\mathscr{A}$ has trivial null-space. Suppose $f\in \mathcal{H}$ is non-zero. Let $P=\{x\in X\;|\;f(x)\neq 0\}.$ Then $\mu(P)>0.$ Since the sets $B_{r_k/4}(x_k)$ for $k\in \mathbb{Z}^+$ cover $X,$ there must be a $k$ with $\mu(P\cap B_{r_k/4}(x_k))>0.$ However, $\phi_k$ is strictly positive on $B_{r_k/4}(x_k)$. Hence $\phi_k(x)f(x)\neq 0$ for all $x$ in a set of positive measure. It follows $f$ is not in the null-space of $\mathscr{A}.$ Therefore the null-space of $\mathscr{A}$ is trivial. 

We compute the commutant of $\mathscr{A}.$ Clearly $L^\infty(X,\mu,\mathbb{C})\subset \mathscr{A}'.$ Suppose $A\in \mathscr{A}'.$ Let $x\in X.$ Choose $j$ with $x\in B_{r_j/4}(x_j).$ Then let $f(x)=\frac{1}{\phi_j(x)}{(A\phi_j)(x)}.$ Suppose also $x\in B_{r_k/4}(x_k).$ Since $\phi_j A= A\phi_j,$ $\phi_k(x)(A\phi_j)(x)=\phi_j(x)(A\phi_k)(x).$ Hence $f(x)=\frac{1}{\phi_k(x)}{A\phi_k(x)}.$ So $f$ is well defined. 

By variable Ahlfors regularity, there exists a constant $C_1>0$ such that for all $x,y\in X$ and $0<r<\diam(X)/2$, $\frac{\mu(B_r(x))}{\mu(B_{3r/8}(y))}\leq C_1\frac{r^{\alpha(x)}}{r^{\alpha(y)}}.$ Moreover, by the argument given in the proof of \ref{main}, there exists a constant $C_2>0$ such that $\frac{r^{\alpha(x)}}{r^{\alpha(y)}}\leq C_2$ for all $x\in X,$ $y\in X$ with $d(x,y)<r$ and $0<r<\min\{\diam(X)/2, 1/2\}.$  We recall the Lebesgue differentiation theorem for doubling spaces. See \cite{semmes} for a proof. 
\begin{proposition} (Lebesgue differentiation) Let $X$ be a metric space with doubling measure $\mu.$ Then if $g\in L^1(X,\mu),$ \[\limsup_{r\rightarrow 0^+} \frac{1}{\mu(B_r(x))}\int_{B_r(x)}|g(y)-g(x)|d\mu(y)=0\] for $\mu$-a.e. $x\in X.$
\end{proposition}

Recall, from the notation of Theorem \ref{main}, that $D$ is a countable dense set in $X,$ $\mathscr{C}$ is the collection of all closed balls centered at points in $D$ with positive rational radii, and $\mathscr{B}$ is the collection of the balls in $\mathscr{C}$ contained entirely in $B.$ Then $(B_k)_{k=1}^\infty=\mathscr{B}$ and $\phi_k$ is the mean exit time function on $B_k$ for each $k.$ Let $x\in X.$ Let $\delta=\dist(x,B^c)>0.$  Then for $0<r<\min\{\delta,1/2\},$ we choose a ball $C_r(x)\in \mathscr{B}$ as follows. Let $y\in D\cap B_{\frac{r}{4}}(x).$ Then $\dist(y,B^c)\geq \frac{3r}{4}.$  Then choose $q\in \mathbb{Q}\cap (\frac{3r}{8},\frac{3r}{4}).$ Then let $C_r(x)=B_q[y]\in \mathscr{B}.$ Then $x\in C_r(x)\subset B_r(x).$ Then for $g\in L^1(X,\mu),$
\[\frac{1}{\mu(C_r(x))}\int_{C_r(x)}|g(y)-g(x)|d\mu(y)\leq \frac{C}{\mu(B_r(x))}\int_{B_r(x)}|g(y)-g(x)|d\mu(y),\] where $C=C_1C_2$. Hence, for $g\in L^1(X,\mu),$ since 
\[\left| \frac{1}{\mu(C_r(x))}\int_{C_r(x)}g(y)d\mu(y)-g(x)\right| \leq \frac{1}{\mu(C_r(x))}\int_{C_r(x)}|g(y)-g(x)|d\mu(y),\] by the Lebesgue differentiation theorem, \[g(x)=\lim_{r\rightarrow 0^+}\frac{1}{\mu(C_r(x))}\int_{C_r(x)}g(y)d\mu(y)\; \mu\mbox{-a.e.}\] Then let $\phi_{r,x}$ be the mean exit time function on $C_r(x).$ Note $\phi_{r,x}$ is $0$ outside of $C_r(x).$ Then $\phi_{r,x}(x)f(x)=(A\phi_{r,x})(x)$ for all $0<r<\min\{\delta,1/2\}.$ Let $g\in L^2(B,\mu).$ functions in $L^2(X,\mu)$ are also in $L^1(X,\mu)$ since $\mu$ is a finite measure. Then for $0<r<\min\{\delta,1/2\},$ 
\begin{equation*}
\begin{split}
&\frac{1}{\mu(C_r(x))}\int_{C_r(x)}\overline{(A^*g)(y)}\phi_{r,x}(y)d\mu(y)=\frac{1}{\mu(C_r(x))}\langle A^*g, \phi_{r,x}\rangle\\ &= \frac{1}{\mu(C_r(x))}\langle g, A\phi_{r,x}\rangle=\frac{1}{\mu(C_r(x))}\int_{C_r(x)}(A\phi_{r,x})(y)\overline{g(y)}d\mu(y).
\end{split}
\end{equation*}
Hence by the Lebesgue differentiation theorem and the definition of $f,$ 
\[(A^*g)y
\begin{singlespace}  
	\setlength\bibitemsep{\baselineskip}  
	\printbibliography[title={References}]
\end{singlespace}